\documentclass[11pt]{memo-l}

\usepackage[utf8]{inputenc}
\usepackage{mathrsfs}
\usepackage{proof}
\usepackage{amsmath}
\usepackage{stmaryrd}
\usepackage{amssymb}
\usepackage{amsthm}
\usepackage{enumerate}
\usepackage{comment}
\usepackage{interval}
\usepackage{hyperref}
\usepackage{todonotes}
\usepackage{verbatim}
\usepackage{graphicx}
\usepackage{thmtools}
\usepackage{thm-restate}
\usepackage{bm}
\usepackage[all]{xy}\usepackage{tikz}
\usetikzlibrary{arrows,calc, patterns, positioning, shapes.geometric, decorations.pathreplacing,decorations.markings,decorations.pathmorphing,backgrounds,shapes,snakes}
\usepackage{graphicx}
\usepackage{rotating}
\usepackage{textcomp}
\usepackage{mathtools}
\usepackage{color}

\usepackage{color,soul}

\theoremstyle{plain}
\newtheorem{theorem}{Theorem}[chapter]
\newtheorem{corollary}[theorem]{Corollary}%[section]
\newtheorem{lemma}[theorem]{Lemma}%[section]
\newtheorem{claim}[theorem]{Claim}%[section]
\theoremstyle{definition}
\newtheorem{definition}[theorem]{Definition}%[section]
\newtheorem{remark}[theorem]{Remark}%[section]
\newtheorem{question}{Question}%[section]
\newtheorem*{question*}{Question}%[section]
\newtheorem*{notation*}{Notation}%[section]
\newtheorem*{fact*}{Fact}
\newtheorem{fact}[theorem]{Fact}
\newtheorem{statement}[theorem]{Statement}

\numberwithin{section}{chapter}
\numberwithin{equation}{chapter}
\numberwithin{figure}{chapter}

\usepackage{cleveref}

\providecommand\Cc{\mathcal{C}}
\providecommand\Dc{\mathcal{D}}

\providecommand\Fc{\mathcal{F}}

\providecommand\Gb{\mathbb{G}}

\providecommand\Jb{\mathbb{J}}

\providecommand\Kb{\mathbb{K}}

\providecommand\Mc{\mathcal{M}}
\providecommand\Nb{\mathbb{N}}

\providecommand\Pc{\mathcal{P}}
\providecommand\Fb{\mathbb{F}}
\providecommand\Pb{\mathbb{P}}
\providecommand\Qb{\mathbb{Q}}

\providecommand\Sc{\mathcal{S}}

\providecommand\Uc{\mathcal{U}}

\providecommand\Xb{\mathbb{X}}

\providecommand\dom{\mathrm{dom}}

\providecommand\PA{\mathrm{PA}}

\providecommand\RCA{\mathrm{RCA}}
\providecommand\WKL{\mathrm{WKL}}

\providecommand\ACA{\mathrm{ACA}}

\providecommand\ATR{\mathrm{ATR}}

\providecommand\RT[2]{\mathrm{RT}^{#1}_{#2}}

\providecommand\COH{\mathrm{COH}}

\providecommand\Baire{{\omega^{\omega}}}
\providecommand\baire{{\omega^{<\omega}}}
\providecommand\Cantor{{2^{\omega}}}
\providecommand\cantor{{2^{<\omega}}}

\providecommand\om{\omega}

\providecommand\Si[2]{\Sigma^{#1}_{#2}}

\providecommand\forces{{\Vdash}}

\providecommand\N{\mathbb{N}}

\def\qt#1{``#1''}%

\mathchardef\mhyphen="2D % Define a "math hyphen"
\providecommand{\meet}{\wedge}
\providecommand{\emptystr}{\epsilon}
\providecommand{\incomp}{\bot}
\providecommand{\MT}[2]{\mathrm{MTT}^{#1}_{#2}}
\providecommand{\PMT}[2]{\mathrm{PMTT}^{#1}_{#2}}
\providecommand{\sTT}{\mathrm{sTT}}
\providecommand{\TT}{\mathrm{TT}}

\providecommand{\embfont}[1]{\mathfrak{#1}}

\providecommand{\Subtree}[2]{\mathcal{S}_{#1}({#2})}

\providecommand{\SubtreeLeaves}[2]{\mathcal{S}^{l}_{#1}({#2})}

\providecommand{\meetclosure}[1]{{#1}^{\meet}}
\providecommand{\lvlclosure}[1]{#1^{\mathrm{lvl}}}
\providecommand{\land}{\ \&\ }
\DeclareMathOperator{\leaves}{leaves}
\DeclareMathOperator{\roots}{roots}
\providecommand{\uh}{\upharpoonright}
\providecommand{\fhl}{h_{\mathrm{HL}}}

\providecommand{\fwidg}{H}
\providecommand{\lgth}[1]{|#1|}

\providecommand{\RG}{\mathrm{RG}}
\providecommand{\JRG}{\mathrm{JRG}}
\providecommand{\DT}[2]{\mathrm{DT}^{#1}_{#2}}
\providecommand{\JDT}[2]{\mathrm{JDT}^{#1}_{#2}}

\providecommand{\CHMTT}{\mathrm{CHMTT}}
\providecommand{\vec}[1]{\bm{#1}}

\providecommand{\exprodtree}[2]{\bigcup_n {#1}_0(n) \times \dots \times {#1}_{#2-1}(n)}
\providecommand{\exleavesprodtree}[3]{ {#1}_0(#3-1) \times \dots \times {#1}_{#2-1}(#3-1)}

\definecolor{lightblue}{rgb}{1,.84,0}

\DeclareRobustCommand{\pelliot}[1]{\sethlcolor{pink}\hl{#1}}

\providecommand{\cmmnt}[1]{}

\providecommand{\Epn}{E_{\mathrm{pn}}}

\providecommand{\meetlevel}[2]{\llbracket #1 , #2 \rrbracket}
\providecommand{\ltlex}{<_{\mathrm{lex}}}
\providecommand{\lelex}{\leq_{\mathrm{lex}}}
\providecommand{\gtlex}{>_{\mathrm{lex}}}

\providecommand{\Tequiv}{\equiv_{\text{\upshape T}}}
\providecommand{\Tred}{\leq_{\text{\upshape T}}}
\providecommand{\nTred}{\nleq_{\text{\upshape T}}}

\providecommand{\cred}{\leq_{\text{\upshape c}}}

\providecommand{\str}[1]{#1}

\providecommand{\seq}[1]{( #1 )}

\providecommand{\NN}{\mathbb{N}}

\providecommand\Jo[1]{(\texttt{J#1})}

\providecommand{\JRel}{\mathsf{R}}

\providecommand{\setminus}{\smallsetminus}
\renewcommand{\setminus}{\smallsetminus}

\providecommand\case[2]{\medskip\noindent \textbf{Case #1:}~\textit{#2}}
\providecommand\construction{\medskip\noindent \textbf{Construction.}~}
\providecommand\verification{\medskip\noindent \textbf{Verification.}~}

\renewcommand{\langle}{(}
\providecommand{\rangle}{)}

\makeindex

\begin{document}

\frontmatter

\title{Milliken's tree theorem and its applications: a computability-theoretic perspective}

\author{Paul-Elliot Angl\`{e}s d'Auriac}
\address{Institut Camille Jordan\\
Universit\'{e} Claude Bernard Lyon 1\\
43 boulevard du 11 novembre 1918, F-69622 Villeurbanne Cedex, France}
%\curraddr{}
\email{peada@free.fr}
%\thanks{}

\author{Peter A. Cholak}
\address{Department of Mathematics, University of Notre Dame, 255 Hurley Building, Notre Dame, Indiana 46556, U.S.A.}
%\curraddr{}
\email{cholak@nd.edu}
%\thanks{}

\author{Damir D. Dzhafarov}
\address{Department of Mathematics, University of Connecticut, 341 Mansfield Road, Storrs, Connecticut 06269, U.S.A.}
%\curraddr{}
\email{damir@math.uconn.edu}
%\thanks{}

\author{Beno\^{i}t Monin}
\address{LACL, D\'{e}partement d'Informatique, Facult\'{e} des Sciences et Technologie, 61 avenue du G\'{e}n\'{e}ral de Gaulle, 94010 Cr\'{e}teil Cedex}
%\curraddr{}
\email{benoit.monin@computability.fr}
%\thanks{}

\author{Ludovic Patey}
\address{Institut Camille Jordan\\
Universit\'{e} Claude Bernard Lyon 1\\
43 boulevard du 11 novembre 1918, F-69622 Villeurbanne Cedex, France}
%\curraddr{}
\email{ludovic.patey@computability.fr}



\date{\today}

\begin{abstract}
	Milliken's tree theorem is a deep result in combinatorics that generalizes a vast number of other results in the subject, most notably Ramsey's theorem and its many variants and consequences. In this sense, Milliken's tree theorem is paradigmatic of structural Ramsey theory, which seeks to identify the common combinatorial and logical features of partition results in general. Its investigation in this area has consequently been extensive.
	
	Motivated by a question of Dobrinen, we initiate the study of Milliken's tree theorem from the point of view of computability theory. The goal is to understand how close it is to being algorithmically solvable, and how computationally complex are the constructions needed to prove it. This kind of examination enjoys a long and rich history, and continues to be a highly active endeavor. Applied to combinatorial principles, particularly Ramsey's theorem, it constitutes one of the most fruitful research programs in computability theory as a whole. The challenge to studying Milliken's tree theorem using this framework is its unusually intricate proof, and more specifically, the proof of the Halpern-La\"{u}chli theorem, which is a key ingredient.
	
	Our advance here stems from a careful analysis of the Halpern-La\"{u}chli theorem which shows that it can be carried out effectively (i.e., that it is computably true). We use this as the basis of a new inductive proof of Milliken's tree theorem that permits us to gauge its effectivity in turn. The key combinatorial tool we develop for the inductive step is a fast-growing computable function that can be used to obtain a finitary, or localized, version of Milliken's tree theorem. This enables us to build solutions to the full Milliken's tree theorem using effective forcing. The principal result of this is a full classification of the computable content of Milliken's tree theorem in terms of the jump hierarchy, stratified by the size of instance. As usual, this also translates into the parlance of reverse mathematics, yielding a complete understanding of the fragment of second-order arithmetic required to prove Milliken's tree theorem.
	
	We apply our analysis also to several well-known applications of Milliken's tree theorem, namely Devlin's theorem, a partition theorem for Rado graphs, and a generalized version of the so-called tree theorem of Chubb, Hirst, and McNicholl. These are all certain kinds of extensions of Ramsey's theorem for different structures, namely the rational numbers, the Rado graph, and perfect binary trees, respectively. We obtain a number of new results about how these principles relate to Milliken's tree theorem and to each other, in terms of both their computability-theoretic and combinatorial aspects. In particular, we establish new structural Ramsey-theoretic properties of the Rado graph theorem and the generalized Chubb-Hirst-McNicholl tree theorem using Zucker's notion of big Ramsey structure.
\end{abstract}

\maketitle

%\begin{acknowledgements}
%	The authors thanks
%\end{acknowledgements}

%    Dedication.  If the dedication is longer than a line or two,
%    remove the centering instructions and the line break.
%\cleardoublepage
%\thispagestyle{empty}
%\vspace*{13.5pc}
%\begin{center}
%  Dedication text (use \\[2pt] for line break if necessary)
%\end{center}
%\cleardoublepage

%    Change page number to 6 if a dedication is present.
\setcounter{page}{6}

\tableofcontents

%    Include unnumbered chapters (preface, acknowledgments, etc.) here.
%\include{}

\mainmatter
%    Include main chapters here.

\chapter*{Acknowledgments}

Cholak and Dzhafarov were partially supported by a Focused Research Group grant from the National Science Foundation of the United States, DMS-1854136 and DMS-1854355, respectively. Patey was partially supported by grant ANR ``ACTC'' \#ANR-19-CE48-0012-01.

The authors express their gratitude and appreciation to the Institut Henri Poincar\'{e} in Paris for kindly hosting them in February of 2020, during which time this project was conceived, and much of it completed.

They also thank Natasha Dobrinen, Lu Liu, and Andy Zucker, for their valuable insights, suggestions, and clarifications during the preparation of this manuscript.

Finally, the authors wish to thank the anonymous referees for their thorough reading, and for critical comments that helped improve this work.

\chapter{Introduction}
This monograph is part of the longstanding project of exploring connections between logic and combinatorics. Our focus is, more specifically, on studying the computable (or effective) content of combinatorial theorems. This has a long history, as we survey below. The interest stems from the realization that combinatorial notions tend to be computability-theoretically natural, and vice-versa. Traditionally, this has led to fine-grained analyses of different combinatorial constructions, often resulting in new, more computationally efficient proofs of various combinatorial results.

Over time, this work has made increasing use of powerful set-theoretic and combinatorial techniques, whose adaptation to the realm of computability theory has produced new insights into unsolved problems. Such will be the case for our investigation here of Milliken's tree theorem (named for its author, and originally proved in~\cite{Milliken1979RTforTrees}; cf. also~\cite{Milliken1881}). This is a deep result whose significance in Ramsey theory and related areas has made it the objective of much attention in combinatorics and set theory. This makes all the more surprising its near complete absence from the computability-theoretic literature. To our knowledge, the only published mentions are by Carlson and Simpson \cite[Section 3]{Carlson1984dual} and Chubb, Hirst, and McNicholl~\cite{Chubb2009Reverse}. The authors of the former paper introduce the so-called dual Ramsey's theorem, and give as a consequence a new proof of the Halpern-La\"{u}chli theorem, an important result for understanding Milliken's tree theorem that we investigate at length also here. The latter paper focuses on what is ultimately a kind of weak or degenerate form of Milliken's tree theorem, which has garnered a great deal of interest in its own right. See \Cref{sec:GenCHMTT}, where we give a full account of the theorem of Chubb, Hirst, and McNicholl and how it relates to Milliken's in the context of our work here. (We add that during the writing of this manuscript, we learned of a concurrent project of Chong, Li, Liu, and Yang in progress, whose focus is the Chubb, Hirst, and McNicholl theorem but which also obtains results about Milliken's tree theorem proper. Specifically, the authors obtain by independent means our \Cref{thm:milliken-aca} below.)

The problem of determining the computable content of Milliken's tree theorem was proposed by Dobrinen~\cite{Dobrinen-2018}. A related question, about the so-called Rado graph theorem, was asked also in Dorbinen, Laflamme, and Sauer~\cite[Question 6.3]{DLS-2016}. We give a complete analysis here, using the tools of computability theory and reverse mathematics. As we will show, Milliken's tree theorem turns out to be surprisingly rich and intricate in this respect, reflecting its centrality among other partition theorems, including Ramsey's theorem and its many variants.

\section{Milliken's tree theorem and Ramsey theory}

Ramsey theory is a vast area of combinatorics, broadly interested in results about when some sort of regularity is unavoidable when a large given structure is partitioned into a small number of pieces. (Here ``large'' is typically taken to mean a particular finite or infinite cardinality, and ``small'' is understood relative to this cardinality.) Canonical examples include, of course, the finite and infinite Ramsey's theorems, both due to F.~P.~Ramsey~\cite{Ramsey1929}, which we recall. Let $\NN$ denote the set of natural numbers, $\{0,1,2,\ldots\}$, and given a set $X \subseteq \NN$ and integer $n \geq 1$, let $[X]^n = \{ ( x_0,\ldots,x_{n-1} ) \in X^n: x_0 < \cdots < x_{n-1}\}$. We identify each $k \in \NN$ with the set of its predecessors, $\{0,1,\ldots,k-1\}$.
%Given a function $f$ defined on a set $X$ let $f \upharpoonright Y$ denote its restriction to $Y \subseteq X$.

\index{Ramsey's theorem!finite}
\begin{theorem}[Finite Ramsey's theorem]
	For all $n,k, \geq 1$ and $m_1$,$\ldots$, $m_{k-1} \in \NN$ there is a number $M \in \NN$ such that for every $f: [M]^n \to k$ there is an $i < k$ and a set $H \subseteq M$ of size $m_i$ such that $f(\vec{x}) = i$ for all $\vec{x} \in [H]^n$.
\end{theorem}

\index{Ramsey's theorem!infinite}
\begin{theorem}[Infinite Ramsey's theorem]
	For all $n,k \geq 1$ and every $f: [\NN]^n \to k$ there is an $i < k$ and an infinite set $H \subseteq \NN$ such that $f(\vec{x}) = i$ for all $\vec{x} \in [X]^n$.
\end{theorem}

\index{homogeneous set!for a coloring}
\noindent The sets $H$ above are called \emph{homogeneous sets} for the coloring $f$. There are also versions of Ramsey's theorem for colorings of uncountable sets, but we will restrict our attention here to the countable setting.

In broad strokes, Ramsey's theorem(s) can be seen as saying that in any configuration of integers, however complicated or random, some amount of order is necessary. Understanding this order, and how it arises, is naturally captivating, and its study has resulted in important advances across mathematics, from combinatorics to logic to number theory. These include, for example, the celebrated Szemer\'{e}di's theorem (cf.~\cite{Szemeredi1975,Szemeredi1975b}), the various proofs of which over the years, and the myriad mathematical ideas used in them, led to it be called the ``Rosetta stone'' of mathematics by Tao~\cite{Tao2007}. We will explore a number of other examples in this monograph. For a general introduction to Ramsey theory, we refer the reader to the book of Graham, Rothschild, and Spencer~\cite{GRS2013}. For more background on the kind of combinatorics most relevant to us here, we refer to Todorcevic~\cite{Todorcevic2010Ramsey}.

The main subject of the present monograph, Milliken's tree theorem, is a strong generalization of the infinite Ramsey's theorem. We state it here in a restricted form in order to be able to begin discussing it. The full statement requires more nuanced definitions that we delay until the next chapter. For now, we recall that $2^{<\omega}$ denotes the set of all finite binary strings, i.e., finite sequences of $0$s and $1$s. For $\sigma \in 2^{<\omega}$, we write $|\sigma|$ for the length of $\sigma$, i.e., the number of bits occurring in $\sigma$, and we let $2^n$ and $2^{<n}$ denote the sets of $\sigma \in 2^{<\omega}$ with $|\sigma| = n$ and $|\sigma| < n$, respectively. For $\sigma,\tau \in 2^{<\omega}$ we write $\sigma \preceq \tau$ to mean that $\sigma$ is an initial segment (not necessarily proper) of $\tau$, and $\sigma \prec \tau$ to mean $\sigma \preceq \tau$ and $\sigma \neq \tau$. We also write $\sigma \meet \tau$ for the longest common initial segment of $\sigma$ and $\tau$. The crucial notion in the statement of Milliken's tree theorem is the following: $S \subseteq 2^{<\omega}$ is a \emph{strong subtree of $2^{<\omega}$} if $S$ is closed under $\meet$, and $(S,\preceq)$ is isomorphic, as a structure, to either $(2^{<\omega},\preceq)$ or $(2^{< n},\preceq)$ for some $n$, via a map that preserves whether or not a pair of nodes has the same length. Thus, for instance, $\{ \str{01}, \str{0101}, \str{0110} \}$ is a strong subtree of $2^{<\omega}$, whereas $\{ \str{01}, \str{0100}, \str{0101} \}$ and $\{ \str{01}, \str{0101}, \str{011} \}$ are not, even though all three sets, under $\preceq$, are isomorphic to $(2^{< n},\preceq)$. (See \Cref{f:intro}.)

\begin{figure}[h]
\centering
\begin{tikzpicture}[font=\footnotesize]
	\tikzset{
		empty node/.style={circle,inner sep=0,fill=none},
		solid node/.style={circle,draw,inner sep=1.5,fill=black},
		hollow node/.style={circle,draw,inner sep=1.5,thick}
	}
	\tikzset{snake it/.style={decorate, decoration=snake, line cap=round}}
	\node(01)[solid node,label=below:{$01$}] at (0,0) {};
	\node(010)[hollow node,label=left:{$010$}] at (-0.7,1) {};
	\node(011)[hollow node,label=right:{$011$}] at (0.7,1) {};
	\node(0100)[hollow node,label=left:{$0100$}] at (-1,2) {};
	\node(0101)[solid node,label=above:{$0101$}] at (-0.4,2) {};
	\node(0110)[solid node,label=above:{$0110$}] at (0.4,2) {};
	\node(0111)[hollow node,label=right:{$0111$}] at (1,2) {};
	\draw[-,thick] (01) to (010);
	\draw[-,thick] (01) to (011);
	\draw[-,thick] (010) to (0100);
	\draw[-,thick] (010) to (0101);
	\draw[-,thick] (011) to (0110);
	\draw[-,thick] (011) to (0111);
	\node at (0,-0.75) {\normalsize $S_0$};
\end{tikzpicture}
\begin{tikzpicture}[font=\footnotesize]
	\tikzset{
		empty node/.style={circle,inner sep=0,fill=none},
		solid node/.style={circle,draw,inner sep=1.5,fill=black},
		hollow node/.style={circle,draw,inner sep=1.5,thick}
	}
	\tikzset{snake it/.style={decorate, decoration=snake, line cap=round}}
	\node(01)[solid node,label=below:{$01$}] at (0,0) {};
	\node(010)[hollow node,label=left:{$010$}] at (-0.7,1) {};
	\node(011)[hollow node,label=right:{$011$}] at (0.7,1) {};
	\node(0100)[solid node,label=left:{$0100$}] at (-1,2) {};
	\node(0101)[solid node,label=above:{$0101$}] at (-0.4,2) {};
	\node(0110)[hollow node,label=above:{$0110$}] at (0.4,2) {};
	\node(0111)[hollow node,label=right:{$0111$}] at (1,2) {};
	\draw[-,thick] (01) to (010);
	\draw[-,thick] (01) to (011);
	\draw[-,thick] (010) to (0100);
	\draw[-,thick] (010) to (0101);
	\draw[-,thick] (011) to (0110);
	\draw[-,thick] (011) to (0111);
	\node at (0,-0.75) {\normalsize  $S_1$};
\end{tikzpicture}
\begin{tikzpicture}[font=\footnotesize]
	\tikzset{
		empty node/.style={circle,inner sep=0,fill=none},
		solid node/.style={circle,draw,inner sep=1.5,fill=black},
		hollow node/.style={circle,draw,inner sep=1.5,thick}
	}
	\tikzset{snake it/.style={decorate, decoration=snake, line cap=round}}
	\node(01)[solid node,label=below:{$01$}] at (0,0) {};
	\node(010)[hollow node,label=left:{$010$}] at (-0.7,1) {};
	\node(011)[solid node,label=right:{$011$}] at (0.7,1) {};
	\node(0100)[hollow node,label=left:{$0100$}] at (-1,2) {};
	\node(0101)[solid node,label=above:{$0101$}] at (-0.4,2) {};
	\node(0110)[hollow node,label=above:{$0110$}] at (0.4,2) {};
	\node(0111)[hollow node,label=right:{$0111$}] at (1,2) {};
	\draw[-,thick] (01) to (010);
	\draw[-,thick] (01) to (011);
	\draw[-,thick] (010) to (0100);
	\draw[-,thick] (010) to (0101);
	\draw[-,thick] (011) to (0110);
	\draw[-,thick] (011) to (0111);
	\node at (0,-0.75) {\normalsize $S_2$};
\end{tikzpicture}
\caption{Three subsets, $S_0$, $S_1$, and $S_2$, of $2^{<\omega}$. Solid circles indicate strings in the set, hollow circles strings not in the set. Only $S_0$ is a strong subtree of $2^{<\omega}$.}\label{f:intro}
\end{figure}
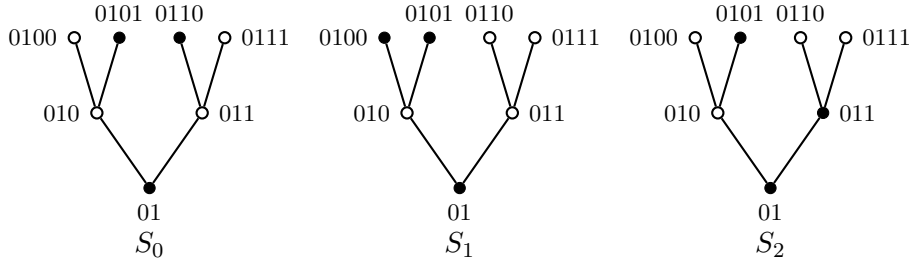

Given $T \subseteq 2^{<\omega}$, let $\mathcal{S}_\omega(T)$ denote the set of strong subtrees of $2^{<\omega}$ that are contained in $T$ and isomorphic to $(2^{<\omega},\preceq)$. For $n \geq 1$, let $\mathcal{S}_n(T)$ denote the set of strong subtrees of $2^{<\omega}$ that are contained in $T$ and isomorphic to $(2^{< n},\preceq)$.

\index{Milliken's tree theorem!for $2^{<\omega}$}
\index{tree theorem!Milliken's tree theorem}
\index{theorem!Milliken's tree theorem}
\begin{theorem}[Milliken's tree theorem for $2^{<\omega}$]
	For all $n,k \geq 1$ and all $f: \mathcal{S}_n(2^{<\omega}) \to k$ there exists $i < k$ and a $T \in \mathcal{S}_\omega(2^{<\omega})$ such that $f(S) = i$ for all $S \in \mathcal{S}_n(T)$.
\end{theorem}

To begin, note that the infinite Ramsey's theorem is a straightforward consequence of (even this version of) Milliken's tree theorem. Indeed, given a coloring $f: [\NN]^n \to k$, we define $g: \mathcal{S}_n(2^{<\omega}) \to k$ as follows. For each $S \in \mathcal{S}_n(2^{<\omega})$, let $\vec{x}_S = \{ |\sigma|: \sigma \in S\}$, which is a set of size $n$ and so can be viewed as an element of $[\NN]^n$. Let $g(S) = f(\vec{x}_S)$. Now if $T \in \mathcal{S}_\omega(2^{<\omega})$ is as given by Milliken's tree theorem for this $g$, then $H = \{|\sigma|: \sigma \in T\}$ is easily seen to be an infinite homogeneous set for $f$.

Indeed, it is well-known that Milliken's tree theorem implies a great many partition theorems, including a number that are significantly more difficult to prove than Ramsey's. We will look at several of these theorems in this manuscript, and show that their implications from Milliken's tree theorem can be made constructive in the sense of computability theory and reverse mathematics, which we discuss below. Much of this will rely on carefully identifying and examining features common between strong subtrees of $2^{<\omega}$ and the combinatorial structures underlying these other principles, using a combination of ideas that have previously been successfully employed in structural Ramsey theory, along with techniques newly developed here.

We refer the reader to Todorcevic~\cite[Chapter 6]{Todorcevic2010Ramsey} for an in-depth discussion of Milliken's tree theorem, and a careful development of a proof. As with Ramsey's theorem, the proof has an inductive form based on the exponent, $n$, of the colorings being considered. Thus, we prove it first for finite colorings of $\mathcal{S}_1(2^{<\omega})$, and then assuming it holds for finite colorings of $\mathcal{S}_n(2^{<\omega})$, we prove it for finite colorings of $\mathcal{S}_{n+1}(2^{<\omega})$. The base case, $n = 1$, is actually not difficult to prove directly (though it is less trivial than the $n = 1$ case of Ramsey's theorem, i.e., the infinitary pigeonhole principle). However, unlike in standard proofs of Ramsey's theorem, where the inductive step uses just the $n = 1$ case to increase the exponent, in the case of Milliken's tree theorem a stronger result is needed. This is the so-called Halpern-La\"{u}chli theorem, due originally to Halpern and La\"{u}chli \cite{HalperbLauchli1966}, and independently Laver (unpublished) and Pincus \cite{Pincus1974} (see \cite{Pincus1981} for more on the history).

Given $d \geq 1$ and $T_0,\ldots,T_{d-1} \subseteq 2^{<\omega}$, let $\mathcal{S}_\alpha(T_0,\ldots,T_{d-1})$ for $\alpha \in \NN \cup \{\omega\}$ be the collection of all tuples $(S_0,\ldots,S_{d-1})$ such that for each $i < d$ we have $S_i \in \mathcal{S}_\alpha(T_i)$, and for all $i,j < d$ and all $\sigma \in S_i$ and $\tau \in S_j$, we have that $\sigma$ has the same number of initial segments in $S_i$ as $\tau$ does in $S_j$ if and only if $|\sigma| = |\tau|$.

\index{Halpern-La\"{u}chli theorem!for $2^{<\omega}$}
\index{theorem!Halpern-La\"{u}chli theorem}
\begin{theorem}[Halpern-La\"{u}chli theorem for $2^{<\omega}$]
	For all $d,k \geq 1$ and all $f: \bigcup_{n \in \NN} (2^n)^d \to k$ there exists $i < k$ and
	\[
		(T_0,\ldots,T_{d-1}) \in \mathcal{S}_\omega(2^{<\omega},\ldots,2^{<\omega})
	\]
	such that $f(\vec{\sigma})=i$ for all $\vec{\sigma} = (\sigma_0,\ldots,\sigma_{d-1}) \in T_0 \times \cdots \times T_{d-1}$ with $|\sigma_0| = \cdots = |\sigma_{d-1}|$.
\end{theorem}

\noindent Prima facie, this theorem appears as a kind of parallelized version of Milliken's tree theorem for colorings of $\mathcal{S}_1(2^{<\omega})$, and one may expect it to be not much more complicated to prove. In fact, this is misleading, and the Halpern-La\"{u}chli theorem largely encompasses the entire combinatorial core of (the full) Milliken's tree theorem. We will analyze the Halpern-La\"{u}chli theorem in detail in this monograph, and use it in a careful way to give a more effective proof of Milliken's tree theorem.

\section{Computable combinatorics}

\index{$\leq_T$}
\index{$\equiv_T$}
\index{computable}
\index{Turing!reduction}
\index{halting problem}
\index{Turing!equivalence}
The principal theme of modern computability theory is \emph{relative} computability: a set $X \subseteq \NN$ is \emph{computable from} (or \emph{Turing reducible to}) a set $Y$, written $X \Tred Y$, if there is an algorithm to decide which numbers belong to $X$ using information about which numbers belong to $Y$. If $X \Tred Y$ and $Y \Tred X$ we write $X \Tequiv Y$. This notion, along with a precise formalization of the concept of an algorithm, was a seminal achievement of Turing in the 1930s. Sets can be classified in numerous ways, such as in terms of their structural properties or by their syntactic descriptions.

As a rule, all such properties can be \emph{relativized}, leading to increasingly larger classes of sets. For example, the \emph{halting problem relative to $X$}, denoted $X'$ and also called the \emph{(Turing) jump of $X$}, refers to the set of $e \in \NN$ such that the $e$th algorithm in some fixed listing, with access to information about $X$, halts on input $e$. For every $X$ we have that $X \Tred X'$ but $X' \nTred X$, which yields in particular a canonical example of a non-computable set. The complexity of a set of natural numbers in computability theory (or by extension, of any object that can be naturally represented or encoded by such a set) is a measure of ``how far'' it is from being computable, according to various hierarchies of classes of sets obtained in this fashion. For general background in computability theory, we refer the reader to Soare~\cite{Soare2016Turing} and to Downey and Hirschfeldt~\cite{Downey2010Algorithmic}.

Computability theory lends itself to analyzing a vast collection of problems that are sometimes called \emph{instance-solution problems}\index{problem}. This refers to theorems having the form
\begin{equation}\label{eqn:Pi12}
	\forall A~[\mathrm{P}(A) \implies \exists B~\mathrm{Q}(A,B)],
\end{equation}
where $\mathrm{P}$ and $\mathrm{Q}$ are some sort of properties of $A$, and of $A$ and $B$, respectively. One can regard such a theorem as the problem, ``Given an $A$ such that $\mathrm{P}(A)$ holds, find a $B$ such that $\mathrm{Q}(A,B)$ holds''. In this context, we call such $A$ the \emph{instances} of the problem (or theorem), and all such $B$ the \emph{solutions} to $A$. This is a natural way of thinking about theorems of this shape. For example, the instances of Ramsey's theorem are all finite colorings of $[\omega]^n$ for some $n$, and the solutions to any such coloring are its infinite homogeneous sets.

One way to gauge the complexity of an instance-solution problem is by studying the relationship between the complexity of instances and solutions, when these can be presented as subsets of $\NN$, as will be the case in all the examples we consider in this manuscript. From this perspective, a problem that is \emph{computably true}\index{computably true}, i.e., one each of whose instances has at least one solution computable in that instance, is trivial from the algorithmic standpoint. By contrast, a problem that has an instance all of whose solutions compute the jump of that instance, is strictly harder, being, in a certain sense, at least as difficult as ``solving the halting problem''. In general, the further apart the instances and solutions are in this sense, the more algorithmically complex it is. We can thus directly compare different problems in terms of their complexity, yielding a notion of algorithmic or computability-theoretic strength. For a thorough introduction to this kind of analysis, which is generally called \emph{computable mathematics}, see the book of Hirschfeldt~\cite{Hirschfeldt2015Slicing}.

A complementary approach is provided by the foundational program of \emph{reverse mathematics}\index{reverse mathematics}, developed by Friedman and Simpson in the late 1970s. The setting here is second-order arithmetic, a formal system strong enough to express countable analogues of most results of classical mathematics. Its axioms include the usual ordered semi-ring axioms for the natural numbers, together with \emph{comprehension axioms} asserting that the set of all numbers $x$ satisfying a given formula (property) exists. By restricting to only certain kinds of formulas we get various \emph{subsystems} of second-order arithmetic, the most basic of which is called $\RCA_0$ and roughly corresponds to computable mathematics. The traditional approach in the subject has been to compare a given theorem with several benchmark subsystems ($\WKL_0$, $\ACA_0$, $\ATR_0$) extending $\RCA_0$, corresponding to increasing levels of non-constructibility. Isolating the weakest such system that the theorem can be proved in, and the strongest that can in turn be proved from it over the base system $\RCA_0$, yields a measure of its proof-theoretic strength. There is a fruitful and well-understood interplay between reverse mathematics and computability theory, with ideas and results from one often leading to results in the other (see Shore~\cite{Shore1975Splitting}). This has been made even more pronounced in recent years by the introduction of various Weihrauch-style reducibilities to the subject, which have come to be viewed largely as an extension and refinement of the traditional program of reverse mathematics. Computable reducibility\index{computable!reducibility}, in particular, which is a non-uniform analogue of Weihrauch reducibility originally introduced in~\cite{Dzhafarov2014Cohesive}, will figure in a number of our results here.

The standard reference on reverse mathematics is Simpson~\cite{Simpson2009Subsystems}. Weih\-rauch reducibility was introduced by Weihrauch \cite{Weihrauch1992} in the 1990s, and has since been widely deployed in computable analysis and other fields; for a recent survey, see Brattka, Gherardi, and Pauly~\cite{BrattkaSurvey}.

Of course, instance-solution problems are ubiquitous across mathematics, but problems from combinatorics have figured especially prominently in the above frameworks for many decades. The classification and differentiation of combinatorial theorems according to their computability-theoretic and proof-theoretic strength is nowadays called \emph{computable combinatorics}. Perhaps the earliest result here is the following one from the late 1960s, stating that Ramsey's theorem for pairs is not computably true.

\begin{theorem}[Specker~\cite{Specker-1971}]
	There is a computable $f: [\omega]^2 \to 2$ with no computable infinite homogeneous set.
\end{theorem}
 
\noindent (In the parlance of reverse mathematics, this shows that Ramsey's theorem for colorings of pairs is not provable in the base theory, $\RCA_0$.) This result was greatly extended in the seminal 1972 paper of Jockusch~\cite{Jockusch1972Ramseys}, which set off an industry of research on Ramsey's theorem in computability theory that is still highly active today.
 
The computability-theoretic perspective offers insights that are not readily discernible in combinatorics alone. In the case of Ramsey's theorem, a well-known example is provided by the following pair of results.

\begin{theorem}[Jockusch~\cite{Jockusch1972Ramseys}, Theorem 5.7]
	For each $n \geq 3$, there is a computable $f: [\omega]^n \to 2$ each of whose infinite homogeneous sets computes $\emptyset^{(n-2)}$ (and in particular $\emptyset'$).
\end{theorem}
 
\begin{theorem}[Seetapun; see~\cite{Seetapun1995strength}]\label{thm:Seetapun}
	Every computable $f: [\omega]^2 \to 2$ has an infinite homogeneous set that does not compute $\emptyset'$.	
\end{theorem}
 
\noindent Thus, there is a direct computational distinction between Ramsey's theorem for colorings of pairs and Ramsey's theorem for colorings triples and larger tuples. (Formalizing these results in $\RCA_0$ yields that Ramsey's theorem for colorings of triples implies the system $\ACA_0$ over $\RCA_0$, whereas Ramsey's theorem for colorings of pairs does not.)

Such threshold phenomenon, where an increase in a parameter changes a theorem from not being able to encode specific non-computable information to being able to do so, are observed quite widely. For example, as was shown by Dzhafarov and Patey~\cite{Dzhafarov2017Coloring}, this is the case for the aforementioned theorem introduced by Chubb, Hirst, and McNicholl \cite{Chubb2009Reverse}. And more recently, Chong et al.~\cite{Chong2019Strengthc} obtained similar results for a theorem of Erd\H{o}s and Rado about colorings of pairs of rationals. We will likewise establish threshold phenomena for Milliken's tree theorem and the various consequences of it we consider.

Another computability-theoretic feature that will feature prominently in our work is \emph{cone avoidance}\index{cone avoidance}. In the subject, a \emph{cone} refers to a set of subsets of $\omega$ closed upward under $\Tred$. As a case in point, the set of all sets $X$ that compute $\emptyset'$ is a cone, and Seetapun's theorem (\Cref{thm:Seetapun} above) can be seen as saying that every computable instance of Ramsey's theorem for colorings of pairs has a solution that lies outside (or avoids) this cone. The emphasis here is on the restriction to \emph{computable instances}, however; indeed, it is easy to see that there is a (necessarily non-computable) $f: [\omega]^2 \to 2$, each of whose infinite homogeneous sets \emph{does} compute $\emptyset'$. By contrast, some instance-solution problems enjoy a stronger property called \emph{strong cone avoidance}\index{strong cone avoidance}\index{cone avoidance!strong}, whereby \emph{every} instance (computable or not) has at least one solution that avoids the cone of sets that compute $\emptyset'$. This is the case, for example, for the infinitary pigeonhole principle, as shown by Dzhafarov and Jockusch~\cite[Lemma 3.2]{Dzhafarov2009Ramseys}. We shall investigate both cone avoidance and strong cone avoidance for versions of Milliken's tree theorem, and in particular, for the Halpern-La\"{u}chli theorem. It is worth noting, too, that while every computably true problem obviously possesses cone avoidance, not every such problem satisfies strong cone avoidance. (For example, consider the identity problem, whose instances are all $X \subseteq \omega$, and the only solution of $X$ is $X$ itself.)
 
\section{Plan of the manuscript}

The manuscript is organized as follows. In \Cref{sec:defns}, we give further background and definitions to allow us to state the full versions of Milliken's tree theorem and the Halpern-La\"{u}chli theorem. In \Cref{sect:hl-theorem}, we proceed to the computability-theoretic analysis of the Halpern-Lauchli theorem, as a bootstrap to understanding the computational content of Milliken's tree theorem. In particular, we prove that the Halpern-Lauchli theorem is computably true (\Cref{thm:halpern-lauchli-computably-true}) and admits strong cone avoidance (\Cref{thm:hl-strong-cone-avoidance}). Then, in \Cref{sect:milliken-theorem}, we analyse a product version of Milliken's tree theorem. We prove that the statement is equivalent to $\ACA_0$ for strong subtrees of height at least $3$ (\Cref{thm:milliken-aca}), and that its restriction to colorings of strong subtrees of height 2 admits cone avoidance (\Cref{thm:cone-avoidance-MTT2}). Lastly, we prove that a weakening to the product version of Milliken's tree theorem for height 3, for which the solutions have now at most 2 colors instead of 1, admits cone avoidance (\Cref{thm:pmtt3k2-cone-avoidance}). We then study three applications of Milliken's tree theorem for pairs, namely: Devlin's theorem concerning colorings of tuples of rationals (\Cref{sec:devlin}); a theorem about colorings of finite subgraphs of the Rado graph (\Cref{sec:radomain}); and a generalization of the combinatorial theorem of Chubb, Hirst, and McNicholl discussed above (\Cref{sec:GenCHMTT}). Finally, in \Cref{sect:open-questions}, we state some questions that our investigation leaves open.

\chapter{Definitions}\label{sec:defns}
The aim of this chapter is to review key concepts to make the rest of this monograph more easily accessible to computability theorists, set theorists, and combinatorialists. Our terminology and notation will for the most part be standard, following, e.g., \cite{Downey2010Algorithmic} and \cite{Todorcevic2010Ramsey}. Where there is less uniformity in the literature, we highlight our particular usage in this chapter and, as the need arises, in the sequel. In \Cref{sec:bkg_strings} we set out our notation for finite strings, operations on them, and spaces of subsets of $\NN$, which are largely common across these fields. \Cref{sec:bkg_comp,sec:bkg_rm} provide an overview of some technical notions from computability theory and reverse mathematics. In \Cref{sec:bkg_trees,sec:bkg_forests}, we review combinatorial definitions relevant to stating Milliken's tree theorem and some of its corollaries, which we then present in \Cref{sec:bkg_stmts}. Finally, in \Cref{sec:bigRamsey}, we review some terminology from structural Ramsey theory that helps give a common framing for these principles.

We begin with some basics. We use $\sqcup$\index{$\sqcup$} to denote disjoint union. For every set $X$, we denote by $\Pc(X)$\index{$\Pc(X)$} the power set of $X$. And given a function $f$ on $X$, we let $f \upharpoonright Y$\index{$\upharpoonright$} denote the restriction of $f$ to $Y \subseteq X$.

Throughout, we use $(\, \cdots )$ to denote (ordered) tuples\index{ordered tuples} of objects, and given a function $f$ defined on a tuple $(a_0,\ldots,a_n)$ we write $f(a_0,\ldots,a_n)$ in place of $f((a_0,\ldots,a_n))$. In the computability-theoretic setting, we do not make a notational distinction for coded tuples (of numbers, subsets of $\NN$, or combinations thereof). Thus, we also let $(\, \cdots )$ denote a fixed computable bijection from finite ordered tuples of natural numbers to $\NN$, e.g., as in \cite[p.~xxxii]{Soare2016Turing}. For $X_0,\ldots,X_{n-1} \subseteq \NN$ we will sometimes use $(X_0,\ldots,X_{n-1})$ as an alternative notation for the join, $X_0 \oplus \cdots \oplus X_{n-1} = \{(x,i): x \in X_i, i < n\} \subseteq \NN$\index{$\oplus$}. In the case that some $X_i$ is a singleton, say containing $x$, we will write simply $(X_0,\ldots,x,\ldots,X_{n-1})$ in place of $(X_0,\ldots,\{x\},\ldots,X_{n-1})$.

Given a countable collection of sets $\{X_0,X_1,\ldots\}$ indexed by the natural numbers, we write $\bigcup_n X_n$ for $\bigcup_{n \in \NN} X_n$.

\section{Strings and subsets of $\NN$}\label{sec:bkg_strings}

The following definition is included for completeness.

\begin{definition}
\
	\begin{enumerate}
		\item $\baire$\index{$\baire$} denotes the set of all finite strings of natural numbers, i.e., functions $\sigma: n \to \omega$ for some $n \in \NN$.
		\item $\cantor$\index{$\cantor$} denotes the subset of $\baire$ of binary ($\{0,1\}$-valued) strings.
		\item The \emph{length} of $\sigma \in \baire$ is the cardinality of its domain, and is denoted by $|\sigma|$\index{$\lgth{\sigma}$}\index{$\lgth{\sigma}$}.
		\item The unique string of length $0$ is denoted by $\emptystr$\index{$\emptystr$}.
		\item For $n \in \NN$, $\omega^n$\index{$\omega^n$} and $\omega^{<n}$\index{$\omega^{<n}$} denote the sets of $\sigma \in \baire$ with $|\sigma| = n$ and $|\sigma| < n$, respectively.
		\item For $n \in \NN$, $2^n$\index{$2^n$} and $2^{<n}$\index{$2^{<n}$} denote the sets of $\sigma \in \cantor$ with $|\sigma| = n$ and $|\sigma| < n$, respectively.
	\end{enumerate}
\end{definition}

As is customary, we will alternate between the function and sequence point of view for elements of $\omega^{<\omega}$. For $\sigma \in \omega^{<\omega}$ and $i < |\sigma|$ we will thus speak of $\sigma(i)$ and the $(i+1)$st element of $\sigma$ (or $(i+1)$st \emph{bit}, if $\sigma \in \cantor$) interchangeably, or as convenient. We will sometimes specify $\sigma$ explicitly as $\seq{\sigma(0)\sigma(1)\cdots \sigma(|\sigma|-1)}$.

\begin{definition}
	Fix $\sigma,\tau \in \omega^{<\omega}$.
	\begin{enumerate}
		\item $\sigma$\index{initial segment!string} is an \emph{initial segment} of $\tau$, and $\tau$ is an \emph{extension}\index{string!extension}\index{extension!string} of $\sigma$, written $\sigma \preceq \tau$\index{$\preceq$}, if $\sigma = \tau \upharpoonright n$ for some $n \leq |\tau|$.
		\item $\sigma$ is a \emph{proper initial segment}\index{initial segment!proper}\index{proper initial segment} of $\sigma$, and $\tau$ is a \emph{proper extension}\index{proper extension}\index{extension!proper} of $\sigma$, written $\sigma \prec \tau$\index{$\prec$}, if $\sigma = \tau \upharpoonright n$ for some $n < |\tau|$, i.e., if $\sigma \preceq \tau$ and $\sigma \neq \tau$.
		\item $\sigma$ and $\tau$ are \emph{incompatible}\index{string!incompatible}\index{incompatible!string}, written $\sigma\incomp\tau$, if $\sigma \npreceq \tau$ and $\tau \npreceq \sigma$.
		\item The \emph{meet}\index{meet!string}\index{string!meet} of $\sigma$ and $\tau$, denoted by $\sigma \meet \tau$, is the longest common initial segment of $\sigma$ and $\tau$, i.e., $\sigma \meet \tau = \sigma \upharpoonright n$ for the longest $n$ such that $\sigma \upharpoonright n = \tau \upharpoonright n$.
		\item The \emph{concatenation}\index{concatenation}\index{string!concatenation} of $\sigma$ by $\tau$ is the string $\sigma\tau: |\sigma| + |\tau| \to \omega$ with $\sigma\tau(i) = \sigma(i)$ for all $i < |\sigma|$ and $\sigma\tau(i) = \tau(i - |\sigma|)$ for all $|\sigma| \leq i < |\sigma| + |\tau|$.
	\end{enumerate}
\end{definition}

So, for the sake of completeness, notice that if $\sigma \meet \tau = \sigma$ then $\sigma \preceq \tau$. Observe too that $\emptystr$ is an initial segment of every $\sigma$, and $\emptystr\sigma = \sigma\emptystr = \sigma$. Finally, if $\sigma,\tau \in \cantor$ then so is $\sigma\tau$.

\begin{definition}
	\
	\begin{enumerate}
		\item $\Baire$\index{$\Baire$} denotes the set of all functions $X: \NN \to \NN$, and $\Cantor$\index{$\Cantor$} the set of all $\{0,1\}$-valued such functions.
		\item $\sigma \in \baire$ is an \emph{initial segment}\index{initial segment!sequence} of $X \in \Baire$, and $X$ is an \emph{extension} of $\sigma$, written $\sigma \prec X$, if $\sigma(i) = X(i)$ for all $i < |\sigma|$.
	\end{enumerate}
\end{definition}

\noindent When convenient, we identify sets with their characteristic functions, which gives us the usual equivalence between elements of $\Cantor$ and elements of $\Pc(\NN)$. For this reason, we use $X \uh \ell$ for $\ell \in \NN$, which denotes the restriction of the characteristic function of $X$ to $\ell$, also as shorthand for $\{x \in X: x < \ell\}$. 

The sets $\Baire$ and $\Cantor$ each have natural topologies defined on them, respectively generated by basic open sets of the form
\index{$[\sigma]$}
\[
	[\sigma] = \{X \in \Baire: \sigma \prec X\}.
\]
for $\sigma \in \baire$, and
\[
	[\sigma] = \{X \in \Cantor: \sigma \prec X\}.
\]
for $\sigma \in \cantor$. This turns $\Baire$ into a Baire space and $\Cantor$ into a Cantor space. For our purposes here, the main relevant topological consideration will be that $\Cantor$ is compact.

\section{Computability and reverse mathematics}\label{sec:bkg_comp}

Everywhere, we adopt the Church-Turing thesis, and therefore forego any specifics of our model of computation. We take as fixed some listing $\Phi_0,\Phi_1,\ldots$\index{$\Phi_e$} of all partial computable functions such that from each $e$ we can computably determine the program of $\Phi_e$, and conversely, from each program we can computably find an $e$ such that $\Phi_e$ executes this program. Nominally, we think of $e$ as being a code for the sequence of steps in the program under a G\"{o}del coding (see, e.g., \cite{Soare2016Turing}, Definitions 1.5.1 and 1.7.2).

Recall that a set $W \subseteq \NN$ is \emph{computably enumerable}\index{computably enumerable} (\emph{c.e.})\ if it is the domain of some partial computable function, i.e., the set of inputs on which a given Turing program halts in finite time. We denote the domain of $\Phi_e$ by $W_e$.

\begin{definition}
	A \emph{Turing functional}\index{Turing!functional} is a c.e.\ set $\Gamma$ of pairs $\seq{\sigma,\tau} \in \cantor \times \cantor$ (coded as numbers) such that if $\seq{\sigma,\tau}$ and $\seq{\sigma',\tau'}$ belong to $\Phi$ and $\sigma \preceq \sigma'$ then $\tau \preceq \tau'$. In this case, for every set $X \subseteq \NN$, we also define the following.
	\begin{enumerate}
		\item $\Gamma^X = \bigcup \{ \tau \in \cantor: (\exists \sigma \prec X)[(\sigma,\tau) \in \Gamma]\}$.
		\item We write $\Gamma^X(x) = y$ or $\Gamma^X(x) \downarrow = y$ if $\tau(x) = y$ for some (and hence all) $(\sigma,\tau) \in X$ with $\sigma \prec X$ and $|\tau| > x$; we write $\Gamma^X(x) \downarrow$ if $\Gamma^X(x) = y$ for some $y$, and otherwise we write $\Gamma^X(x) \uparrow$.
		\item $\Gamma^X$ is \emph{total}\index{total functional} if $\Gamma^X(x) \downarrow$ for all $x \in \NN$.
	\end{enumerate}
\end{definition}

\noindent Note that if $\Gamma^X$ is total then it is, in fact, equal to an element of $\Cantor$. In particular, if $\Gamma^X$ is total for all $X \in \Cantor$ then $\Gamma$ is a continuous map $\Cantor \to \Cantor$. If $\Gamma = W_e$, then for all $X$ we also denote $\Gamma^X$ by $\Phi^X_e$ when convenient.

For simplicity, we abuse notation and write $\Phi_e$ instead of $\Phi_e^{\emptyset}$. (Formally, this is only incorrect up to a fixed computable permutation of $\NN$. Indeed, given any computable set $X$ there is a computable bijection $f: \NN \to \NN$ such that $\Phi^X_e = \Phi_{f(e)}$ for all $e \in \NN$.) This highlights the fact that the main role of Turing functionals is to facilitate relativization of computability-theoretic notions to arbitrary subsets of $\NN$. For example, a set $Y \subseteq \NN$ is computable \emph{relative to $X$} (or \emph{from $X$}, or is \emph{$X$-computable}), if $Y = \Gamma^X$ for some Turing functional $\Gamma$, in which case we write $Y \Tred X$; $Y$ is computably enumerable \emph{relative to $X$} (or \emph{$X$-c.e.}) if $Y$ is the domain of $\Gamma^X$ for some Turing functional $\Gamma$; etc. Recall, too, that for each set $X$, the \emph{jump}\index{Turing!jump}\index{jump!Turing} of $X$ is the $X$-c.e.\ set $X' = \{e \in \NN: \Phi_e^X(e) \downarrow \}$.

An important object in investigations like ours is the following.

\begin{definition}
	A class $\Cc \subseteq \Cantor$ is a a \emph{$\Pi^0_1$ class}\index{class!$\Pi^0_1$}\index{$\Pi^0_1$} if there is a c.e.\ set $W$, viewed as a subset of $\cantor$, such that $\Cc = \Cantor \setminus \bigcup_{\sigma \in W} [\sigma] = \{Y \in \Cantor: (\forall \sigma \in W)[\sigma \nprec Y]\}$.
\end{definition}

\noindent If we take $W$ in the definition to be $X$-c.e.\ rather than c.e.,\ we get the relativized concept of a \emph{$\Pi^{0,X}_1$ class}\index{$\Pi^{0,X}_1$}. Such classes are ubiquitous, often showing up as the collection of sets satisfying some natural computability-theoretic or combinatorial property. A prototypical example, given an infinite set $X$ and a Turing functional $\Gamma$, is the class $\Cc_{X,\Gamma}$ of all pairs of sets $(Y_0,Y_1)$ such that $Y_0 \cup Y_1 = X$ and for each $i < 2$, each $x \in \NN$, and every finite subset $F$ of $Y_i$, $\Gamma^F(x) \uparrow$. It is easy to verify that $\Cc_{X,\Gamma}$ is a $\Pi^{0,X}_1$ class.

Note that a $\Pi^0_1$ class is, in particular, a closed subset of $\Cantor$. (The additional property, worth emphasizing, is that a $\Pi^0_1$ class is one whose complement is effectively generated.) Every closed subset of $\Cantor$ is also compact, which yields the following simple but significant result.

\begin{lemma}[Compactness for $\Pi^0_1$ classes]
	If $W$ is c.e.\ and $\Cc = \Cantor \setminus \bigcup_{\sigma \in W} [\sigma] = \emptyset$, then there is an $\ell \in \NN$ such that $\sigma \in \cantor$ has an initial segment in $W$ of length at most $\ell$.
\end{lemma}

\noindent For instance, if the class $\Cc_{X,\Gamma}$ mentioned above is empty, then compactness yields an $\ell$ such that for every partition of $X$ into two sets, $Y_0$ and $Y_1$, there is an $i < 2$ and a finite subset $F$ of $Y_i \uh \ell$ with $\Gamma^F(x) \downarrow$ for some $x$. Our use of compactness will often take this form.

Equally important for us will be the case when a $\Pi^0_1$ class we are dealing with is non-empty. To study the members of such classes, we typically employ basis theorems of various kinds, a \emph{basis}\index{basis theorem} in this context  being a collection of subsets of $\NN$ that intersects every non-empty $\Pi^0_1$ class. The most celebrated example of this is the low basis theorem\index{basis theorem!low basis}\index{low!basis theorem} of Jockusch and Soare \cite[Theorem 2.1]{Jockusch197201}, which shows that the collection of low sets $Y$ with $Y' \Tred \emptyset'$ forms a basis. In this monograph, we will most often use the following \emph{cone avoidance basis theorem}\index{basis theorem!cone avoidance}\index{cone avoidance!basis theorem}.

\begin{theorem}[Jockusch and Soare \cite{Jockusch1972Degrees}, Corollary 2.11]
	Let $C \subseteq \NN$ be non-computable. Every non-empty $\Pi^0_1$ class contains a member $Y$ such that $C \nTred Y$.
\end{theorem}

\noindent Observe that to relativize the cone avoidance basis theorem to a set $X$, we need $C$ above to be not only non-computable, but non-$X$-computable. Without this additional condition the result would be false, as can be easily seen, for example, by noticing that the singleton $\{X\}$ is a $\Pi^{0,X}_1$ class. This distinction---computing a given non-computable set on the one hand, and computing it together with a given other set on the other---turns out to be an important one, and we will return to it in the next chapter.

\section{Second-order arithmetic and computable reducibility}\label{sec:bkg_rm}

As mentioned, our main focus in this manuscript is a computability-theoretic one. As such, our contributions to reverse mathematics here are largely ancillary, and except where noted otherwise, will follow by straightforward formalization of our computability results. The framework of reverse mathematics nonetheless provides a convenient way to succinctly state many relationships between the various theorems we will be considering, and also motivates many questions we look at. Indeed, many of these questions would not arise otherwise. We thus begin with a brief overview of this framework.

Let $\mathsf{L}_2$ denote the (two-sorted, first-order) language of second-order arithmetic. We use lowercase letters $x,y,\ldots$ to range over first-order variables, and uppercase letters $X,Y,\ldots$ to range over second-order variables. All formulas discussed may include both first- and second-order variables and parameters.

\begin{definition}\label{def:subsystems}
	The following axiomatic systems are defined in the language of second-order arithmetic.
	\begin{enumerate}
		\item $\PA^-$ consists of the algebraic axioms of Peano arithmetic (i.e., all axioms except for induction).
		\item $\RCA_0$\index{$\RCA_0$}\index{subsystem!$\RCA_0$} consists of the axioms of $\PA^-$, together with \emph{$\Delta^0_1$ comprehension} (i.e., the scheme
		\[
			(\forall x)[\phi(x) \iff \psi(x)] \to (\exists X)(\forall x)[x \in X \iff \phi(x)],
		\]
		where $\phi$ is a $\Sigma^0_1$ formula and $\psi$ is $\Pi^0_1$) and \emph{$\Sigma^0_1$ induction} (i.e., the scheme
		\[
			(\phi(0) \wedge (\forall x)[\phi(x) \to \phi(x+1)]) \to (\forall x)[\phi(x)]
		\]
		where $\phi$ is a $\Sigma^0_1$ formula).
		\item $\ACA_0$\index{$\ACA_0$}\index{subsystem!$\ACA_0$} consists of the axioms of $\RCA_0$, together with \emph{arithmetic comprehension} (i.e., the scheme
		\[
			(\exists X)(\forall x)[x \in X \iff \phi(x)]
		\]
		where $\phi$ is a $\Sigma^0_n$ formula for some $n \in \NN$).
	\end{enumerate}	
\end{definition}

$\RCA_0$ corresponds more or less to formalized computable mathematics, since by Post's theorem, being computable from a set is the same as being $\Delta^0_1$ definable from it. Thus, morally, all effectively true theorems ought to be provable in $\RCA_0$. The one complicating factor in this is the restriction in $\RCA_0$ to $\Sigma^0_1$ induction, as even effective arguments sometimes require induction beyond this level, and so may fail in a non-standard model of $\RCA_0$. While this can lead to interesting questions concerning the first-order content of mathematical principles, the majority of our results in this monograph can be readily formalized in $\RCA_0$. Therefore, we will follow the common practice of presenting all our arguments semantically (i.e., we will not give formal proofs in second-order arithmetic), and obtain provability results in $\RCA_0$ implicitly.

The preceding definition lists two of the so-called ``big five'' subsystems of second-order arithmetic, as these will be the only ones of interest to us. In the classical program of reverse mathematics, $\RCA_0$ serves as the base theory, over which implications between (formal versions of) various mathematical theorems are considered, giving a measure of their relative proof-theoretic and computability-theoretic strength. Implications to and from $\ACA_0$ over $\RCA_0$, in particular, constitute an important benchmark in this measurement, as we discuss further below.

We now discuss the models of $\RCA_0$ and $\ACA_0$.

\begin{definition}
	A model\index{model!second-order arithmetic} of second-order arithmetic is a pair $(N,\Sc)$, where $N$ is (the domain of) a model of first-order arithmetic and $\Sc \subseteq \Pc(N)$. If $N = \NN$, then this is an \emph{$\omega$-model}. 
\end{definition}

\noindent Thus, an $\omega$-model\index{$\omega$-model}\index{model!$\omega$-model} is specified entirely by the collection $\Sc$ of subsets of $\NN$ that it includes. The following is immediate.

\begin{lemma}
	Let $(\NN,\Sc)$ be an $\omega$-model.
	\begin{enumerate}
		\item $(\NN,\Sc) \models \RCA_0$ if and only if $\Sc$ is closed under $\oplus$ and under $\Tred$ (i.e., if $\Sc$ is a \emph{Turing ideal}\index{Turing!ideal}).
		\item $(\NN,\Sc) \models \ACA_0$ if and only if $\Sc$ is closed under $\oplus$, $\Tred$, and the map $X \mapsto X'$ (i.e., if $\Sc$ is a \emph{jump ideal}\index{jump!ideal}).
	\end{enumerate}	
\end{lemma}

All the theorems we consider, from Milliken's tree theorem onward, can be expressed by $\Pi^1_2$ formulas in the language of second-order arithmetic, and more specifically, in the form given by \Cref{eqn:Pi12} above. As discussed in the introduction, we think of these as problems, in the following sense.

\begin{definition}
	An \emph{instance-solution problem}	(or just \emph{problem}\index{problem|textbf}) is a relation $\mathsf{P} \subseteq \Cantor \times \Cantor$. For every $(X,Y) \in \mathsf{P}$, $X$ is a \emph{instance} of $\mathsf{P}$ (or \emph{$\mathsf{P}$-instance}) and $Y$ is a \emph{solution} to $X$ for the problem $\mathsf{P}$ (or \emph{$\mathsf{P}$-solution} to $X$).
\end{definition}

\noindent It should be noted that every $\Pi^1_2$ problem can be written in the syntactic form of \Cref{eqn:Pi12} in many different ways. In practice, however, there is a canonical such form one works with, and whenever we refer to a $\Pi^1_2$ statement in this monograph we will have this form in mind.

Not all instance-solution problems naturally come from $\Pi^1_2$ principles (see, e.g., \cite{GohTA,Marcone2020}), but this will be the case in all of the examples we consider. We will move freely between the two perspectives, as convenient. The main practical connection comes from the following definition and basic observation.

\begin{definition}
	Let $\mathsf{P}$ and $\mathsf{Q}$ be problems. $\mathsf{Q}$ is \emph{computably reducible}\index{computable reducibility}\index{computable!reducibility} to $\mathsf{P}$, written $\mathsf{Q} \cred \mathsf{P}$\index{$\cred$}, if every $\mathsf{Q}$-instance $X$ computes a $\mathsf{P}$-instance $\widehat{X}$ such that if $\widehat{Y}$ is any $\mathsf{P}$-solution to $\widehat{X}$ then $X \oplus \widehat{Y}$ computes a $\mathsf{Q}$-solution $Y$ to $X$.
\end{definition}

\begin{lemma}
	Let $\mathsf{P}$ and $\mathsf{Q}$ be $\Pi^1_2$ statements. If $\mathsf{Q} \cred \mathsf{P}$ as problems, then every $\omega$-model of $\RCA_0 \wedge \mathsf{P}$ is a model of $\mathsf{Q}$.
\end{lemma}

\noindent Computable reducibility is a convenient tool for making certain natural constructions in reverse mathematics more explicit. For example, the most common way of showing that a $\Pi^1_2$ statement $\mathsf{P}$ implies $\ACA_0$ over $\RCA_0$ is to show that for every set $A \subseteq \NN$, there is an $A$-computable $\mathsf{Q}$-instance $X$, all of whose solutions $Y$ satisfy $A' \Tred A \oplus Y$. If we let $\mathsf{Q}$ be the problem whose instances are all $X \in \Cantor$, such that the only solution to each $X$ is $X'$, then the preceding precisely says that $\mathsf{Q} \cred \mathsf{P}$.

We conclude this section with a note on non-implications.

\begin{definition}
	Let $\mathsf{P}$ be a problem.
	\begin{enumerate}
		\item\label{cone_avoidance_def} $\mathsf{P}$ admits \emph{cone avoidance}\index{cone avoidance|textbf} if for all sets $A,C \subseteq \NN$ with $C \nTred A$, every $A$-computable $\mathsf{P}$-instance $X$ has a solution $Y$ so that $C \nTred A \oplus Y$.
		\item\label{strong_cone_avoidance_def} $\mathsf{P}$ admits \emph{strong cone avoidance}\index{cone avoidance!strong|textbf} if for all sets $A,C  \subseteq \NN$ with $C \nTred A$, every $\mathsf{P}$-instance $X$ has a solution $Y$ so that $C \nTred A \oplus Y$.
	\end{enumerate}
\end{definition}

\noindent The distinction to note well is that the instance $X$ in \cref{strong_cone_avoidance_def} can be arbitrary, and in particular, need \emph{not} be $A$-computable. As pointed out in the introduction, all computably true principles satisfy cone avoidance, but not necessarily strong cone avoidance. Indeed, strong cone avoidance is a fairly special property which makes it possible to freely use a principle in a construction without increasing its overall complexity, as we will do, e.g., with the Halpern-La\"{u}chli theorem in the next chapters.

Ordinary cone avoidance suffices for the following important result, which we will make repeated use of. We include a proof for completeness.

\begin{lemma}\label{lem:cone-avoidance-not-aca}
	If $\mathsf{P}$ is a $\Pi^1_2$ statement that, as a problem, admits cone avoidance, then there is a $\omega$-model of $\RCA_0 \wedge \mathsf{P}$ in which $\ACA_0$ does not hold. In particular, $\mathsf{P}$ does not imply $\ACA_0$ over $\RCA_0$.
\end{lemma}

\begin{proof}
	Let $C = \emptyset'$. We inductively define $A_0,A_1,\ldots \subseteq \NN$ as follows. Let $A_0 = \emptyset$, and suppose we have defined $A_s$ for some $s \in \NN$ and that $C \nTred A_s$. If $s \neq (e,t)$ for some $e \in \NN$ and some $t < s$, or if $\Phi_e^{A_t}$ is not a $\mathsf{P}$-instance, then let $A_{s+1} = A_s$. Otherwise, by cone avoidance of $\mathsf{P}$ choose a solution $Y$ to $X = \Phi_e^{A_t}$ so that $C \nTred A_s \oplus Y$, and let $A_{s+1} = A_s \oplus Y$.
	
	Let $\mathcal{S} = \{ Z: (\exists s)[Z \Tred A_s]\}$, which is a Turing ideal since $A_t \Tred A_s$ for all $t \leq s$. By construction, if $X$ is any instance of $\mathsf{P}$ in $\mathcal{S}$ then $\mathcal{S}$ contains a solution to $X$. (Indeed, if $X = \Phi_e^{A_t}$, say, let $s = (e,t)$; then a solution to $X$ is computable from $A_{s+1}$.) It follows that $(\NN,\Sc)$ is a model of $\RCA_0 \wedge \mathsf{P}$. But $\emptyset' \nTred A_s$ for all $s \in \NN$, hence $\emptyset' \notin \Sc$. This means $\Sc$ is not a jump ideal and so $(\NN,\Sc)$ is not a model of $\ACA_0$.
\end{proof}

\section{Trees and strong subtrees}\label{sec:bkg_trees}

Trees have different meanings in different areas of mathematics, and what is noteworthy for us here, is that we will \emph{not} be following the common definition used in computability theory.
%Specifically, trees for us need not be downard-closed under the initial segment relation, $\preceq$.

\begin{definition}\label{def:trees}
	A \emph{tree}\index{tree} is a non-empty subset $T$ of $\baire$ 	satisfying the following properties:
	\begin{enumerate}
		\item there exists $\rho \in T$, called a \emph{root}\index{root!tree}\index{tree!root} of $T$, such that $\rho \preceq \sigma$ for all $\sigma \in T$;
		\item if $\sigma,\tau \in T$ then $\sigma \meet \tau \in T$;
		\item for every $\sigma \in T$ there are at most finitely many $\tau \in T$ such that $\sigma \prec \tau$ and such that there is no $\tau' \in T$ with $\sigma \prec \tau' \prec \tau$.
	\end{enumerate}
\end{definition}

\noindent Thus, in brief, our trees are rooted, meet-closed, finitely-branching subsets of $\omega^{<\omega}$. However, they need not be downward closed under the initial segment, $\preceq$, relation, as is the case with the trees commonly used in computability theory. Thus, every tree in the sense of the latter is also a tree in our sense here, but not conversely. Our trees also differ from those used in the context of the so-called Chubb-Hirsct-McNicholl tree theorem, which we will discuss shortly. There, trees are also not necessarily downward closed, but nor are they closed under meets. 

Moving forward, we will use trees exclusively in the sense of Definition \ref{def:trees}, except in \Cref{sec:GenCHMTT} where we deliberately look at the relationship of the two.

As usual, we will refer to the elements of a tree as its \emph{nodes}\index{node}\index{tree!node}.

\begin{definition}\label{def:treeconcepts}
	Let $T$ be a tree.
	\begin{enumerate}
		\item The \emph{level}\index{level}\index{tree!level}\index{node!level} of $\sigma \in T$ is $|\{ \tau \in T: \tau \prec \sigma \}|$. We say $\sigma$ is \emph{at} this level in $T$.
		\item For $n \in \NN$, $T(n)$ denotes the set of all $\sigma \in T$ at level $n$ in $T$.
		\item The \emph{height}\index{height}\index{tree!height} of $T$ is the least ordinal $\alpha$ larger than the level of every $\sigma \in T$.
		\item If $\sigma,\tau \in T$ with $\sigma \prec \tau$ and there is no $\tau' \in T$ with $\sigma \prec \tau' \prec \tau$, then $\tau$ is a \emph{direct extension}\index{direct extension} of $\sigma$ in $T$.
		\item For $k \in \omega$, a node $\sigma \in T$ is \emph{$k$-branching}\index{$k$-branching}\index{node!$k$-branching} in $T$ if it has exactly $k$ many direct extensions in $T$.
		\item A node $\sigma \in T$ is a \emph{leaf}\index{leaf}\index{node!leaf} of $T$ if it is $0$-branching in $T$. The set of leaves of $T$ is denoted by $\leaves(T)$\index{$leaves(T)$}.
		\item $T$ is \emph{$k$-branching}\index{$k$-branching!tree}\index{tree!$k$-branching} if every $\sigma \in T$ is $k$-branching in $T$ or a leaf.
	\end{enumerate}
\end{definition}

\noindent Note that all direct extensions of a given node in a tree $T$ must be pairwise incomparable. The height of a tree is always at most $\omega$, and as trees are non-empty, the height is always defined and at least $1$. Since all trees are finitely-branching by definition, a tree is of height $\omega$ if and only if it is infinite.

\begin{remark}\label{rem:pseudotrees}
	It is worth stressing that if two nodes of $T$ are at the same level, they need \emph{not} have the same length. This is because length is not a structural property of a tree as a graph, but rather of its presentation  (i.e., the labeling of its nodes). The same is true of being closed under meets. Thus, in general, any result we state or prove for trees will apply also, after appropriate relabeling, to any subset $S$ of $\baire$ such that $(S,\preceq)$ is isomorphic to $(T,\preceq)$ for some tree $T$. For any such set $S$ we can freely employ the terms in the preceding definition, since these are independent of presentation.
\end{remark}

\begin{definition}
Given $b:\om\to\om$, a tree $T$ is \emph{$b$-bounded}\index{tree!bounded}\index{bounded!tree}, or \emph{bounded by $b$}, if for every $\sigma \in T$ we have $\sigma(i) < b(i)$ for all $i < |\sigma|$. $T$ is \emph{computably bounded}\index{tree!computably bounded}\index{computably bounded!tree} if $T$ is $b$-bounded for some computable $b$.
\end{definition}

\noindent A $k$-branching tree is thus one which is bounded by precisely the functions whose ranges lie in the interval $[k,\infty)$. Clearly, every finitely branching tree is computably bounded relative to its Turing jump.

Notice, however, that because our trees are not closed downwards under $\preceq$, computably bounded trees here do not necessarily enjoy the usual effectivity properties familiar from computability theory (see, e.g., \cite{Soare2016Turing}, Chapter 3). For example, the set of infinite paths through a computable, computably bounded tree need not be a $\Pi^0_1$ class.

A subset $S$ of a tree $T$ may not itself be a tree, and even if it is, it may not preserve all the structure of $T$. For example, two nodes at the same level in $S$ may be at different levels in $T$, or a node may have fewer direct extensions in $S$ than it did in $T$. This motivates the following definition.

\begin{definition}\label{def:strong-subtree}
  A tree $S$ of height $\alpha$ is a \emph{strong subtree}\index{strong subtree}\index{subtree!strong}\index{tree!strong subtree} of a tree $T$ if it satisfies the following two properties:
  \begin{enumerate}
  \item\label{pt:preserve-levels} there exists a function $f:\alpha\to\om$, called a \emph{level function}\index{level!function}\index{strong subtree!level function}, such that for all $n < \alpha$, if $\sigma \in S(n)$ then $\sigma \in T(f(n))$;
  \item\label{pt:preserve-branching} for all $k$, a node in $S$ which is not at level $\alpha-1$ in $S$ is $k$-branching in $S$ if and only if it is $k$-branching in $T$.
  \end{enumerate}
\end{definition}

\noindent See \Cref{fig:strong-subtree} for a visual representation of a strong subtree.
Given a tree $T$ and $1 \leq \alpha \leq \omega$, we let $\Subtree{\alpha}{T}$\index{$\Subtree{\alpha}{}$!tree} be the collection of all strong subtrees of $T$ of height $\alpha$.

A strong subtree $S$ of a tree $T$ is itself a tree, and so is closed under meets. The branching in $S$ is thus completely determined by the direct extensions in $T$ of the (non-trivial) meets of nodes in $S$. The level function $f$ ensures that if $\sigma\in S\cap T(f(n))$ is not a leaf of $S$, then for every $\tau\in T(f(n)+1)$ extending $\sigma$, there exists a unique $\rho\in S\cap T(f(n+1))$ extending $\tau$. (See Figure \ref{fig:strong-subtree}.)

If the height of $T$ is $\alpha < \omega$ then $\Subtree{\beta}{T} = \emptyset$ for all $\beta > \alpha$, and it is also easy to see that the only element of $\Subtree{\alpha}{T}$ in this case is $T$ itself. Being a strong subtree of a tree is a transitive relation, so in particular, if $S \in \Subtree{\alpha}{T}$ and $U \in \Subtree{\beta}{S}$ for some $\beta \leq \alpha$ then $U \in \Subtree{\beta}{T}$.

%\pelliot{TODO, update the figure according to the screenshot in Whatsapp.}

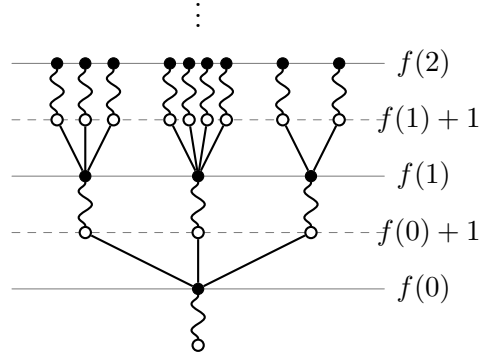
\begin{figure}[htbp]
  \centering
	\begin{tikzpicture}[scale=1.5,font=\normalsize]
		\tikzset{
		empty node/.style={circle,inner sep=0,fill=none},
		solid node/.style={circle,draw,inner sep=1.5,fill=black},
		hollow node/.style={circle,draw,thick,inner sep=1.5,fill=white}
		}
		\tikzset{snake it/.style={decorate, decoration=snake, line cap=round}}
		\tikzset{gray line/.style={line cap=round,color=gray}}
		\tikzset{thin line/.style={line cap=round}}
		\draw[color=gray,-,thin line] (-1.65,0.5) to (1.65,0.5);
		\draw[color=gray,-,thin line] (-1.65,1.5) to (1.65,1.5);
		\draw[color=gray,-,thin line] (-1.65,2.5) to (1.65,2.5);
		\draw[color=gray,-,thin line,dash pattern={on 3pt off 3pt}] (-1.65,1) to (1.65,1);
		\draw[color=gray,-,thin line,dash pattern={on 3pt off 3pt}] (-1.65,2) to (1.65,2);
		\node(a)[hollow node] at (0,0) {};
		\node(b)[solid node] at (0,0.5) {};
		\node(b0)[hollow node] at (-1,1) {};
		\node(b1)[hollow node] at (0,1) {};
		\node(b2)[hollow node] at (1,1) {};
		\node(b0')[solid node] at (-1,1.5) {};
		\node(b1')[solid node] at (0,1.5) {};
		\node(b2')[solid node] at (1,1.5) {};
		\node(b0'0)[hollow node] at (-1.25,2) {};
		\node(b0'1)[hollow node] at (-1,2) {};
		\node(b0'2)[hollow node] at (-0.75,2) {};
		\node(b1'0)[hollow node] at (-0.25,2) {};
		\node(b1'1)[hollow node] at (-0.08,2) {};
		\node(b1'2)[hollow node] at (0.08,2) {};
		\node(b1'3)[hollow node] at (0.25,2) {};
		\node(b2'0)[hollow node] at (0.75,2) {};
		\node(b2'1)[hollow node] at (1.25,2) {};
		\node(dots)[empty node] at (0,3) {$\vdots$};
		\node[empty node] at (2,0.5) {$f(0)$};
		\node[align=flush left] at (2,1) {~~~~~$f(0)+1$};
		\node[empty node] at (2,1.5) {$f(1)$};
		\node[align=flush left] at (2,2) {~~~~~$f(1)+1$};
		\node[align=flush left] at (2,2.5) {$f(2)$};

		\node(b0'0')[solid node] at (-1.25,2.5) {};
		\node(b0'1')[solid node] at (-1,2.5) {};
		\node(b0'2')[solid node] at (-0.75,2.5) {};
		\node(b1'0')[solid node] at (-0.25,2.5) {};
		\node(b1'1')[solid node] at (-0.08,2.5) {};
		\node(b1'2')[solid node] at (0.08,2.5) {};
		\node(b1'3')[solid node] at (0.25,2.5) {};
		\node(b2'0')[solid node] at (0.75,2.5) {};
		\node(b2'1')[solid node] at (1.25,2.5) {};

		\draw[thick,snake it] (a) to (b);
		\draw[thick,-] (b) to (b0);
		\draw[thick,-] (b) to (b1);
		\draw[thick,-] (b) to (b2);
		\draw[thick,snake it] (b0) to (b0');
		\draw[thick,snake it] (b1) to (b1');
		\draw[thick,snake it] (b2) to (b2');
		\draw[thick,-] (b0') to (b0'0);
		\draw[thick,-] (b0') to (b0'1);
		\draw[thick,-] (b0') to (b0'2);
		\draw[thick,-] (b1') to (b1'0);
		\draw[thick,-] (b1') to (b1'1);
		\draw[thick,-] (b1') to (b1'2);
		\draw[thick,-] (b1') to (b1'3);
		\draw[thick,-] (b2') to (b2'0);
		\draw[thick,-] (b2') to (b2'1);
		\draw[thick,snake it] (b0'0) to (b0'0');
		\draw[thick,snake it] (b0'1) to (b0'1');
		\draw[thick,snake it] (b0'2) to (b0'2');
		\draw[thick,snake it] (b1'0) to (b1'0');
		\draw[thick,snake it] (b1'1) to (b1'1');
		\draw[thick,snake it] (b1'2) to (b1'2');
		\draw[thick,snake it] (b1'3) to (b1'3');
		\draw[thick,snake it] (b2'0) to (b2'0');
		\draw[thick,snake it] (b2'1) to (b2'1');
		%\draw[gray line,out=20,in=-90] (a) to (2,3);
		%\draw[gray line,out=170,in=-90] (a) to (-2,3);
	\end{tikzpicture}
  	\caption{A strong subtree $S$ of a tree $T$, with level function $f$. The circles represent nodes in $T$; the solid circles in $S$, the hollow circles are in $T \setminus S$. The levels of $S$ are included in the level{s} of $T$; solid gray horizontal lines represent levels in $S$, dashed gray horizontal lines levels in $T \setminus S$. A node connected to another below it by a straight black line denotes a direct extension in $T$. Wavy lines indicate omitted (skipped over) portions of $T$. Note that all branchings are preserved: a nodes in $S$ has the same number of direct extensions in $S$ as in $T$.}
  	\label{fig:strong-subtree}
\end{figure}

\section{Forests and products of trees}\label{sec:bkg_forests}

As mentioned above, in order to study the proof of Milliken's tree theorem we will need to examine the Halpern-La\"{u}chli theorem, whose statements requires us to consider multiple trees in parallel.

\begin{definition}\label{def:forest}
	A \emph{forest}\index{forest} is a non-empty subset $X$ of $\omega^{<\omega}$ such that if a pair of nodes $\sigma,\tau \in X$ has a common initial segment in $X$ then also $\sigma \meet \tau \in X$.
%A \emph{forest} is a non-empty set $X \subseteq \omega^{<\omega}$ such that if a pair of nodes $\sigma, \tau \in X$ admits a common prefix in $X$, then $\sigma \meet \tau \in X$. A node $\sigma \in X$ is a \emph{root} of $X$ if it has no prefix in $X$.   We write $X(n)$ for the set of nodes of $X$ such that $|\{\tau\in X:\tau\prec\sigma\}|=n$. The \emph{level} of a node is the number $n$ such that $\sigma\in X(n)$. A node $\sigma\in X(n)$ is \emph{$k$-branching} iff $\sigma$ has exactly $k$ pairwise incomparable extensions in $X(n+1)$. A \emph{leaf} is a 0-branching node. % The \emph{meet in $X$} of two nodes $\sigma,\tau\in X$, if it exists, is the common predecessor of $\sigma$ and $\tau$ with the largest level in $X$, written $\sigma\meet_X\tau$.
\end{definition}

Since every pair of nodes in a tree has at least one common initial segment (the root), it is clear that every tree is a forest. Indeed, the following is easy to see: $X \subseteq \baire$ is a forest if and only if it is a union of disjoint trees. For this reason, we refer to the elements of a forest as nodes, and lift all other terminology from trees to forests. For definiteness, we make this explicit in the following definition.

\begin{definition}
	Let $X$ be a forest.
	\begin{enumerate}
		\item A \emph{root}\index{forest!root}\index{root!forest} of $X$ is any $\rho \in T$ having no proper initial segment in $X$. The set of all roots of $X$ is denoted by $\roots(X)$\index{$\roots(X)$}.
		\item The \emph{level}\index{node!level}\index{level} of $\sigma \in X$ is $|\{ \tau \in X: \tau \prec \sigma\}|$. We say $\sigma$ is \emph{at} this level in $X$.
		\item For $n \in \NN$, $X(n)$ denotes the set of all $\sigma \in X$ at level $n$ in $X$.
		\item The \emph{height}\index{forest!height}\index{height!forest} of $X$ is the least ordinal $\alpha$ larger than the level of every $\sigma \in X$.
		\item If $\sigma,\tau \in X$ with $\sigma \prec \tau$ and there is no $\tau' \in X$ with $\sigma \prec \tau' \prec \tau$, then $\tau$ is a \emph{direct extension}\index{direct extension!forest}\index{node!direct extension} of $\sigma$ in $X$.
		\item For $k \in \omega$, a node $\sigma \in X$ is \emph{$k$-branching} in $X$ if it has exactly $k$ many direct extensions in $X$.
		\item A node $\sigma \in X$ is a \emph{leaf}\index{leaf}\index{node!leaf} of $X$ if it is $0$-branching. The set of leaves of $X$ is denoted $\leaves(X)$.
	\end{enumerate}
\end{definition}

\noindent Thus, a forest $X$ is a tree if and only if $\roots(X)$ is a singleton. The height of $X$ is the maximum of the heights of the disjoint trees that comprise it.

Given a forest $X$ and a node $\sigma \in X$, we let $X \uh \sigma = \{ \tau \in X: \tau \succeq \sigma \}$. In particular, whenever $\sigma \in X$ we have that $X \uh \sigma$ is a tree with root $\sigma$.

\begin{definition}
  A forest $Y$ of height $\alpha \leq \omega$ is a \emph{strong subforest}\index{forest!strong subforest}\index{strong subforest} of a forest $X$ if it satisfies the following two properties:
  \begin{enumerate}
  \item there exists a function $f: \alpha \to \omega$, called a \emph{level function}, such that for all $n \leq \alpha$, if $\sigma \in X(n)$ then $\sigma \in X(f(n))$;
  %\item There exists a function $f:\om\to\om$ mapping levels to levels, such that $\forall n\forall \sigma\in Y(n), \sigma\in X(f(n))$.  We shall later refer to this function as the \emph{level function}.
  \item for all $k$, a node in $Y$ which is not at level $\alpha$ in $Y$ is $k$-branching in $Y$ if and only if it is $k$-branching in $X$.
  %\item If $\sigma\in Y\cap X(f(n))$, then for every $\tau\in X(f(n)+1)$ extending $\sigma$, there exists a unique $\rho\in Y\cap X(f(n+1))$ extending $\tau$.
  \end{enumerate}
\end{definition}

\noindent Given a forest $X$ and an $\alpha \leq \omega$, we let $\Subtree{\alpha}{X}$\index{$\Subtree{\alpha}{}$!forest} be the collection of all strong subforests of $X$ of height $\alpha$. We also add the following slightly more general definition.

\begin{definition}
	For each $d \geq 1$, if $T_0,\ldots,T_{d-1}$ are trees then
	\[
		\Subtree{\alpha}{T_0,\dots, T_{d-1}}\index{$\Subtree{\alpha}{}$!product of trees}
	\]
	for $\alpha \leq \omega$ is the collection of all tuples $(S_0,\ldots,S_{d-1})$ such that for each $i < d$ we have $S_i \in \Subtree{\alpha}{T_i}$, witnessed by one and the same level function. In addition, $\Subtree{<\alpha}{T_0,\dots, T_{d-1}}$\index{$\Subtree{<\alpha}{}$} denotes $\bigcup_{n < \alpha} \Subtree{n}{T_0,\dots, T_{d-1}}$.
\end{definition}

\noindent Thus, if $X = \bigcup_{i < d} T_i$, where $T_0,\ldots,T_{d-1}$ are disjoint trees, then $\Subtree{\alpha}{X} = \Subtree{\alpha}{T_0,\ldots,T_{d-1}}$. However, the preceding definition applies to arbitrary trees $T_0,\ldots,T_{d-1}$, disjoint or not.

We include one final definition, which is standard in other investigations of Milliken's tree theorem and will be important to us going forward.

\begin{definition}
	Fix $m \geq 1$.
	\begin{enumerate}
		\item For a forest $X$ and node $\sigma \in X$, a subset $P$ of $X$ is \emph{$m$-$\sigma$-dense}\index{$m$-$\sigma$-dense}\index{dense!subset}\index{dense!$m$-$\sigma$-dense} if every $\tau \in X(m)$ that extends $\sigma$ has an extension in $P$.
		\item For forests $X_0,\ldots,X_{d-1}$ and tuple $\pi = (\sigma_0,\ldots,\sigma_{d-1}) \in \bigcup_{n} X_0(n) \times \cdots \times X_{d-1}(n)$, a subset $P$ of $X_0 \times \cdots \times X_{d-1}$ is an \emph{$m$-$\pi$-dense matrix}\index{dense!matrix}\index{$m$-$\pi$-dense matrix} if $P = P_0 \times \cdots \times P_{d-1}$ where $P_i$ is an $m$-$\sigma_i$-dense subset of $X_i$, for each $i < d$.
	\end{enumerate}
\end{definition}

\noindent We will of course only be interested in the case where $m$ is larger than the level of $\sigma$ in $X$, respectively, of the (common) level in $X_i$ of each of the entries $\sigma_i$ of $\pi$. In the latter case, we will call this common level the \emph{level of $\pi$} in $X$.

The main point in item 2 above is that if $P$ is an $m$-$\pi$-dense matrix then for every $\tau_i \in X_i(m)$ that extends $\sigma_i$ we can find a $\rho_i$ such that $(\rho_0,\ldots,\rho_{d-1}) \in P$, and the latter is true for every choice of possible $\rho_i$. Note that the $\rho_i$ do not have to be at the same level in their respective forests. Also, notice that the $X_i$ need not be disjoint, and so their union need not be itself a forest.

\section{Statements of theorems}\label{sec:bkg_stmts}

In this section, we can finally define Milliken's tree theorem and its combinatorial variants that we will investigate in \Cref{sect:hl-theorem,sect:milliken-theorem}, as well as the various application of Milliken's tree theorem that we will discuss in \Cref{sec:devlin,sec:radomain,sec:GenCHMTT}.

\begin{theorem}[Milliken's tree theorem]\label{th:milliken-theorem}\index{tree theorem!Milliken's tree theorem|textbf}\index{theorem!Milliken's tree theorem|textbf}\index{Milliken's tree theorem|textbf}
	Let $T$ be an infinite tree with no leaves. For all $n,k \geq 1$ and all $f:\Subtree{n}{T} \to k$ there is an $S \in \Subtree{\omega}{T}$ such that $f$ is constant on $\Subtree{n}{S}$.
\end{theorem}

By analogy with Ramsey's theorem, we will break this statement up into sub-statements, in this case according to the height of the subtrees being colored. Thus, we define the following:

\begin{statement}\index{statement!$\MT n{}$|textbf}\index{Milliken's tree theorem!$\MT n{}$|textbf}
For all $n \geq 1$, $\MT n{}$ is the restriction of Milliken's tree theorem to colorings to strong subtrees of height~$n$.
\end{statement}

\noindent We will sometimes also refer to $\MT n{}$ as \emph{Milliken's tree theorem for height $n$} in the sequel. From the computability-theoretic point of view, we will regard an instance of $\MT n{}$ as being a tuple $\seq{T,b,f,k}$, where $T$ is an infinite $b$-bounded tree with no leaves, and $f$ is a map $\Subtree{n}{T} \to k$. In effect, this means all computable instances of $\MT n {}$ are computably bounded.

As discussed in the introduction, the next theorem is the analogue of the pigeonhole principle in the proof of Milliken's tree theorem.

\begin{theorem}[Halpern-La\"{u}chli theorem]\label{th:strong-hl}\index{tree theorem!Halpern-La\"{u}chli theorem|textbf}\index{theorem!Halpern-La\"{u}chli theorem|textbf}\index{Halpern-La\"{u}chli theorem|textbf}
	Let $T_0,\ldots,T_{d-1}$ be infinite trees with no leaves. For all $k \geq 1$ and all $f: \bigcup_{n} T_0(n) \times \cdots \times T_{d-1}(n) \to k$ there exists $(S_0,\ldots,S_{d-1}) \in \Subtree{\omega}{T_0,\ldots,T_{d-1}}$ such that $f$ is constant on $\bigcup_{n} S_0(n) \times \cdots \times S_{d-1}(n)$.
\end{theorem}

\noindent Again, one would naturally expect $\MT 1{}$ to play this role, so the need for the Halpern-La\"{u}chli theorem is not a priori obvious. In fact, the original paper~\cite{Milliken1979RTforTrees} that introduced what we now call Milliken's tree theorem actually proved a version for products that looks much more like the ``general case'' of the Halpern-La\"{u}chli theorem. In many ways, this is really the more natural result, and Milliken's tree theorem is merely a restriction that suffices for most applications.

% following theorem was proven by Milliken in~\cite{Milliken1979RTforTrees}.

\begin{theorem}[Product version of Milliken's tree theorem]\label{th:product-MTT}\index{tree theorem!Milliken's tree theorem for product|textbf}\index{theorem!Milliken's tree theorem for product|textbf}\index{Milliken's tree theorem!product version|textbf}
	Fix infinite trees $T_0,\ldots,T_{d-1}$ with no leaves. For all $n,k \geq 1$ and all colorings $f: \Subtree{n}{T_0,\ldots,T_{d-1}} \to k$ there exists $(S_0,\ldots,S_{d-1}) \in \Subtree{\omega}{T_0,\ldots,T_{d-1}}$ such that $f$ is constant on $\Subtree{n}{S_0,\ldots,S_{d-1}}$.
\end{theorem}

\begin{statement}\index{statement!$\PMT n{}$|textbf}\index{Milliken's tree theorem!$\PMT n{}$|textbf}
  For all $n \geq 1$, $\PMT n{}$ is the restriction of the product version of Milliken's tree theorem for height $n$.
\end{statement}

\noindent The Halpern-La\"{u}chli theorem is exactly $\PMT 1{}$, since for all $T_0,\ldots,T_{d-1}$ we have
\[
	\Subtree{1}{T_0,\dots, T_{d-1}}=\bigcup_{n}T_0(n)\times\dots\times T_{d-1}(n).
\]
In our analysis, we will regard an instance of $\PMT n{}$ as a tuple
\[
	\seq{d,T_0,\ldots,T_{d-1},b,f,k},
\]
where the $T_i$ are infinite $b$-bounded trees with no leaves, and $f$ is a map $\Subtree{n}{T_0,\ldots,T_{d-1}} \to k$.

We now state some further applications of Milliken's tree theorem, which concern various structures besides trees. Each of these structures will be countable and, unless otherwise stated, infinite, and will have a countable, relational underlying language. For a finite substructure $\mathcal{A}$ of a structure $\mathcal{B}$, let %$[\mathcal{B}]^\mathcal{A}$
$\mathcal{B} \choose \mathcal{A}$ \index{$\mathcal{B} \choose \mathcal{A}$}
denote the set of (isomorphic) copies of $\mathcal{A}$ contained in $\mathcal{B}$. Recall also that if $X$ is a set and $n$ is a positive integer then $[X]^n$ \index{$[X]^n$} denotes the set of $n$-element subsets of $X$. In particular, if $B$ is the domain of $\mathcal{B}$, then each element of $[B]^n$ may be regarded as a substructure of $\mathcal{B}$ by restriction since the language of $\mathcal{B}$ is relational. (In general, however, $[B]^n$ need not equal $\mathcal{B} \choose \mathcal{A}$ for any one $\mathcal{A}$.) When convenient, we may also write $[\mathcal{B}]^n$ for $[B]^n$.

The first application of Milliken's tree theorem we consider is \emph{Devlin's theorem}, also called \emph{Devlin's second theorem}, e.g., in \cite{Todorcevic2010Ramsey}, Chapter 6.
\begin{theorem}[Devlin's theorem]\index{theorem!Devlin's theorem|textbf}\index{Devlin's theorem|textbf}
	For every $n \geq 1$ there exists $\ell \geq 1$ such that for every $k \geq 1$ and every $f: [\mathbb{Q}]^n \to k$ there is a dense suborder $S$ of $\mathbb{Q}$ with no endpoints satisfying $|f ([S]^n)| \leq \ell$.
\end{theorem}

\noindent The key here is that the bound $\ell$ does not depend on $k$ or the particular coloring, but only on $n$. As an instance-solution problem, we will study Devlin's theorem in the following form:

\begin{statement}\label{stmt:DT}\index{statement!$\DT n{k,\ell}$|textbf}
  For all $n, k, \ell \geq 1$, $\DT{n}{k,\ell}$ is the assertion that for every $f: [\mathbb Q]^n \to k$ there is a dense suborder $S$ of $\mathbb{Q}$ with no endpoints satisfying $|f ([S]^n)| \leq \ell$.
\end{statement}

\noindent Note that $\DT{n}{k,\ell}$ is merely a formal statement, not a necessarily a true theorem for all possible $n$, $k$, and $\ell$. For example, it is easy to see that $\DT{1}{k,1}$ is true for all $k$. However, $\DT{2}{2,1}$ is false. 
%\noindent It is easy to see that if $n = 1$ then we can take $\ell = 1$. However, for $n \geq 2$, this is no longer the case, so Devlin's theorem is not a direct extension of Ramsey's theorem.
To see this, let $(q_n)_{n\in\Nb}$ be an enumeration of the rationals, and define $f: [\mathbb Q]^2\to 2$ by letting $f(q_n, q_m) = 0$ if $q_n<q_m\iff n<m$, and $f(q_n, q_m) = 1$ otherwise. Then it is readily seen that every subset $S\subseteq\mathbb Q$ of order-type $\mathbb Q$ (or even $\mathbb Z$) must contain pairs of both colors under $f$. For $n = 2$, this situation turns out to be as bad as it can be, as $\DT{2}{k,2}$ is true for all $k$. This fact was originally observed by Galvin (unpublished). For general $n$, the corresponding $\ell$ values were obtained by Devlin \cite[Chapter 4]{Devlin1980}.

The second application we consider concerns graph colorings. We use $\mathcal{G}$ as generic notation for a graph, and unless otherwise specified, assume the set of vertices of $\mathcal{G}$ is $G$, and the set of edges, $E$. For $x,y \in G$, we write $xEy$ if $(x,y) \in E$ and $\lnot x E y$ if $(x,y) \notin E$. The graph $\mathcal{G}$ is a \emph{Rado graph} (or \emph{random graph}) \index{graph!Rado} \index{graph!random|see{Rado}} \index{Rado Graph} if for every two disjoint finite sets of vertices $F_0,F_1 \subseteq G$ there exists $x \in G$ such that $xEy$ for all $y \in F_0$ and $\lnot x E y$ for all $y \in F_1$. Such a graph is, in particular, universal, containing every finite graph as an induced subgraph. All Rado graphs are isomorphic by the standard back and forth construction, so we usually speak just of \emph{the} Rado graph, and assume we have fixed a canonical computable representative of it, denoted by $\mathcal{R}$. The principle of interest to us is following, which we will call the \emph{Rado graph theorem} here for definiteness.
\begin{theorem}[Rado graph theorem]\index{theorem!Rado graph theorem}\index{Rado graph theorem}
	For every finite graph $\mathcal{G}$ there exists $\ell \geq 1$ such that for every $k \geq 1$ and every $f: {\mathcal{R} \choose \mathcal{G}} \to k$ there is an isomorphic subgraph $\mathcal{R}'$ of $\mathcal{R}$ satisfying $|f'' {\mathcal{R}' \choose \mathcal{G}}| \leq \ell$.
\end{theorem}

Again, the bound $\ell$ does not depend on $k$, but only, in this case, on the particular subgraph $\mathcal{G}$. The precise bounds here were obtained by Sauer \cite{Sauer2006} and Laflamme, Sauer, and Vuksanovic \cite{LSV2006}. The result shares much in common with Devlin's theorem, as we will see further below. Both results are well-known consequences of Milliken's theorem. (See, e.g., Todorcevic \cite{Todorcevic2010Ramsey}, Theorems 6.23 and 6.25 for direct proofs.) We give a more effective proof of the Rado graph theorem from Milliken's tree theorem in \Cref{subsect:rado-from-mtt}.

We will investigate the Rado graph theorem in the following two forms.

\begin{statement}\index{statement!$\RG^\mathcal{G}_{k,\ell}$}
For all finite graphs $\mathcal{G}$ and all $k,\ell \geq 1$, $\RG^\mathcal{G}_{k,\ell}$ is the assertion that for every coloring $f: {\mathcal{R} \choose \mathcal{G}} \to k$, there is an isomorphic subgraph $\mathcal{R}'$ of $\mathcal{R}$ satisfying $|f''{\mathcal{R}' \choose \mathcal{G}}| \leq \ell$.
\end{statement}

\begin{statement}\index{statement!$\RG^n_{k,\ell}$}
For all $n,k,\ell \geq 1$, $\RG^n_{k,\ell}$ is the assertion that for every coloring $f: [\mathcal{R}]^n \to k$, there is an isomorphic subgraph $\mathcal{R}'$ of $\mathcal{R}$ satisfying $|f''[\mathcal{R}']^n| \leq \ell$.
\end{statement}

\noindent Since there are, up to isomorphism, only finitely many graphs $G$ of a given finite size, we immediately get the implication
\[
	(\forall \mathcal{G})(\exists \ell)(\forall k)[\RG^\mathcal{G}_{k,\ell}] \to (\forall n)(\exists \ell)(\forall k)[\RG^n_{k,\ell}].
\]

The final application we look at, unlike the previous two, is not a familiar one in set theory. However, it has been studied extensively in computable combinatorics and reverse mathematics (see, e.g., \cite{Chong2019Strengtha, Chong2019Strengthb, Chong2019Strengthc, Dzhafarov2017Coloring, Patey2016strength} for some very recent papers). This is the tree theorem of Chubb, Hirst, and McNicholl \cite{Chubb2009Reverse}, which we will refer to as the \emph{Chubb-Hirst-McNicholl (CHM) tree theorem} in this monograph, to avoid confusion with Milliken's tree theorem. The CHM tree theorem concerns a weaker structure of tree than in Definition \ref{def:trees}, where we do not insist on being closed under meets. A tree is thus any subset of $2^{<\omega}$ with a root. The theorem asserts the existence, for every finite coloring of the $n$-tuples of \emph{comparable} nodes of $\cantor$, of an infinite monochromatic perfect subtree in this weaker sense. The restriction to comparable nodes comes from wanting to extend Ramsey's theorem to these ``weak'' trees. And indeed, as in Devlin's theorem, it is easy to devise a coloring of arbitrary tuples of nodes here where no monochromatic solution exists (e.g., consider coloring all comparable pairs of strings $0$, and all incomparable pairs of strings $1$). As it turns out, this restriction loses a great deal of combinatorial structure, which becomes apparent if we look not for monochromatic solutions, but merely for bounds on the numbers of colors used in a solution. It is this generalization of the CHM tree theorem that we investigate.

\begin{theorem}[Generalized CHM tree theorem]\index{theorem!generalized CHM tree theorem|textbf}\index{tree theorem!generalized CHM tree theorem|textbf}\index{theorem!Chubb, Hirst, McNicholl|see{generalized CHM tree theorem}}\index{generalized CHM tree theorem|textbf}
	For every $n \geq 1$ there exists $\ell \geq 1$ such that for every $k \geq 1$ and every $f: [2^{<\omega}]^n \to k$ there is an $S \subseteq 2^{<\omega}$ such that $(S,\preceq)$ is isomorphic to $(2^{<\omega},\preceq)$ and $|f ([S]^n)| \leq \ell$.
\end{theorem}

\begin{statement}\index{statement!$\CHMTT^n_{k,\ell}$}
	For all $n,k,\ell \geq 1$,	$\CHMTT^n_{k,\ell}$ is the assertion that for every $f: [2^{<\omega}]^n \to k$ there is an $S \subseteq 2^{<\omega}$ such that $(S,\preceq)$ is isomorphic to $(2^{<\omega},\preceq)$ and $|f ([S]^n)| \leq \ell$.
\end{statement}

\noindent As with the previous two principles, the CHM tree theorem is a consequence of Milliken's tree theorem. We include a proof in Theorem \ref{th:strong_gen_treeth} below.

\section{Big Ramsey degrees and structures}\label{sec:bigRamsey}

Though we will study each of Devlin's theorem, the Rado graph theorem, and the CHM tree theorem separately and in its own right, we mention a common framework within which all three can be presented, and which better highlights some of the main similarities between the three. Some of the terminology here will also be convenient in our discussions later on.

All three principles can be stated more succinctly using the concept of big Ramsey degrees, which we now review. Recall that if $\mathcal{B}$ is an infinite structure and $\mathcal{A}$ is a finite substructure of $\mathcal{B}$, then for positive numbers $\ell \leq k$ the notation
\[
\mathcal{B} \to (\mathcal{B})^{\mathcal{A}}_{k,\ell}
\]
means that for every coloring $f: {\mathcal{B} \choose \mathcal{A}} \to k$ there exists an isomorphic substructure $\mathcal{B}'$ of $\mathcal{B}$ such that $|f '' {\mathcal{B}' \choose \mathcal{A}}| \leq \ell$. The following terminology is standard in structural Ramsey theory.

\begin{definition}\label{D:bigRamsey}
	Let $\mathcal{B}$ be a structure.
	\begin{itemize}
		\item For a finite substructure $\mathcal{A}$ of $\mathcal{B}$, the \emph{big Ramsey degree of $\mathcal{A}$ in $\mathcal{B}$} \index{big Ramsey degree} \index{Ramsey!big degree} is the least number $\ell \in \omega$, if it exists, such that $\mathcal{B} \to (\mathcal{B})^\mathcal{A}_{k,\ell}$ for all $k \in \omega$, in which case we say that the big Ramsey degree of $\mathcal{A}$ is \emph{finite}.
		\item We say that a structure \emph{$\mathcal{B}$ has finite big Ramsey degrees} \index{big Ramsey degree!finite} if, for every finite substructure $\mathcal{A}$ of $\mathcal{B}$ has finite big Ramsey degree.
	\end{itemize}
\end{definition}

\noindent In the parlance of this definition, then, the Rado graph theorem is simply the assertion that the Rado graph has finite big Ramsey degrees. Similarly, Devlin's theorem is the assertion that $(\mathbb{Q},<)$ has finite big Ramsey degrees, since up to isomorphism $(\mathbb{Q},<)$ has exactly one finite substructure $\mathcal{A}$ of each size $n \geq 1$, and so ${(\mathbb{Q},<) \choose \mathcal{A}} = [\mathbb{Q}]^n$. For the generalized CHM tree theorem the situation is slightly different. While $(2^{<\omega},\preceq)$ can have more than one non-isomorphic substructure of a given finite size, it still has only finitely many. Thus, the generalized CHM tree theorem is equivalent to the statement that $(2^{<\omega},\preceq)$ has finite big Ramsey degrees.

The bounds $\ell$ in each of Devlin's theorem, the Rado graph theorem, and the generalized CHM tree theorem are not determined purely by properties of the underlying structures. For example, even though $\mathbb{Q}$ has only one substructure of size $2$ up to isomorphism, we saw that we could differentiate two \emph{types} of substructure of size $2$ by enriching the structure by an enumeration of the domain. Enrichments of this kind play an important role in these computations, since they can be taken into account in designing colorings with a certain number of unavoidable colors.

A precise formalization of the concept of ``enrichment'' is given by Zucker \cite{Zucker-2019}.

\begin{definition}[Zucker \cite{Zucker-2019}, Definition 1.3]\label{def:big-ramsey-structure}\index{big Ramsey structure}\index{Ramsey!big structure}
	Let $\mathcal{B}$ be a structure in a language $\mathscr{L}$. A \emph{big Ramsey structure for $\mathcal{B}$} is a structure $\hat{\mathcal{B}}$ in a language $\hat{\mathscr{L}}$ satisfying the following properties:
	\begin{enumerate}
		\item $\mathscr{L} \subseteq \hat{\mathscr{L}}$;
		\item the restriction of $\hat{\mathcal{B}}$ to $\mathscr{L}$ is $\mathcal{B}$;
		\item for every finite substructure $\mathcal{A}$ of $\mathcal{B}$ there is a number $t_{\hat{\mathcal{B}}}(\mathcal{A})$ such that, up to isomorphism, there are exactly $t_{\hat{\mathcal{B}}}(\mathcal{A})$ many different substructures $\hat{\mathcal{A}}$ of $\hat{\mathcal{B}}$ whose restriction to $\mathscr{L}$ is a copy of $\mathcal{A}$;
		\item every finite substructure $\mathcal{A}$ of $\mathcal{B}$ has big Ramsey degree equal to $t_{\hat{\mathcal{B}}}(\mathcal{A})$;
		\item for every finite substructure $\mathcal{A}$ of $\mathcal{B}$ and choice $\hat{\mathcal{A}}_0,\ldots,\hat{\mathcal{A}}_{t_{\hat{\mathcal{B}}}(\mathcal{A})-1}$ of substructures of $\hat{\mathcal{B}}$ as in property 3, the coloring $f: {\mathcal{B} \choose \mathcal{A}} \to t_{\hat{\mathcal{B}}}(\mathcal{A})$ mapping each copy of $\mathcal{A}'$ of $\mathcal{A}$ in $\mathcal{B}$ to the unique $i < t_{\hat{\mathcal{B}}}(\mathcal{A})$ such that $\mathcal{A}'$, viewed as a substructure of $\hat{\mathcal{B}}$ by restriction, is isomorphic to $\hat{\mathcal{A}}_i$ witnesses that the big Ramsey degree of $\mathcal{A}$ in $\mathcal{B}$ is at least $t_{\hat{\mathcal{B}}}(\mathcal{A})$.
	\end{enumerate}
\end{definition}

\noindent The idea here is that for every finite substructure $\mathcal{A}$ of $\mathcal{B}$, the substructures $\hat{\mathcal{A}}_0,\ldots,\hat{\mathcal{A}}_{t_{\hat{\mathcal{B}}}(\mathcal{A})}$ of $\hat{\mathcal{B}}$ satisfying property 3 represent all recognizable or describable types of the copies of $\mathcal{A}$ in $\mathcal{B}$, and the additional structure of $\hat{\mathcal{B}}$ facilitates these descriptions. In the literature, these instances are called more specifically \emph{embedding types} or \emph{Devlin types} based on the specific structure $\mathcal{B}$.

Zucker \cite[Theorem 7.1]{Zucker-2019} provides some sufficient (and somewhat technical) conditions for a structure to admit a big Ramsey structure. For our purposes here, it is enough to know that each of $(\mathbb{Q},\leq)$, the Rado graph, and $(2^{<\omega},\preceq)$ does. We will study the big Ramsey structure of the Rado graph in detail (see also \cite{Zucker-2019}, Section 6.3), and we will carefully develop the appropriate notion of type in the sense of the big Ramsey structure for the generalized CHM tree theorem. For an account of a big Ramsey structure for $(\mathbb{Q},\leq)$, see \cite[Section 6.2]{Zucker-2019}.

%%% Local Variables:
%%% mode: latex
%%% TeX-master: "../embryon"
%%% End:

\chapter{The Halpern-La\"{u}chli theorem}\label{sect:hl-theorem}
We begin our analysis of Milliken's tree theorem by studying the computable content of the Halpern-La\"{u}chli theorem (\Cref{th:strong-hl}). The two main theorems of this chapter are \Cref{thm:halpern-lauchli-computably-true}, that the Halpern-La\"{u}chli theorem is computably true, and \Cref{thm:hl-strong-cone-avoidance}, that it admits strong cone avoidance. The first result will be used in the proof that the product version of Milliken's tree theorem admits arithmetical solutions. The second result will be used to prove that the product version of Milliken's tree theorem for colorings of strong subtrees of height 2 admits cone avoidance, in the same way that strong cone avoidance of the pigeonhole principle can be used to prove cone avoidance of Ramsey's theorem for pairs (see, e.g., Hirschfeldt~\cite{Hirschfeldt2015Slicing}, Section 6.7).
\index{Halpern-La\"{u}chli theorem}
\section{An effective proof of the Halpern-La\"{u}chli theorem}

Our effectivization of the the Halpern-La\"{u}chli theorem is based on the proof of that theorem given in Todorcevic \cite{Todorcevic2010Ramsey}, where it appears as Theorem 3.2. We include that proof here largely in full, emphasizing the effective analysis when it shows up, with the exception of one technical lemma that we present first. For trees $T_0,\ldots,T_{d-1}$ and a tuple $\pi \in T_0(n) \times \cdots \times T_{d-1}(n)$ for some $n \in \NN$, we call $n$ the \emph{level} of $\pi$.

%\begin{lemma}
%	Let $T_0,\ldots,T_{d-1}$ be infinite trees with no leaves. For all $k \geq 1$ and all $f: \bigcup_{n} T_0(n) \times \cdots \times T_{d-1}(n) \to k$ there is a tuple $\pi \in \bigcup_{n} T_0(n) \times \cdots \times T_{d-1}(n)$ such that for every $m$ larger than the level of $\pi$ there is an $m$-$\pi$-dense matrix $P \subseteq \bigcup_{n} T_0(n) \times \cdots \times T_{d-1}(n)$ on which $f$ is constant.
%\end{lemma}

%\noindent The proof of this lemma can be found at the beginning of the proof of Theorem 3.2 in \cite{Todorcevic2010Ramsey}. It is worth remarking that the account there deals with trees that are not necessarily closed under meets, but this only means that the content of the above lemma is proved in a more general form than we will need here.

%For trees $T_0,\ldots,T_{d-1}$ and a tuple $\pi \in T_0(n) \times \cdots \times T_{d-1}(n)$ for some $n$, we call $n$ the \emph{level} of $\pi$.

\begin{lemma}[Halpern and La\"{u}chli \cite{HalperbLauchli1966}, Theorem 1]\label{lem:HLuneven}
	Let $T_0,\ldots,T_{d-1}$ be infinite tree with no leaves. For all $k \geq 1$ and all $g:T_0 \times \cdots \times T_{d-1} \to k$ there is a $\pi \in \bigcup_n T_0(n) \times \cdots \times T_{d-1}(n)$, an $m$ larger than the level of $\pi$, and an $m$-$\pi$-dense matrix $P$ for $T_0,\ldots,T_{d-1}$ on which $g$ is constant.
\end{lemma}

\noindent Nota bene that the coloring $g$ above is defined on the full product $T_0 \times \cdots \times T_{d-1}$, rather than the level product $\bigcup_{n} T_0(n) \times \cdots T_{d-1}(n)$. However, we can obtain a level version, as follows.

\begin{lemma}\label{lem:HLlevel}
	Let $T_0,\ldots,T_{d-1}$ be infinite trees with no leaves. For all $k \geq 1$ and all $f: \bigcup_{n} T_0(n) \times \cdots \times T_{d-1}(n) \to k$ there is a $\pi \in \bigcup_n T_0(n) \times \cdots \times T_{d-1}(n)$, an $m$ larger than the level of $\pi$, and an $m$-$\pi$-dense matrix $P \subseteq \bigcup_{n} T_0(n) \times \cdots \times T_{d-1}(n)$ on which $f$ is constant.
\end{lemma}

\begin{proof}[Proof (from Theorem 3.2 in \cite{Todorcevic2010Ramsey})]
	Fix $T_0,\ldots,T_{d-1}$. By compactness, for every $k \geq 1$ there is an $n_k \geq 1$ such that for every coloring $g:T_0 \times \cdots \times T_{d-1} \to k$ we can find a $\pi$, an $m$, and an $m$-$\pi$-dense matrix $P = P_0 \times \cdots \times P_{d-1}$ as in \Cref{lem:HLuneven} with $P_i \subseteq \bigcup_{n < n_k} T_i(n)$ for all $i < d$.

	Consider now $f: \bigcup_{n} T_0(n) \times \cdots \times T_{d-1}(n) \to k$. We define $g: T_0 \times \cdots \times T_{d-1} \to k$ as follows. First, for each $\sigma \in \bigcup_{n < n_k} T_i(n)$ fix an extension $\hat{\sigma} \in T_i(n_k)$. Now for all $(\sigma_0,\ldots,\sigma_{d-1}) \in T_0 \times \cdots \times T_{d-1}$, set
	\[
		g(\sigma_0,\ldots,\sigma_{d-1}) =
		\begin{cases}
			f(\hat{\sigma}_0,\ldots,\hat{\sigma}_{d-1}) & \text{if } \sigma_i \in \bigcup_{n < n_k} T_i(n) \text{ for all } i < d,\\
			0 & \text{otherwise.}
		\end{cases}
	\]
	By choice of $n_k$ there is a $\pi \in \bigcup_n T_0(n) \times \cdots \times T_{d-1}(n)$, an $m$ larger than the level of $\pi$, and an $m$-$\pi$-dense matrix $Q = Q_0 \times \cdots \times Q_{d-1}$ such that $Q_i \subseteq \bigcup_{n < n_k} T_i(n)$ for all $i < d$ and $g$ is constant on $Q$. For each $i < d$, let $P_i = \{\hat{\rho}: \rho \in Q\}$, so that now $P_i \subseteq T_i(n_k)$. By definition of $g$, we have that $f$ is constant on $P = P_0 \times \cdots \times P_{d-1}$. Thus, $\pi$, $m$, and $P$ are as desired.
\end{proof}

One final critical lemma for us is the following, which is a consequence of the previous one. We include the proof for completeness.

\begin{lemma}\label{lem:HLdensity}
	Let $T_0,\ldots,T_{d-1}$ be infinite trees with no leaves. For all $k \geq 1$ and all $f: \bigcup_{n} T_0(n) \times \cdots \times T_{d-1}(n) \to k$ there is a tuple $\pi \in \bigcup_{n} T_0(n) \times \cdots \times T_{d-1}(n)$ such that for every $m$ larger than the level of $\pi$ there is an $m$-$\pi$-dense matrix $P \subseteq \bigcup_{n} T_0(n) \times \cdots \times T_{d-1}(n)$ on which $f$ is constant.
\end{lemma}

\begin{proof}[Proof (following Theorem 3.2 in \cite{Todorcevic2010Ramsey})]
	Suppose otherwise. Then for every $\seq{\sigma_0,\ldots,\sigma_{d-1}} \in \bigcup_n T_0(n) \times \cdots \times T_{d-1}(n)$ there is an $m \geq 1$ such that $f$ is not constant on any $m$-$\pi$-dense matrix $P \subseteq \bigcup_{n} T_0(n) \times \cdots \times T_{d-1}(n)$. Let $m_\pi$ be the least such $m$. Then choose $m_0 < m_1 < \cdots$ so that $m_0 = 0$ and for all $s \geq 0$, $m_\pi < m_{s+1}$ for all tuples $\pi \in \bigcup_{n \leq m_s} T_0(n) \times \cdots \times T_{d-1}(n)$. For each $i < d$, define $S_i = \bigcup_s T_i(m_s)$, and note that the structure $(S_i,\preceq)$ is isomorphic to a tree, so $S_0 \times \cdots \times S_{d-1}$ can be regarded as a product of trees. Using Remark \ref{rem:pseudotrees}, apply Lemma \ref{lem:HLlevel} to the restriction of $f$ to $S_0 \times \cdots \times S_{d-1}$ to get a tuple $\pi \in \bigcup_n S_0(n) \times \cdots \times S_{d-1}(n)$, an $m$ larger than the level of $\pi$ in this product, and an $m$-$\pi$-dense matrix $P \subseteq \bigcup_n S_0(n) \times \cdots \times S_{d-1}(n)$ on which $f$ is constant. But by construction, the level of $\pi$ must be equal to $m_s$ for some $s$, and $m$ must be equal to $m_t$ for some $t > s$. So $f$ cannot, in fact, be constant on $P$, which is a contradiction.
\end{proof}

We now come to proving our first main theorem of this chapter.

\begin{theorem}\label{thm:halpern-lauchli-computably-true}
 The Halpern-La\"{u}chli theorem is computably true (i.e., every instance computes a solution for itself).
\end{theorem}

\begin{proof}
	Fix an instance of the Halpern-La\"{u}chli theorem, which is to say, infinite trees $T_0,\ldots,T_{d-1}$ with no leaves (and which, recall, we take to be presented with an explicit bound) and a coloring $f: \bigcup_n T_0(n) \times \cdots \times T_{d-1}(n) \to k$ for some $k \geq 1$. We exhibit an $(f \oplus T_0 \oplus \cdots \oplus T_{d-1})$-computable solution, i.e., $(S_0,\ldots,S_{d-1}) \in \Subtree{\omega}{T_0,\ldots,T_{d-1}}$ such that $f$ is constant on $\bigcup_n S_0(n) \times \cdots \times S_{d-1}(n)$.

	Fix, non-effectively, a $\pi = (\sigma_0,\ldots,\sigma_{d-1}) \in \bigcup_n T_0(n) \times \cdots \times T_{d-1}(n)$ as in Lemma \ref{lem:HLdensity}, and say $m_0$ is the level of $\pi$. By the pigeonhole principle, we can also fix a $j < k$ such that for every $m > m_0$ there is an $m$-$\pi$-dense matrix $P \subseteq \bigcup_{n} T_0(n) \times \cdots \times T_{d-1}(n)$ such that $f(\tau_0,\ldots,\tau_{d-1}) = j$ for all $(\tau_0,\ldots,\tau_{d-1}) \in P$. Call such a $P$ \emph{good above $m$}.

	Notice that given $m  > m_0$ and a set $P \subseteq \bigcup_{n} T_0(n) \times \cdots \times T_{d-1}(n)$, it is computable in $f$ and the $T_i$ whether or not $P$ is good above $n$. Hence, we can $(f \oplus T_0 \oplus \cdots \oplus T_{d-1})$-computably define sequences of numbers $m_1 < m_2 < \cdots$ and sets $P_1,P_2,\ldots$ such that $m_0 < m_1$ and each $P_s$ is good above $m_s$. Now, for each $i < d$, define $S_i \subseteq T_i$ inductively as follows: add $\sigma_i$ to $S_i$, and having added $\tau \succeq \sigma_i$ choose the least $s$ such that $P_s$ contains an extension of each direct extension of $\tau$ in $T_i$, and add these extensions to $S_i$. Then $(S_0,\ldots,S_{d-1}) \in \Subtree{\omega}{T_0,\ldots,T_{d-1}}$, and $f(\tau_0,\ldots,\tau_{d-1}) = j$ for all $(\tau_0,\ldots,\tau_{d-1}) \in \bigcup_n S_0(n) \times \cdots \times S_{d-1}(n)$. Clearly, $(S_0,\ldots,S_{d-1})$ is computable from the $T_i$ and the sequences of $m_s$ and $P_s$, hence from $f$ and the $T_i$, as desired.
\end{proof}

In the next section, we will design a good notion of forcing for building infinite strong subtrees, and use this to give more effective proofs of Milliken's tree theorem and its product version. We will need a forest version of the Halpern-La\"{u}chli theorem.

\begin{theorem}[Halpern-La\"{u}chli theorem for forests]\label{lem:hl-forest}
  Let $T_0,\dots, T_{d-1}$ be infinite trees with no leaves, and $X_0,\dots, X_{d-1} \subseteq \baire$ be forests such that for each $i<d$, $X_i$ is a strong subforest of $T_i$ of height $\omega$, with common level function. For all $k \geq 1$ and all $f:\bigcup_n T_0(n)\times\dots\times T_{d-1}(n)\to k$ there exist strong subforests $Y_0,\ldots,Y_{d-1}$ of $X_0,\ldots,X_{d-1}$, respectively, with common level function, such that:
  \begin{enumerate}
  	\item for each $i < d$, every root of $X_i$ is extended by some root of $Y_i$;
  	%\item there exists a coloring $g:\roots(X_0)\times\dots\times \roots(X_{d-1})\to k$ such that for all\[(\sigma_0,\ldots,\sigma_{d-1}) \in \roots(X_0) \times \cdots \times \roots(X_{d-1})\] and all \[(\tau_0,\ldots,\tau_{d-1}) \in \bigcup_n~(Y_0 \uh \sigma_0)(n) \times \cdots \times (Y_{d-1} \uh \sigma_{d-1})(n)\]we have $f(\tau_0,\ldots,\tau_{d-1}) = g(\sigma_0,\ldots,\sigma_{d-1})$.
  	\item for each $(\sigma_0,\ldots,\sigma_{d-1}) \in \roots(X_0) \times \cdots \times \roots(X_{d-1})$, $f$ is constant on $\bigcup_n~(Y_0 \uh \sigma_0)(n) \times \cdots \times (Y_{d-1} \uh \sigma_{d-1})(n)$.
  \end{enumerate}
\end{theorem}

\noindent In other words, the lemma asserts that no part of any of the forests $X_i$ above any given root is wholly omitted in passing to the subforest $Y_i$, and the color under $f$ of a tuple in $\bigcup_n Y_0(n)\times\dots\times Y_{d-1}(n)$ depends only on which roots of $X_0,\dots, X_{d-1}$ the elements of the tuple extend.

As with the ordinary Halpern-La\"{u}chli theorem, our interest will be more in an effective version, which we now prove using \Cref{thm:halpern-lauchli-computably-true} above.

\begin{theorem}\label{lem:hl-forest-computably-true}
  \Cref{lem:hl-forest} is computably true.
\end{theorem}
\begin{proof}
  Fix a collection of trees $T_0,\dots,T_{d-1}$ along with strong subforests $X_0,\dots, X_{d-1}$ with a common level function, and a finite coloring $f:\bigcup_n T_0(n)\times\dots\times T_{d-1}(n)\to k$. For every $i<d$ and $\sigma\in\roots(X_i)$, the set
  %$T_i^{\sigma}=T_i\upharpoonright\sigma$
  $T_i^\sigma = X_i \uh \sigma$ is a tree. The result will come from an application of \Cref{thm:halpern-lauchli-computably-true} to the collection of $T_i^\sigma$ for $i<d$ and $\sigma\in\roots(X_i)$. Define a coloring  \[h:
        % g\colon
    % \begin{array}{rcl}
      \bigcup_n \prod_{\substack{i<d, \\ \sigma\in \roots(X_i)}}T_i^\sigma(n)\to k^{|\roots(T_0)|\times\dots\times|\roots(T_{d-1})|}
%      \bigcup_{n} T_0(n) \times \dots \times T_{d-1}(n) &\to& 2\\
      % \bigcup_{n<h} {T}_0(n) \times \dots \times {T}_{d-1}(n) &\to& k \\
%      \langle \tau_i^\sigma\rangle_{\sigma\in\roots(X_i)}&\mapsto&\langle f(\tau_0^{\sigma_0},\dots,\tau_{d-1}^{\sigma_{d-1}})
%      (\sigma_0,\dots, \sigma_{d-1})&\mapsto& f((e(\sigma_0),\dots, e(\sigma_{d-1})))
%      (\sigma_0,\dots,\sigma_{d-1}) &\mapsto& f(l_{\sigma_0},\dots, l_{\sigma_{d-1}})
%    \end{array}
    \]
%   \benoit{It should be $g:\bigcup_n\prod_{i \leq d, \sigma\in \roots(X_i)}T_i^\sigma(n)\to K$}
	%where $K=k^{|\roots(T_0)|\times\dots\times|\roots(T_{d-1})|}$, and
	such that to a tuple $\pi=(\tau_i^{\sigma} \in T^\sigma_i: i \leq d, \sigma\in \roots(X_i))$, $h$ associates the tuple of all values that $f$ can take on the elements of $\pi$. That is,
    \[
      % g:t\mapsto \langle f(\tau_0,\dots,\tau_{d-1})\rangle_{(\tau_0,\dots,\tau_{d-1})\in t\cap T_0\times\dots\times t\cap T_{d-1}}
      h(\pi) = \seq{f(\tau_0^{\sigma_0},\dots,\tau_{d-1}^{\sigma_{d-1}}):  (\sigma_0, \dots, \sigma_{d-1}) \in \roots(X_0) \times \dots \times \roots(X_{d-1})}.
    \]
%    \benoit{I would write $g(t) = \langle f(\tau_0^{\sigma_0},\dots,\tau_{d-1}^{\sigma_{d-1}})\rangle_{(\sigma_0, \dots, \sigma_{d-1}) \in \roots(X_0) \times \dots \times \roots(X_{d-1})}$}
	Note that $h$ is computable from $f$ and the $T_i^\sigma$, hence from $f$, the $T_i$, and the $X_i$.

    Apply \Cref{thm:halpern-lauchli-computably-true} to define a sequence of strong subtrees $S_i^\sigma$ of $T_i^\sigma$, for $i<d$ and $\sigma\in \roots(X_i)$, with a common level function, and computable from $h$ and the $T_i^\sigma$. For $i<d$, define $Y_i=\bigcup_{\sigma\in\roots(X_i)}S_i^\sigma$, so that $S^\sigma_i = Y_i \uh \sigma$. It is clear that each $Y_i$ is a strong subforest of $X_i$, and that every root of $X_i$ has an extension in $Y_i$. Moreover,
    if $(\tau_0,\ldots,\tau_{d-1})$ and $(\tau'_0,\ldots,\tau'_{d-1})$ both belong to $\bigcup_n S_0^{\sigma_0}(n) \times \cdots \times \ S_{d-1}^{\sigma_{d-1}}(n)$ for some $(\sigma_0,\ldots,\sigma_{d-1}) \in \roots(X_0) \times \cdots \times \roots(X_{d-1})$, then we must have $f(\tau_0,\ldots,\tau_{d-1}) = f(\tau'_0,\ldots,\tau'_{d-1})$ since $h$ is monochromatic on $\bigcup_n \prod_{\substack{i<d, \sigma\in \roots(X_i)}}S_i^\sigma(n)$.
\end{proof}

\section{Product tree forcing}\label{subsect:product-tree-forcing}

We now design the main notion of forcing for building strong subtrees. Variants of this notion of forcing will be used throughout the manuscript.
Fix a collection of finitely branching trees with no leaves $T_0, \dots,\allowbreak T_{d-1}$.

\begin{definition}\label{def:product-tree-condition}\index{product tree condition}\index{condition!product tree}\index{forcing!product tree}
	A \emph{product tree condition} is a tuple
	\[
		(F_0, \dots, F_{d-1}, X_0, \dots,X_{d-1})
	\]
	as follows:
	\begin{enumerate}
		\item $(F_0, \dots, F_{d-1}) \in \Subtree{n}{T_0, \dots, T_{d-1}}$, for some $n \in \NN$;
		\item  $X_0, \dots, X_{d-1}$ are infinite strong subforests of $T_0, \dots, T_{d-1}$, respectively, with a common level function;
		\item for every $j < d$ and every leaf $\sigma$ of $F_j$, say at level $k$ in $T_j$, $\roots(X_j)$ is $(k+1)$-$\sigma$-dense in $T_j$.
	\end{enumerate}
\end{definition}

\noindent Thus, the last condition asserts that every node $\tau \in T_j(k+1)$ extending $\sigma$ has an extension in $\roots(X_j)$.

For instance, let $d = 1$ and $T_0 = \cantor$,
with $F_0 = \{01, 01001, 01100\}$ and $X_0$ any strong subforest of $\cantor$
with $\roots(X_0) = \{ 0100100110, 0100110101,\allowbreak 0110000010, 0110011100 \}$.
Then $(F_0, X_0)$ is a product tree condition. The leaves of $F_0$ are $01001$ and $01100$
and are at level 5 in $\cantor$. The roots $0100100110$ and $0100110101$ of $X_0$ witness $(5+1)$-$\sigma$-density of $\roots(X_0)$ for $\sigma = 01001$, since the extensions of $\sigma$ at level 6 in $\cantor$ are $010010$ and $010011$.

\begin{definition}\label{def:product-tree-extension}\index{product tree condition!extension}\index{extension!product tree condition}
A product tree condition \[d = (\hat{F}_0, \dots \hat{F}_{d-1}, \hat{X}_0, \dots, \hat{X}_{d-1})\] \emph{extends} $c = (F_0, \dots, F_{d-1}, X_0, \dots, X_{d-1})$, written $d \leq c$,  if for every $j < d$, $F_j \subseteq \hat{F}_j$, $\hat{X}_j \subseteq X_j$ and $\hat{F}_j \setminus F_j \subseteq X_j$.
\end{definition}

\begin{remark}\label{remark:product-tree-condition-roots}
Given a product tree condition
\[
	c = (F_0, \dots, F_{d-1}, X_0, \dots,X_{d-1}),
\]
it is not necessarily the case that $F_0 \cup X_0, \dots, F_{d-1} \cup X_{d-1}$ are strong subtrees of $T_0, \dots, T_{d-1}$, respectively, as witnessed by the same level function.
Indeed, the forests may have extra roots unrelated to the finite trees.
However, by removing some roots of the forests, one can always obtain an extension
$d = (F_0, \dots, F_{d-1}, Y_0, \dots, Y_{d-1})$ for which it is the case. We can therefore assume this when convenient. However, in the proof of strong cone avoidance of the Halpern-Lauchli theorem (\Cref{thm:hl-strong-cone-avoidance}), we will use the degree of freedom of being able to have extra roots for {the} construction {of} multiple product tree conditions all sharing the same forests.
\end{remark}

We now define a forcing relation for product tree conditions. We follow a standard approach to forcing in arithmetic, using strong forcing; see, e.g., Shore \cite[Chapter 3]{Shore-2016} for a complete introduction. We use $\Vdash$ (``forces'') for the forcing relation irrespective of the underlying forcing notion, as no confusion will arise in our treatment. As is usual, we write $\cdots \not\Vdash \cdots$ (``$\cdots$ does not force $\cdots$'') as an abbreviation $\neg ( \cdots \forces \cdots)$. Throughout, we work in the language of second-order arithmetic. We follow the usual convention that for a $\Delta^{0,Z}_0$ formula $\varphi(G)$ with a free set parameter $G$, if $\varphi(F)$ holds for a finite set $F$, then so does $\varphi(F \cup E)$ for every finite set $E$ such that $\min E \setminus F > \max F$. We also assume our pairing function is such that if $\sigma,\tau \in \baire$ and $|\tau| > |\sigma|$ then the code for $\tau$ is larger than the code for $\sigma$. So for example, if $F$ is viewed as a subset of Baire and the length of $\tau \in \baire$ is larger than the length of every string  in $F$, then the code of $\tau$ is larger than $\max F$. In particular, if every string in $E \setminus F$ is longer than every string in $F$ and $\varphi(F)$ holds then so does $\varphi(F \cup E)$.

\begin{definition}\label{def:product-tree-forcing-relation}\index{$\Vdash$!product tree condition}
Let $c = (F_0, \dots, F_{d-1}, X_0, \dots, X_{d-1})$ be a product tree condition, $Z \subseteq \NN$ a set, and $\varphi(G_0, \dots, G_{d-1}, x)$ a $\Delta^{0,Z}_0$ formula with a free set parameters $G_0, \dots, G_{d-1}$ and a free integer parameter $x$.
\begin{enumerate}
	\item\label{force:exists} $c \Vdash (\exists x)\varphi(G_0, \dots, G_{d-1}, x)$ if $\varphi(F_0, \dots, F_{d-1}, x)$ holds for some $x \in \NN$.
	\item\label{force:forall} $c \Vdash (\forall x)\varphi(G_0, \dots, G_{d-1}, x)$ if $\varphi(F_0 \cup E_0, \dots, F_{d-1} \cup E_{d-1}, x)$ holds for all $x \in \NN$ and all finite subsets $E_0,\ldots,E_{d-1}$ of $X_0,\ldots,X_{d-1}$, respectively, such that $F_0 \cup E_0, \dots, F_{d-1} \cup E_{d-1}$ are finite strong subtrees of $T_0, \dots, T_{d-1}$, respectively, with a common level function.
\end{enumerate}
\end{definition}

\noindent Of course, \Cref{force:forall} should abstractly be defined as there being no $d$ extending $c$ such that  $d \Vdash (\exists x) \varphi(G_0,\ldots,G_{d-1},x)$, but this is easily seen to be equivalent to the given formulation. We give it explicitly in the definition since we will make frequent use of it.

Every filter $\Uc$ on the set of product tree conditions
induce{s} a $d$-tuple of (finite or infinite) strong subtrees $G^\Uc_0, \dots, G^\Uc_{d-1}$ of $T_0, \dots, T_{d-1}$, respectively, with common level function. Moreover, if $c \Vdash (\exists x)\varphi(G_0, \dots, G_{d-1}, x)$ or $c \Vdash (\forall x)\varphi(G_0, \dots, G_{d-1}, x)$
for some condition $c \in \Uc$ and $\Delta^{0,Z}_0$ formula $\varphi$, then $(\exists x)\varphi(G_0^{\Uc}, \dots, \allowbreak G_{d-1}^\Uc, x)$ holds or $(\forall x)\varphi(G_0^{\Uc}, \dots, G_{d-1}^\Uc, x)$ holds, respectively.

Given a Turing functional $\Gamma$, sets $C,Z \subseteq \NN$, and a condition $c$, we write $c \Vdash \Gamma^{G_0 \oplus \dots \oplus G_{d-1} \oplus Z} \neq C$
if there is an $x \in \NN$ such that either $c \Vdash \Gamma^{G_0 \oplus \dots \oplus G_{d-1} \oplus Z}(x)\uparrow$ or $c \Vdash \Gamma^{G_0 \oplus \dots \oplus G_{d-1} \oplus Z}(x)\downarrow \neq C(x)$. (Note that $C(x)$ is a definite value, so $C$ is not a parameter in the latter formula.) The following lemma states that given a filter $\Uc$ on the set of product tree conditions,
if for every Turing functional $\Gamma$ there is a condition $c \in \Uc$
such that $c \Vdash \Gamma^{G_0 \oplus \dots \oplus G_{d-1} \oplus Z} \neq C$, then $G^\Uc_0, \dots, G^\Uc_{d-1}$ are all infinite.

\begin{lemma}\label{lem:product-tree-genericity-implies-infinity}
For every $n \in \NN$, and all sets $C,Z \subseteq \NN$, there is a Turing functional $\Gamma$ such that
if $c = (F_0, \dots, F_{d-1}, X_0, \dots, X_{d-1})$ is any product tree condition satisfying
$$
c \Vdash \Gamma^{G_0 \oplus \dots \oplus G_{d-1} \oplus Z} \neq C
$$
then $F_0, \dots, F_{d-1}$ all have height at least $n$.
\end{lemma}
\begin{proof}
Let $\Gamma$ be the Turing functional such that for all sets $F_0,\ldots,F_{d-1}$ coding strong subtrees, if the height of each $F_j$ is not at least $n$ then $\Gamma^{F_0 \oplus \dots \oplus F_{d-1} \oplus Z}(x) \uparrow$ for all $x \in \NN$ and $Z \subseteq \NN$, and otherwise
\[
	c \Vdash \Gamma^{G_0 \oplus \dots \oplus G_{d-1} \oplus Z}(x) \downarrow = 0
\]
for all $x$ and $Z$.

Now suppose $c \Vdash \Gamma^{G_0 \oplus \dots \oplus G_{d-1} \oplus Z} \neq C$. If $c \Vdash \Gamma^{G_0 \oplus \dots \oplus G_{d-1} \oplus Z}(x)\downarrow \neq C(x)$ for some $x \in \NN$, then by \Cref{def:product-tree-forcing-relation}(1), $\Gamma^{F_0 \oplus \dots \oplus F_{d-1} \oplus Z}(x)\downarrow \neq C(x)$, so by choice of $\Gamma$, $F_0, \dots, F_{d-1}$ must all have height at least $n$.

If $c \Vdash \Gamma^{G_0 \oplus \dots \oplus G_{d-1} \oplus Z}(x)\uparrow$ for some $x \in \NN$, then by \Cref{def:product-tree-forcing-relation}(2),
\[
	\Gamma^{(F_0 \cup E_0) \oplus \dots \oplus (F_{d-1} \cup E_{d-1}) \oplus Z}(x)\uparrow
\]
for all finite subsets $E_0,\ldots,E_{d-1}$ of $X_0,\ldots,X_{d-1}$, respectively, such that $F_0 \cup E_0, \dots, F_{d-1} \cup E_{d-1}$ are finite strong subtrees of $T_0, \dots, T_{d-1}$ with a common level function. But since $X_0, \dots, X_{d-1}$ are infinite, we can find some such $E_0, \dots, E_{d-1}$ with $F_0 \cup E_0, \dots, F_{d-1} \cup E_{d-1}$ all of height at least $n$, contradicting the definition of $\Gamma$.
\end{proof}

\begin{remark}
	The definition of product tree condition is made with respect to our particular choice of trees $T_0,\ldots,T_{d-1}$. We will always work with a single such choice at any given time, and so do not decorate our conditions by these trees explicitly. In particular, when a product tree condition is mentioned it should be understood as being with respect to whichever trees $T_0,\ldots,T_{d-1}$ are currently under discussion.
\end{remark}

\section{Strong cone avoidance of $\MT{1}{}$}\label{sec:sca_mt1}

Before proving strong cone avoidance of the product version of Milliken's tree theorem,
we prove a similar result for its non-product version. The proof is simpler and is actually sufficient to prove cone avoidance of the non-product version of Milliken's tree theorem for height 2. The techniques involved are a variation of the notion of \emph{$k$-hierarchy} of Chong et al~\cite[Section 4]{Chong2019Strengthb}. The theorem proven in this section will not be used in the remainder of the monograph, but can be seen as an instructive warm-up to the proof of \Cref{thm:hl-strong-cone-avoidance}.

\begin{theorem}\label{thm:mtt1-strong-cone-avoidance}
  $\MT{1}{}$ admits strong cone avoidance.
\end{theorem}

In what follows, fix two sets $C,Z \subseteq \NN$ such that $C \nTred Z$. Also fix an infinite $Z$-computable $Z$-computably bounded tree $T$ with no leaves and an arbitrary 2-partition $A_0 \sqcup A_1 = T$ representing an instance of $\MT{1}{2}$. Our task is to exhibit an $\MT{1}{}$-solution to whose join with $Z$ still does not compute $C$.

%Call a tree condition $(F, X)$ \emph{cone avoiding} if $C \nTred X \oplus Z$.

Given a finite strong subtree $F$ of
%an infinite perfect binary tree
$T$,
a \emph{cover} of $F$ is a set $E \subseteq T$ such that
for every leaf $\sigma$ of $F$, every immediate extension of $\sigma$ in $T$
has an extension in~$E$.

\begin{definition}\index{condition!tree}\index{tree condition}\index{tree!condition}\index{forcing!tree}
\
%Given a finitely branching tree with no leaves $T \subseteq \baire$, a
\begin{enumerate}
	\item A \emph{tree condition} is a pair $(F, X)$ such that $F$ is a finite strong subtree of $T$, $X$ is an infinite strong subforest of $T$, and $\roots(X)$ is a cover of $F$.
	\item A tree condition $(\hat{F}, \hat{X})$ \emph{extends} $(F, X)$, written $(\hat{F}, \hat{X}) \leq (F, X)$, if $F \subseteq \hat{F}$, $\hat{X} \subseteq X$ and $\hat{F} \setminus F \subseteq X$.
	\item $(F,X)$ is \emph{cone avoiding} if $C \nTred X \oplus Z$.
\end{enumerate}
\end{definition}

\noindent Note that a tree condition is nothing but a product tree condition (\Cref{def:product-tree-condition}) relative to the $1$-tuple $T$.

A tree condition inherits the forcing relation
%for for $\Sigma^{0,Z}_1$ and $\Pi^{0,Z}_1$ formulas
from the one for product tree conditions (\Cref{def:product-tree-forcing-relation}).

\begin{definition}\label{def:mtt1-sca-forcing-relation}\index{$\Vdash$!tree condition}
Let $(F, X)$ be a tree condition and $\varphi(G, x)$ a $\Delta^{0,Z}_0$ formula with a free set parameter $G$ and a free integer parameter $x$.
\begin{itemize}
	\item[1.] $(F, X) \Vdash (\exists x)\varphi(G, x)$ if $\varphi(F, x)$ holds for some $x \in \NN$.
	\item[2.] $(F, X) \Vdash (\forall x)\varphi(G, x)$ if $\varphi(F \cup E, x)$ holds for all $x \in \NN$ and all finite $E \subseteq X$ such that $F \cup E$ is a finite strong subtree of $T$.
\end{itemize}
\end{definition}

Every filter $\Uc$ on the set of tree conditions
induces a (finite or infinite) strong subtree $G_\Uc$ of $T$.
Moreover, if $(F, X) \Vdash (\exists x)\varphi(G, x)$ or $(F, X) \Vdash (\forall x)\varphi(G, x)$
for some tree condition $(F, X) \in \Uc$ and $\Delta^{0,Z}_0$ formula $\varphi$, then $(\exists x)\varphi(G_\Uc, x)$ or $(\forall x)\varphi(G_\Uc, x)$ holds, respectively.

Given a Turing functional $\Gamma$, we write $(F, X) \Vdash \Gamma^{G \oplus Z} \neq C$
if there is an $x \in \NN$ such that either $(F, X) \Vdash \Gamma^{G \oplus Z}(x)\uparrow$ or $(F, X) \Vdash \Gamma^{G \oplus Z}(x)\downarrow \neq C(x)$. We have the following analogue of Lemma \ref{lem:product-tree-genericity-implies-infinity}, which is proved in the same way.

\begin{lemma}\label{lem:mtt1-sca-genericity-implies-infinity}
For every $n \in \NN$, there is a Turing functional $\Gamma$ such that
for every tree condition $(F, X)$, if $(F, X) \Vdash \Gamma^{G \oplus Z} \neq C$
then $F$ has height at least $n$.
\end{lemma}

\begin{definition}\index{condition!compound tree condition}\index{compound tree condition}\index{compound tree condition!cone avoiding}
A \emph{compound tree condition} is a tuple $(F, \Fc, X)$ such that $(F, X)$ is a tree condition with $F \subseteq A_0$, and $\Fc$ is a finite collection of finite sets as follows:
\begin{enumerate}
	%\item[1.] $(F, X)$ is a tree condition such that $F \subseteq A_0$;
	\item for every $E \in \Fc$, $(E, X)$ is a tree condition with $E \subseteq A_1$;
	%\item the set of roots of the trees in $\Fc$ form a cover of $F$.
	\item $\bigcup_{E \in \Fc} \roots(E)$ is a cover of $F$.
\end{enumerate}
A compound tree condition $(F, \Fc, X)$ is \emph{cone avoiding} if $C \nTred X \oplus Z$.
\end{definition}

\noindent Equivalently, $(F, \Fc, X)$ is cone avoiding if $(F, X)$ is cone avoiding as a tree condition, and so is $(E, X)$ for every $E \in \Fc$. Note that we do not require the finite strong subtrees in $\Fc$ to be witnessed by the same level function.

\begin{lemma}\label{lem:mtt1-sca-compound-creator}\
\begin{itemize}
	\item[1.] For every tree condition $(F, X)$ with $F \subseteq A_0$, and every level $\ell \in \NN$ such that $X(\ell) \cap A_1$ is a cover of $F$, $(F, \Fc, Y)$ is a compound tree condition, where $\Fc = \{ \{ \rho \}: \rho \in X(\ell) \cap A_1 \}$ and $Y = X \setminus \bigcup_{s \leq \ell} X(s)$.
	\item[2.] For every compound tree condition $(F, \Fc, X)$, every  $E \in \Fc$, and every extension $(\hat{E}, \hat{X}) \leq (E, X)$ with
	$\hat{E} \setminus E \subseteq A_1$ and such that every root of $X$ is extended by a root of $\hat{X}$, $(F, \hat{\Fc}, \hat{X})$ is a compound tree condition, where $\hat{\Fc} = \{\hat{E}\} \cup (\Fc \setminus \{E\})$.
\end{itemize}
\end{lemma}

\begin{proof}
	Immediate from the definitions.	
\end{proof}

\begin{lemma}\label{lem:mtt1-sca-existence-colored-cover}
Suppose there is no infinite strong perfect subtree $S \subseteq T$ such that $S \subseteq A_0$ and $C \nTred S \oplus Z$.
Then for every cone avoiding tree condition $(F, X)$, there is a level $\ell \in \NN$ such that $X(\ell) \cap A_1$ is a cover of $F$.
\end{lemma}
\begin{proof}
Suppose first there is some level $\ell$ such that every root $\rho$ of $X$ has an extension $\sigma \in X(\ell) \cap A_1$. Since $\roots(X)$ is a cover of $F$, then so is $X(\ell) \cap A_1$.

So now suppose that for every level $\ell$, there is some root $\rho$ of $X$ all of whose extensions $\sigma \in X(\ell)$ belong to $A_0$. We claim there is an infinite strong  subtree $S \subseteq T$ such that $S \subseteq A_0$ and $C \nTred S \oplus Z$, contrary to the hypothesis of the lemma. Let $f: \NN \to \roots(X)$ be the function which to $\ell$ associates such a root $\rho$. By strong cone avoidance of $\RT{1}{2}$ (\cite{Dzhafarov2009Ramseys}, Lemma 3.2), there is an infinite set of levels $H$ which is $f$-homogeneous for some root $\rho$ of $X$ and such that $C \nTred H \oplus X \oplus Z$. In particular, for every level $\ell \in H$ and every node $\sigma$ at level $\ell$ in $X$ extending $\rho$ we have $\sigma \in A_0$. But now we can $H \oplus X$-computably build an infinite strong subtree $S \subseteq T$ among these $\sigma$. Then $S \subseteq A_0$, and since $S \Tred H \oplus X$ we also have $C \nTred S \oplus Z$.
\end{proof}

\begin{lemma}\label{lem:mtt1-sca-forcing-requirement}
For every cone avoiding compound tree condition $(F, \Fc, X)$ and every tuple of Turing functionals $\langle \Gamma_F, \Gamma_E: E \in \Fc \rangle$, one of the following holds:
\begin{itemize}
	\item[1.] There is a cone avoiding extension $(\hat{F}, \hat{X}) \leq (F, X)$
	such that $(\hat{F}, \hat{X}) \Vdash \Gamma^{G \oplus Z}_F \neq C$ and $\hat{F} \setminus F \subseteq A_0$ ;
	\item[2.] There is a cone avoiding extension $(\hat{E}, \hat{X}) \leq (E, X)$ for some $E \in \Fc$
	such that $(\hat{E}, \hat{X}) \Vdash \Gamma^{G \oplus Z}_E \neq C$ and $\hat{E} \setminus E \subseteq A_1$ and every root of $X$ is extended by a root of $\hat{X}$.
\end{itemize}
\end{lemma}
\begin{proof}
Let $W$ be the set of pairs $(x, v) \in \NN \times \{0,1\}$ such that for every 2-partition $B_0 \sqcup B_1 = X$ one of the following holds:
\begin{itemize}
	\item[(a)] there is a finite set $H \subseteq X \cap B_0$ such that $F \cup H$ is a finite strong subtree of $T$ and $\Gamma_F^{(F \cup H) \oplus Z}(x)\downarrow = v$;
	\item[(b)] there is some $E \in \Fc$ and a finite set $H_E \subseteq X \cap B_1$ such that $E \cup H_E$ is a finite strong subtree of $T$ and $\Gamma_E^{(E \cup H_E) \oplus Z}(x)\downarrow = v$.
\end{itemize}
By compactness, the set $W$ is $X \oplus Z$-c.e.\ There are three cases:

\case{1}{$(x, 1-C(x)) \in W$ for some $x \in \NN$.} Let $B_0 = X \cap A_0$ and $B_1 = X \cap A_1$. If (a) holds with witness $H$, then let $\ell$ be the level of the leaves of $F \cup H$ in $X$, and $\hat{X} = X \setminus \bigcup_{s \leq \ell} X(s)$. Now $(F \cup H, \hat{X})$ is a tree condition satisfying item 1 of the lemma.
If (b) holds for some $E \in \Fc$ with witness $H_E$, then let $\ell$ be the level of the leaves of $E \cup H_E$ in $X$, and $\hat{X} = X \setminus \bigcup_{s \leq \ell} X(s)$. Now $(E \cup H_E, \hat{X})$ is a tree condition satisfying item 2 of the lemma.

\case{2}{$(x, C(x)) \not\in W$ for some $x \in \NN$.} Let $\Cc$ be the $\Pi^{0,X \oplus Z}_1$ class of all sets $B_0 \oplus B_1$ such that $B_0 \sqcup B_1 = X$ and neither (a) nor (b) holds for the pair $(x, C(x))$. By assumption, $\Cc \neq \emptyset$.
	By the cone avoidance basis theorem (\cite{Jockusch1972Degrees}, Corollary 2.11), there is a $B_0 \oplus B_1 \in \Cc$ such that $C \nTred B_0 \oplus B_1 \oplus X \oplus Z$. For $\sigma \in X$, write $B(\sigma)$ for the unique $i < 2$ such that $\sigma \in B_i$.
	Let $I = \roots(X)$.
	By \Cref{thm:halpern-lauchli-computably-true} applied to the finite $I$-tuple of infinite trees $\langle X \uh \rho: \rho \in I \rangle$ and the coloring $g$ defined on $\bigcup_n \prod_{\rho \in I} (X \uh \rho)(n)$ by $g(\sigma_\rho: \rho \in I) = (B(\sigma_\rho): \rho \in I)$, 
	%(\langle \sigma_\rho: \rho \in I \rangle) = \langle B(\sigma_\rho): \rho \in I \rangle$,
	there is a $B_0 \oplus B_1 \oplus X$-computable finite tuple of infinite strong subtrees $( Y_\rho: \rho \in I )$ of $( X \uh \rho: \rho \in I )$ with common level function, together with a tuple of colors $( i_\rho  \in \{0,1\}: \rho \in I )$ such that $Y_\rho \subseteq B_{i_\rho}$ for every $\rho \in I$. For every $E \in \Fc$, let $I_E$ be the set of nodes in $I$ extending the root of $E$. By assumption, $I_E$ is a cover of $E$.

	If $I_E \subseteq \{ \rho \in I: i_\rho = 1 \}$ for some $E \in \Fc$, then $(E, \bigcup_{\rho \in I} Y_\rho)$ is a cone avoiding extension of $(E, X)$ such that every root of $X$ is extended by a root of $\bigcup_{\rho \in I} Y_\rho$, and
	forcing $\Gamma^{G \oplus Z}_E(x) \uparrow$ or $\Gamma^{G \oplus Z}_E(x) \downarrow \neq C(x)$.

	If $I_E \cap \{ \rho \in I: i_\rho = 0 \} \neq \emptyset$ for every $E \in \Fc$, then in particular, every root of every $E \in \Fc$ has an extension in $\{ \rho \in I: i_\rho = 0 \}$. Since the set of roots of the trees in $E$ form a cover of $F$, then $\{ \rho \in I: i_\rho = 0 \}$ is a cover of $F$. Thus, $(F, \bigcup_{i_\rho = 0} Y_\rho)$ is a cone avoiding extension of $(F, X)$
	forcing $\Gamma^{G \oplus Z}_F(x) \uparrow$ or $\Gamma^{G \oplus Z}_F(x) \downarrow \neq C(x)$.

\case{3}{otherwise.} Then for $x,y \in \NN$ we have $(x,y) \in W$ if and only if $y = C(x)$. But as $W$ is $X \oplus Z$-c.e., this implies that $C \Tred X \oplus Z$, which is a contradiction.
\end{proof}

%The following lemma states that given a filter $\Uc$ of tree conditions,
%if for every Turing functional $\Gamma$ there is a tree condition $(F, X) \in \Uc$
%such that $(F, X) \Vdash \Gamma^{G \oplus Z} \neq C$, then $G_\Uc$ is infinite.
%
%\begin{lemma}\label{lem:mtt1-sca-genericity-implies-infinity}
%For every $n \in \NN$, there is a Turing functional $\Gamma$ such that
%for every tree condition $(F, X)$, if $(F, X) \Vdash \Gamma^{G \oplus Z} \neq C$
%then $F$ has height at least $n$.
%\end{lemma}
%\begin{proof}
%Let $\Gamma^{F \oplus Z}$ be the 0-constant function if $F$ has height at least $n$. Otherwise $\Gamma^{F \oplus Z}$ is the nowhere defined function. Suppose $(F, X) \Vdash \Gamma^{G \oplus Z} \neq C$. If $(F, X) \Vdash \Gamma^{G \oplus Z}(x)\downarrow \neq C(x)$ for some $x \in \NN$, then by \Cref{def:mtt1-sca-forcing-relation}(1), $\Gamma^{F \oplus Z}(x)\downarrow \neq C(x)$, so by choice of $\Gamma$, $F$ has height at least~$n$. If $(F, X) \Vdash \Gamma^{G \oplus Z}(x)\uparrow$ for some $x \in \NN$, then by \Cref{def:mtt1-sca-forcing-relation}(2), $\Gamma^{(F \cup E) \oplus Z}(x)\uparrow$ for every set $E \subseteq X$ such that $F \cup E$ is a finite strong subtree of $T$. Since $X$ is infinite, one can find some $E$ such that $F \cup E$ has height at least $n$, contradicting $\Gamma^{(F \cup E) \oplus Z}(x)\uparrow$.
%\end{proof}

We are now ready to prove strong cone avoidance of Milliken's tree theorem for height 1.

\begin{proof}[Proof of \Cref{thm:mtt1-strong-cone-avoidance}]
Suppose first there is a filter $\Uc$ of cone avoiding tree conditions such that
$F \subseteq A_0$ for every $(F, X) \in \Uc$, and such that
for every Turing functional $\Gamma$ there is a tree condition $(F, X) \in \Uc$
with $(F, X) \Vdash \Gamma^{G \oplus Z} \neq C$.
Then by definition of a tree condition, $G_\Uc$ is a strong subtree of $T$.
Moreover, by assumption, $G_\Uc \subseteq A_0$ and $C \nTred G_\Uc \oplus Z$.
Last, by \Cref{lem:mtt1-sca-genericity-implies-infinity}, $G_\Uc$ is infinite, thus $G_\Uc$ satisfies the statement of the theorem.

Suppose now there is no such filter. Then there is a cone avoiding tree condition $(F, X)$
such that $F \subseteq A_0$ and a Turing functional $\Gamma_F$ such that
for every cone avoiding extension $(\hat{F}, \hat{X})$ with $\hat{F} \setminus F \subseteq A_0$ we have $(\hat{F}, \hat{X}) \nVdash \Gamma_F^{G \oplus Z} \neq C$.

Assume there is no infinite strong subtree $S \subseteq T$ such that $S \subseteq A_0$ and $C \nTred S \oplus Z$, otherwise we are done. By \Cref{lem:mtt1-sca-existence-colored-cover}, there is a level $\ell \in \NN$
such that $X(\ell) \cap A_1$ is a cover of $F$. Let $I = X(\ell) \cap A_1$.

We claim there exists an infinite sequence of cone avoiding compound tree conditions
\[
	(F, \Fc_0, X_0), (F, \Fc_1, X_1), \dots
\]
such that for every $s \in \NN$, letting $s = \langle \Gamma_\rho: \rho \in I \rangle$, the following holds:
\begin{enumerate}
	\item $\Fc_s  = \{ E_{s, \rho}: \rho \in I \}$;
	\item $X_s \subseteq X$;
	\item $\Fc_{s+1} \setminus \{E_{s+1,\rho}\} = \Fc_s \setminus \{E_{s, \rho}\}$
	for some $\rho \in I$ with $(E_{s+1, \rho}, X_{s+1}) \leq (E_{s, \rho}, X_s)$
	and  $(E_{s+1, \rho}, X_{s+1}) \Vdash \Gamma_\rho^{G \oplus Z} \neq C$.
\end{enumerate}

By \Cref{lem:mtt1-sca-compound-creator}(1), letting $\Fc_0 = \{ \{\rho \}: \rho \in I\}$ and $X_0 = X \setminus \bigcup_{t \leq \ell} X(t)$, the tuple $(F, \Fc_0, X_0)$ is a cone avoiding compound tree condition. Given a compound tree condition $(F, \Fc_s, X_s)$ and letting $s = \langle \Gamma_\rho: \rho \in I \rangle$, by \Cref{lem:mtt1-sca-forcing-requirement},
either there is a cone avoiding extension $(\hat{F}, \hat{X}) \leq (F, X)$
	such that $(\hat{F}, \hat{X}) \Vdash \Gamma^{G \oplus Z}_F \neq C$ and $\hat{F} \setminus F \subseteq A_0$, or there some $\rho \in I$ and a cone avoiding extension $(E_{s+1,\rho}, X_{s+1}) \leq (E_{s,\rho}, X_s)$
	such that $(E_{s+1, \rho}, X_{s+1}) \Vdash \Gamma_\rho^{G \oplus Z} \neq C$ and $E_{s+1,\rho}  \setminus E_{s,\rho} \subseteq A_1$ and every root of $X_s$ extends in a root of $X_{s+1}$. The former case cannot happen, so the latter case holds, and we can define $(F, \Fc_{s+1}, X_{s+1})$ accordingly by \Cref{lem:mtt1-sca-compound-creator}(2). This proves our claim.

By a pairing argument, there is a $\rho \in I$ such that for every Turing functional $\Gamma$ there is an $s \in \NN$ such that $(E_{s,\rho}, X_s) \Vdash \Gamma^{G \oplus Z} \neq C$. By construction, the conditions $(E_{s,\rho}, X_s)$ for this fixed $\rho$ are compatible for all $s$. Thus, we can fix a filter $\Uc$ containing all of them.
%Let $\Uc$ be the smallest filter containing the set of tree conditions $\{ (E_{s,\rho}, X_s): s \in \NN \}$.
Again, by definition of a tree condition, $G_\Uc$ is a strong subtree of $T$. By assumption, $G_\Uc \subseteq A_1$ and $C \nTred G_\Uc \oplus Z$.
Last, by \Cref{lem:mtt1-sca-genericity-implies-infinity}, $G_\Uc$ is infinite, thus $G_\Uc$ satisfies the statement of the theorem. This completes the proof of \Cref{thm:mtt1-strong-cone-avoidance}.
\end{proof}

\section{Strong cone avoidance of the Halpern-La\"{u}chli theorem}\label{sec:sca_hl}

We now prove that the Halpern-La\"{u}chli theorem admits strong cone avoidance. This will be used in multiple parts of the rest of the manuscript, to prove that the product version of Milliken's tree theorem for height 2 admits cone avoidance (\Cref{thm:hl-strong-cone-avoidance}) and hence does not imply $\ACA_0$ over $\RCA_0$, and also to prove the same for the product version of Milliken's tree theorem for height 3, but where at most 2 colors are allowed in the solution (\Cref{thm:pmtt3k2-cone-avoidance}).

\begin{theorem}\label{thm:hl-strong-cone-avoidance}
The Halpern-La\"{u}chli theorem admits strong cone avoidance.
\end{theorem}

The meta-analysis of a theorem sometimes requires the use of the classical version of the theorem itself. In order to prove \Cref{thm:hl-strong-cone-avoidance}, we need the following version of the Halpern-La\"{u}chli theorem:

\begin{theorem}\label{thm:combinatorial-finite-hapern-lauchli}
	Let $T_0,\ldots,T_{d-1}$ be infinite trees with no leaves. For all $k \geq 1$, there is an $N \in \NN$ such that for every $f: T_0(N) \times \cdots \times T_{d-1}(N) \to k$ there is an $\ell < N$, a $\pi \in T_0(\ell) \times \cdots T_{d-1}(\ell)$, and an $(\ell+1)$-$\pi$-dense matrix $P \subseteq T_0(N) \times \cdots \times T_{d-1}(N)$ on which $f$ is constant.
%For every tuple of finitely branching trees with no leaves $T_0, \dots,\allowbreak T_{d-1}$ and every number of colors $k \geq 1$,
%there is a level $N \in \NN$ such that for every coloring $f: T_0(N) \times \dots \times T_{d-1}(N) \to k$, there is some $\ell < N$, some $\pi \in T_0(\ell) \times \dots \times T_{d-1}(\ell)$ and some $(\ell+1)$-$\pi$-dense matrix $M \subseteq T_0(N) \times \dots \times T_{d-1}(N)$ monochromatic for $f$.
\end{theorem}
\begin{proof}
  Fix $k$. The Halpern-La\"{u}chli thoerem (\Cref{th:strong-hl}) implies that for every $k$-coloring of $\bigcup_n T_0(n)\times\dots\times T_{d-1}(n)$ there exists an $N \in \NN$ and a tuple of strong subtrees $(S_0,\ldots,S_{d-1}) \in \Subtree{2}{T_0,\ldots,T_{d-1}}$ with level function bounded by $N$ such that $f$ is constant on $\bigcup_{n<2}S_0(n)\times\dots\times S_{d-1}(n)$. By compactness of the space of $k$-colorings of $\bigcup_n T_0(n)\times\dots\times T_{d-1}(n)$, we can choose a single such $N$ that works for all $k$-colorings: that is, for every $k$-coloring, the trees $S_0,\dots, S_{d-1}$ can be taken as strong subtrees of $\bigcup_{n<N} T_0(n),\dots,\bigcup_{n<N} T_{d-1}(n)$, respectively. The claim is that this $N$ also witnesses the theorem.

  Let $f: T_0(N) \times \dots \times T_{d-1}(N) \to k$ be a coloring. For any $i<d$, $n\in\NN$ and $\sigma\in T_i(n)$, let $e(\sigma)$ be any element of $T_i(N)$ compatible with $\sigma$: either an extension, or a prefix. Define the coloring $g: \bigcup_{n} T_0(n) \times \dots \times T_{d-1}(n) \to k$ by $g(\sigma_0,\ldots,\sigma_{d-1}) = f(e(\sigma_0),\ldots,e(\sigma_{d-1}))$ for all $(\sigma_0,\ldots,\sigma_{d-1})$.
  %\[
  %  g\colon
  %  \begin{array}{rcl}
  %    \bigcup_{n} T_0(n) \times \dots \times T_{d-1}(n) &\to& k\\
  %    % \bigcup_{n<h} {T}_0(n) \times \dots \times {T}_{d-1}(n) &\to& k \\
  %    (\sigma_0,\dots, \sigma_{d-1})&\mapsto& f((e(\sigma_0),\dots, e(\sigma_{d-1})))
% %     (\sigma_0,\dots,\sigma_{d-1}) &\mapsto& f(l_{\sigma_0},\dots, l_{\sigma_{d-1}})
  %  \end{array}
  %\]
%  $g: \bigcup_{n} T_0(n) \times \dots \times T_{d-1}(n) \to 2$ by $g((\sigma_0,\dots, \sigma_{d-1}))=f((e(\sigma_0),\dots, e(\sigma_{d-1})))$.
  By choice of $N$, we can find $(S_0,\ldots,S_{d-1}) \in \Subtree{2}{T_0,\ldots,T_{d-1}}$ with level function bounded by $N$ so that $g$ is constant on $S_0(1)\times\dots\times S_{d-1}(1)$. Thus, $f$ is constant on $P = e(S_0(1))\times\dots\times e(S_{d-1}(1))$. Moreover, if we let $\ell$ be the (common) first level of the $S_i$ in $T_i$, and let $\pi$ be the unique element of $S_0(0) \times \cdots \times S_{d-1}(0)$, then $P$ is an $(\ell+1)$-$\pi$-dense matrix.
  %Moreover, if $\ell$ is the first level of $S_0$ in $T_0$, then the unique $\pi\in S_0(0)\times\dots\times S_{d-1}(0)$ is such that $P$ is a $(\ell+1)$-$\pi$-dense matrix.
% \todo[inline]{This can be obtained from Theorem 3.9 page 53 in the book about Ramsey spaces, using compactness and the level normalization trick.}
\end{proof}

%\begin{definition}\label{def:product-tree-extension}
%A product tree condition $d = (\hat{F}_0, \dots \hat{F}_{d-1}, \hat{X}_0, \dots, \hat{X}_{d-1})$ \emph{extends} $c = (F_0, \dots, F_{d-1}, X_0, \dots, X_{d-1})$, written $d \leq c$,  if for every $j < d$, $F_j \subseteq \hat{F}_j$, $\hat{X}_j \subseteq X_j$ and $\hat{F}_j \setminus F_j \subseteq X_j$.
%\end{definition}

In what follows, fix two sets $C$ and $Z$ such that $C \nTred Z$.
Also fix a tuple of infinite $Z$-computable $Z$-computably bounded trees with no leaves $T_0, \dots, T_{d-1}$ and an arbitrary $k$-partition $A_0 \sqcup \dots \sqcup A_{k-1} = \exprodtree{T}{d}$ representing an instance of the Halpern-La\"{u}chli theorem (for $k$-colorings).

For this section, we will need to strengthen the extension relation for product tree conditions (relative to these $T_i$).

\begin{definition}\label{def:strong-product-tree-extension}\index{product tree condition!strong extension}
Let $T_0,\ldots,T_{d-1}$ be infinite trees with no leaves. A product tree condition $d = (\hat{F}_0, \dots \hat{F}_{d-1}, \hat{X}_0, \dots, \hat{X}_{d-1})$ (relative to these $T_i$) \emph{extends} $c = (F_0, \dots, F_{d-1}, X_0, \dots, X_{d-1})$, written $d \leq c$,  if for every $j < d$, $F_j \subseteq \hat{F}_j$, $\hat{X}_j \subseteq X_j$ and $\hat{F}_j \setminus F_j \subseteq X_j$, and moreover, every root of $X_j$ is extended by a root of $\hat{X}_j$.
\end{definition}

\begin{definition}\label{def:levhom3}\index{product tree condition!cone avoiding}\index{product tree condition!level homogeneous}
A product tree condition $(F_0, \dots, F_{d-1}, X_0, \dots, X_{d-1})$ is \emph{cone avoiding} if $C \nTred X_0 \oplus \dots \oplus X_{d-1} \oplus Z$. It is \emph{level-homogeneous} if for every $n$, there is some color $i < k$ such that $F_0(n) \times \dots \times F_{d-1}(n) \subseteq A_i$.
\end{definition}

\noindent In particular, if $d$ extends $c$ in the sense of \Cref{def:strong-product-tree-extension},
then $d$ extends $c$ in the sense of \Cref{def:product-tree-extension}.

Any product tree condition of the form
$$
(\{\rho_0\}, \dots, \{\rho_{d-1}\}, X_0,  \dots, X_{d-1})
$$
is level-homogeneous. Let $\Pb$ be the set of cone avoiding level-homogeneous product tree conditions, ordered by the stronger relation of \Cref{def:strong-product-tree-extension}.
The following lemma is the core of the argument. The proof of \Cref{lem:hl-sca-density-below-a-cone} shows that the witnessed condition $c$ can actually be chosen so that its stems are singletons.
% \pelliot{What you really prove is that $c$ can be taken of the form $(\{\rho_0\}, \dots, \{\rho_{d-1}\}, X^s_0, \dots, X^s_{d-1})$ which is obviously level-homogeneous. I don't know, but it might be clearer to state it like this, as I was wondering where exactly you used the ``level-homogeneous'' thing}

\begin{lemma}\label{lem:hl-sca-density-below-a-cone}
There is a condition $c \in \Pb$
such that for every Turing functional $\Gamma$, the set of conditions $c' \in \Pb$
satisfying $c' \Vdash \Gamma^{G_0 \oplus \dots \oplus G_{d-1} \oplus Z} \neq C$
is $\Pb$-dense below $c$.
\end{lemma}
\begin{proof}
Assume for the sake of contradiction that for every condition $c \in \Pb$,
there is a Turing functional $\Gamma$ and a $\Pb$-extension, every further $\Pb$-extension $c'$ of which satisfies $c' \not\Vdash \Gamma^{G_0 \oplus \dots \oplus G_{d-1} \oplus Z} \neq C$.

We build non-effectively a $d$-tuple of subsets $S_0, \dots, S_{d-1}$ of $T_0, \dots, T_{d-1}$, respectively. Formally, these sets will not be trees, as specified in \Cref{def:trees}, since they will not be closed under $\meet$. However, the prefix relation induces a tree structure, and seen as such, the $S_j$ will be finitely branching trees with no leaves. (In fact, the $S_j$ will have a common level function.) We may thus use Remark \ref{rem:pseudotrees} to think of the $S_j$ as trees, and in particular, we may apply \Cref{thm:combinatorial-finite-hapern-lauchli} to them.
%Recall that $\Pc(X)$ denotes the power set of a set $X$.

Along with $S_0, \dots, S_{d-1}$, we define the following functions:
\begin{enumerate}
	\item $\operatorname{sets}: \NN \to \Pc(\baire) \times \dots \times \Pc(\baire)$ which to a level $\ell \in \NN$
	associates a $d$-tuple $X_0, \dots, X_{d-1}$ of infinite strong subforests of $T_0, \dots, T_{d-1}$, respectively, with a common level function, such that $C \nTred X_0 \oplus \dots \oplus X_{d-1} \oplus Z$ and for every $j < d$, $S_j(\ell+1) = \roots(X_j)$;
	\item $\operatorname{stems}: \exprodtree{S}{d} \to \Subtree{<\omega}{T_0, \dots, T_{d-1}}$, which to a $\pi \in S_0(\ell) \times \dots \times S_{d-1}(\ell)$ associates a tuple $(F_0, \dots, F_{d-1})$ whose roots pointwise extend $\pi$, and such that $(F_0, \dots, F_{d-1}, \operatorname{sets}(\ell))$ is a $\Pb$-condition;
	\item $\operatorname{req}: \exprodtree{S}{d} \to \NN$, which to a $\pi \in S_0(\ell) \times \dots \times S_{d-1}(\ell)$ associates the index $e \in \NN$ of a Turing functional $\Phi_e$
	such that for every $\Pb$-extension $c'$ of the condition $(\operatorname{stems}(\pi), \operatorname{sets}(\ell))$,
	$c' \nVdash \Gamma_e^{G_0 \oplus \dots \oplus G_{d-1} \oplus Z} \neq C$.
\end{enumerate}
Moreover, we ensure that for every level $\ell \in \NN$,
$\operatorname{sets}(\ell+1)$ is a tuple of strong subforests of $\operatorname{sets}(\ell)$ with common level function.
% \pelliot{maybe there should be a picture here? I needed to draw one to understand better.}
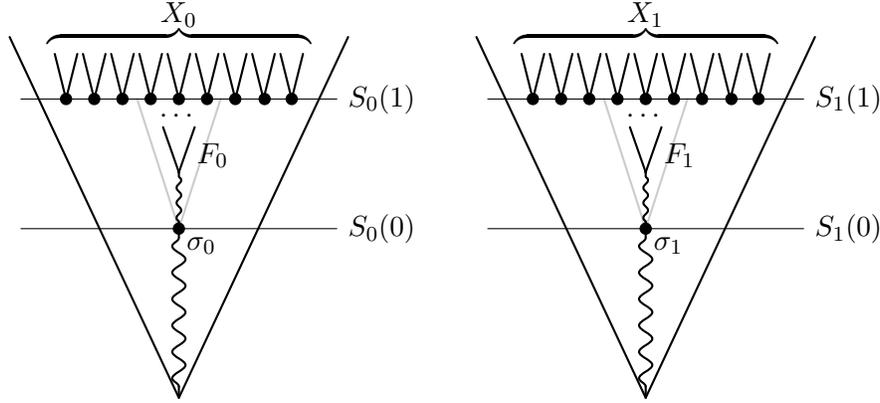
\begin{figure}[h!]
  \begin{center}
    \begin{tikzpicture}[scale=1.5]
		\tikzset{
			empty node/.style={circle,inner sep=0,outer sep=0,fill=none},
			solid node/.style={circle,draw,inner sep=1.5,fill=black},
			hollow node/.style={circle,draw,inner sep=1.5,fill=white},
			gray node/.style={circle,draw={rgb:black,1;white,4},inner sep=1,fill={rgb:black,1;white,4}}
		}
		\tikzset{snake it/.style={decorate, decoration=snake, line cap=round}}
		\tikzset{gray line/.style={line cap=round,thick,color={rgb:black,1;white,4}}}
		\tikzset{thick line/.style={line cap=round,rounded corners=0.1mm,thick}}
		\tikzset{thin line/.style={line cap=round,rounded corners=0.1mm}}
		\node (a)[empty node] at (0,-0.5) {};
		\node (a')[empty node] at (0,-0.4) {};
		\node (sigma0)[solid node] at (0,1) {};
		\node (b)[empty node] at (0,1.5) {};
		\node ()[solid node] at (-0.25,2.15) {};
		\node ()[solid node] at (-0.5,2.15) {};
		\node ()[solid node] at (-0.75,2.15) {};
		\node ()[solid node] at (-1,2.15) {};
		\node ()[solid node] at (0,2.15) {};
		\node ()[solid node] at (0.25,2.15) {};
		\node ()[solid node] at (0.5,2.15) {};
		\node ()[solid node] at (0.76,2.15) {};
		\node ()[solid node] at (1,2.15) {};
		\draw[thick line,decorate,decoration={snake,amplitude=.3mm,segment length=2mm,post length=0.01mm}] (sigma0) to (b.center);
		\begin{pgfonlayer}{background}
		\draw[thick line] (a) to (a'.center);
		\draw[thick,snake it] (a'.center) to (sigma0);
		\draw[thick line] (a.center) to (1.5,2.7);
		\draw[thick line] (a.center) to (-1.5,2.7);
		\draw[thick line] (b.center) to (0.14,1.9);
		\draw[thick line] (b.center) to (-0.14,1.9);
		\draw[gray line] (sigma0.center) to (0.37,2.15);
		\draw[gray line] (sigma0.center) to (-0.37,2.15);
		\draw[thin line,thick] (0,2.15) to (-0.1,2.55);
		\draw[thin line,thick] (0,2.15) to (0.1,2.55);
		\draw[thin line,thick] (0.25,2.15) to (-0.1+.25,2.55);
		\draw[thin line,thick] (0.25,2.15) to (0.1+.25,2.55);
		\draw[thin line,thick] (0.5,2.15) to (-0.1+.5,2.55);
		\draw[thin line,thick] (0.5,2.15) to (0.1+.5,2.55);
		\draw[thin line,thick] (0.75,2.15) to (-.1+.75,2.55);
		\draw[thin line,thick] (0.75,2.15) to (.1+0.75,2.55);
		\draw[thin line,thick] (1,2.15) to (-0.1+1,2.55);
		\draw[thin line,thick] (1,2.15) to (0.1+1,2.55);
		\draw[thin line,thick] (-0.25,2.15) to (-0.1-.25,2.55);
		\draw[thin line,thick] (-0.25,2.15) to (0.1-.25,2.55);
		\draw[thin line,thick] (-0.5,2.15) to (-0.1-.5,2.55);
		\draw[thin line,thick] (-0.5,2.15) to (0.1-.5,2.55);
		\draw[thin line,thick] (-0.75,2.15) to (-.1-.75,2.55);
		\draw[thin line,thick] (-0.75,2.15) to (.1-0.75,2.55);
		\draw[thin line,thick] (-1,2.15) to (-0.1-1,2.55);
		\draw[thin line,thick] (-1,2.15) to (0.1-1,2.55);
		\draw[thin line] (-1.4,1) to (1.4,1);
		\draw[thin line] (-1.4,2.15) to (1.4,2.15);
		\node(dots)[empty node] at (0,2) {$\cdots$};
		\node(F0)[empty node] at (0.3,1.65) {$F_0$};
		\node(sigma0label)[empty node] at (0.2,0.85) {$\sigma_0$};
		\node(X0)[empty node] at (0,2.9) {$X_0$};
		\node(S0)[empty node] at (1.8,2.15) {$S_0(1)$};
		\node(S0)[empty node] at (1.8,1) {$S_0(0)$};
		\node(brace)[empty node] at (0,2.7) {$\overbrace{\hspace{35mm}}$};
		\end{pgfonlayer}
	\end{tikzpicture}
	\hspace{5mm}
	\begin{tikzpicture}[scale=1.5]
		\tikzset{
			empty node/.style={circle,inner sep=0,outer sep=0,fill=none},
			solid node/.style={circle,draw,inner sep=1.5,fill=black},
			hollow node/.style={circle,draw,inner sep=1.5,fill=white},
			gray node/.style={circle,draw={rgb:black,1;white,4},inner sep=1,fill={rgb:black,1;white,4}}
		}
		\tikzset{snake it/.style={decorate, decoration=snake, line cap=round}}
		\tikzset{gray line/.style={line cap=round,thick,color={rgb:black,1;white,4}}}
		\tikzset{thick line/.style={line cap=round,rounded corners=0.1mm,thick}}
		\tikzset{thin line/.style={line cap=round,rounded corners=0.1mm}}
		\node (a)[empty node] at (0,-0.5) {};
		\node (a')[empty node] at (0,-0.4) {};
		\node (sigma0)[solid node] at (0,1) {};
		\node (b)[empty node] at (0,1.5) {};
		\node ()[solid node] at (-0.25,2.15) {};
		\node ()[solid node] at (-0.5,2.15) {};
		\node ()[solid node] at (-0.75,2.15) {};
		\node ()[solid node] at (-1,2.15) {};
		\node ()[solid node] at (0,2.15) {};
		\node ()[solid node] at (0.25,2.15) {};
		\node ()[solid node] at (0.5,2.15) {};
		\node ()[solid node] at (0.76,2.15) {};
		\node ()[solid node] at (1,2.15) {};
		\draw[thick line,decorate,decoration={snake,amplitude=.3mm,segment length=2mm,post length=0.01mm}] (sigma0) to (b.center);
		\begin{pgfonlayer}{background}
		\draw[thick line] (a) to (a'.center);
		\draw[thick,snake it] (a'.center) to (sigma0);
		\draw[thick line] (a.center) to (1.5,2.7);
		\draw[thick line] (a.center) to (-1.5,2.7);
		\draw[thick line] (b.center) to (0.14,1.9);
		\draw[thick line] (b.center) to (-0.14,1.9);
		\draw[gray line] (sigma0.center) to (0.37,2.15);
		\draw[gray line] (sigma0.center) to (-0.37,2.15);
		\draw[thin line,thick] (0,2.15) to (-0.1,2.55);
		\draw[thin line,thick] (0,2.15) to (0.1,2.55);
		\draw[thin line,thick] (0.25,2.15) to (-0.1+.25,2.55);
		\draw[thin line,thick] (0.25,2.15) to (0.1+.25,2.55);
		\draw[thin line,thick] (0.5,2.15) to (-0.1+.5,2.55);
		\draw[thin line,thick] (0.5,2.15) to (0.1+.5,2.55);
		\draw[thin line,thick] (0.75,2.15) to (-.1+.75,2.55);
		\draw[thin line,thick] (0.75,2.15) to (.1+0.75,2.55);
		\draw[thin line,thick] (1,2.15) to (-0.1+1,2.55);
		\draw[thin line,thick] (1,2.15) to (0.1+1,2.55);
		\draw[thin line,thick] (-0.25,2.15) to (-0.1-.25,2.55);
		\draw[thin line,thick] (-0.25,2.15) to (0.1-.25,2.55);
		\draw[thin line,thick] (-0.5,2.15) to (-0.1-.5,2.55);
		\draw[thin line,thick] (-0.5,2.15) to (0.1-.5,2.55);
		\draw[thin line,thick] (-0.75,2.15) to (-.1-.75,2.55);
		\draw[thin line,thick] (-0.75,2.15) to (.1-0.75,2.55);
		\draw[thin line,thick] (-1,2.15) to (-0.1-1,2.55);
		\draw[thin line,thick] (-1,2.15) to (0.1-1,2.55);
		\draw[thin line] (-1.4,1) to (1.4,1);
		\draw[thin line] (-1.4,2.15) to (1.4,2.15);
		\node(dots)[empty node] at (0,2) {$\cdots$};
		\node(F0)[empty node] at (0.3,1.65) {$F_1$};
		\node(sigma0label)[empty node] at (0.2,0.85) {$\sigma_1$};
		\node(X0)[empty node] at (0,2.9) {$X_1$};
		\node(S0)[empty node] at (1.8,2.15) {$S_1(1)$};
		\node(S0)[empty node] at (1.8,1) {$S_1(0)$};
		\node(brace)[empty node] at (0,2.7) {$\overbrace{\hspace{35mm}}$};
		\end{pgfonlayer}
	\end{tikzpicture}
    \caption{A representation of the construction of $S_0$, $S_1$, and the functions $\operatorname{sets}$ and $\operatorname{stems}$. If $\pi=(\sigma_0,\sigma_1)$, then $\operatorname{stems}(\pi)=(F_0, F_1)$ and $\operatorname{sets}(1)=(X_0, X_1)$.}\label{fig:S3Constr}
    \label{fig:sca-HL}
\end{center}
\end{figure}

\construction We define $S_0, \dots, S_{d-1}$ and the functions $\operatorname{sets}$, $\operatorname{stems}$ and  $\operatorname{req}$ level by level. For convenience of notation, let $\operatorname{sets}(-1) = (T_0, \dots, T_{d-1})$.
At level $\ell \geq 0$, assume $\operatorname{sets}(\ell-1)$ is already defined.
Say $(Y_0, \dots, Y_{d-1}) = \operatorname{sets}(\ell-1)$.
For every $j < d$, let $S_j(\ell) = \roots(Y_j)$.
Now let $\pi_0, \dots, \pi_{r-1}$ be a finite listing of all the elements in  $S_0(\ell) \times \dots \times S_{d-1}(\ell)$. We define  $\operatorname{stems}(\pi_s)$ and $\operatorname{req}(\pi_s)$ successively for each $s < r$, together with a decreasing sequence of $d$-tuples of cone avoiding strong subforests
\[
	(X^0_0, \dots, X^0_{d-1}), \dots, (X^r_0, \dots, X^r_{d-1}).
\]
Then $\operatorname{sets}(\ell) = (X^r_0, \dots, X^r_{d-1})$.
Initially, let $(X^0_0, \dots, X^0_{d-1})$ be the tuple $(Y_0 \setminus Y_0(0), \dots, Y_{d-1} \setminus Y_{d-1}(0))$.
At stage $s < r$, assume $(X^s_0, \dots, X^s_{d-1})$ is defined.
Say $\pi_s = (\rho_0, \dots, \rho_{d-1})$.
In particular,
$$
(\{\rho_0\}, \dots, \{\rho_{d-1}\}, X^s_0, \dots, X^s_{d-1})
$$
is a $\Pb$-condition. By assumption, this has a $\Pb$-extension $(F_0, \dots, F_{d-1}, X^{s+1}_0, \allowbreak \dots, X^{s+1}_{d-1})$ for which there is a Turing functional $\Phi_e$
such that every further $\Pb$-extension $c'$ satisfies $c' \not\Vdash \Gamma^{G_0 \oplus \dots \oplus G_{d-1} \oplus Z} \neq C$. So we have $(X^{s+1}_0, \allowbreak \dots, X^{s+1}_{d-1})$, and we set $\operatorname{stems}(\pi_s) = (F_0, \dots, F_{d-1})$ and $\operatorname{req}(\pi_s) = e$. Now if $s < r-1$, proceed to $s+1$.
This finishes the construction. (See \Cref{fig:S3Constr}.)

%\benoit{From the construction, the trees $S_i \subseteq T_i$ are not strong subtrees of $T_i$: It is clear why from the picture above. In fact I now agree with the rest of the proof and its redaction. We just need to find a way to make $S_i$ matches condition of Theorem 3.21 (they are not even meet-closed trees): I suggest some encoding so that every node of level $l+1$ extending $\sigma$ of level $l$ are of the form $\sigma n$ for some $n \in \NN$.}

\verification

\begin{claim}\label{fact:hl-sca-density-below-a-cone-condition-extension}
For every $\ell_0 < \ell_1$ and every $\pi \in S_0(\ell_0) \times \dots \times S_{d-1}(\ell_0)$,
the tuple $(\operatorname{stems}(\pi), \operatorname{sets}(\ell_1))$ is a $\Pb$-extension of $(\operatorname{stems}(\pi), \operatorname{sets}(\ell_0))$.
%In particular, $d \nVdash \Gamma_e^{G_0 \oplus \dots \oplus G_{d-1} \oplus Z} \neq C$ where $e = \operatorname{req}(\pi)$.
\end{claim}
\begin{proof}
Say $\operatorname{sets}(\ell_0) = (X_0, \dots, X_{d-1})$ and $\operatorname{sets}(\ell_1) = (Y_0, \dots, Y_{d-1})$. By an immediate induction, $\operatorname{sets}(\ell_1)$ is a tuple of strong subforests of $\operatorname{sets}(\ell_0)$ with common level function.
By construction, for all $j < d$ we have $S_j(\ell_0+1) = \roots(X_j)$ and  $S_j(\ell_1+1) = \roots(Y_j)$, and since we are dealing with extension in $\Pb$ here, this implies that every root of $X_j$ is extended by a root of $Y_j$.
It follows that $d = (\operatorname{stems}(\pi), \operatorname{sets}(\ell_1))$ is a $\Pb$-extension of $(\operatorname{stems}(\pi), \operatorname{sets}(\ell_0))$.
\end{proof}

From the preceding fact, it follows that the $S_j$ are as claimed. The rest of Properties 1--3 above are evident from the construction.

By \Cref{thm:combinatorial-finite-hapern-lauchli}, there is a level $N \in \NN$
such that for every coloring $h: S_0(N) \times \dots \times S_{d-1}(N) \to k$,
there is some $\ell < N$, some $\pi \in S_0(\ell) \times \dots \times S_{d-1}(\ell)$
and some $(\ell+1)$-$\pi$-dense matrix $M \subseteq S_0(N) \times \dots \times S_{d-1}(N)$
on which $h$ is constant.
Fix such an $N$. Let $(X_0, \dots, X_{d-1}) = \operatorname{sets}(N-1)$. In particular, for every $j < d$, $S_j(N) = \roots(X_j)$.

Let $W$ be the set of pairs $(x, v) \in \NN \times \{0,1\}$ such that for every $k$-partition $B_0 \sqcup \dots \sqcup B_{k-1} = \bigcup_n X_0(n) \times \dots \times X_{k-1}(n)$, there is some $\ell < N$, some $\pi \in S_0(\ell) \times \dots \times S_{d-1}(\ell)$, and for every $j < d$, a finite set $H_j \subseteq X_j$ such that if $ (F_0, \dots, F_{d-1}) = \operatorname{stems}(\pi)$ then the following hold:
\begin{itemize}
	\item[(a)] $(F_0 \cup H_0, \dots, F_{d-1} \cup H_{d-1}) \in \Subtree{<\omega}{T_0, \dots, T_{d-1}}$;
	\item[(b)] $\bigcup_n H_0(n) \times \dots \times H_{d-1}(n) \subseteq B_i$ for some $i < k$;
	\item[(c)] $\Phi_e^{(F_0 \cup H_0) \oplus \dots \oplus (F_{d-1} \cup H_{d-1}) \oplus Z}(x)\downarrow = v$, where $e = \operatorname{req}(\pi)$.
\end{itemize}
Note that although the trees $S_0, \dots, S_{d-1}$ and the functions $\operatorname{sets}$, $\operatorname{stems}$ and $\operatorname{req}$ are built non-effectively, only their restrictions to the height $N$ are used. Therefore, since every finite object is computable, they do not add to the complexity of the set $W$. By compactness, the set $W$ is $X_0 \oplus \dots \oplus X_{d-1} \oplus Z$-c.e. We break into three cases.

\case{1}{$(x, 1-C(x)) \in W$ for some $x \in \NN$.} For $i < k$, let $B_i = A_i \cap \bigcup_n X_0(n) \times \dots \times X_{d-1}(n)$. Let $\ell<N$, $\pi = (F_0, \dots, F_{d-1})$ and $H_0, \dots, H_{d-1}$ witness that $(x, 1-C(x)) \in W$ for the partition    $B_0, \dots, B_{k-1}$.
	Let $\ell_1$ be the common level of the leaves of $F_j \cup H_j$ in $X_j$, and $\hat{X}_j = X_j \setminus \bigcup_{\ell_0 \leq \ell_1} X_j(\ell_0)$. Then $c' = (F_0 \cup H_0, \dots, F_{d-1} \cup H_{d-1}, \hat{X}_0, \dots, \hat{X}_{d-1})$ is a $\Pb$-extension of the condition $(F_0, \dots, F_{d-1}, X_0, \dots, X_{d-1})$ which, by Fact \ref{fact:hl-sca-density-below-a-cone-condition-extension},
		is a $\Pb$-extension of $(\operatorname{stems}(\pi), \operatorname{sets}(\ell))$ since $\ell_1 \geq \ell$.
		Moreover
		$$
			c' \Vdash \Phi_e^{G_0 \oplus \dots \oplus G_{d-1} \oplus Z} \neq C
		$$
		where $e = \operatorname{req}(\pi)$. This contradicts Property 3 above,
		according to which $(\operatorname{stems}(\pi), \operatorname{sets}(\ell))$ has no such $\Pb$-extension.

\case{2}{$(x, C(x)) \not\in W$ for some $x \in \NN$.} Let $\Cc$ be the $\Pi^{0,X_0 \oplus \dots \oplus X_{d-1} \oplus Z}_1$ class of all sets $B_0 \oplus \dots \oplus B_{k-1}$ such that $B_0 \sqcup \dots \sqcup B_{k-1} = \bigcup_n X_0(n) \times \dots \times X_{d-1}(n)$ and such that for every $\ell < N$, every $\pi \in S_0(\ell) \times \dots \times S_{d-1}(\ell)$ and every $H_0 \subseteq X_0, \dots, H_{d-1} \subseteq X_{d-1}$, one of (a), (b) or (c) in the definition of $W$ fails for the pair $(x, C(x))$.
	By assumption, $\Cc \neq \emptyset$.

	By the cone avoidance basis theorem, there is some $B_0 \oplus \dots \oplus B_{k-1} \in \Cc$ such that $C \nTred B_0 \oplus \dots \oplus B_{k-1} \oplus X_0 \oplus \dots \oplus X_{d-1} \oplus Z$. For $\pi \in \bigcup_n X_0(n) \times \dots \times X_{d-1}(n)$, write $B(\pi)$ for the unique $i < k$ such that $\pi \in B_i$. Recall that for every $j < d$, $S_j(N) = \roots(X_j)$.
	We define a finite coloring $g$ on $\bigcup_n X_0(n) \times \dots \times X_{d-1}(n)$ by
	%$$
	%	\bigcup_n  \prod_{j < d} X_j(n)
	%$$
	by
	\[
		g(\sigma_0,\ldots,\sigma_{d-1}) = B(\sigma_0, \dots,\sigma_{d-1}).
	\]
	By \Cref{lem:hl-forest-computably-true} applied to $g$,
	there is a $B_0 \oplus \dots B_{k-1} \oplus X_0 \oplus \dots \oplus X_{d-1}$-computable tuple of infinite strong subtrees $( Y_{j,\rho}: j < d, \rho \in S_j(N))$ of $( X_j \uh \rho: j < d, \rho \in S_j(N))$ with common level function, together with a coloring $h: S_0(N) \times \dots \times S_{d-1}(N) \to k-1$,
		such that
		$$
			\bigcup_n Y_{0,\rho_0}(n) \times \dots \times Y_{d-1,\rho_{d-1}}(n)  \subseteq B_{h(\pi), }
		$$
		for every $\pi = (\rho_0, \dots, \rho_{d-1}) \in S_0(N) \times \dots \times S_{d-1}(N)$.

	By choice of $N$, there is some $\ell < N$, some $\pi = (\nu_0, \dots, \nu_{d-1}) \in S_0(\ell) \times \dots \times S_{d-1}(\ell)$ and some $(\ell+1)$-$\pi$-dense matrix $M \subseteq S_0(N) \times \dots \times S_{d-1}(N)$ on which $h$ is constant. Say $M = M_0 \times \dots \times M_{d-1}$ and let $i < k$ be the color of $h$ on this matrix.
	For every $j < d$, let $P_j$ be the set of nodes in $S_j(N)$ which are not extensions of $\nu_j$. For every $j < k$, let $\hat{Y}_j = \bigcup_{\rho \in M_j \cup P_j} Y_{j,\rho}$.

	\begin{claim}\label{fact:hl-sca-case2-exts}
	$(\operatorname{stems}(\pi), \hat{Y}_0, \dots, \hat{Y}_{d-1})$
	$\Pb$-extends $(\operatorname{stems}(\pi), \operatorname{sets}(\ell))$.
	\end{claim}
	\begin{proof}
	Let $(\hat{X}_0, \dots, \hat{X}_{d-1}) = \operatorname{sets}(\ell)$. Since $\ell < N$ and since $\operatorname{sets}(N-1) = (X_0,\ldots,X_{d-1})$, it follows by Fact \ref{fact:hl-sca-density-below-a-cone-condition-extension} that the $X_j$ are strong subtrees of the $\hat{X}_j$ with common level function. Hence, so are the $Y_j$. Furthermore, by Property 1, for every $j < k$ we have that $\roots(\hat{X}_j) = S_j(\ell+1)$. So every root of $\hat{X}_j$ is extended by a root of $\hat{Y}_j$.
	\end{proof}

	It follows by
	%\Cref{fact:hl-sca-density-below-a-cone-condition-extension} and
	Property 3 that $(\operatorname{stems}(\pi), \hat{Y}_0, \dots, \hat{Y}_{d-1}) \nVdash \Phi_e^{G_0 \oplus \dots \oplus G_{d-1} \oplus Z} \neq C$ where $e = \operatorname{req}(\pi)$. Now, since the forcing relation depends only on part of the reservoirs extending the roots of the stems, the following fact holds. However, we have the following contradictory fact:

	\begin{claim}\label{fact:hl-sca-case2-force-diag}
	$(\operatorname{stems}(\pi), \hat{Y}_0, \dots, \hat{Y}_{d-1}) \Vdash \Phi_e^{G_0 \oplus \dots \oplus G_{d-1} \oplus Z} \neq C$, where $e = \operatorname{req}(\pi)$.
	\end{claim}
	\begin{proof}
	%We claim that $d \Vdash \Gamma_e^{G_0 \oplus \dots \oplus G_{d-1} \oplus Z}(x) \neq C(x)$,
	%where the inequality means that either the left part diverges, or halts and is different.
	For every $j < d$, let $H_j \subseteq \hat{Y}_j$ be such that
	$F_0 \cup H_0, \dots, F_{d-1} \cup H_{d-1}$ are finite strong subtrees of $T_0, \dots, T_{d-1}$, respectively, with common level function. Since the rots of the $F_j$ pointwise extend $\pi$, so do the nodes of each of the $H_j$. In particular, for every $j < d$, $H_j \subseteq  \bigcup_{\rho \in M_j} Y_{j,\rho}$. It follows that $\bigcup_n H_0(n) \times \dots \times H_{d-1}(n) \subseteq B_i$. But $B_0 \oplus  \dots \oplus B_{k-1} \in \Cc$, so $\Phi_e^{(F_0 \cup H_0) \oplus \dots \oplus (F_{d-1} \cup H_{d-1}) \oplus Z}(x)$ either diverges or is different from $C(x)$.
	Since the $H_j$ were arbitrary, the claim is proved.
	\end{proof}

	The contradiction completes Case 2.

\case{3}{otherwise.} Then $(x,y) \in W$ if and only if $y = C(x)$, which, since $W$ is $X_0 \oplus \dots \oplus X_{d-1} \oplus Z$-c.e., implies $C \Tred X_0 \oplus \dots \oplus X_{d-1} \oplus Z$, a contradiction.
\end{proof}

We are now ready to prove strong cone avoidance of the Halpern-La\"{u}chli theorem.

\begin{proof}[Proof of \Cref{thm:hl-strong-cone-avoidance}]
Fix two sets $C$ and $Z$ such that $C \nTred Z$.
Also fix a tuple of infinite $Z$-computable $Z$-computably bounded trees with no leaves $T_0, \dots, T_{d-1} \subseteq \baire$ and an arbitrary $k$-partition $A_0 \sqcup \dots \sqcup A_{k-1} = \exprodtree{T}{d}$. Let $\Pb$ be the set of cone avoiding level-homogeneous product tree conditions (relative to these givens).

By \Cref{lem:hl-sca-density-below-a-cone}, there is some $c \in \Pb$ below which, for every Turing functional $\Gamma$, the set
\[
D_\Gamma = \{ c' \in \Pb: c' \Vdash \Gamma^{G_0 \oplus \dots \oplus G_{d-1} \oplus Z} \neq C \}
\]
is $\Pb$-dense.
Let $\Uc$ be a $\Pb$-filter which intersects every set $D_\Gamma$.
Then by definition of a product tree condition, $G^\Uc_0, \dots, G^\Uc_{d-1}$ are strong subtrees of $T_0, \dots, T_{d-1}$. Moreover, since all conditions in $\Pb$ are level-homogeneous, so are $G^\Uc_0, \dots, G^\Uc_{d-1}$. Since  $\Uc$ intersects every set $D_\Gamma$, then $C \nTred G^\Uc_0 \oplus \dots \oplus G^\Uc_{d-1} \oplus Z$.
Last, by \Cref{lem:product-tree-genericity-implies-infinity}, $G^\Uc_0, \dots, G^\Uc_{d-1}$ are all infinite.

Let $f: \NN \to k$ be the function which on level $\ell$ associates the color $i < k$
such that $G^\Uc_0(\ell) \times \dots \times G^\Uc_{d-1}(\ell) \subseteq A_i$.
By strong cone avoidance of $\RT{1}{k}$, there is an infinite set of levels $H \subseteq \NN$ on which $f$ is constant and
%$f$-homogeneous for some color $i < k$ and such that
$C \nTred H \oplus G^\Uc_0 \oplus \dots \oplus G^\Uc_{d-1} \oplus Z$. Say $f$ takes the color $i < k$ on $H$. In particular, for every $\ell \in H$, $G^\Uc_0(\ell) \times \dots \times G^\Uc_{d-1}(\ell) \subseteq A_i$, so we can $H \oplus G^\Uc_0 \oplus \dots \oplus G^\Uc_{d-1} \oplus Z$-computably thin out to infinite strong subtrees $S_0, \dots, S_{d-1}$
of  $G^\Uc_0, \dots, G^\Uc_{d-1}$ with common level function,
and such that $\bigcup_n S_0(n) \times \dots \times S_{d-1}(n) \subseteq A_i$.
In particular, $C \nTred S_0 \oplus \dots \oplus S_{d-1} \oplus Z$.
This completes the proof of \Cref{thm:hl-strong-cone-avoidance}.
\end{proof}

\chapter{Milliken's tree theorem}\label{sect:milliken-theorem}

We now turn to the computability-theoretic analysis of the product and non-product versions Milliken's tree theorem, the base cases of which we already studied through the Halpern-Lauchli theorem in the previous chapter. As the product version obviously implies the non-product, we formulate our upper bounds in terms of the former and our lower bounds in terms of the latter. More specifically, we obtain the following.
%This section is divided into three parts:
In \Cref{subsect:proof-pmt-aca}, we provide an inductive proof of the product version of Milliken's tree theorem in $\ACA_0$, using the notion of prehomogeneous tree. Using standard methods, it is easy to obtain a reversal for (even the non-product version of) Milliken's tree theorem for height at least $3$. For height $1$, we already saw in the previous chapter that the product version of Milliken's tree theorem for height 1 is computably true, and hence does not imply $\ACA_0$. This leaves the situation for trees of height $2$, which we address in \Cref{subsect:cone-avoidance-pmt2}. Since Milliken's tree theorem for height two implies Ramsey's theorem for pairs, it is not computably true, but we show that the product version admits cone avoidance, and so is strictly weaker than $\ACA_0$. Finally, in \Cref{subsect:thin-milliken}, we study a weakening of Milliken's tree theorem that allows more than one color in the solutions. We prove that the product version of Milliken's tree theorem for height 3, but where up to two colors are allowed in the solution, admits cone avoidance, and hence does not imply $\ACA_0$. We will make use of this result in our discussion of Devlin's theorem in \Cref{sec:devlin}.

\section{A proof of $\PMT{n}{}$ in $\ACA_0$}\label{subsect:proof-pmt-aca}

%Fix $n\in\NN$,  a collection of finitely branching trees with no leaves $T_0, \dots, T_{d-1}$ and a coloring $f: \Subtree{n+1}{T_0, \dots, T_{d-1}} \to k$.
Given a tree $F$ of height $\alpha \leq \omega$ and an a number $n < \alpha$,  we write $F \uh n$ \index{$\uh$} for the subtree of $F$ of height $n$.

\begin{definition}\index{prehomogeneous!product tree condition}\index{prehomogeneous!tuple}\index{product tree condition!prehomogeneous}
	%[Prehomogeneity]
	Fix $n\in\NN$,  a collection of trees $T_0, \dots, T_{d-1}$ with no leaves and a coloring $f: \Subtree{n+1}{T_0, \dots, T_{d-1}} \to k$.
	\begin{enumerate}
		\item A tuple $(S_0, \dots, S_{d-1}) \in \Subtree{\omega}{T_0, \dots, T_{d-1}}$ is \emph{prehomogeneous} for $f$ if the color of every $(E_0, \dots, E_{d-1}) \in \Subtree{n+1}{S_0, \dots, S_{d-1}}$ depends only on $(E_0 \uh n, \dots, E_{d-1} \uh n)$.
		\item A product tree condition $(F_0, \dots, F_{d-1}, X_0, \dots,X_{d-1})$ is \emph{prehomogeneous} for $f$  if the color of every
		\[
		(E_0, \dots, E_{d-1}) \in \Subtree{n+1}{F_0 \cup X_0, \dots, F_{d-1} \cup X_{d-1}}
		\]
		%\benoit{do you mean $\mathcal{S}_{n+1}$  ?}
		depends only on $( E_0 \uh n, \dots, E_{d-1} \uh n )$ whenever $E_j \uh n \subseteq F_j$ for every $j < d$.
	\end{enumerate}
\end{definition}

%\begin{definition}
%A product tree condition $(F_0, \dots, F_{d-1}, X_0, \dots,X_{d-1})$ is \emph{prehomogeneous} for $f$  if the color of every
%$$
%(E_0, \dots, E_{d-1}) \in \Subtree{n+1}{F_0 \cup X_0, \dots, F_{d-1} \cup X_{d-1}}
%$$
%%\benoit{do you mean $\mathcal{S}_{n+1}$  ?}
%depends only on $( E_0 \uh n, \dots, E_{d-1} \uh n )$ whenever $E_j \uh n \subseteq F_j$ for every $j < d$.
%\end{definition}

\noindent In particular, note that the product tree condition $(\emptyset, \dots, \emptyset, T_0, \dots,T_{d-1})$ is prehomogeneous for a given $f$ as above.

We add several other useful definitions.

\begin{definition}
	Fix a collection of infinite trees $T_0, \dots, T_{d-1}$ with no leaves. A product tree condition $c = (F_0, \dots, F_{d-1}, X_0, \dots,X_{d-1})$ is \emph{computable} if $X_0, \dots, X_{d-1}$ are all computable and computably bounded. An \emph{index} of $c$ is a finite tuple  $(F_0, \dots, F_{d-1}, e_0, \dots,e_{d-1})$ such that $\Phi_{e_j} = X_j$ for every $j < d$.\index{product tree condition!index}\index{index!product tree condition}\index{product tree condition!computable}
\end{definition}

%We will call product tree condition $c = (F_0, \dots, F_{d-1}, X_0, \dots,X_{d-1})$ \emph{computable}
%if $X_0, \dots, X_{d-1}$ are all computable and computably bounded. An \emph{index} of $c$ is a finite tuple 
%$(F_0, \dots, F_{d-1}, e_0, \dots,e_{d-1})$ such that $\Phi_{e_j} = X_j$ for every $j < d$.
%Assume from now on that the coloring $f$ is computable.

%We shall use the following notation for the remainder of the paper.

\begin{definition}\label{def:subtreeleaves}\index{$\SubtreeLeaves{n}{}$!tree}\index{$\SubtreeLeaves{n}{}$!product}
  If $n\geq 1$ and $T$ is a finite tree, then%define the following notation for subtrees that keep the leaves:
  % Given a finite tree $T$ of height $h$ and an embedding type $\embfont e$ of height $n$,
   \[
     \SubtreeLeaves{n}{T} = \{S\in\Subtree{n}{T}: \leaves(S)\subseteq \leaves(T)\}.
   \]
   % And if $\embfont e$ is a product embedding type: %
   More generally, if {$T_0, \dots, T_{d-1}$} are finite trees, then $\SubtreeLeaves{n}{T_0,\dots,T_{d-1}}$ equals
  % trees.
  % \todo{If we speak about product of trees more than just here, we need to define $\Subtree{\embfont e}{T_0,\dots,T_{d-1}}$ in the intro on trees.}
  \[
    \{(S_0,\dots, S_{d-1})\in\Subtree{n}{T_0,\dots,T_{d-1}}: (\forall i<d)[S_i\in \SubtreeLeaves{n}{T_i}]\}.
  \]
%  \benoit{wait, what ? what is $\mathcal{S}_n^l(T_i)$ ? From the latex file, it seems that these are the subtrees whose leaves are the leaves of $T_i$. Is that correct ? in this case $T_i$ should be introduced as finite.}
\end{definition}

The main combinatorial result of this section is the following density lemma.

\begin{lemma}\label{thm:milliken-prehomogeneous-one-step}
Fix $n \in \NN$, a collection of computable, computably bounded trees with no leaves $T_0, \dots, T_{d-1}$, and a computable coloring
\[
	f: \Subtree{n+1}{T_0, \dots, T_{d-1}} \to k.
\]
For every computable product tree condition $c = (F_0, \dots, F_{d-1}, X_0, \dots,X_{d-1})$ which is prehomogeneous (for $f$), there is a computable prehomogeneous product tree condition
$\hat{c} = (\hat{F}_0, \dots, \hat{F}_{d-1}, \hat{X}_0, \dots, \hat{X}_{d-1})$ extending $c$
such that $F_j \subsetneq \hat{F}_j$ for every $j < d$. Moreover, an index of $d$ can be found uniformly $\emptyset''$-computably from an index of $c$.
\end{lemma}
\begin{proof}
By definition of a product tree condition (\Cref{def:product-tree-condition}), for every $j < d$ and every leaf $\sigma$ of $F_j$,
$\roots(X_j)$ is $(t+1)$-$\sigma$-dense with respect to $T_j$, where $t$ is the level of the leaves of $F_j$ within $T_j$.
For every $j < k$, let $\hat{F}_j$ be $F_j$ augmented by the roots of $X_j$ extending the leaves of $F_j$. By \Cref{remark:product-tree-condition-roots}, we can assume that $(\hat{F}_0, \dots, \hat{F}_{d-1}) \in \Subtree{<\omega}{T_0, \dots, T_{d-1}}$.  Let 
\[
( E^0_0, \dots, E^0_{d-1} ), \dots, ( E^{p-1}_0, \dots, E^{p-1}_{d-1} )
\]
be the (finite) enumeration of all the tuples in $\SubtreeLeaves{n}{\hat{F}_0, \dots, \hat{F}_{d-1}}$, meaning tuples of strong subtrees $( E_0, \dots, E_{d-1} )$ such that the leaves of $E_j$ are among the leaves of $\hat{F}_j$, i.e., belong to $X_j(0)$.

We inductively define a finite sequence of $d$-tuples of computable forests
$$
( Y^0_0, \dots, Y^0_{d-1} ), \dots, ( Y^p_0, \dots, Y^p_{d-1} )
$$
such that for every $s < p$:
\begin{enumerate}
	\item $Y^{s+1}_0, \dots, Y^{s+1}_{d-1}$ are infinite strong subforests of $Y^s_0, \dots, Y^s_{d-1}$, respectively, with common level function;
	\item $(\hat{F}_0 \cup Y^{s+1}_0, \dots, \hat{F}_{d-1} \cup Y^{s+1}_{d-1}) \in \Subtree{\omega}{T_0, \dots, T_{d-1}}$;
	\item there is some color $i < k$ such that for every level $\ell \in \NN$, every $j < d$, and every $H_j \subseteq Y^{s+1}_j(\ell)$ for which $( E^s_0 \cup H_0, \dots, E^s_{d-1} \cup H_{d-1} ) \in \Subtree{n+1}{T_0, \dots, T_{d-1}}$,
		$f( E^s_0 \cup H_0, \dots, E^s_{d-1} \cup H_{d-1} ) = i$.
\end{enumerate}
Let $Y^0_0, \dots, Y^0_{d-1}$ be $X_0, \dots, X_{d-1}$, respectively, trimmed by their first levels.
Assume $Y^s_0, \dots, Y^s_{d-1}$ is defined for $s < p$. Let $m$ be the common level of the leaves of $\hat{F}_0, \dots, \hat{F}_{d-1}$ in $T_0, \dots, T_{d-1}$, respectively. For every $j < d$, let $R_j = \roots(Y^s_j)$, and for every $\rho \in R_j$, let $Y_{j,\rho} = Y^s_j \uh \rho$.
We can see $Y^s_0, \dots, Y^s_{d-1}$ as a tuple $( Y_{j,\rho}: j < d, \rho \in R_j )$ of trees.

%$G_j = E^s_j \cup \{\sigma_{j,\rho} \in Y_{j,\rho}: \rho \in U_j\}$.
%$G_j = E^s_j \cup \bigcup_{\rho \in U_j} Y_{j,\rho}$.
Define a coloring $g$ of%\benoit{the variable $n$ is already used globally}
$$
		\bigcup_m  \left( \prod_{\rho \in R_0} Y_{0, \rho}(m) \right) \times \dots \times \left( \prod_{\rho \in R_{d-1}} Y_{d-1,\rho}(m) \right)
$$
as follows. For every $j < d$, let $U_j = \{\rho \in R_j: (\exists \mu \in \leaves(E^s_j)) [\rho \succeq \mu ] \}$, and note that $( E^s_0 \cup U_0, \dots, E^s_{d-1} \cup U_{d-1} ) \in \Subtree{n+1}{T_0, \dots, T_{d-1}}$. Now, given $\pi = \{\sigma_{j,\rho} \in Y_{j,\rho}: j < d, \rho \in R_j\}$ in the domain of $g$, let
\[
	G_j = E^s_j \cup \{\sigma_{j,\rho}: \rho \in U_j\}.
\]
for each $j$. So $( G_0, \dots G_{d-1} ) \in \Subtree{n+1}{T_0, \dots, T_{d-1}}$. Set
\[
	g(\pi) = f(G_0,\ldots,G_{d-1}).
\]
% $g(\{\sigma_{j,\rho} \in Y_{j,\rho}: j < d, \rho \in R_j\}) = f(G_0, \dots, G_{d-1})$. %\benoit{I don't think $H_i$ is a tree of height $n+1$. Maybe none of them are because the product over all the roots maybe too big, as we may have many roots. I think this is annoying to write properly: the products of root we select should depend on each $E^s_j$. }\ludovic{Do you maintain this claim with this new proof?}
%defined for every $j < d$ by $U_j = \{\rho \in R_j: \exists \mu \in \leaves(E^s_j) \rho \succeq \mu\}$ and $G_j = E^s_j \cup \{\sigma_{j,\rho}: \rho \in U_j\}$.
%Note that $\langle E^s_0 \cup U_0, \dots, E^s_{d-1} \cup U_{d-1} \rangle \in \Subtree{n+1}{T_0, \dots, T_{d-1}}$,
%hence  $\langle G_0, \dots G_{d-1} \rangle \in \Subtree{n+1}{T_0, \dots, T_{d-1}}$.

Since the Halpern-La\"{u}chli theorem is computably true%(\Cref{thm:halpern-lauchli-computably-true})
, there is a computable tuple $( Z_{j,\rho}: j < d, \rho \in R_j )$ of strong subtrees of $( Y_{j,\rho}: j < d, \rho \in R_j )$, respectively, with common level function, together with a color $i < k$
such that for every $\ell \in \NN$, every $j < k$, if $H_j \subseteq \prod_{\rho \in R_j} Z_{j,\rho}(\ell)$ is such that $E^s_j \cup H_j \in \mathcal{S}_{n+1}(T_0, \dots, T_{d-1})$ then $f( E^s_0 \cup H_0, \dots, E^s_{d-1} \cup H_{d-1} ) = i$.
For every $j < k$, let $Y^{s+1}_j = \bigcup_{\rho \in R_j} Z_{j,\rho}(\ell)$. This completes the construction of the sequence.

Let $\hat{c} = (\hat{F}_0, \dots, \hat{F}_{d-1}, Y^p_0, \dots, Y^p_{d-1})$. By items 1 and 2, 
$\hat{c}$ is a computable product tree condition extending $c$.
Moreover, by item 3 and the fact that $c$ is prehomogeneous for $f$, so is $\hat{c}$. 

One can $\emptyset''$-computably search for a finite tuple
\[
	(E_0, \dots, E_{d-1}, e_0, \dots, e_{d-1})
\]
such that for every $j < d$, $\Phi_{e_j}$ is total, and  \[(E_0, \dots, E_{d-1}, \Phi_{e_0}, \dots, \Phi_{e_{d-1}})\] is a product tree condition extending $c$ and prehomogeneous for $f$. Indeed, being a strong subforest of $T_j$ is $\Pi^0_2$ since $T_j$ is computable and computably bounded. Thus, being a product tree condition is $\emptyset''$-decidable. Moreover, being prehomogeneous is $\Pi^0_1$ since $f$ is computable, and being an extension of a product tree condition is also $\Pi^0_2$. Since we prove the existence of such an extension, an exhaustive search will always terminate, and the procedure is $\emptyset''$-computable, uniformly in an index of $c$.
This completes the proof of \Cref{thm:milliken-prehomogeneous-one-step}.
\end{proof}

\begin{lemma}\label{thm:milliken-prehomogeneous}
Fix $n \in \NN$, a collection of computable, computably bounded trees with no leaves $T_0, \dots, T_{d-1}$, and a computable coloring
\[
	f: \Subtree{n+1}{T_0, \dots, T_{d-1}} \to k.
\]
There is a $\Delta^0_3$ sequence $S_0, \dots, S_{d-1}$ of strong subtrees of $T_0, \dots,\allowbreak T_{d-1}$, respectively, with common level function, such that the tuple $( S_0, \dots, S_{d-1} )$ is prehomogeneous for~$f$.
\end{lemma}
\begin{proof}
By iterating \Cref{thm:milliken-prehomogeneous-one-step},
build a $\Delta^0_3$ descending sequence of computable prehomogeneous product tree conditions
$c_0 \geq c_1 \geq \dots$ where 
$$
c_s = (F^s_0, \dots, F^s_{d-1}, X^s_0, \dots, X^s_{d-1})
$$
and such that  $F^s_j \subsetneq F^{s+1}_j$ for every $j < d$ and $s \in \NN$.
For every $j < d$, let $S_j = \bigcup_s F^s_j$. Since the $F^s_j$ are strictly increasing in $s$, it follows by definition of a product tree condition that $S_0, \dots, S_{d-1}$ are strong subtrees of $T_0, \dots,\allowbreak T_{d-1}$, respectively, with common level function.
Moreover, $S_0, \dots, S_{d-1}$ are $\Delta^0_3$, and by definition of a prehomogeneous condition, $( S_0, \dots, S_{d-1})$ is prehomogeneous for~$f$.
\end{proof}

\begin{theorem}\label{thm:milliken-arithmetic}
For every $n \geq 1$ and every set $X$, every $X$-computable instance of the product version of Milliken's tree theorem for height $n$ admits a $\Delta^{0,X}_{2n-1}$ solution.
\end{theorem}
\begin{proof}
By induction on $n$. For $n = 1$, the product version of Milliken's tree theorem for height 1 is the Halpern-La\"{u}chli theorem, which is computably true by \Cref{thm:halpern-lauchli-computably-true}.

Suppose the property holds for $n$, and fix a set $X$,
and an $X$-computable sequence of $X$-computably bounded trees with no leaves $T_0, \dots, T_{d-1} \subseteq \baire$. 
Let $f: \Subtree{n+1}{T_0, \dots, T_{d-1}} \to k$ be an $X$-computable coloring.
By \Cref{thm:milliken-prehomogeneous}, relativized to $X$, there is a $\Delta^{0,X}_3$ tuple
\[
	(S_0, \dots, S_{d-1}) \in \Subtree{\omega}{T_0, \dots,\allowbreak T_{d-1}}
\]
prehomogeneous for~$f$.
Let $g: \Subtree{n}{S_0, \dots, S_{d-1}} \to k$ be defined by
\[
	g( E_0, \dots, E_{d-1} ) = f( E_0 \cup H_0, \dots, E_{d-1} \cup H_{d-1} )
\]
for any $H_0 \subseteq S_0(n), \dots, H_{d-1} \subseteq S_{d-1}(n)$ such that $( E_0 \cup H_0, \dots, E_{d-1} \cup H_{d-1} ) \in \Subtree{n+1}{S_0, \dots, S_{d-1}}$. Such a coloring is well defined by prehomogenenity.
The coloring $g$ can be seen as a $\Delta^{0,X''}_1$ instance of the product version of Milliken's tree theorem for height $n$.
By induction hypothesis, there is a $\Delta^{0, X''}_{2n-1}$ (hence $\Delta^{0,X}_{2(n+1)-1}$) solution to $g$, which is by prehomogeneity also a solution to $f$. This completes the proof of \Cref{thm:milliken-arithmetic}.
\end{proof}

\begin{corollary}\label{thm:milliken-aca}
  For every $n \geq 1$, the product version of Milliken's tree theorem for height $n$ is provable in
  $\ACA_0$, and the product version of Milliken's tree theorem itself is provable in $\ACA'_0$.
\end{corollary}
\begin{proof}
The proof of \Cref{thm:milliken-arithmetic} is formalizable in $\ACA_0$.
The induction on $n$ can then be carried out in $\ACA'_0$.
\end{proof}

%\todo[inline]{``Unbury'' this theorem?}

\begin{theorem}\label{thm:milliken-rt}
  Milliken's tree theorem for height $n$ implies $\RT{n}{}$.
\end{theorem}
\begin{proof}
Let $f: [\NN]^n \to k$ be an instance of $\RT{n}{}$.
Let $T = 1^{<\omega} = \{\epsilon, 0, 00, \dots \}$ be the unary finitely branching tree with no leaves. Define $g: \Subtree{n}{T} \to k$ by $g(\sigma_0, \dots, \sigma_{n-1}) = f(|\sigma_0|, \dots, |\sigma_{d-1}|)$. 
Now if $S$ is a strong subtree of $T$ such that $\Subtree{n}{S}$
is monochromatic for~$g$ then $H = \{|\sigma|: \sigma \in S \}$ is homogeneous for~$f$.
\end{proof}

\begin{corollary}
  For every $n\geq 3$, $\PMT n{}$ and $\MT n{}$ are equivalent to $\ACA_0$ over
  $\RCA_0$. Moreover the product version of Milliken's tree theorem and Milliken's tree theorem are equivalent to $\ACA'_0$.
\end{corollary}
\begin{proof}
For every $n \geq 3$, by \Cref{thm:milliken-aca}, $\ACA_0$ implies $\PMT n{}$,
which generalizes $\MT n{}$. By \Cref{thm:milliken-rt}, $\MT n{}$ imples $\RT{n}{}$,
and by formalization of a result of Jockusch~\cite[Theorem 5.7]{Jockusch1972Ramseys} (as formalized e.g. in~\cite{Simpson2009Subsystems}, Lemma III.7.5), $\RT{n}{}$ implies $\ACA_0$.
Moreover, by \Cref{thm:milliken-aca}, $\ACA'_0$ implies $(\forall n)\PMT n{}$
which generalizes $(\forall n)\MT n{}$. By \Cref{thm:milliken-rt}, $(\forall n)\MT n{}$ implies $(\forall n)\RT{n}{}$,
which is itself known to imply $\ACA'_0$ (for a proof, see~Hirschfeldt~\cite{Hirschfeldt2015Slicing}, Theorem 6.27).
\end{proof}

\section{Cone avoidance of $\PMT{2}{}$}\label{subsect:cone-avoidance-pmt2}

This section is devoted to the proof of cone avoidance of the product version of Milliken's tree theorem for height 2. As in the proof of cone avoidance for Ramsey's theorem for pairs (see Cholak, Jockusch and Slaman~\cite{Cholak2001strength}, Sections 3 and 4) the proof of \Cref{thm:cone-avoidance-MTT2} will be decomposed into two steps, using the notion of stability.

\begin{definition}\index{stable!coloring}\index{coloring!stable}
%[Stability]
Fix $n \geq 1$ and a collection of trees $T_0, \dots, T_{d-1}$ with no leaves. A coloring $f: \Subtree{n+1}{T_0, \dots, T_{d-1}} \to k$ is \emph{stable} if
for every $( F_0, \dots, F_{d-1} ) \in \Subtree{n}{T_0, \dots, T_{d-1}}$, there is a threshold $t \in \NN$ and a color $i < k$ such that for every level $\ell \geq t$ and all $E_0 \subseteq T_0(\ell), \dots, E_{d-1} \subseteq T_{d-1}(\ell)$
for which $( F_0 \cup E_0, \dots, F_{d-1} \cup E_{d-1} ) \in \Subtree{n+1}{T_0, \dots, T_{d-1}}$, $f( F_0 \cup E_0, \dots, F_{d-1} \cup E_{d-1} ) = i$.
\end{definition}

We refer to the $i < k$ above as the \emph{limit color} of the tuple $( F_0, \dots, F_{d-1} )$. Any stable coloring $f: \Subtree{n+1}{T_0, \dots, T_{d-1}} \to k$ induce{s}
a coloring 
$$
g: \Subtree{n}{T_0, \dots, T_{d-1}} \to k
$$
which to $( F_0, \dots, F_{d-1} ) \in \Subtree{n}{T_0, \dots, T_{d-1}}$ associates its limit color $i < k$.
We shall call $g$ the \emph{limit coloring} of $f$\index{coloring!limit}. Note that $g$ is $\Delta^0_2$ in $f$ and the sequence $T_0, \dots, T_{d-1}$. The notion of stability is therefore as bridge between computable instances of $\PMT{n+1}{}$ and arbitrary instances of $\PMT{n}{}$. This gives rise to a two step proof of cone avoidance of $\PMT{2}{}$.

The first step consists of proving that for every instance of the product version of Milliken's tree theorem for height 2 there exist cone avoiding strong subtrees on which the coloring is stable. We will actually prove a more general theorem for products of trees, and subtrees of arbitrary height.% embedding types \benoit{embedding types has never been mentioned so far} of arbitrary height. 

The second step consists of applying \emph{strong} cone avoidance of the product version of Milliken's tree theorem for height 1, which is just a particular case of the Halpern-La\"{u}chli theorem, and then computably thinning out the result to obtain a solution to the original instance of the product version of Milliken's tree theorem of height 2.

%The proof of cone avoidance of the existence of stable subdomain for every coloring uses cone avoidance of the Halpern-La\"{u}chli theorem. Since the Halpern-La\"{u}chli theorem admits strong cone avoidance, we can prove an actually stronger statement:

We begin with the first step.

\begin{theorem}\label{thm:cmtt-admits-strong-cone-avoidance}
Fix sets $C,Z \subseteq \NN$ with $C \nTred Z$, an $n \geq 1$, a 
%Fix $n \geq 1$, Fix two sets $C$ and $Z$ such that $C \nTred Z$.
$Z$-computable collection of $Z$-computably bounded trees $T_0, \dots, T_{d-1}$ with no leaves, and a coloring $f: \Subtree{n+1}{T_0, \dots, T_{d-1}} \to k$. 
%For every coloring $f: \Subtree{n+1}{T_0, \dots, T_{d-1}} \to k$,
There exists $(S_0, \dots,\allowbreak S_{d-1}) \in\Subtree{\om}{T_0, \dots, T_{d-1}}$ such that $f$ is stable on $\Subtree{n+1}{S_0, \dots, S_{d-1}}$ and such that $C \nTred S_0 \oplus \dots \oplus S_{d-1} \oplus Z$.
\end{theorem}

The proof of \Cref{thm:cmtt-admits-strong-cone-avoidance} will employ a refinement of the forcing with product tree conditions. We will require some definitions and preliminary lemmas.
%To this end, assume $C$, $Z$, and the $T_i$ are fixed.
%, and $f$ are fixed.

%From now on, fix $C$, $Z$, and $T_0, \dots, T_{d-1}$. Also fix $n\geq 1$ and a coloring $f: \Subtree{n+1}{T_0, \dots, T_{d-1}} \to k$.

\begin{definition}\index{product tree condition!cone avoiding}\index{product tree condition!stable} \index{stable!product tree condition}
	Fix sets $C,Z \subseteq \NN$ with $C \nTred Z$, an $n \geq 1$, a 
	%Fix $n \geq 1$, Fix two sets $C$ and $Z$ such that $C \nTred Z$.
	$Z$-computable collection of $Z$-computably bounded trees $T_0, \dots, T_{d-1}$ with no leaves, and a coloring $f: \Subtree{n+1}{T_0, \dots, T_{d-1}} \to k$. Let
	\[
		c = (F_0, \dots, F_{d-1}, X_0, \dots, X_{d-1})
	\]
	be a product tree condition (with respect to the $T_i$).
	\begin{enumerate}
		\item $c$ is \emph{cone avoiding} if $C \nTred X_0 \oplus \dots \oplus X_{d-1} \oplus Z$.
		\item $c$ is \emph{stable} for $f$ if for every tuple $( E_0, \dots, E_{d-1} ) \in \Subtree{n}{F_0, \dots, F_{d-1}}$, there is a color $i < k$ such that for every level $\ell \in \NN$ and every $H_0 \subseteq X_0(\ell), \dots, H_{d-1} \subseteq X_{d-1}(\ell)$ for which $( E_0 \cup H_0, \dots, E_{d-1} \cup H_{d-1} ) \in \Subtree{n+1}{T_0, \dots, T_{d-1}}$, $f( E_0 \cup H_0, \dots, E_{d-1} \cup H_{d-1} ) =~i$.
	\end{enumerate}
%A product tree condition $(F_0, \dots, F_{d-1}, X_0, \dots, X_{d-1})$ is \emph{cone avoiding} if $C \nTred X_0 \oplus \dots \oplus X_{d-1} \oplus Z$. 
%It is \emph{stable} for $f$  if for every $\langle E_0, \dots, E_{d-1} \rangle \in \Subtree{n}{F_0, \dots, F_{d-1}}$, there is a color $i < k$ such that for every level $\ell \in \NN$ and every $H_0 \subseteq X_0(\ell), \dots, H_{d-1} \subseteq X_{d-1}(\ell)$
%such that $\langle E_0 \cup H_0, \dots, E_{d-1} \cup H_{d-1} \rangle \in \Subtree{n+1}{T_0, \dots, T_{d-1}}$, $f(\langle E_0 \cup H_0, \dots, E_{d-1} \cup H_{d-1} \rangle) =~i$.
\end{definition}

%For the remainder of the section, let $n \geq 1$ and $f: \Subtree{n+1}{T_0, \dots, T_{d-1}} \to k$ be fixed. Let $\Pb$ be the partial order of all stable cone avoiding product tree conditions (with respect to $f$ and $Z$).
%As explained in \Cref{subsect:product-tree-forcing}, every $\Pb$-filter $\Uc$ induces a $d$-tuple of (finite or infinite) strong subtrees $G^\Uc_0, \dots, G^\Uc_{d-1}$ of $T_0, \dots, T_{d-1}$, respectively, with common level function. 

% The following lemma is useless by \Cref{lem:product-tree-genericity-implies-infinity}.
%\begin{lemma}\label{thm:mtt2-stable-one-step}
%For every $\Pb$-condition $c = (F_0, \dots, \allowbreak F_{d-1}, X_0, \dots,X_{d-1})$, there is a $\Pb$-condition
%$d = (\hat{F}_0, \dots, \hat{F}_{d-1}, \hat{X}_0, \dots, \hat{X}_{d-1})$ extending $c$
%such that $F_j \subsetneq \hat{F}_j$ for every $j < d$. 
%\end{lemma}
%\begin{proof}
%The proof is exactly the same as \Cref{thm:milliken-prehomogeneous-one-step}.
%Indeed, given a cone avoiding stable product tree condition, the proof of \Cref{thm:milliken-prehomogeneous-one-step} yields a stable product tree condition whose forests are computable from the former condition, hence cone avoiding.
%\end{proof}

Making progress in satisfying the cone avoidance requirements will demand the use of a computable function dominating the levels of a tuple of strong subtrees with certain nice combinatorial properties. %We need the following definition:

% Recall that strong subtrees preserve the number of direct extension, therefore the leaves of a strong subtree $S$ of a tree $T$ must be included in the leaves of $T$. However, this might not be the case for level-closed subtrees, justifying the definition $\SubtreeLeaves{\cdot}{\cdot}$. 

For now, we will take for granted the following technical result, which is a finite version of Milliken's tree theorem where all subtrees are assumed to keep the leaves and the level function is bounded. For a given tree $T$, recall the notation $\SubtreeLeaves{n}{T}$ from \Cref{def:subtreeleaves} which denotes the collection of strong subtrees of $T$ of height $n$ whose leaves are among those of $T$.

%\begin{restatable}{theorem}{widget}
%  \label{thm:main-widget-theorem}
%  Fix a number of colors $k\in\Nb$, heights $N,n\geq 1$, a function $b:\om\to\om$ and a level $\ell\in\Nb$. Then there exists $h=\fwidg(N,\ell,n+1,k,d,b)$ as follows. If $T_0,\dots, T_{d-1}$ is a sequence of finite $b$-bounded trees of height $h$, and
%  %$f$ is any coloring
%  \[f:\SubtreeLeaves{n+1}{T_0,\dots,T_{d-1}}\to k\]
%  is any coloring where $f(F_0,\dots, F_{d-1})$ depends only on $(F_0\uh n,\dots, F_{d-1}\uh n)$ whenever $F_{i}\uh n\subseteq T_i\uh \ell$ for every $i$,  then there exists $(S_0, \dots, S_{d-1}) \in\SubtreeLeaves{\ell+N+1}{T_0,\dots, T_{d-1}}$ such that:
%  \begin{enumerate}
%  \item\label{item:stems-widget} $S_i\uh \ell = T_i\uh \ell$;
%  \item\label{item:hbound-widget} for any $i<d$, the level function of $S_i$ as a subset of $T_i$ is bounded by the function defined by $x\mapsto\fwidg(x,\ell, n+1,k,d, b)$ if $x>\ell$, and $x\mapsto x$ if $x\leq\ell$;
%  \item\label{item:monochr-widget}
%    the color of $(F_0,\dots, F_{d-1})\in \SubtreeLeaves{n+1}{S_0,\dots,S_{d-1}}$ depends only on $(F_0\uh n,\dots, F_{d-1}\uh n)\in \Subtree{n}{S_0,\dots,S_{d-1}}$.
%  \end{enumerate}
%\end{restatable}

\begin{restatable}{theorem}{widget}
  \label{thm:main-widget-theorem}
	%There exists a function $\fwidg: \NN^6 \to \NN$ as follows.
	Fix a level $\ell \in \NN$, a height $n \geq 1$, a number of colors $k \in \NN$, an arity $d \geq 1$, and a function $b:\om\to\om$. There exists a function $N \mapsto \fwidg(N,\ell,n+1,k,d,b)$, uniformly $b$-computable in $\ell$, $n$, $k$, and $d$, as follows. If $U_0,\dots, U_{d-1}$ is a sequence of finite $b$-bounded trees of height $h = \fwidg(N,\ell,n+1,k,d,b)$ for some fixed $N \in \NN$, and
  %$f$ is any coloring
  \[\chi:\SubtreeLeaves{n+1}{U_0,\dots,U_{d-1}}\to k\]
  is any coloring where $\chi(F_0,\dots, F_{d-1})$ depends only on $(F_0\uh n,\dots, F_{d-1}\uh n)$ whenever $F_{i}\uh n\subseteq U_i\uh \ell$ for every $i$,  then there exists $(V_0, \dots, V_{d-1}) \in\SubtreeLeaves{\ell+N+1}{U_0,\dots, U_{d-1}}$ such that:
  \begin{enumerate}
  \item\label{item:stems-widget} $V_i\uh \ell = U_i\uh \ell$ for each $i < d$;
  \item\label{item:hbound-widget} for any $i<d$, the level function of $V_i$ as a subset of $U_i$ is bounded by the function defined by $x\mapsto\fwidg(x,\ell, n+1,k,d, b)$ if $x>\ell$, and $x\mapsto x$ if $x\leq\ell$;
  \item\label{item:monochr-widget}
    the color of $(F_0,\dots, F_{d-1})\in \SubtreeLeaves{n+1}{V_0,\dots,V_{d-1}}$ depends only on $(F_0\uh n,\dots, F_{d-1}\uh n)\in \Subtree{n}{V_0,\dots,V_{d-1}}$.
  \end{enumerate}
\end{restatable}
 
To help understand the statement of \Cref{thm:main-widget-theorem},
suppose $S_0, \dots, S_{d-1}$ are infinite, computable and computably bounded trees with no leaves. Also fix a coloring $g: \Subtree{n+1}{S_0, \dots, S_{d-1}} \to k$.
Consider a product tree condition $(E_0, \dots, E_{d-1}, X_0, \dots, X_{d-1})$ for these $S_i$ which is stable for $g$. Say the $E_i$ are of height $\ell$. One would like to extend the stems with $N$ new levels in one step, so that the resulting stems are of height $\ell+N$, while keeping the resulting product tree condition stable for $g$. 
\Cref{thm:main-widget-theorem} provides a sufficient bound $h = \fwidg(N,\ell,n+1,k,d,b)$ depending on the number $N$ of new levels we would like to add, on the height $\ell$ of the stems, the parameters $n+1$ and $k$ of the coloring $g$,
%$g: \Subtree{n+1}{S_0, \dots, S_{d-1}} \to k$,
on the number $d$ of trees in the product tree condition, and on the computable bound $b$ over the trees $E_0 \cup X_0, \dots, E_{d-1} \cup X_{d-1}$,
 	so that one can always find such an extension of the stems where the new elements are taken among the first $h$ first levels of  $E_0 \cup X_0, \dots, E_{d-1} \cup X_{d-1}$.
 	
 In the statement of \Cref{thm:main-widget-theorem}, the finite trees $U_0, \dots, U_{d-1}$
 correspond to the trees $E_0 \cup X_0, \dots, E_{d-1} \cup X_{d-1}$ up to level $h$, respectively. 
  Let $Y_0, \dots, Y_{d-1}$ be the forests obtained from the trees $E_0 \cup X_0, \dots, E_{d-1} \cup X_{d-1}$ by removing their first $h-1$ many levels.
 For each $j < d$, the tree $U_j$ therefore has three parts. First, we have the first $\ell$ levels, which correspond to to the stem $E_j$. Second, we have the levels up to the one before the leaves, which will serve to extend the stem $E_j$. Very few of these levels will be kept, but $h$ is chosen large enough so that we can always extend with $N$ new levels. Last, the leaves of $U_j$ correspond to the roots of the forest $Y_j$.
 
Fixing strong subtrees $(F_0, \dots, F_{d-1}) \in \SubtreeLeaves{n+1}{U_0, \dots, U_{d-1}}$ of height $n+1$ should actually be understood as fixing strong subtrees $(F_0 \uh n, \dots, F_{d-1} \uh n)$ of height $n$ from the trees $U_0, \dots, U_{d-1}$ trimmed from their leaves,
and then picking a set of roots from $Y_0, \dots, Y_{d-1}$ (or equivalently picking a set of leaves from $U_0, \dots, U_{d-1}$). This induces a product coloring of the nodes in $Y_0, \dots, Y_{d-1}$ pointwise extending the product of the roots chosen, by considering which color the function $g$ assigns to the strong subtrees $(F_0 \uh n, \dots, F_{d-1} \uh n)$ augmented by these nodes. Multiple applications of the Halpern-Lauchli theorem yield subforests $Z_0, \dots, Z_{d-1}$ of $Y_0, \dots, Y_{d-1}$ with the same set of roots such that the induced coloring has a limit color on products of nodes from the $Z_i$.
%, and on which each induced function has a limit color.
This limit color therefore depends only on the choice of element from $\SubtreeLeaves{n+1}{U_0, \dots, U_{d-1}}$. This is how we define the limit function $\chi: \SubtreeLeaves{n+1}{U_0, \dots, U_{d-1}}$.

%It is important to keep in mind, in the statement of \Cref{thm:main-widget-theorem}, that $f$ is not the local value of the strong subtrees in $\SubtreeLeaves{n+1}{U_0, \dots, U_{d-1}}$, but the limit color of these strong subtrees when we replace their leaves by any pointwise extension in the reservoirs $Z_0, \dots, Z_{d-1}$.

%\benoit{I find suspicious that we need such complex combinatorics to prove 4.10. Let $f$ be a computable $2$-color on elements of $\mathcal{S}_{n+1}(T_0, \dots, T_{d-1})$. Given a condition 
%
%$(F_0, \dots, F_{d-1}, Y_{0}, \dots, Y_{d-1})$ 
%
%and given a tree $T \in \mathcal{S}_n(T_0, \dots, T_{d-1})$, we can define $g$ on 
%
%$\bigcup_{n} \Pi_{j < d, \rho \in \roots(Y_j)} (Y_j \upharpoonright {\rho}) (n)$ (roughly)
%
%to be $g(H) = f(T \cup H)$ if $T \cup H$ belongs to $\mathcal{S}_{n+1}(T_0, \dots, T_{d-1})$ and a third color otherwise. By applying HL on the reservoirs we make the color stable for $T$ on our generic (maybe by making that no finite subtree of height $n+1$ starts with $T$). Why cannot we mix this for every $T \in \mathcal{S}_{n}(T_0, \dots, T_{d-1})$ with the forcing conditions for cone avoidance ?}

The proof of \Cref{thm:main-widget-theorem} requires some rather heavy combinatorial development, and so we postpone it to the next section. Instead, we first show how to use the theorem to obtain \Cref{thm:cmtt-admits-strong-cone-avoidance}.

\begin{lemma}\label{thm:cmtt-admits-strong-cone-avoidance-req}
Fix sets $C,Z \subseteq \NN$ with $C \nTred Z$, an $n \geq 1$, a 
	%Fix $n \geq 1$, Fix two sets $C$ and $Z$ such that $C \nTred Z$.
	$Z$-computable collection of $Z$-computably bounded trees with no leaves $T_0, \dots, T_{d-1}$, and a coloring $f: \Subtree{n+1}{T_0, \dots, T_{d-1}} \to k$. Let $\Pb$ be the partial order of all stable cone avoiding product tree conditions (with respect to the givens). For every $\Pb$-condition $c$ and every Turing functional $\Gamma$, there is a $\Pb$-condition $c'$ extending $c$ such that $c' \Vdash \Gamma^{G_0 \oplus \dots \oplus G_{d-1} \oplus Z} \neq C$.
\end{lemma}
\begin{proof}
Fix $c = (F_0, \dots, F_{d-1}, X_0, \dots, X_{d-1})$. By \Cref{remark:product-tree-condition-roots}, we can assume that $(F_0 \cup X_0, \dots, F_{d-1} \cup X_{d-1}) \in \Subtree{\omega}{T_0, \dots, T_{d-1}}$. Let $b: \NN \to \NN$ be a $Z$-computable function bounding the trees $F_0 \cup X_0, \dots, F_{d-1} \cup X_{d-1}$,
	and let $\ell$ be the height of $F_0, \dots, F_{d-1}$.

Let $W$ be the set of all pairs $(x, v) \in \NN \times \{0,1\}$ such that for every $d$-tuple of strong subforests $Y_0, \dots, Y_{d-1}$ of $X_0, \dots, X_{d-1}$, respectively, with common level function dominated by $N \mapsto \fwidg(N,\ell,n+1,k,d, b)$, and such that for every $j < d$, every root of $X_j$ is extended by a root of $Y_j$, there is some $d$-tuple $H_0 \subseteq Y_0, \dots, H_{d-1} \subseteq Y_{d-1}$
	with $(F_0 \cup H_0, \dots, F_{d-1} \cup H_{d-1}) \in \Subtree{<\omega}{T_0, \dots, T_{d-1}}$ and
$$
\Gamma^{(F_0 \cup H_0) \oplus \dots \oplus (F_{d-1} \cup H_{d-1}) \oplus Z}(x)\downarrow = v.
$$
By compactness, the set $W$ is $X_0 \oplus \dots \oplus X_{d-1} \oplus Z$-c.e. We have three cases.

\case{1}{$(x, 1-C(x)) \in W$ for some $x \in \NN$.} By compactness, there is some height $N_0 \in \NN$ such that the property holds for every $d$-tuple
%$V_0, \dots, V_{d-1}$
of strong subforests of $X_0, \dots, X_{d-1}$, respectively, of height $N_0$ with common level function dominated by $N \mapsto \fwidg(N,\ell,n+1,k,d, b)$.
Let $U_0, \dots, U_{d-1}$ be the finite trees obtained by restricting $F_0 \cup X_0, \dots, F_{d-1} \cup X_{d-1}$, respectively, to their first $\fwidg(N_0,\ell,n+1,k,d, b)$ many levels.
In particular, $U_0, \dots, U_{d-1}$ are $b$-bounded trees of height $\fwidg(N_0,\ell,n+1,k,d, b)$.

Fixing a tuple $(E_0, \dots, E_{d-1}) \in \SubtreeLeaves{n+1}{U_0, \dots, U_{d-1}}$,
the coloring $f$ induces a function 
$$
g: \bigcup_m \prod_{j < d} \prod_{\rho \in \leaves(E_j)} (X_j \uh \rho)(m) \to k
$$
define for all tuples $\pi = (\sigma^\rho_j \in (X_j \uh \rho)(m): j < d, \rho \in \leaves(E_j))$ by
\[
	g(\pi) = f(\{ (E_j \uh n) \cup \{\sigma^\rho_j\:\  \rho \in \leaves(E_j)\}:  j < d\}).
\]
Thus, by iteratively applying strong cone avoidance of the Halpern-La\"{u}chli theorem (\Cref{thm:hl-strong-cone-avoidance}), there is a $d$-tuple of strong subforests $Y_0, \dots, Y_{d-1}$ of $X_0, \dots, X_{d-1}$, respectively, with common level function, such that:
\begin{itemize}
	\item[(a)] for every $j < d$, every leaf of $U_j$ is extended by exactly one root of $Y_j$;
	\item[(b)] for every $(E_0, \dots, E_{d-1}) \in \SubtreeLeaves{n+1}{U_0, \dots, U_{d-1}}$, there is a color $i < k$ such that for every
	\[
		( \sigma^\rho_j: j < d, \rho \in \leaves(E_j) ) \in \bigcup_m \prod_{\rho \in \leaves(E_j)} (Y_j \uh \rho)(m),
	\]
	we have $f(\{ (E_j \uh n) \cup \{\sigma^\rho_j\:\  \rho \in \leaves(E_j)\}:  j < d\}) = i$;
	\item[(c)] $C \nTred Y_0 \oplus \dots \oplus Y_{d-1} \oplus Z$.
\end{itemize}
Item (b) induces a coloring $\chi: \SubtreeLeaves{n+1}{U_0,\dots,U_{d-1}}\to k$
which associates to $(E_0, \dots, E_{d-1})$ the unique color $i < k$ as specified there.
By \Cref{thm:main-widget-theorem}, there are finite strong subtrees $V_0, \dots, V_{d-1}$ of $U_0, \dots, U_{d-1}$, respectively, of height $N_0 + \ell$ with common level function, such that for every $j < d$, $V_j \uh \ell = F_j$, the level function of $V_j$ is bounded by $N \mapsto \fwidg(N,\ell,n+1, k,d, b)$ if $N > \ell$, and the color of $(E_0, \dots, E_{d-1}) \in \SubtreeLeaves{n+1}{V_0,\dots,V_{d-1}}$ with respect to $\chi$ depends only on $(E_0 \uh n, \dots, E_{d-1} \uh n)$. By choice of $N_0$, there are some $H_0 \subseteq V_0, \dots, H_{d-1} \subseteq V_{d-1}$
	such that $F_0 \cup H_0, \dots, F_{d-1} \cup H_{d-1}$ are finite strong subtrees of $T_0, \dots, T_{d-1}$, respectively, with common level function, and such that
$$
\Gamma^{(F_0 \cup H_0) \oplus \dots \oplus (F_{d-1} \cup H_{d-1}) \oplus Z}(x)\downarrow = v.
$$
The tuple $c' = (F_0 \cup H_0, \dots, F_{d-1} \cup H_{d-1}, Y_0, \dots, Y_{d-1})$
	is therefore a cone avoiding stable product tree condition extending $c$ that satisfies
$$
c' \Vdash \Gamma^{G_0 \oplus \dots \oplus G_{d-1} \oplus Z} \neq C.
$$

\case{2}{$(x, C(x)) \not \in W$ for some $x \in \NN$.} Let $\Cc$ be the class of all strong subforests $Y_0, \dots, Y_{d-1}$ of $X_0, \dots, X_{d-1}$, respectively, with common level function dominated by $N \mapsto \fwidg(N,\ell,n+1,k,d, b)$ such that for every $j < d$, every root of $X_j$ is extended in a root of $Y_j$, and for every $d$-tuple $H_0 \subseteq Y_0, \dots, H_{d-1} \subseteq Y_{d-1}$
	for which $F_0 \cup H_0, \dots, F_{d-1} \cup H_{d-1}$ are finite strong subtrees of $T_0, \dots, T_{d-1}$, respectively, again with common level function, we have
$$
\Gamma^{(F_0 \cup H_0) \oplus \dots \oplus (F_{d-1} \cup H_{d-1}) \oplus Z}(x) \uparrow \text{ or } \Gamma^{(F_0 \cup H_0) \oplus \dots \oplus (F_{d-1} \cup H_{d-1}) \oplus Z}(x) \downarrow \neq v.
$$
%where inequality means either divergence, or halting on a different value.
Since the trees $T_0, \dots, T_{d-1}$ are $Z$-computably bounded and the level function of $Y_0, \dots, Y_{d-1}$ is dominated by the $Z$-computable function $\fwidg$, it follows that $\Cc$ is  a $\Pi^0_1$ class relative to $X_0 \oplus \dots \oplus X_{d-1} \oplus Z$. Moreover, by assumption, $\Cc$ is non-empty.

By the cone avoidance basis theorem, there is some $( Y_0, \dots, Y_{d-1}) \in \Cc$ such that $C \nTred Y_0 \oplus \dots \oplus Y_{d-1} \oplus Z$. The tuple \[c' = (F_0, \dots, F_{d-1}, Y_0, \dots, Y_{d-1})\] is then a $\Pb$-condition extending $c$ such that
$$
c' \Vdash \Gamma^{G_0 \oplus \dots \oplus G_{d-1} \oplus Z} \neq C.
$$

\case{3}{otherwise.} Then we have that $(x,y) \in W$ if and only if $C(x) = y$, so $C \Tred X_0 \oplus \cdots \oplus X_{d-1} \oplus Z$.
%Then $W$ is an $X_0 \oplus \dots \oplus X_{d-1} \oplus Z$-c.e.\ graph of the characteristic function of $C$, hence $C \leq X_0 \oplus \dots \oplus X_{d-1} \oplus Z$. Contradiction.
\end{proof}

\begin{proof}[Proof of \Cref{thm:cmtt-admits-strong-cone-avoidance}]
Fix two sets $C$ and $Z$ such that $C \nTred Z$.
Also fix a $Z$-computable collection of $Z$-computably bounded trees with no leaves $T_0, \dots,\allowbreak T_{d-1} \subseteq \baire$.
 Let $n\geq 1$ and $f: \Subtree{n+1}{T_0, \dots, T_{d-1}} \to k$ be a coloring. Let $\Pb$ be the partial order of all cone avoiding product tree conditions which are stable for $f$, and let $\Uc$ be a sufficiently generic $\Pb$-filter.
Let $G^\Uc_0, \dots, G^\Uc_{d-1}$ be the strong subtrees of $T_0, \dots, T_{d-1}$ induced by $\Uc$. By \Cref{thm:cmtt-admits-strong-cone-avoidance-req},
for every Turing functional $\Gamma$, there is some $\Pb$-condition $c \in \Uc$
such that $c \Vdash \Gamma^{G_0 \oplus \dots \oplus G_{d-1} \oplus Z} \neq C$.
Hence, $C \nTred G^\Uc_0 \oplus \dots \oplus G^\Uc_{d-1} \oplus Z$.
Moreover, by \Cref{lem:product-tree-genericity-implies-infinity}, $G^\Uc_0, \dots, G^\Uc_{d-1}$ are all infinite. And finally, since $\Uc$ contains only stable conditions, $f$ is stable on $\Subtree{n}{G_0, \dots, G_{d-1}}$.
This completes the proof of \Cref{thm:cmtt-admits-strong-cone-avoidance}.
\end{proof}

We are ready to prove cone avoidance of $\PMT{2}{}$.

\begin{theorem}\label{thm:cone-avoidance-MTT2}
  The product version of Milliken's tree theorem for height 2 admits cone avoidance.
\end{theorem}
\begin{proof}
Fix two sets $C$ and $Z$ such that $C \nTred Z$.
Also fix a $Z$-computable collection of $Z$-computably bounded trees with no leaves $T_0, \dots,\allowbreak T_{d-1} \subseteq \baire$ and a $Z$-computbale coloring $f: \Subtree{2}{T_0, \dots, T_{d-1}} \to k$.
 
By \Cref{thm:cmtt-admits-strong-cone-avoidance}, there are strong subtrees $S_0, \dots, S_{d-1}$ of $T_0, \dots, T_{d-1}$, respectively, with common level function, such that $f$ is stable on
\[
	\Subtree{2}{S_0, \dots, S_{d-1}},
\]
and such that $C \nTred S_0 \oplus \dots \oplus S_{d-1} \oplus Z$. By stability, the coloring $f$ induces a $k$-partition $A_0 \sqcup \dots \sqcup A_{k-1} = \bigcup_n S_0(n) \times \dots \times S_{d-1}(n)$ by letting
$A_i$ be the set of tuples $(\sigma_0, \dots, \sigma_{d-1}) \in \bigcup_n S_0(n) \times \dots \times S_{d-1}(n)$ such that for all but finitely many levels $\ell \in \NN$,
whenever $(\{\sigma_0\} \cup H_0, \dots, \{\sigma_{d-1}\} \cup H_{d-1}) \in \Subtree{2}{S_0, \dots, S_{d-1}}$
%where $H_j \subseteq S_j(\ell)$ for each $j < d$,
then $f( \{\sigma_0\} \cup H_0, \dots, \{\sigma_{d-1}\} \cup H_{d-1} ) = i$.

By \Cref{thm:hl-strong-cone-avoidance}, there is some color $i < k$
and some strong subtrees $U_0, \dots, U_{d-1}$ of $S_0, \dots, S_{d-1}$, respectively, with common level function, such that $\bigcup_n U_0(n) \times \dots \times U_{d-1}(n) \subseteq A_i$ and $C \nTred U_0 \oplus \dots \oplus U_{d-1} \oplus Z$.
By $U_0 \oplus \dots \oplus U_{d-1} \oplus Z$-computably thinning out the set of levels,
we can obtain a tuple of strong subtrees $V_0, \dots, V_{d-1}$ of $U_0, \dots, U_{d-1}$, respectively, with common level function, such that
$\Subtree{2}{V_0,\dots, V_{d-1}}$ is monochromatic for color $j$ with respect to $f$.
In particular, by transitivity of the strong subtree relation, $V_0, \dots, V_{d-1}$ are strong 
subtrees of $T_0, \dots, T_{d-1}$ with common level function, and $C \nTred V_0 \oplus \dots \oplus V_{d-1} \oplus Z$.
This completes the proof.
\end{proof}

\begin{corollary}
$\RCA_0 \wedge \PMT 2{}\not\vdash\ACA_0$.
\end{corollary}
\begin{proof}
Immediate by \Cref{thm:cone-avoidance-MTT2} and \Cref{lem:cone-avoidance-not-aca}.
\end{proof}

%The remainder of this section is devoted to the proof of \Cref{thm:main-widget-theorem}.
%\bigskip

\section[Proof of Theorem 4.13]{Proof of Theorem \ref{thm:main-widget-theorem}}

We now prove the main technical result used in the preceding section. We shall restate it in full below for convenience. First, we have the following lemma.
\begin{lemma}[Finitary Halpern-La\"{u}chli theorem for leaves]\label{lem:finitary-strong-hl-tuple}
	Fix a number of colors $k \in \NN$, an arity $d \geq 1$, and a function $b:\om\to\om$. There exists a function $N \mapsto \fhl(N, k, d, b)$, uniformly $b$-computable in $k$ and $d$ as follows.
  %There exists a computable function $\fhl(N, d, k, b)$ such that
  If $U_0,\dots, U_{d-1}$ is a sequence of finite $b$-bounded trees of height $h=\fhl(N,k,d,b)$ for some fixed $N \geq 1$, and
  %for any finite coloring of tuple of leaves
  \[g:\exleavesprodtree{U}{d}{h}\to k\]
  is any coloring of the $d$-tuples of leaves from this sequence, then there exists % finite perfect strong subtrees
  $(V_0,\dots,V_{d-1})\in\SubtreeLeaves{N}{U_0,\dots, U_{d-1}}$, % of respectively $T_0,\dots, T_{d-1}$ with a common level function and height
  % \footnote{We require the resulting tree to be of height $N+1$ as we are interested in the height of the tree without the leaves, which is $N$.}
%  $N$,
  such that $g$ is constant on the product of the leaves
%  \begin{enumerate}
%  \item $S_0$ and $S_1$ are of height $N$,
%  \item the leaves of $S_m$ are included in the leaves of $T_m$, that is, $S_m(N-1)\subseteq T_m(h-1)$ for any $m<d$,
%\item
  %the product of the leaves
  \[\exleavesprodtree V d N.\]
  %is monochromatic for $f$.
%  \end{enumerate}
\end{lemma}
%Recall that as $S_i$ is a strong subtree of $T_i$, the leaves of $S_i$ are included in the leaves of $T_i$, so $S_i(N-1)\subseteq T_i(h-1)$ for $i<d$.
\begin{proof}
  % \todo[inline]{  By compacity, the strong Halpern-La\"{u}chli theorem and the ``normalization coloring''. }
  %Fix $k$ and $d$ in $\Nb$.
  Let $\Cc$ be the space of all functions
%  By the compactness of the space
  \[
    %\left\{
    f: \bigcup_n T_0(n)\times\dots\times T_{d-1}(n) \to k
    %T_0,\dots,T_{d-1}\text{ $b$-bounded trees}\right\}
  \]
  where $T_0,\dots,T_{d-1}$ are $b$-bounded trees. By compactness of $\Cc$,
%  of colorings into $k$ colors of product of nodes from $b$-bounded trees at the same level \benoit{I think this should be rephrased: If I understand the compactness theorem is applied to both the space of trees and the space of coloring, for a unique level function fixed in advance},
  the Halpern-La\"{u}chli theorem (\Cref{th:strong-hl}) yields the existence of a function
  \[
  	\fhl(\cdot,k,d, b):\Nb\to\Nb
  \]
  such that for any $N$, any collections of $b$-bounded trees $T_0,\dots, T_{d-1}$ of height $\fhl(N,k,d,b)$, and any $f: \bigcup_{n<\fhl(N,k,d,b)} {T}_0(n) \times \dots \times {T}_{d-1}(n) \to k$, there exists $(S_0,\dots,S_{d-1}) \in \Subtree{N}{T_0,\ldots,T_{d-1}}$ such that $f$ is constant on $\bigcup_{n<N} {S}_0(n) \times \dots \times {S}_{d-1}(n)$.
  
  Now, consider the given trees $U_0,\ldots,U_{d-1}$ of height $h = \fhl(N,k,d,b)$, and the given coloring $g$.
%  The function $\fhl$ will be a witness of the lemma. However, we are not yet done, as the strong subtrees are not required to include the leaves of the initial trees. Consider any collections of perfect trees $T_0,\dots, T_{d-1}$ of height $h=\fhl(N,d,k,b)$, and a coloring $f$ of, this time, the product of \emph{leaves}: \[f:{T}_0(h-1) \times \dots \times {T}_{d-1}(h-1)\to k.\]% into $k$ colors.
  Define
  \[
  	f: \bigcup_{n<h} {U}_0(n) \times \dots \times {U}_{d-1}(n) \to k
  \]
  by $f(\sigma_0,\ldots,\sigma_{d-1}) = g(l_{\sigma_0},\dots, l_{\sigma_{d-1}})$,
  where $l_\sigma$ for each $\sigma \in U_i$ denotes a choice of leaf extending $\sigma$. 
  
%  $g$ of the product of \emph{nodes} in order to apply the property of $\fhl$. For any node $\sigma$ of a tree $T_i$ for some $i<d$, let $l_\sigma\in T_i(h-1)$ be a leaf extending $\sigma$, for instance the leftmost one. Then,
 % \[
 %   g\colon
 %   \begin{array}{rcl} 
 %     \bigcup_{n<h} {T}_0(n) \times \dots \times {T}_{d-1}(n) &\to& k \\ 
 %     (\sigma_0,\dots,\sigma_{d-1}) &\mapsto& f(l_{\sigma_0},\dots, l_{\sigma_{d-1}})
 %   \end{array}
 % \]
  
  By the property of $\fhl$, let $S_0,\dots S_{d-1}$ be strong subtrees of $T_0,\dots, T_{d-1}$ of height $N$ and with a common level function such that $f$ is constant on $\bigcup_{n<N} {S}_0(n) \times \dots \times {S}_{d-1}(n)$. For $i<d$, set
  %define $V_i$ by replacing the last level by their extension chosen by $l$, that is,
  \[V_i=\bigcup_{n<N-1}S_i(n)\cup\{l_\sigma:\sigma\in S_i(N-1).\}\]
  Thus, $(V_0,\ldots,V_{d-1}) \in \SubtreeLeaves{N}{U_0,\dots, U_{d-1}}$, and as $f$ is constant on $S_0(N-1)\times\dots\times S_{d-1}(N-1)$ it follows that $g$ is constant on $V_0(N-1)\times\dots\times V_{d-1}(N-1)$. (Note that by definition of the $V_i$ and the $l_\sigma$, $V_0(N-1)\times\dots\times V_{d-1}(N-1)$ is a subset of $T_0(h-1)\times\dots\times T_{d-1}(h-1)$, the domain of $g$.)
%  The collection $\hat{S}_0,\dots,\hat{S}_{d-1}$ consists of strong subtrees of $T_0,\dots,T_{d-1}$ of height $N$ and with a common level function. Moreover, as $S_0(N-1)\times\dots\times S_{d-1}(N-1)$ is monochromatic for $g$, the set $\hat{S}_0(N-1)\times\dots\times\hat{S}_{d-1}(N-1)$ is a monochromatic for $f$. Note that by definition of $\hat S_i(N-1)$ and $l_\sigma$, $\hat{S}_0(N-1)\times\dots\times\hat{S}_{d-1}(N-1)$ is a subset of $T_0(h-1)\times\dots\times T_{d-1}(h-1)$, the domain of $f$.
\end{proof}

We are now ready to prove \Cref{thm:main-widget-theorem} stated earlier. Recall that it is a finitary version of Milliken's tree theorem for $\SubtreeLeaves{n+1}{\cdot}$, meaning that we color strong subtrees of a certain height that also preserve the leaves. We recall the full statement.

\widget*

We begin by giving the definition of the function $H$.

\begin{definition}
	Fix a level $\ell \in \NN$, a height $n \geq 1$, a number of colors $k \in \NN$, an arity $d \geq 1$, and a function $b:\om\to\om$. Define a function $N \mapsto \hat{H}(N,\ell,n,k,d,b)$ inductively as follows:
	\begin{enumerate}
		\item $\hat\fwidg(0,\ell,n,k,d,b)=0$;
		\item if $\hat\fwidg(N-1,\ell, n,k,d,b) = H_{N-1}$ is defined, then
		\[
   			\hat\fwidg(N,\ell,n,k,b)= \hat\fwidg(N-1,\ell,n,k,d,b) + \fhl(2, K, D, B),
  		\]
  		where
  		\begin{itemize}
  			\item $K$ is the cardinality of the set of all $k$-valued functions defined on
  			\[
  			\Pc(U_0 \uh H_N) \times \cdots \times \Pc(U_{d-1} \uh H_N) \times \Pc(U_0(H_N)) \times \cdots \times \Pc(U_{d-1}(H_N))
 			 \]
  			for some $b$-bounded trees $T_0,\ldots,T_{d-1}$;
  			%\item $K$ is the cardinality of the set of finite functions with domain the product of $\Pc(T_0\uh H_{N-1})\times\dots\times \Pc(T_{d-1}\uh H_{N-1})$ and $T_0(H_{N-1})\times\dots\times T_{d-1}(H_{N-1})$;
  			\item $D=d\times\prod_{i<\ell}b(i)\prod_{i<H_{N-1}}b(\ell+i)$;
  			\item $B$ is the function $n \mapsto b(n+H_{N-1})$.
  		\end{itemize}
	\end{enumerate}
	Define $\fwidg$ by
	\[
		\fwidg(N,\ell, n,k,d,b)= \ell+\hat\fwidg(N,\ell,n,k,d,b).
	\] 
\end{definition}

Note that $D$ corresponds to a bound on the number of leaves of $d$ many $b$-bounded trees of height $\ell+H_{N-1}$, and that $B$ is a bounding function for subtrees of a $b$-bounded tree that contains all the level{s} starting from $H_{N-1}$. \Cref{fig:widget} helps shed light on some of the parameters given to $\fhl$ in the definition of $\fwidg$.

\begin{proof}[Proof of \Cref{thm:main-widget-theorem}]
  %ix $n,k,d$.
  We proceed by induction on $N$, starting with $N=0$. The base case holds by taking any \[(V_0,\dots, V_{d-1}) \in\SubtreeLeaves{\ell+1}{U_0,\dots,U_{d-1}}\] with $V_i \uh \ell = U_i \uh \ell$ for all $i < d$.
  %$(V_0\uh\ell,\dots, V_{d-1}\uh\ell) = (U_0\uh \ell,\dots U_{d-1}\uh \ell)$.
  These trees satisfy \Cref{item:stems-widget,item:hbound-widget} by construction. Moreover, by assumption on $\chi$, they also satisfy \Cref{item:monochr-widget}.

%  for any collection of trees $T_0,\dots, T_{d-1}$ of height $\fwidg(N,2^n,k,d)$, a witness of the theorem can be any $(S_0,\dots, S_{d-1})\in \SubtreeLeaves{\embfont e_N}{T_0}\times\dots\times \SubtreeLeaves{\embfont e_N}{T_{d-1}}$ with an $\fwidg$-bounded level function, as $\SubtreeLeaves{\embfont e_n}{S_0}\times\dots\times \SubtreeLeaves{\embfont e_n}{S_{d-1}}$  is a singleton: there is only one perfect subtree of height $n$ included in a tree of height $N$ when $N=n$.

  % one can find the set $[T_0(0)]^n\times\dots\times [T_{d-1}(0)]^n$ is either empty or a singleton, and therefore monochromatic.

  Now, suppose the result is true for some $N\geq 0$.
  To simplify notation, define $H_{N}=\fwidg(N,\ell,n+1,k,d, b)$ and $H_{N+1}=\fwidg(N+1,\ell,n+1,k,d, b)$. The construction of the solution $V_0,\dots, V_{d-1}$ is divided into three steps, summarized as follows.
  \begin{enumerate}
  \item We apply \Cref{lem:finitary-strong-hl-tuple} to the collection of trees $U_i^\sigma = U_i \uh \sigma$ for $\sigma \in U_i(H_N)$
  %\item and $T_i^\sigma=\{\tau\in T_i:\tau\succeq\sigma\}$,
  and a certain coloring with a large number of colors. This will yields strong subtrees $V_i^\sigma$ of $U_i^\sigma$ of height $2$ with a common level function. In turn, these will induce a coloring of $\SubtreeLeaves{n+1}{U_0\uh H_N,\dots, U_{d-1}\uh H_N}$.
  \item We apply the inductive hypothesis to $U_0\uh H_N,\dots, U_{d-1}\uh H_N$ and the induced coloring, obtaining strong subtrees $\hat V_0,\dots, \hat V_{d-1}$.
  \item For each $i < d$, we replace the leaves of $\hat V_i$ by $V_i^\sigma$ to get $V_i$.
  \end{enumerate}
  % Let $T_0,\dots, T_{d-1}$ be perfect trees of height $H_N$.
  We now give the details of each step of the construction.
  
  \construction
  
  \medskip
  \noindent \textbf{Step 1.} We define a coloring
  \[
  	g:\prod_{i<d}\prod_{\sigma\in U_i(H_N)}\leaves({U_i^\sigma})\to K,
  \]
  where $K$ is the finite set of all functions
  \[
  	\zeta: \Pc(U_0 \uh H_N) \times \cdots \times \Pc(U_{d-1} \uh H_N) \times \Pc(U_0(H_N)) \times \cdots \times \Pc(U_{d-1}(H_N)) \to k.
  \]
  Let $\pi$ be an element of the domain of $g$, meaning a tuple $((\tau_i^\sigma)_{\sigma\in U_i(H_N)})_{i<d}$ consisting of one leaf $\tau^\sigma_i$ from each tree $U_i^\sigma$. Then $g(\pi)$ is the function $\zeta$ defined as follows. Given $F_i \subseteq U_i \uh H_N$ and $G_i \subseteq U_i(H_N)$ for each $i < d$,
  \[
  	\zeta(F_0,\ldots,F_{d-1},G_0,\ldots,G_{d-1}) = \chi((F_i ~\cup~\{ \tau_i^\sigma: \sigma \in G_i\})_{i < d})
  \]
  if $(F_i \cup \{ \tau_i^\sigma: \sigma \in G_i\})_{i < d} \in \SubtreeLeaves{n+1}{U_0,\ldots,U_{d-1}}$, and
  \[
  	\zeta(F_0,\ldots,F_{d-1},G_0,\ldots,G_{d-1}) = 0
  \]
  otherwise. So in particular, $g(\pi)$ records the values of $\chi$ on all strong subtrees of height $n+1$ that have leaves in $\pi$ and all other nodes below level $H_N$ in $U_0,\ldots,U_{d-1}$.

  By \Cref{lem:finitary-strong-hl-tuple} applied to the collection of $U_i^\sigma$ with the coloring $g$, using the fact that the height, $H_{N+1}-H_N$, of the trees is sufficiently large by definition of $H$, we obtain strong subtrees $V_i^\sigma$ of $U_i^\sigma$ of height 2 and with common level function such that $g$ is constant on the product of the leaves of the $V_i^\sigma$. Call the value assumed by $g$ on this product $\zeta_0 \in K$.
  
  \medskip
  \noindent \textbf{Step 2.} The function $\zeta_0$ naturally induces a coloring
  \[
  	\chi_N: \SubtreeLeaves{n+1}{U_0 \uh H_N+1,\ldots,U_{d-1} \uh H_N+1} \to k
  \]
  as follows. Given $(F_0,\ldots,F_{d-1})$ in the domain of $\chi_N$, let
  \[
  	\chi_N(F_0,\ldots,F_{d-1}) = \zeta_0(F_0 \uh n,\ldots,F_{d-1} \uh n,\leaves(F_0),\ldots,\leaves(F_{d-1})).
  \]
  Note that by choice of the $V^\sigma_i$, if $((\tau_i^\sigma)_{\sigma\in U_i(H_N)})_{i<d}$ is any tuple consisting of one leaf $\tau^\sigma_i$ from each tree $V_i^\sigma$, then $(F_i \uh n~\cup~\{ \tau_i^\sigma: \sigma \in \leaves(F_i)\})_{i < d} \in \SubtreeLeaves{n+1}{U_0,\ldots,U_{d-1}}$, so by definition we also have
  \[
  	\chi_N(F_0,\ldots,F_{d-1}) = \chi((F_i \uh n~\cup~\{ \tau_i^\sigma: \sigma \in \leaves(F_i)\})_{i < d}).
  \]
  By assumption on $\chi$, it follows that if $F_i \uh \ell \subseteq (U_i \uh H_N +1) \uh \ell = U_i \uh \ell$ for all $i < d$, then $\chi_N(F_0,\ldots,F_{d-1})$ depends only on $(F_0 \uh n,\ldots,F_{d-1} \uh n)$. We may thus apply the induction hypothesis to $\chi_N$ and the trees $U_0 \uh H_N+1,\ldots,U_{d-1} \uh H_N+1$ to obtain a tuple of strong subtrees $(\hat V_0,\dots, \hat V_{d-1}) \in \SubtreeLeaves{\ell + N+ 1}{U_0 \uh H_N+1,\ldots,U_{d-1} \uh H_N+1}$.
  
  \medskip
  \noindent \textbf{Step 3.} Finally, we glue the trees $\hat V_i$ to the trees $V_i^\sigma$ to finish the construction of the solution. More precisely, we let
  \[
  	V_i = \hat V_i \setminus \leaves(\hat V_i)~\cup~{\bigcup_{\sigma \in \leaves(\hat V_i)} V^\sigma_i}.
  \]
  Note that the height of $V_i$ is $\ell + N + 2$, as desired. This completes the construction.
  
  \verification We now prove that the collection of $V_i$ is a solution. \Cref{item:stems-widget} is satisfied since it is satisfied by $\hat V_i$. This is because $V_i$ extends $\hat V_i\setminus\leaves(\hat V_i)$, and the height of $\hat{V_i}$ is at least $\ell + 1$, so we have ${V_i \uh \ell} = (\hat V_i\setminus\leaves(\hat V_i)) \uh \ell = {\hat V_i \uh \ell} = U_i \uh \ell$.
  
  \Cref{item:hbound-widget} is satisfied by construction.
  
  It remains to verify \Cref{item:monochr-widget}. Suppose $(F_0,\dots, F_{d-1}) \in \SubtreeLeaves{n+1}{V_0,\dots, V_{d-1}}$. We consider two cases.
  
  \case{1}{$F_i(n-1) \subseteq V_i(\ell+N)$ for each $i < d$.} Since $F_i \in \SubtreeLeaves{n+1}{V_i}$ for each $i < d$ and there is only one level in $V_i$ above $\ell+N$, the elements of $F_i(n) = \leaves(F_i)$ are uniquely determined by those of $F_i(n-1)$. Namely, $F_i(n) = \{\sigma \in V_i(\ell + n + 1): (\exists \tau \in F_i(n-1))[\tau \prec \sigma]\}$. Thus, $F_i$ is completely determined by $F_i \uh n$, and so also $\chi(F_0,\ldots,F_{d-1})$ depends only on $(F_0 \uh n,\ldots,F_{d-1} \uh n)$.
  
  \case{2}{$F_i(n-1) \subseteq V_i\uh \ell+N$ for each $i < d$.} In this case, we have $F_i \uh n \subseteq \hat{V}_i \setminus \leaves(\hat{V}_i) \subseteq U_i \uh H_N$. So, if we define
  \[
  	\hat F_i = F_i \uh n \cup \{ \sigma \in U_i(H_N): (\exists \tau \in \leaves(F_i))[\sigma \prec \tau]\}
  \]
  then $(\hat F_0,\ldots,\hat F_{d-1}) \in \SubtreeLeaves{n+1}{\hat{V}_0,\ldots,\hat{V}_{d-1}}$. By choice of the $\hat V_i$, we know that $\chi_N(\hat{F}_0,\ldots,\hat F_{d-1})$ depends only on $(\hat{F}_0 \uh n,\ldots,\hat{F}_{d-1} \uh n) = (F_0 \uh n,\ldots,F_{d-1} \uh n)$.
  
  Separately, by definition of $\chi_N$ and choice of the $V^\sigma_i$, we have that if $((\tau_i^\sigma)_{\sigma\in U_i(H_N)})_{i<d}$ is any tuple consisting of one leaf $\tau^\sigma_i$ from each tree $V_i^\sigma$, then 
  \[
  	\begin{array}{lll}
  		\chi_N(\hat{F}_0,\ldots,\hat F_{d-1}) & = &  \chi((\hat F_i \uh n~\cup~\{ \tau_i^\sigma: \sigma \in \leaves(\hat F_i)\})_{i < d})\\
  		& = & \chi((F_i \uh n~\cup~\{ \tau_i^\sigma: \sigma \in \leaves(\hat F_i)\})_{i < d}).
  	\end{array}
  \]
  Since the leaves of $F_0,\ldots,F_{d-1}$ form precisely such a tuple $((\tau_i^\sigma)_{\sigma\in U_i(H_N)})_{i<d}$ and $F_i \uh n \cup \leaves(F_i) = F_i$ for each $i < d$, we conclude
  \[
  	\chi_N(\hat{F}_0,\ldots,\hat F_{d-1}) = \chi(F_0,\ldots,F_{d-1}).
  \]
  
  Combining the previous two paragraphs, we find that $\chi(F_0,\ldots,F_{d-1})$ depends only on $(F_0 \uh n,\ldots,F_{d-1} \uh n)$, as was to be shown.
\end{proof} 

\begin{figure}[h!]
  \begin{center}
    \begin{tikzpicture}[scale=1.5]
		\tikzset{
			empty node/.style={circle,inner sep=0,outer sep=0,fill=none},
			solid node/.style={circle,draw,inner sep=1.5,fill=black},
			hollow node/.style={circle,draw,inner sep=1.5,fill=white},
			gray node/.style={circle,draw={rgb:black,1;white,4},inner sep=1,fill={rgb:black,1;white,4}}
		}
		\tikzset{snake it/.style={decorate, decoration=snake, line cap=round}}
		\tikzset{gray line/.style={line cap=round,thick,color={rgb:black,1;white,4}}}
		\tikzset{gray thin line/.style={line cap=round,color={rgb:black,1;white,4}}}
		\tikzset{thick line/.style={line cap=round,rounded corners=0.1mm,thick}}
		\tikzset{thin line/.style={line cap=round,rounded corners=0.1mm}}
		\node (a)[empty node] at (0.5,-2) {};
		\node (a')[empty node] at (0,-0.4) {};
		\node (sigma0)[empty node] at (0-0.45,0.8) {};
		\node (sigma1)[empty node] at (0+0.2,0.8) {};
		\node (sigma2)[empty node] at (0+1.45,0.8) {};
		\node (b0)[solid node] at (0-0.45,1.7) {};
		\node (b1)[solid node] at (0+0.2,1.7) {};
		\node (b2)[solid node] at (0+1.45,1.7) {};
		\node (c0)[empty node] at (0-0.45,1.1) {};
		\node (c1)[empty node] at (0+0.2,1.1) {};
		\node (c2)[empty node] at (0+1.45,1.1) {};
		\node (b00)[solid node] at (0.13-0.45,2.45) {};
		\node (b01)[solid node] at (-0.15-0.45,2.45) {};
		\node (b10)[solid node] at (0.13+0.2,2.45) {};
		\node (b11)[solid node] at (-0.15+0.2,2.45) {};
		\node (b20)[solid node] at (0.13+1.45,2.45) {};
		\node (b21)[solid node] at (-0.15+1.45,2.45) {};
		\begin{pgfonlayer}{background}
		\draw[gray thin line] (-1.4,1.7) to (2.4,1.7);
		\draw[gray thin line] (-1.4,1.9) to (2.4,1.9);
		\draw[gray thin line] (-1.4,0.8) to (2.4,0.8);
		\draw[gray thin line] (-1.4,2.45) to (2.4,2.45);
		\draw[thick line] (a.center) to (2.3,2.7);
		\draw[thick line] (a.center) to (-1.3,2.7);
		\draw[thick line,decorate,decoration={snake,amplitude=-.3mm,segment length=2.5mm,pre length=3mm}] (b0.center) to (0.13-0.45,2.45);
		\draw[thick line,decorate,decoration={snake,amplitude=.3mm,segment length=2.5mm,pre length=3mm}] (b0.center) to (-0.15-0.45,2.45);
		\draw[thick line,decorate,decoration={snake,amplitude=-.3mm,segment length=2.5mm,pre length=3mm}] (b1.center) to (0.13+0.2,2.45);
		\draw[thick line,decorate,decoration={snake,amplitude=.3mm,segment length=2.5mm,pre length=3mm}] (b1.center) to (-0.15+0.2,2.45);
		\draw[thick line,decorate,decoration={snake,amplitude=-.3mm,segment length=2.5mm,pre length=3mm}] (b2.center) to (0.13+1.45,2.45);
		\draw[thick line,decorate,decoration={snake,amplitude=.3mm,segment length=2.5mm,pre length=3mm}] (b2.center) to (-0.15+1.45,2.45);
		\draw[gray line] (sigma0.center) to (0.3-0.45,2.45);
		\draw[gray line] (sigma0.center) to (-0.3-0.45,2.45);
		\draw[gray line] (sigma1.center) to (0.3+0.2,2.45);
		\draw[gray line] (sigma1.center) to (-0.3+0.2,2.45);
		\draw[gray line] (sigma2.center) to (0.3+1.45,2.45);
		\draw[gray line] (sigma2.center) to (-0.3+1.45,2.45);
		\draw[thick line,decorate,decoration={snake,amplitude=.3mm,segment length=2.5mm}] (sigma0) to (b0.center);
		\draw[thick line,decorate,decoration={snake,amplitude=.3mm,segment length=2.5mm}] (sigma1) to (b1.center);
		\draw[thick line,decorate,decoration={snake,amplitude=.3mm,segment length=2mm,post length=0.01mm}] (sigma2) to (b2.center);
		\node(dots)[empty node,fill=white] at (0.85,1.7) {$\,\cdots\,$};
		\node(S0)[empty node,label=right:{$H_{N+1}$}] at (2.45,2.45) {};
		\node(S0)[empty node,label=right:{$H_N$}] at (2.45,0.8) {};
		\end{pgfonlayer}
	\end{tikzpicture}
    \caption{The construction of a tree in Theorem~\ref{thm:main-widget-theorem} when $d=1$.
    % As illustrated by Figure~\ref{fig:widget} (where $d=1$),
      Given a tree $T$ of height $H_N$, cutting at level $H_{N}$ yields a collection of finite perfect trees whose roots are nodes at level $H_{N}$. A finite coloring of $\SubtreeLeaves{n+1}{T}$ yields a coloring of product of leaves from the collection, by merging the colors of all possible closure into a tree of height $n+1$. The level $H_{N+1}$ is chosen large enough above $H_{N}$ so that on can apply Lemma~\ref{lem:finitary-strong-hl-tuple} to obtain strong subtrees of height 2, represented in bold.
      As explained in the proof, this yields a coloring of $\SubtreeLeaves{n}{T\uh H_{N}}$, and one can apply the induction hypothesis. % iterate the argument to obtain a collection of subtrees of height 2, in each interval between $\fwidg(i-1)$ and $\fwidg(i)$. A strong perfect tree of height $N$ satisfying Theorem~\ref{thm:main-widget-theorem} can be extracted from this collection. The leaves of this tree will be included in the leaves of the last level, while its branching nodes will be included in the branching nodes of the height 2 subtrees, the nodes represented with higher radius in the figure.% and which allow to extract the tree satisfying Theorem~\ref{thm:main-widget-theorem}, represented in bold.
      %The number of levels $\fwidg(N,n,k,d)$ is large enough so one can apply Lemma~\ref{lem:widget-one-level} on all subtrees rooted at level $\fwidg(N-1)$ and whose leaves are at level $\fwidg(N)$% in between level $\fwidg(N-1,n,k,d)$ and $\fwidg(N,n,k,d)$
      % , to get strong subtrees with one level in addition to the leaves. This yields a coloring of nodes at level $\fwidg(N-1,n,k,d)$, allowing to iterate $N$ times.
    }
    \label{fig:widget}
\end{center}
\end{figure}
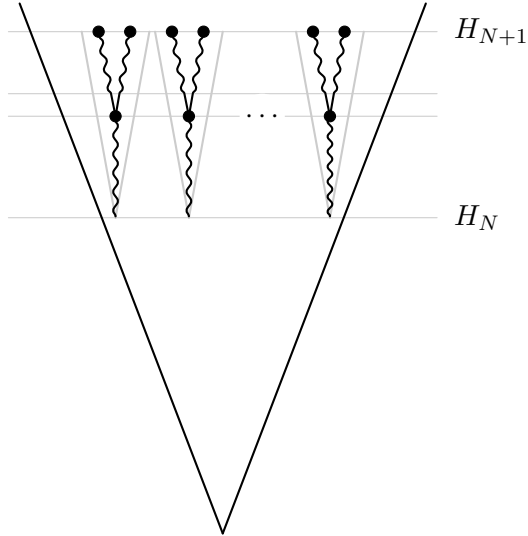

\section{Milliken's tree theorem with more colors}\label{subsect:thin-milliken}

As we have seen in the preceding sections, there is a computably detectable difference between Milliken's tree theorem for heights 2 and 3 that parallels that for Ramsey's theorem for pairs and triples. More specifically, Milliken's tree theorem for height 2 admits cone avoidance while the version for height 3 does not. In the case of Ramsey's theorem, more can be said.
%Although Ramsey's theorem for pairs admits cone avoidance
% (see Seetapun and Slaman~\cite{Seetapun1995strength}),
%when considering larger tuples, this is not anymore the case (see Jockusch~\cite{Jockusch1972Ramseys}). In particular, there exists a computable coloring of $[\NN]^3$ such that every solution computes the halting set. Like Ramsey's theorem, the product version of Milliken's tree theorem for height 2 admits cone avoidance (\Cref{thm:cone-avoidance-MTT2}) while even Milliken's tree theorem for height 3 does not, since it implies Ramsey's theorem for triples.
Wang~\cite[Theorem 3.2]{Wang2014Some}
%surprisingly
proved the surprising result that if we weaken Ramsey's theorem for $n$-tuples to permit a larger number $\ell$ of colors in the solution (instead of just one, which is to say, requiring the solutions to be homogeneous sets), and if $\ell$ is sufficiently large relative to $n$, then
 %whenever this number of colors $\ell$ is sufficiently large with respect to the size $n$ of the colored tuples,
the resulting statement admits strong cone avoidance. More recently, Cholak and Patey~\cite[Corollary 4.17]{Cholak2019Thin} gave explicit bounds on the relationship between $\ell$ and $n$, proving that cone avoidance holds so long as $\ell$ is at least as large as the $n$th Catalan number.

\begin{statement}[Ramsey's theorem for $n$-tuples and $k$ colors]\index{statement!$\RT{n}{k, \ell}$}
  $\RT{n}{k, \ell}$ is the statement: ``For any coloring $f: [\Nb]^n \to k$, there exists an infinite set $H\subseteq\Nb$ such that $f$ uses at most $\ell$ colors on $[H]^n$''.
\end{statement}

In this section, we prove a similar result for the product version of Milliken's tree theorem for height 2. More precisely, we show that whenever the number of colors in the solutions is allowed to be at least 2, then the resulting statement for height 2 admits strong cone avoidance (\Cref{thm:pmtt2k2-strong-cone-avoidance}), while the statement for height 3 admits cone avoidance (\Cref{thm:pmtt3k2-cone-avoidance}).

The notion of level-homogeneous coloring sets a bridge between Ramsey's theorem and Milliken's tree theorem. Let $T_0, \dots, T_{d-1} \subseteq \baire$ be finitely branching trees with no leaves. Recall that the level function witnessing a strong subtree is the function mapping the levels of the strong subtree to the levels in the original tree (see \Cref{def:strong-subtree}).

\begin{definition}\index{coloring!level-homogeneous}\index{level-homogeneous!coloring}\index{product tree condition!level-homogeneous}\index{level-homogeneous!product tree condition}
A coloring \[f: \Subtree{n}{T_0, \dots, T_{d-1}} \to k\] is \emph{level-homogeneous}
if the color of $(E_0, \dots, E_{d-1}) \in \Subtree{n}{T_0, \dots, T_{d-1}}$ depends only on its level function. A product tree condition \[(F_0, \dots, F_{d-1}, X_0, \dots, X_{d-1})\] is \emph{level-homogeneous} for $f$ if for every \[(E_0, \dots, E_{d-1}) \in \Subtree{n}{F_0 \cup X_0, \dots, F_{d-1} \cup X_{d-1}}\] such that $E_j \uh 1 \subseteq F_j$ for every $j < d$,
the color of  $( E_0, \dots, E_{d-1} )$ depends only on its level function.
\end{definition}

\noindent Note that the notion of level-homogeneous here extends that in \Cref{def:levhom3}, which is the particular case when $n = 1$ and $f$ is the function mapping a tuple in $\Subtree{1}{T_0, \dots, T_{d-1}}$ to the unique $i < k$ such that $A_i$ contains this tuple.

Any level-homogeneous coloring $f:\Subtree{n}{T_0, \dots, T_{d-1}}\to k$
induce{s} a coloring $g: [\NN]^n \to k$ which to some $F \in [\NN]^n$ associates 
the color of any element of $\Subtree{n}{T_0, \dots, T_{d-1}}$ whose level 
function has range $F$. This coloring $g$ is well-defined by level-homogeneity of $f$,
and for every homogeneous set $H \subseteq \NN$ for $g$, the \emph{principal function} \index{principal function} $p_H: \NN \to \NN$, which to $x$ associates the ($x+1$)st element of $H$ in natural order, is the level function of a solution to $f$. 

\begin{theorem}\label{thm:pmtt2-level-homogeneous-strong-cone-avoidance}
Fix two sets $C$ and $Z$ such that $C \nTred Z$.
Also fix a $Z$-computable collection of $Z$-computably bounded trees with no leaves $T_0, \dots,\allowbreak T_{d-1}$.
Let $f: \Subtree{2}{T_0, \dots, T_{d-1}} \to k$ be a coloring.
Then, there exist strong
  subtrees $(S_0, \dots, S_{d-1}) \in \Subtree{\omega}{T_0, \dots, T_{d-1}}$ over which $f$ is level-homogeneous, and such that
  $C \nTred S_0 \oplus \dots \oplus S_{d-1} \oplus Z$.
\end{theorem}

\begin{proof}
%We now prove \Cref{thm:pmtt2-level-homogeneous-strong-cone-avoidance}.
Fix $C$, $Z$, $T_0, \dots, T_{d-1}$ and $f$.
By \Cref{thm:cmtt-admits-strong-cone-avoidance}, there are strong subtrees $(U_0, \dots, U_{d-1}) \in \Subtree{\omega}{T_0, \dots, T_{d-1}}$ on which  $f$ is stable, and such that $C \nTred U_0 \oplus \dots \oplus U_{d-1} \oplus Z$.
%Let $g: \Subtree{1}{S_0, \dots, S_{d-1}}$ be the limit coloring induced by stability of $f$. By strong cone avoidance of $\PMT{2}{k,2}$ (\Cref{thm:pmtt2k2-strong-cone-avoidance}),
%there is some color $i_{\mathtt{lim}} < k$
%and some strong subtrees $(U_0, \dots, U_{d-1}) \in \Subtree{\omega}{S_0, \dots, S_{d-1}}$ on which $g$ is monochromatic for color $i_{\mathtt{lim}}$, and $C \nTred U_0 \oplus \dots \oplus U_{d-1} \oplus Z$.

We build strong subtrees $(G_0, \dots, G_{d-1}) \in \Subtree{\omega}{U_0, \dots, U_{d-1}}$  on which $f$ is level-homogeneous, and such that $C \nTred G_0 \oplus \dots \oplus G_{d-1} \oplus Z$. These sets will be constructed by forcing with product tree conditions. Recall that a product tree condition $c = (F_0, \dots, F_{d-1}, X_0, \dots, X_{d-1})$ is \emph{cone avoiding} (with respect to the given set $C$) if $C \nTred X_0 \oplus \dots \oplus X_{d-1} \oplus Z$ (see \Cref{def:levhom3}).
Let $\Pb$ be the collection of all cone avoiding product tree conditions which are level-homogeneous for~$f$.

The proof of the following lemma is very similar to that of \Cref{lem:hl-sca-density-below-a-cone}. In particular, we need again that condition extensions cannot remove roots of forests (see \Cref{def:strong-product-tree-extension}).

\begin{lemma}\label{lem:pmtt2-level-homogeneous-density-below-a-cone}
There is some condition $c \in \Pb$
such that for every Turing functional $\Gamma$, the set of conditions $c' \in \Pb$
such that $c' \Vdash \Gamma^{G_0 \oplus \dots \oplus G_{d-1} \oplus Z} \neq C$
is $\Pb$-dense below $c$.
\end{lemma}
\begin{proof}
Assume for the sake of contradiction that for every condition $c \in \Pb$,
there is a Turing functional $\Gamma$ and some extension, every further extension of which $c'$ satisfies $c' \not\Vdash \Gamma^{G_0 \oplus \dots \oplus G_{d-1} \oplus Z} \neq C$.

As in \Cref{lem:hl-sca-density-below-a-cone}, we build (non-effectively) a $d$-tuple $S_0, \dots, S_{d-1}$ of infinite subsets of $T_0, \dots, T_{d-1}$, respectively, together with three functions:
\begin{itemize}
	\item[1.] $\operatorname{sets}: \NN \to \Pc(\baire) \times \dots \times \Pc(\baire)$ which to a level $\ell \in \NN$
	associates a $d$-tuple $X_0, \dots, X_{d-1}$ of infinite strong subforests of $T_0, \dots, T_{d-1}$, respectively, with common level function, such that $C \nTred X_0 \oplus \dots \oplus X_{d-1} \oplus Z$ and such that for every $j < d$, $S_j(\ell+1) = \roots(X_j)$;
	\item[2.] $\operatorname{stems}: \exprodtree{S}{d} \to \Subtree{<\omega}{T_0, \dots, T_{d-1}}$, which to a $\pi \in S_0(\ell) \times \dots \times S_{d-1}(\ell)$ associates a tuple $(F_0, \dots, F_{d-1})$ whose roots pointwise extend $\pi$, and such that $(F_0, \dots, F_{d-1}, \operatorname{sets}(\ell))$ is a $\Pb$-condition;
	\item[3.] $\operatorname{req}: \exprodtree{S}{d} \to \NN$, which to a $\pi \in S_0(\ell) \times \dots \times S_{d-1}(\ell)$ associates an index $e$ of a Turing functional $\Phi_e$
	such that for every $\Pb$-extension $c'$ of $(\operatorname{stems}(\pi), \operatorname{sets}(\ell))$,
	$c' \nVdash \Phi_e^{G_0 \oplus \dots \oplus G_{d-1} \oplus Z} \neq C$.
\end{itemize}
Moreover, we require that for every level $\ell \in \NN$,
$\operatorname{sets}(\ell+1)$ are strong subforests of $\operatorname{sets}(\ell)$
with common level function.
% \pelliot{maybe there should be a picture here? I needed to draw one to understand better.}

\bigskip
\noindent
The construction is now exactly the same as in the proof \Cref{lem:hl-sca-density-below-a-cone}.
Moreover, the following fact still holds:

\begin{fact}\label{fact:pmtt2-level-homogeneous-density-below-a-cone-condition-extension}
For every $\ell_0 < \ell_1$ and every $\pi \in S_0(\ell_0) \times \dots \times S_{d-1}(\ell_0)$, 
the tuple $(\operatorname{stems}(\pi), \operatorname{sets}(\ell_1))$ is a $\Pb$-extension of $(\operatorname{stems}(\pi), \operatorname{sets}(\ell_0))$. 
\end{fact}

By \Cref{thm:combinatorial-finite-hapern-lauchli}, there is a level $N \in \NN$
such that for every coloring $h: S_0(N) \times \dots \times S_{d-1}(N) \to k$,
there is some $\ell < N$, some $\pi \in S_0(\ell) \times \dots \times S_{d-1}(\ell)$
and some $(\ell+1)$-$\pi$-dense matrix $M \subseteq S_0(N) \times \dots \times S_{d-1}(N)$
monochromatic for~$h$. 
Fix such an $N$. Let $(X_0, \dots, X_{d-1}) = \operatorname{sets}(N-1)$. In particular, for every $j < d$, $S_j(N) = \roots(X_j)$.

Let $W$ be the set of pairs $(x, v) \in \NN \times \{0,1\}$ such that for every $k$-coloring $g: \Subtree{2}{X_0, \dots, X_{d-1}} \to k$, there is some $\ell < N$, some $\pi \in S_0(\ell) \times \dots \times S_{d-1}(\ell)$, and for every $j < d$, there is a finite set $H_j \subseteq X_j$ such that, letting $(F_0, \dots, F_{d-1}) = \operatorname{stems}(\pi)$, the following holds
\begin{itemize}
	\item[(a)] $(F_0 \cup H_0, \dots, F_{d-1} \cup H_{d-1}) \in \Subtree{<\omega}{U_0, \dots, U_{d-1}}$;
	\item[(b)] $g$ restricted to $\Subtree{2}{H_0, \dots, H_{d-1}}$ is monochromatic for some $i < k$;
	\item[(c)] $\Phi_e^{(F_0 \cup H_0) \oplus \dots \oplus (F_{d-1} \cup H_{d-1}) \oplus Z}(x)\downarrow = v$, where $e = \operatorname{req}(\pi)$.
\end{itemize}
By compactness, the set $W$ is $X_0 \oplus \dots \oplus X_{d-1} \oplus Z$-c.e.\ There are three cases:

\case{1}{$(x, 1-C(x)) \in W$ for some $x \in \NN$.} For $i < k$, let $g$ be the restriction of $f$ to $\Subtree{2}{X_0, \dots, X_{d-1}}$. Let $\ell < N$, $\pi = (F_0, \dots, F_{d-1})$ and $H_0, \dots, H_{d-1}$ witness that $(x, 1-C(x)) \in W$ for $g$.
	Let $\ell_1$ be a level large enough to witness stability of $f$ for every level of $H_j$, and let $\hat{X}_j = X_j \setminus \bigcup_{\ell_0 \leq \ell_1} X_j(\ell_0)$. Then $c' = (F_0 \cup H_0, \dots, F_{d-1} \cup H_{d-1}, \hat{X}_0, \dots, \hat{X}_{d-1})$ is a $\Pb$-extension of $(F_0, \dots, F_{d-1}, X_0, \dots, X_{d-1})$ which, by Fact \ref{fact:pmtt2-level-homogeneous-density-below-a-cone-condition-extension},
		is a $\Pb$-extension of $(\operatorname{stems}(\pi), \operatorname{sets}(\ell))$.
		Moreover
		$$
			c' \Vdash \Phi_e^{G_0 \oplus \dots \oplus G_{d-1} \oplus Z} \neq C
		$$
		where $e = \operatorname{req}(\pi)$. This contradicts item 3, 
		according to which $c$ has no such $\Pb$-extension.

\case{2}{$(x, C(x)) \not\in W$ for some $x \in \NN$.} Let $\Cc$ be the $\Pi^{0,X_0 \oplus \dots \oplus X_{d-1} \oplus Z}_1$ class of all colorings  $g: \Subtree{2}{X_0, \dots, X_{d-1}} \to k$ such that for every $\ell < N$, every $\pi \in S_0(\ell) \times \dots \times S_{d-1}(\ell)$ and every $H_0 \subseteq X_0, \dots, H_{d-1} \subseteq X_{d-1}$, one of (a), (b) or (c) fails for the pair $(x, C(x))$.
	By assumption, $\Cc \neq \emptyset$.

	By the cone avoidance basis theorem, there is some $g \in \Cc$ such that $C \nTred g \oplus X_0 \oplus \dots \oplus X_{d-1} \oplus Z$. 
	For every $j < d$, recall that $S_j(N) = \roots(X_j)$.
	We can see $X_0, \dots, X_{d-1}$ as a tuple $( X_j \uh \rho: j < d, \rho \in S_j(N) )$ of trees.
	For every $\theta = ( \rho_0, \dots, \rho_{d-1}) \in S_0(N) \times \dots \times S_{d-1}(N)$, we let $g_\theta$ be the restriction of $g$ over
	$$
		\Subtree{2}{X_0 \uh \rho_0, \dots, X_{d-1} \uh \rho_{d-1}} \to k
	$$
	By successive applications of cone avoidance of $\PMT{2}{}$ (\Cref{thm:cone-avoidance-MTT2}) applied to $g_\theta$ for each  $\theta \in S_0(N) \times \dots \times S_{d-1}(N)$,
	there is a tuple of infinite strong subtrees $( Y_{j,\rho}: j < d, \rho \in S_j(N) )$ of $( X_j \uh \rho: j < d, \rho \in S_j(N) )$ with common level function, together with a coloring $h: S_0(N) \times \dots \times S_{d-1}(N) \to k$,
		such that for every  $\theta = ( \rho_0, \dots, \rho_{d-1}) \in S_0(N) \times \dots \times S_{d-1}(N)$,
		 $g_\theta$ restricted to $\Subtree{2}{X_0 \uh \rho_0, \dots, X_{d-1} \uh \rho_{d-1}}$ is monochromatic for color $h(\theta)$. 
		
	By choice of $N$, there is some $\ell < N$, some $\pi = (\nu_0, \dots, \nu_{d-1}) \in S_0(\ell) \times \dots \times S_{d-1}(\ell)$ and some $(\ell+1)$-$\pi$-dense matrix $M \subseteq S_0(N) \times \dots \times S_{d-1}(N)$ monochromatic for~$h$. Say $M = M_0 \times \dots \times M_{d-1}$ and $i < k$ is the color of monochromaticity.
	For every $j < d$, let $P_j$ be the set of nodes in $S_j(N)$ which are not extensions of $\nu_j$. For every $j < k$, let $\hat{Y}_j = \bigcup_{\rho \in M_j \cup P_j} Y_{j,\rho}$.
	
	\begin{fact}\label{fact:pmtt2-level-homogeneous-case2-exts}
	$c' = (\operatorname{stems}(\pi), \hat{Y}_0, \dots, \hat{Y}_{d-1})$
	is a $\Pb$-extension of \[(\operatorname{stems}(\pi), \operatorname{sets}(\ell)).\]
	\end{fact}
	\begin{proof}
	Let $(\hat{X}_0, \dots, \hat{X}_{d-1}) = \operatorname{sets}(\ell)$.
	By item 1, for every $j < k$, $\roots(\hat{X}_j) = S_j(\ell+1)$. In particular, every root of $\hat{X}_j$ is extended by a root of $\hat{Y}_j$.
	\end{proof}
	
	In particular, by Fact \ref{fact:pmtt2-level-homogeneous-density-below-a-cone-condition-extension} and item 3, $c' \nVdash \Phi_e^{G_0 \oplus \dots \oplus G_{d-1} \oplus Z} \neq C$ where $e = \operatorname{req}(\pi)$.
	Moreover, since the forcing relation depends only on part of the reservoirs extending the roots of the stems, the following fact holds.
	
	\begin{fact}\label{fact:pmtt2-level-homogeneous-case2-force-diag}
	$c' \Vdash \Phi_e^{G_0 \oplus \dots \oplus G_{d-1} \oplus Z} \neq C$, where $e = \operatorname{req}(\pi)$.
	\end{fact}
	\begin{proof}
	We claim that $c' \Vdash \Phi_e^{G_0 \oplus \dots \oplus G_{d-1} \oplus Z}(x) \neq C(x)$,
	where as usual the inequality includes the possibility that the left side diverges. For every $j < d$, let $H_j \subseteq \hat{Y}_j$ be such that 
	$F_0 \cup H_0, \dots, F_{d-1} \cup H_{d-1}$ are finite strong subtrees of $T_0, \dots, T_{d-1}$, respectively, with common level function. In particular, 
	for every $j < d$, $H_j \subseteq  \bigcup_{\rho \in M_j} Y_{j,\rho}$,
	so $g$ restricted to $\Subtree{2}{H_0, \dots, H_{d-1}}$ is monochromatic for color $i$, hence since $g \in \Cc$, $\Phi_e^{(F_0 \cup H_0) \oplus \dots \oplus (F_{d-1} \cup H_{d-1}) \oplus Z}(x)$ either diverges, or is different from $C(x)$.
	This means  \[c' \Vdash \Phi_e^{G_0 \oplus \dots \oplus G_{d-1} \oplus Z}(x) \neq C(x),\] as needed.
	\end{proof}
	
	Fact \ref{fact:pmtt2-level-homogeneous-case2-force-diag} contradicts Fact \ref{fact:pmtt2-level-homogeneous-case2-exts} and item 3 of the construction, according to which $c$ has no such $\Pb$-extension. This completes Case 2.

\case{3}{otherwise.} Then $W$ is an $X_0 \oplus \dots \oplus X_{d-1} \oplus Z$-c.e.\ graph of the characteristic function of $C$, hence $C \leq X_0 \oplus \dots \oplus X_{d-1} \oplus Z$. This is a contradiction.
\end{proof}

We are now ready to complete the proof \Cref{thm:pmtt2-level-homogeneous-strong-cone-avoidance}.
By \Cref{lem:pmtt2-level-homogeneous-density-below-a-cone}, there is some cone avoiding level-homogeneous product tree condition $c$ below which, for every Turing functional $\Gamma$, the set 
$$
D_\Gamma = \{ c' \in \Pb: c' \Vdash \Gamma^{G_0 \oplus \dots \oplus G_{d-1} \oplus Z} \neq C \}
$$
is $\Pb$-dense.
Let $\Uc$ be a $\Pb$-filter which intersects every set $D_\Gamma$.
Then by definition of a product tree condition, $G^\Uc_0, \dots, G^\Uc_{d-1}$ are strong subtrees of $T_0, \dots, T_{d-1}$. Moreover, since all conditions in $\Pb$ are level-homogeneous, so are $G^\Uc_0, \dots, G^\Uc_{d-1}$. Since  $\Uc$ intersects every set $D_\Gamma$, we have $C \nTred G^\Uc_0 \oplus \dots \oplus G^\Uc_{d-1} \oplus Z$.
Lastly, by \Cref{lem:product-tree-genericity-implies-infinity}, $G^\Uc_0, \dots, G^\Uc_{d-1}$ are all infinite.
This completes the proof of \Cref{thm:pmtt2-level-homogeneous-strong-cone-avoidance}.
\end{proof} %END OF PROOF OF THE THEOREM

\begin{statement}\index{statement!$\PMT{n}{k,\ell}$}
	For all $n,k,\ell \geq 1$, $\PMT{n}{k,\ell}$ is the following statement. Let $T_0,\ldots,T_{d-1}$ be infinite trees with no leaves. For all colorings \[f: \Subtree{n}{T_0,\ldots,T_{d-1}} \to k\] there exists $(S_0,\ldots,S_{d-1}) \in \Subtree{\omega}{T_0,\ldots,T_{d-1}}$ such that $f$ takes at most $\ell$ values on $\Subtree{n}{S_0,\ldots,S_{d-1}}$.\end{statement}

\begin{theorem}\label{thm:pmtt2k2-strong-cone-avoidance}
$(\forall k)\PMT{2}{k,2}$ admits strong cone avoidance.
\end{theorem}
\begin{proof}
Fix two sets $C$ and $Z$ such that $C \nTred Z$.
Also fix a $Z$-computable collection of $Z$-computably bounded trees with no leaves $T_0, \dots,\allowbreak T_{d-1} \subseteq \baire$.
Let $f: \Subtree{2}{T_0, \dots, T_{d-1}} \to k$ be a coloring.
By \Cref{thm:pmtt2-level-homogeneous-strong-cone-avoidance}, there exist strong
  subtrees $(S_0, \dots, S_{d-1}) \in \Subtree{\omega}{T_0, \dots, T_{d-1}}$ on which $f$ is level-homogeneous, and such that
  $C \nTred S_0 \oplus \dots \oplus S_{d-1} \oplus Z$.

Let  $g: [\NN]^2 \to k$ which to some $\{x_0 < x_1\} \in [\NN]^2$ associates 
the color of any element of $\Subtree{2}{S_0, \dots, S_{d-1}}$ whose level 
function has for range $\{x_0, x_1\}$. By strong cone avoidance of $\RT{2}{k,2}$ (see Wang~\cite{Wang2014Some}, Theorem 3.2, or Cholak and Patey~\cite{Cholak2019Thin}, Corollary 4.17),
there exists an infinite set $H \subseteq \NN$ such that $g$ restricted to $[H]^2$
uses at most 2 colors. Using $H$, one can compute strong subtrees $(U_0, \dots, U_{d-1}) \in \Subtree{\omega}{S_0, \dots, S_{d-1}}$ whose level function is the principal function of $H$.
By definition of $g$, $f$ uses at most 2 colors over $\Subtree{2}{U_0, \dots, U_{d-1}}$.
And by transitivity of the strong subtree relation, $(U_0, \dots, U_{d-1}) \in \Subtree{\omega}{T_0, \dots, T_{d-1}}$. 
This completes the proof of \Cref{thm:pmtt2k2-strong-cone-avoidance}.
\end{proof}%fembeddin\benoit{I agree}

\begin{theorem}\label{thm:pmtt3k2-cone-avoidance}
$(\forall k)\PMT{3}{k,2}$ admits cone avoidance.
\end{theorem}
\begin{proof}
Fix two sets $C$ and $Z$ such that $C \nTred Z$.
Also fix a $Z$-computable collection of $Z$-computably bounded trees with no leaves $T_0, \dots,\allowbreak T_{d-1} \subseteq \baire$.
Let $f: \Subtree{3}{T_0, \dots, T_{d-1}} \to k$ be a $Z$-computable coloring.
 
By \Cref{thm:cmtt-admits-strong-cone-avoidance}, there are strong subtrees \[(S_0, \dots, S_{d-1}) \in \Subtree{\omega}{T_0, \dots, T_{d-1}}\] on which $f$ is stable, and such that $C \nTred S_0 \oplus \dots \oplus S_{d-1} \oplus Z$.
Let $g: \Subtree{2}{S_0, \dots, S_{d-1}}$ be the limit coloring induced by stability of $f$. By strong cone avoidance of $\PMT{2}{k,2}$ (\Cref{thm:pmtt2k2-strong-cone-avoidance}),
there are strong subtrees $(U_0, \dots, U_{d-1}) \in \Subtree{\omega}{S_0, \dots, S_{d-1}}$ on which $g$ uses at most 2 colors, and $C \nTred U_0 \oplus \dots \oplus U_{d-1} \oplus Z$.
By $U_0 \oplus \dots \oplus U_{d-1} \oplus Z$-computably thinning out the set of levels,
we can obtain a tuple of strong subtrees $(V_0, \dots, V_{d-1}) \in \Subtree{\omega}{U_0, \dots, U_{d-1}}$, on which $f$ uses at most 2 colors. In particular, by transitivity of the strong subtree relation, $(V_0, \dots, V_{d-1}) \in \Subtree{\omega}{T_0, \dots, T_{d-1}}$.
Last, $C \nTred V_0 \oplus \dots \oplus V_{d-1} \oplus Z$.
This completes the proof of \Cref{thm:pmtt3k2-cone-avoidance}.
\end{proof}

\begin{corollary}
$(\forall k)\PMT{3}{k,2}$ does not imply $\ACA_0$ over $\RCA_0$.
\end{corollary}
\begin{proof}
Immediate by \Cref{thm:pmtt3k2-cone-avoidance} and \Cref{lem:cone-avoidance-not-aca}.
\end{proof}

%%% Local Variables:
%%% mode: latex
%%% TeX-master: "../embryon"
%%% End:

\chapter{Devlin's theorem}\label{sec:devlin}
%\pelliot{To put at the right place}
%\begin{theorem}
%  Let $\Bc=(\Nb, <_\Qb)$ be a dense linear order. Then, there exists a Ramsey enrichment $\hat\Bc$ of $\Bc$. The language $\hat\Lc$ consists of $\Lc$ and an additional relation of arity 4.
%\end{theorem}
%\begin{proof}
%  The predicate $P(a,b,c,d)$ we add to the language will be interpreted as $a\meet b<c\meet d$, for a definition of $\meet$ that we will explain in the proof.
%\end{proof}
%As for every $n$, there is only one embedding type in $\Bc$ of size $n$, in order to know the big Ramsey number for a tuple of $n$ elements, we need to count the number of embedding types in $\hat\Bc$ with $n$ elements.
%\begin{theorem}
%  There are ... Joyce trees with $n$ leaves.
%\end{theorem}
%\begin{corollary}
%  The big Ramsey number for $n$-tuples inside a DLO structure is ...
%\end{corollary}
%\pelliot{End To put at the right place}

%In this subsection, we study the big Ramsey number of the dense linear order.

%\begin{statement}[Devlin's theorem]
 % The statement $\DT{n}{k,\ell}$ is the following: ``for every coloring $f: [\mathbb Q]^n \to k$, there exists a set $S$ order-isomorphic to $\mathbb Q$ such that $[S]^n$ uses at most $\ell$ many colors''.
%\end{statement}
\index{Devlin's theorem}

We now turn to some applications of Milliken's tree theorem and of our preceding work. We begin, in this chapter, with Devlin's theorem. This states that the dense linear orders admit big Ramsey numbers, meaning that for every $n$ there exists an $\ell$ such that for any finite coloring of the $n$-tuples of rationals, there exists a dense linear subordering of $\Qb$ on which the coloring takes only $\ell$ colors. This corresponds to the statement $(\forall k)\DT{n}{k,\ell}$. Moreover, the function which associates to each $n$ the minimal such $\ell$ is known, being the sequence of the so-called \index{odd tangent numbers} \index{$t_{\DT{}{}}$} ``odd tangent numbers'' $t_{\DT{}{}}(n)$, as defined in \cite[p.~147]{Todorcevic2010Ramsey}. (To list the first few, we have $t_{\DT{}{}}(1) = 1$, $t_{\DT{}{}}(2) = 2$, $t_{\DT{}{}}(3) = 16$, and $t_{\DT{}{}}(4) = 272$.)

Some variants of Devlin's theorem, such as the Erd\H{o}s-Rado theorem for colorings of rationals (see \Cref{subsect:er-theorem} below), have previously been studied in the reverse mathematics literature. So part of our motivation here is to see what new insights can be obtained using the tools from our earlier sections. Another, of course, is to understand more directly how Devlin's theorem compares to Milliken's tree theorem in its computable content. As remarked following \Cref{D:bigRamsey} above, one important idea here is to distinguish features that are intrinsic to a structure yet somehow hidden, as is the case when a structure has a big Ramsey degree bigger than $1$. This is what we alluded to as being describable by an ``enrichment'' of the language and gives rise to the notion of big Ramsey structure (\Cref{def:big-ramsey-structure}). In the case of Devlin's theorem, this can be made explicit using a representation of the rationals in terms of binary strings and so-called Joyce trees, which we define below. This is a somewhat technical construction, but it eliminates the need for more intricate combinatorial objects, such as embedding types, and will simplify our discussion not only of Devlin's theorem but also of the Rado graph theorem which we consider in the next chapter.

\section{A big Ramsey structure for dense linear orders}

As noted below Statement \ref{stmt:DT} above, the big Ramsey degree of Devlin's theorem for pairs is 2, while there is only one sub-order of size 2. We now describe an enrichment to the language of orders to obtain a big Ramsey structure for the dense linear orders with no endpoints which reflects this fact. We will see that we can represent any countable order as an anti-chain $A$ in $\cantor$ with respect to the prefix relation, equipped with the lexicographic order $\ltlex$. \index{$\ltlex$}\index{ordering!$\ltlex$} Then, given two elements $\sigma, \tau \in A$, some extra structure induced by the string representation can be exploited, such as the comparison between the length of $\sigma$ and the length of $\tau$, but as well with respect to length of their meet $\sigma \meet \tau$. 

%\peter{The defn of meet closure in section 7 needs to be moved much earlier.  I used this in the paragraph above.}\ludovic{I removed the meet-closure from the paragraph above since it should really be $A$, not $A^\wedge$.}

As we will see in \Cref{thm:joyce-orders-can-be-coded} we can always ensure that the length of any string in $A^\wedge = \{ \sigma \meet \tau: \sigma, \tau \in A \}$ is unique\index{$A^\wedge$}. There are then 2 possible cases for a pair  $\sigma \ltlex \tau$ in $A$: either $|\sigma| >_\Nb |\tau|$, or $|\sigma| <_\Nb |\tau|$. The case of the equality has been ruled out since all lengths of the strings in $A^\wedge$ will be unique. While there are examples where their lengths are not unique (see the figure below), the example where they are unique will prove to be very illustrative. 

A finite subset of $n$ elements in $A$ can be represented as a particular kind of binary tree, known as a Joyce tree. A \emph{Joyce tree} \index{Joyce!tree}\index{tree!Joyce}of size $n$ is a labeled tree with $2n-1$ vertices, such that every non-leaf has exactly two immediate children. The labels \index{label!Joyce tree} are among $\{1, \dots, 2n-1\}$ and every child has a label greater than its parent (see Street~\cite{streets}). See \Cref{fig:joyce-trees} for some examples of Joyce trees.

\begin{figure}[htbp]
\centering
	\begin{tikzpicture}[scale=1.5,font=\normalsize]
		\tikzset{
			empty node/.style={circle,inner sep=0,fill=none},
			solid node/.style={circle,draw,inner sep=1.5,fill=black},
			hollow node/.style={circle,draw,inner sep=1.5,fill=white},
			gray node/.style={circle,draw={rgb:black,1;white,4},inner sep=1.5,fill={rgb:black,1;white,4}}
		}
		\tikzset{snake it/.style={decorate, decoration=snake, line cap=round}}
		\tikzset{gray line/.style={line cap=round,thick,color={rgb:black,1;white,4}}}
		\tikzset{thick line/.style={line cap=round,rounded corners=0.1mm,thick}}
		\node (a)[empty node,label=below:{$1$}] at (0,0) {};
		\node (b)[empty node,label=below:{$2$}] at (-0.4,0.5) {};
		\node (c)[empty node,label=above:{$3$}] at (0.4,0.5) {};
		\node (d)[empty node,label=above:{$4$}] at (-0.8,1) {};
		\node (e)[empty node,label=above:{$5$}] at (0,1) {};
		\draw[thick line] (a.center) to (b.center) to (d.center);
		\draw[thick line] (b.center) to (e.center);
		\draw[thick line] (a.center) to (c.center);
	\end{tikzpicture}
	\hspace{5mm}
	\begin{tikzpicture}[scale=1.5,font=\normalsize]
		\tikzset{
			empty node/.style={circle,inner sep=0,fill=none},
			solid node/.style={circle,draw,inner sep=1.5,fill=black},
			hollow node/.style={circle,draw,inner sep=1.5,fill=white},
			gray node/.style={circle,draw={rgb:black,1;white,4},inner sep=1.5,fill={rgb:black,1;white,4}}
		}
		\tikzset{snake it/.style={decorate, decoration=snake, line cap=round}}
		\tikzset{gray line/.style={line cap=round,thick,color={rgb:black,1;white,4}}}
		\tikzset{thick line/.style={line cap=round,rounded corners=0.1mm,thick}}
		\node (a)[empty node,label=below:{$1$}] at (0,0) {};
		\node (b)[empty node,label=below:{$3$}] at (-0.4,0.5) {};
		\node (c)[empty node,label=above:{$2$}] at (0.4,0.5) {};
		\node (d)[empty node,label=above:{$4$}] at (-0.8,1) {};
		\node (e)[empty node,label=above:{$5$}] at (0,1) {};
		\draw[thick line] (a.center) to (b.center) to (d.center);
		\draw[thick line] (b.center) to (e.center);
		\draw[thick line] (a.center) to (c.center);
	\end{tikzpicture}
	\hspace{5mm}
	\begin{tikzpicture}[scale=1.5,font=\normalsize]
		\tikzset{
			empty node/.style={circle,inner sep=0,fill=none},
			solid node/.style={circle,draw,inner sep=1.5,fill=black},
			hollow node/.style={circle,draw,inner sep=1.5,fill=white},
			gray node/.style={circle,draw={rgb:black,1;white,4},inner sep=1.5,fill={rgb:black,1;white,4}}
		}
		\tikzset{snake it/.style={decorate, decoration=snake, line cap=round}}
		\tikzset{gray line/.style={line cap=round,thick,color={rgb:black,1;white,4}}}
		\tikzset{thick line/.style={line cap=round,rounded corners=0.1mm,thick}}
		\node (a)[empty node,label=below:{$1$}] at (0,0) {};
		\node (b)[empty node,label=below:{$2$}] at (-0.4,0.5) {};
		\node (c)[empty node,label=above:{$4$}] at (0.4,0.5) {};
		\node (d)[empty node,label=above:{$3$}] at (-0.8,1) {};
		\node (e)[empty node,label=above:{$5$}] at (0,1) {};
		\draw[thick line] (a.center) to (b.center) to (d.center);
		\draw[thick line] (b.center) to (e.center);
		\draw[thick line] (a.center) to (c.center);
	\end{tikzpicture}
	\hspace{5mm}
	\begin{tikzpicture}[scale=1.5,font=\normalsize]
		\tikzset{
			empty node/.style={circle,inner sep=0,fill=none},
			solid node/.style={circle,draw,inner sep=1.5,fill=black},
			hollow node/.style={circle,draw,inner sep=1.5,fill=white},
			gray node/.style={circle,draw={rgb:black,1;white,4},inner sep=1.5,fill={rgb:black,1;white,4}}
		}
		\tikzset{snake it/.style={decorate, decoration=snake, line cap=round}}
		\tikzset{gray line/.style={line cap=round,thick,color={rgb:black,1;white,4}}}
		\tikzset{thick line/.style={line cap=round,rounded corners=0.1mm,thick}}
		\node (a)[empty node,label=below:{$1$}] at (0,0) {};
		\node (b)[empty node,label=below:{$2$}] at (-0.4,0.5) {};
		\node (c)[empty node,label=above:{$5$}] at (0.4,0.5) {};
		\node (d)[empty node,label=above:{$3$}] at (-0.8,1) {};
		\node (e)[empty node,label=above:{$4$}] at (0,1) {};
		\draw[thick line] (a.center) to (b.center) to (d.center);
		\draw[thick line] (b.center) to (e.center);
		\draw[thick line] (a.center) to (c.center);
	\end{tikzpicture}\\
	\medskip
	\begin{tikzpicture}[scale=1.5,font=\normalsize]
		\tikzset{
			empty node/.style={circle,inner sep=0,fill=none},
			solid node/.style={circle,draw,inner sep=1.5,fill=black},
			hollow node/.style={circle,draw,inner sep=1.5,fill=white},
			gray node/.style={circle,draw={rgb:black,1;white,4},inner sep=1.5,fill={rgb:black,1;white,4}}
		}
		\tikzset{snake it/.style={decorate, decoration=snake, line cap=round}}
		\tikzset{gray line/.style={line cap=round,thick,color={rgb:black,1;white,4}}}
		\tikzset{thick line/.style={line cap=round,rounded corners=0.1mm,thick}}
		\node (a)[empty node,label=below:{$1$}] at (0,0) {};
		\node (b)[empty node,label=below:{$2$}] at (-0.4,0.5) {};
		\node (c)[empty node,label=above:{$3$}] at (0.4,0.5) {};
		\node (d)[empty node,label=above:{$5$}] at (-0.8,1) {};
		\node (e)[empty node,label=above:{$4$}] at (0,1) {};
		\draw[thick line] (a.center) to (b.center) to (d.center);
		\draw[thick line] (b.center) to (e.center);
		\draw[thick line] (a.center) to (c.center);
	\end{tikzpicture}
	\hspace{5mm}
	\begin{tikzpicture}[scale=1.5,font=\normalsize]
		\tikzset{
			empty node/.style={circle,inner sep=0,fill=none},
			solid node/.style={circle,draw,inner sep=1.5,fill=black},
			hollow node/.style={circle,draw,inner sep=1.5,fill=white},
			gray node/.style={circle,draw={rgb:black,1;white,4},inner sep=1.5,fill={rgb:black,1;white,4}}
		}
		\tikzset{snake it/.style={decorate, decoration=snake, line cap=round}}
		\tikzset{gray line/.style={line cap=round,thick,color={rgb:black,1;white,4}}}
		\tikzset{thick line/.style={line cap=round,rounded corners=0.1mm,thick}}
		\node (a)[empty node,label=below:{$1$}] at (0,0) {};
		\node (b)[empty node,label=below:{$3$}] at (-0.4,0.5) {};
		\node (c)[empty node,label=above:{$2$}] at (0.4,0.5) {};
		\node (d)[empty node,label=above:{$5$}] at (-0.8,1) {};
		\node (e)[empty node,label=above:{$4$}] at (0,1) {};
		\draw[thick line] (a.center) to (b.center) to (d.center);
		\draw[thick line] (b.center) to (e.center);
		\draw[thick line] (a.center) to (c.center);
	\end{tikzpicture}
	\hspace{5mm}
	\begin{tikzpicture}[scale=1.5,font=\normalsize]
		\tikzset{
			empty node/.style={circle,inner sep=0,fill=none},
			solid node/.style={circle,draw,inner sep=1.5,fill=black},
			hollow node/.style={circle,draw,inner sep=1.5,fill=white},
			gray node/.style={circle,draw={rgb:black,1;white,4},inner sep=1.5,fill={rgb:black,1;white,4}}
		}
		\tikzset{snake it/.style={decorate, decoration=snake, line cap=round}}
		\tikzset{gray line/.style={line cap=round,thick,color={rgb:black,1;white,4}}}
		\tikzset{thick line/.style={line cap=round,rounded corners=0.1mm,thick}}
		\node (a)[empty node,label=below:{$1$}] at (0,0) {};
		\node (b)[empty node,label=below:{$2$}] at (-0.4,0.5) {};
		\node (c)[empty node,label=above:{$4$}] at (0.4,0.5) {};
		\node (d)[empty node,label=above:{$5$}] at (-0.8,1) {};
		\node (e)[empty node,label=above:{$3$}] at (0,1) {};
		\draw[thick line] (a.center) to (b.center) to (d.center);
		\draw[thick line] (b.center) to (e.center);
		\draw[thick line] (a.center) to (c.center);
	\end{tikzpicture}
	\hspace{5mm}
	\begin{tikzpicture}[scale=1.5,font=\normalsize]
		\tikzset{
			empty node/.style={circle,inner sep=0,fill=none},
			solid node/.style={circle,draw,inner sep=1.5,fill=black},
			hollow node/.style={circle,draw,inner sep=1.5,fill=white},
			gray node/.style={circle,draw={rgb:black,1;white,4},inner sep=1.5,fill={rgb:black,1;white,4}}
		}
		\tikzset{snake it/.style={decorate, decoration=snake, line cap=round}}
		\tikzset{gray line/.style={line cap=round,thick,color={rgb:black,1;white,4}}}
		\tikzset{thick line/.style={line cap=round,rounded corners=0.1mm,thick}}
		\node (a)[empty node,label=below:{$1$}] at (0,0) {};
		\node (b)[empty node,label=below:{$2$}] at (-0.4,0.5) {};
		\node (c)[empty node,label=above:{$5$}] at (0.4,0.5) {};
		\node (d)[empty node,label=above:{$4$}] at (-0.8,1) {};
		\node (e)[empty node,label=above:{$3$}] at (0,1) {};
		\draw[thick line] (a.center) to (b.center) to (d.center);
		\draw[thick line] (b.center) to (e.center);
		\draw[thick line] (a.center) to (c.center);
	\end{tikzpicture}
\caption{Eight Joyce trees among the sixteen Joyce trees with three leaves. The eight remaining Joyce trees are mirror reflections of these along a vertical axis though the root node.}
\label{fig:joyce-trees}
\end{figure}
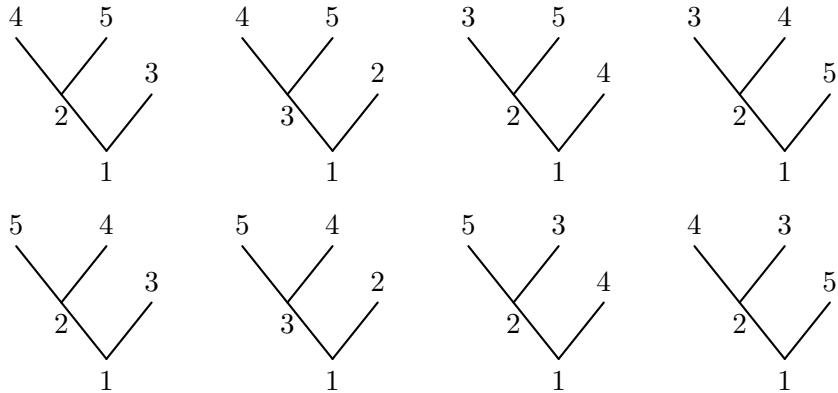

%The labels of the leaves of a Joyce tree correspond to the length of the elements of $A$, while the labels of the branching nodes correspond of the length of their meets.  %% not true in our example. 

Recall that a function is \emph{symmetric} \index{function!symmetric}if its value is the same no matter the order of its arguments. In what follows, we use the string representation of a countable order $(X, <)$ to enrich the order with a symmetric binary function $\meetlevel{\cdot}{\cdot}: X^2\to \Nb$.   We shall later refer to any value of the range of $\meetlevel{\cdot}{\cdot}$ as a \emph{label}. \index{label!Joyce order}  %One can think of a label $\meetlevel{x}{y}$ as the length of the longest common substring of $x$ and $y$, given a string representation of $x$ and $y$. 
The label of an element $x \in X$ is $\meetlevel{x}{x}$ and $\meetlevel{x}{y}$ is the label given to $x \meet y$. Again the illustrative example is when all lengths are unique, to consider the lengths of these nodes as the labels.  

However there are other examples. Consider the first tree in \Cref{fig:joyce-trees} and name its leaves $x$, $y$ and $z$, from left to right. This tree induces a symmetric function $\meetlevel{\cdot}{\cdot}: \{x, y, z\}^2 \to \{1, \dots, 5\}$ as follows: $\meetlevel{x}{x} = 4$, $\meetlevel{y}{y} = 5$, $\meetlevel{z}{z} = 3$, $\meetlevel{x}{y} = 2$, $\meetlevel{x}{z} = 1$, $\meetlevel{y}{z} = 1$. The tree representation also induces an ordering of the labels $\{1, \dots, 5\}$ by reading them from left to right. In this case, $4 < 2 < 5 < 1 < 3$. As we will later see in \Cref{JOtoJT}  $(\{x, y, z\}, <, \meetlevel{\cdot}{\cdot})$ can be used to recover the original Joyce tree.

%This yields the notion of Joyce order, which can be thought of the set of leaves of a Joyce tree.
%There exists exactly two Joyce orders of length 2, each of which admits big Ramsey degree 1.
%Joyce orders make explicit the hidden underlying structure of dense linear orders with respect to the Ramsey property.
%

\begin{definition}\label{def:jo}\index{$\meetlevel{\cdot}{\cdot}$}\index{order!Joyce}\index{Joyce!order}
  A \emph{Joyce order} is an order $(X,<)$ equipped with a symmetric function $\meetlevel{\cdot}{\cdot}: X^2\to \Nb$ such that for every $x, y, z, t\in X$, not all equal, with $x\leq y$ and $z\leq t$:
  \begin{enumerate}
  \item[\Jo{1}] $\meetlevel{x}{y} <_\Nb \meetlevel{x}{z}\implies(x< y\iff z< y)$;
  \item[\Jo{2}] $\meetlevel{x}{y}<_\Nb\meetlevel{x}{z}\implies\meetlevel{x}{y} = \meetlevel{z}{y}$;
  \item[\Jo{3}] $\meetlevel{x}{y}=\meetlevel{z}{t}\implies \meetlevel{x}{y}<_\Nb\min(\meetlevel{x}{z}, \meetlevel{y}{t})$.
  \end{enumerate}
\end{definition}

Note that the axioms of a Joyce order are universal, hence every subset of a Joyce order induces again a Joyce order.

%\peter{I am still confused by the paragraph after 5.1.  Could you first explain what [[ ]] is for some of the Joyce trees in figure 5.  I would like to 
%semantically understand a JS.   Maybe move def 5.5 to right after 5.1 and then explain the JS for the JO in figure 5.  Please write out that 5.2 1. says the label of meet is less than the label of x and y 2. Says all leafs have unique labels, 3 says that all leafs are not meets.  Maybe you want to have this lemma before you have the paragraph after 5.1 which tries to give an intuition to these rules.}

%%One can think of a Joyce order as the set of leaves of a binary tree whose meets have unique label.
% One can think of a label $\ell = \meetlevel{x}{y}$ as the length of the longest common substring of $x$ and $y$, given a string representation of $x$ and $y$. The label of an element $x \in X$ is $\meetlevel{x}{x}$. Following the previous intuition, $\meetlevel{x}{x}$ corresponds to the length of the string representation of $x$. 

Every Joyce tree gives rise to a Joyce order. Let $X$ be the set of leaves, $\meetlevel{x}{y}$ be the label of the node $x \meet y$, and $\ltlex$ be the lexicographical order on $X$. Then, we claim that $(X,\ltlex,\meetlevel\cdot\cdot)$ is a Joyce order: If $\meetlevel x y<_\Nb\meetlevel x z$, then as $x\meet y$ and $x\meet z$ are comparable and every child has a label greater than its parent, $x\meet y\prec x\meet z$. But then, $x\ltlex y\iff (x\meet z)\ltlex y\iff z\ltlex y$, and $x\meet y = (x\meet z)\meet y=z\meet y$, so both \Jo1 and \Jo2 hold. Finally let $x,y,z,t\in X$ not all equal be such that $\meetlevel xy=\meetlevel zt$. By injectivity of the labelling, $x\meet y= z\meet t$. If $x=y$ then $x\meet y =x$ is a leaf, so $z\meet t$ is a leaf, thus $x=y=z=t$, a contradiction. Therefore, $x\neq y$ and $z\neq t$. Now, suppose $y=z$. As a Joyce tree is binary branching, and $x\ltlex y=z\ltlex t$, one cannot have $x\meet y=z\meet t$, a contradiction. Finally suppose that $x,y,z,t$ are all different. The fact that $x\ltlex y$ and $z\ltlex t$ implies $(x\meet y)0\preceq x$ and $(x\meet y)0\preceq z$, so $x\meet y\prec x\meet z$. By the fact that the label of a child is greater than those of its parents, $\meetlevel xy<_\Nb\meetlevel xz$. Similarly, $\meetlevel xy<_\Nb\meetlevel yt$.  %\peter{I will let someone else finish this proof.  We need this since we are counting finite JO, JT and JS of size n as the same.  This is now mentioned after Lemma 5.6.}

We can use the illustrative example when $X \subseteq \omega^{<\omega}$, $<$ is $<_{lex}$, and all lengths in $X^\wedge = \{\sigma \meet \tau: \sigma, \tau \in X \}$ are unique to get an intuition about these rules. Assume $x\meet z$ is longer than $x\meet y$.  \Jo{1} says that either $y$ is to the left  of both $x$ and $z$ (i.e. $y< x$ and $y < z$) or $y$ is to the right of both $x$ and $z$. \Jo{2} says that $x\meet y = z \meet y $ (after all $x \meet y \preceq x \meet z$). 
%The axiom \Jo{1} informally states that if $z$ and $x$ belong to the same branch but $y$ does not, then $y$ must be either on the left or on the right of both elements. Therefore, it asserts that the order has to be compatible with the tree structure. In particular, one can refer to the order to decide whether $x$ belongs to the left or right branch after a meet. The second axiom \Jo{2} is similar to the first one, but for meets: it basically says that the meets have to be compatible with the tree structure.
The third axiom \Jo{3} says that if both pairs $x, y$ and $z, t$ have a meet with the same label, $x< y$, and $z < t$, then $x \meet z$ must properly above $x \meet y$ and similarly for $y \meet t$.  This implies the meets are binary branching in $X^\wedge$. This also implies that different meets must have different labels.

%The following lemma states some essential properties of Joyce orders.

\begin{lemma}\label{lem:prop-of-JO}
  The following is true in any Joyce order $(X,<,\meetlevel{\cdot}{\cdot})$:
  \begin{enumerate}
  \item\label{it:prop-of-JO-0} for all $x,y \in X$ with $x\neq y$, $\meetlevel{x}{y} <_\Nb \min(\meetlevel{x}{x}, \meetlevel{y}{y})$;
  \item\label{it:prop-of-JO-1} for all $x,y \in X$ with $x\neq z$, $\meetlevel{x}{x}\neq \meetlevel{z}{z}$;
  \item\label{it:prop-of-JO-2} for all $x,z,t \in X$ with $z\neq t$, $\meetlevel{x}{x}\neq \meetlevel{z}{t}$.
%  \item\label{it:prop-of-JO-3} $\forall x,z,t\in X$ not all equal, $\lnot(\meetlevel{x}{z}= \meetlevel{z}{t}= \meetlevel{x}{t})$.
%  \item\label{it:prop-of-JO-3} $\forall x,y,t\in X$ with $x\neq y$, $|x\meet y|=|\neq |z\meet t|$.
  \end{enumerate}
\end{lemma}
\begin{proof}
  Item 1: by \Jo{3} with $x=z$ and $y=t$.
  Item 2: by \Jo{3} with $x=y$ and $z=t$, $\meetlevel{x}{x}=\meetlevel{z}{z}\implies \meetlevel{x}{x}<_\Nb\min(\meetlevel{x}{z}, \meetlevel{x}{z})$. By Item 1, $\meetlevel{x}{x}<_\Nb\min(\meetlevel{x}{z}, \meetlevel{x}{z})$ cannot hold, so $\meetlevel{x}{x}\neq\meetlevel{z}{z}$.
  Item 3: by \Jo{3} with $x=y$, $\meetlevel{x}{x}=\meetlevel{z}{t}\implies \meetlevel{x}{x}<_\Nb\min(\meetlevel{x}{z}, \meetlevel{x}{t})$. Since $z \neq t$, then either $x \neq z$ or $x \neq t$. In either case, by Item 1, $\meetlevel{x}{z} <_\Nb \meetlevel{x}{x}$ or $\meetlevel{x}{t} <_\Nb \meetlevel{x}{x}$, so $\meetlevel{x}{x}<_\Nb\min(\meetlevel{x}{z}, \meetlevel{x}{t})$ cannot hold, hence $\meetlevel{x}{x}\neq\meetlevel{z}{t}$.
%  Item 4: The case where two elements are equal is already tackled by item 3 \Jo{3} with $x=z$ and $y=t$. Otherwise, suppose without loss of generality that $x<z<t$, then the result follow by \Jo3 with $y=z$.
\end{proof}

The first item says that the label of $x \meet y$ is less than the labels of $x$ and $y$.  The second says that each leaf has a unique label.  The third says that no meet can have the same label as a leaf.  Items 4 and 5 of the following lemma show that the labels of the leafs and meets are always different. 

%The following lemma continues to relate Joyce orders to the structure of tree. % show that one can reason inductively over Joyce orders like a tree.

\begin{lemma}\label{lem:JO-are-trees}
Let $(X,<,\meetlevel{\cdot}{\cdot})$ be a (finite or infinite) Joyce order with minimal label $\ell \in \omega$. Let $x \leq y \in X$ be such that $\meetlevel{x}{y} = \ell$ and let $X_x = \{ z \in X: \meetlevel{x}{z} >_\Nb \ell \}$
and $X_y = \{ z \in X: \meetlevel{y}{z} >_\Nb \ell \}$. The following holds:
\begin{enumerate}
	\item\label{it:JO-are-trees-0} if $x = y$ then $|X| = 1$ and $X_x = X_y = \emptyset$;
	\item\label{it:JO-are-trees-1} if $x < y$ then $X = X_x \sqcup X_y$ with $x \in X_x$ and $y \in X_y$;
	\item\label{it:JO-are-trees-2} for all $z \in X_x$ and all $t \in X_y, z < t$ and $\meetlevel{z}{t} = \ell$;
	\item\label{it:JO-are-trees-3} the labels over $X_x^2$ and $Y_y^2$ are disjoint;
	\item\label{it:JO-are-trees-4} if $|X| = n$ then there are $2n+1$ distinct labels over $X^2$.
\end{enumerate}
\end{lemma}
\begin{proof}
  \Cref{it:JO-are-trees-0}: Let $z\in X$. If $z\neq x$ we would have by \Cref{it:prop-of-JO-0} of \Cref{lem:prop-of-JO} $\meetlevel{x}{z}<_\Nb\meetlevel{x}{x}$, a contradiction with the minimality of $\ell$. Therefore, $z=x$ and $|X|=1$. As $x\not\in X_x\subseteq X$, $X_x=X_y=\emptyset$.

    \Cref{it:JO-are-trees-1}: Let $z\in X_x$. By \Jo{2}, we must have $\meetlevel{y}{z}=\ell$, and therefore $z\not\in X_y$, so $X_x\cap X_y=\emptyset$. Now, let $t\in X$ such that $\meetlevel{x}{t}=\ell$. By \Jo{3} applied with $x=z$, we have $\meetlevel{x}{y}<\meetlevel{y}{t}$, and so $t\in X_y$. By \Cref{it:prop-of-JO-0} of \Cref{lem:prop-of-JO}, $x \in X_x$ and $y \in X_y$.

    \Cref{it:JO-are-trees-2}: We have $\meetlevel{x}{z}>\ell=\meetlevel{x}{y}$, so by \Jo{2}, we must have $\meetlevel{y}{z}=\ell$. But we also have $\meetlevel{t}{y}>\ell=\meetlevel{y}{z}$, so by another application of \Jo{2}, we must have $\meetlevel{z}{t}=\ell$.

    \Cref{it:JO-are-trees-3}: Let $z_0,t_0\in X_x$ and $z_1,t_1\in X_y$. Suppose $\meetlevel{z_0}{t_0}=\meetlevel{z_1}{t_1}$. By application of \Jo3, we would have $\meetlevel{z_0}{t_0}<_\Nb \meetlevel{z_0}{t_1}$, however by \Cref{it:JO-are-trees-2} $\meetlevel{z_0}{t_1}=\ell$, and $\meetlevel{z_0}{t_0}<_\Nb\ell$ is a contradiction.

\Cref{it:JO-are-trees-4}: By induction over $n \geq 1$. For $n = 1$, $X = \{x\}$, then the unique label is $\meetlevel{x}{x}$. For $n > 1$, assume by induction hypothesis that any non-empty Joyce order of size $m < n$ has $2m-1$ distinct labels. Let $\ell \in \omega$ be the minimal label of $X$ and let $x \leq y \in X$ be such that $\meetlevel{x}{y} = \ell$. Define $X_x$ and $X_y$ as above. By Item 1, since $|X| > 1$, then $x < y$. By Item 2, $X_x \neq \emptyset$ and $X_y \neq\emptyset$ and $X = X_x \sqcup X_y$. By induction hypothesis, there are $2|X_x|-1$ distinct labels over $X_x^2$ and $2|X_y|-1$ distinct labels over $X_y^2$. By item 4, the labels are disjoint, so there are $2(|X_x|+|X_y|)-2 = 2n-2$ distinct labels over $X_x^2 \cup X_y^2$. Last, by Item 3, for every $z \in X_x$ and $t \in X_y$, $\meetlevel{z}{t} = \ell$, so the only label over $X_x \times X_y$ is $\ell$. Therefore there are $2n-1$ distinct labels over $X^2 = X_x^2 \cup X_y^2 \cup (X_x \times X_y)$.
\end{proof}

\begin{lemma}[Representing a finite Joyce order as a Joyce tree] \label{JOtoJT}
There is a computable function $J_X$ such that if $(X,<,\meetlevel{\cdot}{\cdot})$ is a Joyce order of where $|X| = n$ then $J_X$ is Joyce tree of size $n$. We shall refer to $J_X$ as the Joyce tree coded by~$X$.
\end{lemma}

\begin{proof}
	
%	Thanks to \Cref{lem:JO-are-trees}, we can now reason over Joyce orders inductively like binary trees, and formally define the correspondance between a finite Joyce order and a Joyce tree.

%Given a non-empty finite Joyce order $(X,<,\meetlevel{\cdot}{\cdot})$ of size $n$, 

Let $L$ be the set of labels over $X^2$ and $l$ be the minimal label. For every string $\sigma \in 2^{<\omega}$, we will define a binary tree $J_{X, \sigma} \subseteq 2^{<\omega}$ whose root is $\sigma$, with $2n-1$ nodes, such that every non-leaf has two immediate children, and every node has a unique label in $L$. The construction goes inductively as follows:

If $\ell = \meetlevel{x}{x}$ for some $x \in X$, then by \Cref{lem:JO-are-trees}, $X = \{x\}$ and $J_{X,\sigma} = \{\sigma\}$ where $\sigma$ has label $\ell$.
If $\ell = \meetlevel{x}{y}$ with $x < y$, then let $X_x$ and $X_y$ be defined as in \Cref{lem:JO-are-trees}. By \Cref{it:JO-are-trees-1} of \Cref{lem:JO-are-trees}, $X = X_x \sqcup X_y$.
By Items \ref{it:JO-are-trees-2} and \ref{it:JO-are-trees-3}
of \Cref{lem:JO-are-trees}, $L = L_x \sqcup L_y \sqcup \{\ell\}$, where $L_x$ and $L_y$ are the sets of labels over $X_x^2$ and $X_y^2$, respectively.
By induction hypothesis one can define $J_{X_x, \sigma 0}$ and $J_{X_y, \sigma 1}$, which are $L_x$-labelled and $L_y$-labelled, respectively.
Then $J_{X, \sigma} = \{\sigma\} \sqcup J_{X_x, \sigma 0} \sqcup J_{X_y, \sigma 1}$ where $\sigma$ is given label $\ell$.

Let $v: L \to \{1, \dots, 2n-1\}$ be the unique isomorphism between $(L, <_\Nb)$ and $(\{1, \dots, 2n-1\}, <_\Nb)$ seen as linear orders. Then renaming the labels of $J_{X, \epsilon}$ according to $v$, one obtains a Joyce tree $J_X$. 
\end{proof}

%\textit{Representing a finite Joyce order as a Joyce tree.}

\begin{definition}\index{Joyce!structure}\index{Joyce structure}\index{Joyce structure!DLO}\index{big Ramsey structure!DLO}
  The \emph{Joyce structure} of a Joyce order is a structure \mbox{$(X, <, \JRel)$} such that for all $x,y,z,t$, $\JRel(x,y,z,t)\iff\meetlevel{x}{y}<\meetlevel{z}{t}$. A \emph{DLO Joyce structure} is the Joyce structure of a dense linear Joyce order with no endpoints.
\end{definition}

\index{isomorphism!Joyce structure}\index{Joyce structure!isomorphism}By abuse of language, we may say that two Joyce orders are isomorphic whenever their corresponding Joyce structures are isomorphic. The construction of a Joyce tree from a Joyce order does not depend on the labels but on the  ordering of the labels.  Moreover two different orders of the labels yields two different Joyce trees. Hence the following lemma holds.

\begin{lemma}[$\RCA_0$]
Two finite Joyce structures are isomorphic if and only if they yield the same Joyce tree.
\end{lemma}

Just after the definition of a Joyce order, \Cref{def:jo}, we showed every Joyce tree yielded a Joyce order which in turn yields a Joyce structure. Hence the number of Joyce tree of size $n$, Joyce orders of size $n$, and Joyce structure of size $n$ are all the same. Street~\cite{streets} shows that this is the odd tangent number of $n$ (also see \cite[p.~147]{Todorcevic2010Ramsey}).

We now prove that every dense linear order with no endpoints can be enriched into a DLO Joyce structure. Actually, since these orders are computably categorical, that is, any two dense linear orders with no endpoints are isomorphic, and furthermore this isomorphism is computable in the orders, it suffices to prove the existence of a DLO Joyce order. For this, we need to consider the following ordering on $\cantor$:

\begin{definition}[the ordering $<_\Qb$ on $\cantor$]\label{def:leq-Q}\index{$<_\Qb$}\index{ordering!$<_\Qb$}
  Given two strings $\sigma,\tau\in\cantor$, define $\sigma<_\Qb\tau$ if and only if one of the following holds:
  \begin{enumerate}
  \item $\sigma\prec\tau$ and $\tau(|\sigma|)=1$;
  \item $\tau\prec\sigma$ and $\sigma(|\tau|)=0$;
  \item $\sigma$ and $\tau$ are incomparable and $\sigma<_{\mathrm{lex}}\tau$, where $<_{\mathrm{lex}}$ is the lexicographical order.
  \end{enumerate}
%  \[|\sigma|<|\tau|\land \tau(|\sigma)=1\quad\text{ or }\quad|\tau|<|\sigma|\land \sigma(|\tau)=0\]
\end{definition}

Intuitively, if $\sigma<_\Qb\tau$ then $\sigma$ lies to the left of $\tau$ if one draws the standard picture of the tree $\cantor$, growing upwards from the root. (See, e.g., Figure~\ref{fig:Devlin-to-ACA}.) 

\begin{theorem}[$\RCA_0$]\label{thm:dlo-joyce-order-exists}
There exists a DLO Joyce order.
\end{theorem}
\begin{proof}
Let $X$ be the rational language $(000 \cup 100)^{*}01$, that is, the set of strings $\sigma \in 2^{<\omega}$ of length $3n+2$ for some $n \in \omega$, such that $\sigma(3n) = 0$, $\sigma(3n+1) = 1$, and for every $j < n$, $\sigma(3j+1) = \sigma(3j+2) = 0$. For example, $10000010001 \in X$. In particular, $X$ is an infinite antichain with respect to the prefix order. Let $\ltlex$ be the lexicographic order restricted to $X$, that is, $\sigma \ltlex \tau$ if $\sigma(|\sigma \meet \tau|) <_\Nb \tau(|\sigma \meet \tau|)$. Then $(X, \ltlex)$ is a dense linear order with no endpoints. Indeed, letting $f$ be the natural one-to-one map from $X$ to $\cantor$, $f$ is an order isomorphism between $(X, \ltlex)$ and $(\cantor, <_\Qb)$ where $<_\Qb$ is order defined in \Cref{def:leq-Q}.
Last, fix an injective function $v: 2^{<\omega} \to \omega$ such that
for every $\sigma, \tau \in 2^{<\omega}$, if $|\sigma| < |\tau|$ then $v(\sigma) < v(\tau)$,
and for every $\sigma, \tau \in X$, define $\meetlevel{\sigma}{\tau} = v(\sigma \meet \tau)$.
Then $(X, \ltlex, \meetlevel{\cdot}{\cdot})$ is a dense linear Joyce order with no endpoints.

We prove that $(X, \ltlex, \meetlevel{\cdot}{\cdot})$ satisfies axioms \Jo{1}, \Jo{2} and \Jo{3}.
Let $x, y, z, t \in X$, not all equal, with $x \lelex y$ and $z \lelex t$.

Suppose $\meetlevel{x}{y} < \meetlevel{x}{z}$. By definition, $v(x \meet y) < v(x \meet z)$. By choice of the map $v$,  $|x \meet y| \leq |x \meet z|$, so $x \ltlex y$ iff $z \ltlex y$. This shows \Jo{1}.

Now suppose $\meetlevel{x}{y} < \meetlevel{x}{z}$. By definition, $v(x \meet y) < v(x \meet z)$. By choice of the map $v$,  $|x \meet y| \leq |x \meet z|$, so $x \meet y = z \meet y$, hence $v(x \meet y) = v(z \meet y)$. This shows \Jo{2}.

Finally, suppose $\meetlevel{x}{y} = \meetlevel{z}{t}$. By definition, $v(x \meet y) = v(z \meet t)$. By injectivity of the map $v$, $x \meet y = z \meet t$, so $x \meet y \prec x \meet z$ and $x \meet y \prec y \meet t$, hence $v(x \meet y) <_\Nb \min(v(x \meet z), v(y \meet t))$. This shows \Jo{3}.
%\peter{Someone needs to check the proofs of these rules are correct.}
\end{proof}

Depending on the choice of $v$ in the above construction, the DLO Joyce orders won't be isomorphic. Consider the $3$ leafs with the least labels.  The leaf $01$ always has the least label. The other $2$ leafs with minimal labels are always $x=00001$ and $y=10001$. Note that $x < y$. Now the structures yielded by $v_0$ and $v_1$ where $v_0(00001)<v_0(10001)$ (hence the label of $x$ is less than the label of $y$) and $v_1(00001)>v_1(10001)$ (hence the label of $x$ is greater than the label of $y$) are not isomorphic.  

%as the following formula is satisfied by the latter but not the former ``Let $x,y$ that lexicographically minimize $(\meetlevel xx, \meetlevel yy)$. Then $x<y$.''  The 

\begin{corollary}[$\RCA_0$]\label{cor:order-can-be-enriched}
Every dense linear order with no endpoints $(X, <)$ can be equipped with a function $\meetlevel{\cdot}{\cdot}: X^2\to \Nb$ to form a DLO Joyce order.
\end{corollary}
\begin{proof}
Let $(Y, <_Y, \meetlevel{\cdot}{\cdot}_Y)$ be the DLO Joyce order of \Cref{thm:dlo-joyce-order-exists}. By computable categoricity of the dense linear orders with no endpoints, there
exists an order isomorphism $f$ between $(X, <)$ and $(Y, <_Y)$.
Define $\meetlevel{\cdot}{\cdot}: X^2 \to \Nb$ by $\meetlevel{x}{y} = \meetlevel{f(x)}{f(y)}_Y$.
Then $(X, <, \meetlevel{\cdot}{\cdot})$ is a DLO Joyce order.
\end{proof}

%\textit{Representing a Joyce order as a set of strings.}
One can canonically represent any countable Joyce order as a set of strings which are pairwise incomparable under the prefix relation, equipped with the lexicographic order the natural $\meetlevel{\cdot}{\cdot}$ operation, that is, $\meetlevel{\sigma}{\tau} = |\sigma \meet \tau|$.

\begin{definition}\index{coded!Joyce order}\index{Joyce order!coded}
A \emph{coded Joyce order} is a Joyce order of the form $(X, \ltlex, |\cdot \meet \cdot|)$,
with $X \subseteq 2^{<\omega}$, where $|\sigma \meet \tau|$ is the length of the longest common prefix of $\sigma$ and $\tau$, such that for all $\sigma, \tau, \rho \in X$ with $|\rho| > |\sigma \meet \tau|$ and
$\sigma \meet \tau \not \preceq \rho$, then $\rho(|\sigma \meet \tau|) = 0$.

%\pelliot{isn't it more ``$\forall \sigma\in X$, $\forall \ell<|\sigma|$, $\exists \tau\in X$, $|\sigma\meet\tau|=\ell$ or $\sigma(\ell=0)$'', that is, one can only go right at ?}
\end{definition}

In particular, letting $\sigma = \tau$, if $|\rho| > |\sigma|$, then $\rho(|\sigma|) = 0$.
Since a coded Joyce order is fully specified by its set $X$, we shall simply refer to $X$ when talking about the coded Joyce order $(X, \ltlex, |\cdot \meet \cdot|)$. Note that any subset of a coded Joyce order is again a coded Joyce order, since the axioms are universal.

\begin{theorem}[$\RCA_0$]\label{thm:joyce-orders-can-be-coded}
Every countable Joyce order is isomorphic to a coded Joyce order.
\end{theorem}
\begin{proof}
Let $(X, <, \meetlevel{\cdot}{\cdot})$ be a countable Joyce order.
Let $L$ be the set of labels over $X^2$. For every $x \in X$, let $L_x$ be the set of labels $\ell \in L$ such that $\ell <_\Nb \meetlevel{x}{x}$ and such that there is some $y \in X$ such that $y < x$ and $\meetlevel{y}{x} = \ell$.
Let $\sigma_x \in 2^{<\omega}$ be the unique string of length $\meetlevel{x}{x}$, such that for every $j < \meetlevel{x}{x}$, $\sigma_x(j) = 1$ if and only if $j \in L_x$.
Let $Y = \{ \sigma_x: x \in X \}$.

\begin{claim}
$(Y, \ltlex, |\cdot \meet \cdot|)$ is isomorphic to $(X, <, \meetlevel{\cdot}{\cdot})$.
\end{claim}
\begin{proof}
We first prove that for all $x,y,z,t\in X$, $\meetlevel xy<\meetlevel zt\implies |\sigma_x\meet \sigma_y|<_\Nb|\sigma_z\meet \sigma_t|$. We actually prove the stronger fact that for every $x,y\in X$, $\meetlevel xy = |\sigma_x\meet\sigma_y|$. If $x=y$, it is clear as by construction, $\sigma_x$ is of length $\meetlevel xx$. If $x\neq y$, we first prove that $\meetlevel xy \leq |\sigma_x\meet\sigma_y|$: indeed, for all $\ell<\meetlevel xy$, by \Jo2 we have $\ell\in L_x$ if and only if $\ell\in L_y$. It remains to show $\meetlevel xy \geq |\sigma_x\meet\sigma_y|$: if $x<y$ we have that $\meetlevel xy\in L_y\setminus L_x$, and if $y<x$, $\meetlevel xy\in L_x\setminus L_y$. So in any case, $\sigma_x(\meetlevel xy)\neq\sigma_y(\meetlevel xy)$, so $|\sigma_x\meet\sigma_y|\leq\meetlevel xy$.

Let $x<y\in X$. Then, $\meetlevel xy\not\in L_x$, as if $z$ is such that $\meetlevel xz=\meetlevel xy$, then by \Jo3 $\meetlevel xy<_\Nb\meetlevel yz$ and by \Jo1 and the fact that $x<y$, we have $x<z$. So $\sigma_x(|x\meet y|)=\sigma_x(\meetlevel xy)=0$. However, $\meetlevel xy\in L_y$ as witnessed by $x$, so $\sigma_y(|x\meet y|)=\sigma_y(\meetlevel xy)=1$. Therefore, $\sigma_x\ltlex\sigma_y$.
\end{proof}

\begin{claim}
$(Y, \ltlex, |\cdot \meet \cdot|)$ is a coded Joyce order.
\end{claim}
\begin{proof}
By the previous claim, $(Y, \ltlex, |\cdot \meet \cdot|)$ is a Joyce order isomorphic to $(X, <, \meetlevel{\cdot}{\cdot})$.
Fix $\sigma_x, \sigma_y, \sigma_z \in Y$ with $|\sigma_z| > |\sigma_x \meet \sigma_y|$ and $\sigma_x \meet \sigma_y \not \preceq \sigma_z$. Assume $x \leq y$ without loss of generality.  Let $\ell = \meetlevel{x}{y} =  |\sigma_x \meet \sigma_y|$. Suppose for the contradiction that $\ell \in L_z$. Then there is some $u \in X$ with $u < z$ such that $\meetlevel{u}{z} = |\sigma_u \meet \sigma_z| = \ell = \meetlevel{x}{y} = |\sigma_x \meet \sigma_y|$. 
Since $u < z$, then $\sigma_u \ltlex \sigma_z$ and since $x \leq y$ and $u < z$, 
by \Jo{3}, $|\sigma_y \meet \sigma_z|  >_\Nb \ell$ and $|\sigma_u \meet \sigma_x| >_\Nb \ell$.
Let $\ell_0 = |\sigma_y \meet \sigma_z|$. In particular, $\sigma_x \meet \sigma_y = \sigma_x \meet \sigma_y \uh \ell_0 = \sigma_x \meet \sigma_z \uh \ell_0$, so $\sigma_x \meet \sigma_y \preceq \sigma_z$, contradiction. So $\ell \not\in L_z$, hence $\sigma_z(\ell) = 0$. 
%\peter{After L or PE removes the above 2 comments, someone else needs to check the proof.}\ludovic{Comments fixed.}\peter{OK but I did not read carefully}
\end{proof}
This completes the proof of \Cref{thm:joyce-orders-can-be-coded}.
\end{proof}

Note that the proof in \Cref{thm:joyce-orders-can-be-coded} yields a coded Joyce order whose set of lengths correspond exactly to the labels of the original Joyce order.

\begin{remark}\label{remark:labels-initial-segment}
 Since Joyce structures only consider the ordering between the labels and not their actual value, we can always pick a Joyce order isomorphic to the original one, whose set of labels is an initial segment of $\Nb$, and using \Cref{thm:joyce-orders-can-be-coded}, we can represent it as a coded Joyce order whose lengths coincide with the labels, and hence form an initial segment of $\Nb$.
\end{remark}

Todorcevic~\cite[Lemma 6.20]{Todorcevic2010Ramsey} made an explicit construction of a computable coded DLO Joyce order, under a different terminology.

\begin{corollary}\label{cor:true-dlo-joyce-order-exists}
There exists a computable coded DLO Joyce order.
\end{corollary}
\begin{proof}
Immediate by \Cref{thm:joyce-orders-can-be-coded} and \Cref{thm:dlo-joyce-order-exists}.
\end{proof}

\section{A proof of Devlin's theorem}

Note that the substructure of any Joyce structure is a Joyce structure, with the same witness function $|\cdot\meet\cdot|$. Not all DLO Joyce structures are isomorphic, as already remarked in the paragraph below \Cref{thm:dlo-joyce-order-exists}. 
%\peter{I am thinking that the above section should provide an easy proof of this fact. BTW, we should call them universal JS}\ludovic{No. Although every DLO Joyce structure is universal, there are some Joyce Structures which are universal but not DLO. }
However they all contain all Joyce structures. In particular, two DLO Joyce structures might not be isomorphic, but there exists an embedding from the first to the second, as well as from the second to the first.

\begin{theorem}[$\RCA_0$]\label{th:4.1LaflammeDevlin}
  Let $\Xb$ be a DLO Joyce structure, and $\Fb$ be a (finite or infinite) Joyce structure. Then, there exists an embedding from $\Fb$ to $\Xb$.
\end{theorem}
\begin{proof}
  Let $(X,\ltlex,|\cdot\meet\cdot|)$ be a computably coded DLO Joyce order with Joyce structure $\Xb$, which can always be found using \Cref{cor:true-dlo-joyce-order-exists}. Let $(F,\ltlex,|\cdot\meet\cdot|)$ be a Joyce order of $\Fb$. By \Cref{remark:labels-initial-segment}, we can suppose that the length of the elements of $\meetclosure{F}$ form an initial segment of $\Nb$. The \emph{cardinality} of $\meetclosure F$, $S\in\om\cup\{\om\}$, is that for all $s<S$, then there exists an unique element of $\meetclosure F$ of length $s$, $\sigma_s$.   We need to include the meets in our construction, so instead of building a map from $F$ to $X$, we build a map from $\meetclosure F$ to $X$. In the end, restricting the mapping to $F$ will yield the embedding.

  By induction on $s< S$, we build $x_s, a_s$ and $b_s$ such that:
  \begin{enumerate}
  \item\label{it:Dev-LaFlamme-1} $x_s,a_s,b_s\in X$;
  \item\label{it:Dev-LaFlamme-2} $a_s\ltlex x_s\ltlex b_s$; % and $|{a_s}\meet{d_s}|=|{b_s}\meet{c_s}|<|{x_s}|$;
  \item\label{it:Dev-LaFlamme-3} $|{a_{s}}\meet{b_{s}}|<_\Nb|{x_s}|<_\Nb|{a_{s+1}}\meet{b_{s+1}}|$;
  
%  \peter{Something is not correct with 4.  There could be two different t's with different values for $\sigma_s(t)$.  I believe theorem is correct but the proof needs work.  I will revisit this.}\pelliot{You are right, I added the condition: $\sigma_t\prec\sigma_s$. Is it clearer?}\peter{OK now}
%  \item if $y<x_s<z$ are elements of $\Xb_s$, then $\meetlevel y z_X$.
  \item\label{it:Dev-LaFlamme-4} If $t<s$ and $\sigma_t\prec\sigma_s$, then: If $\sigma_s(t)=0$, then $x_s,a_s,b_s$ are in the interval $(a_t,x_t)$. Similarly, if $\sigma_s(t)=1$, then $x_s,a_s,b_s$ are in the interval $(x_t,b_t)$.
  \end{enumerate}
  Suppose that $a_t,b_t$ and $x_{t}$ are defined for $t<s$. Let $t_0<s$ be biggest such that there exists $\tau\in F$ such that $|\sigma_s\meet\tau|=t_0$. Let $A$ be the interval $(a_{t_0},x_{t_0})$ if $\sigma_s(t_0)=0$, and $A$ be the interval $(x_{t_0},b_{t_0})$ if $\sigma_s(t_0)=1$. In either case $A$ is a dense linear Joyce order with no endpoints. Let $a_s,b_s\in A$ such that $|a_s\meet b_s|>_\Nb|{x_{s-1}}|$: They exists as $A$ is infinite and $|A|\geq n$ implies $|\{|\sigma\meet\tau|:\sigma\neq\tau\in A\}|\geq\log_2(n)$. Define $x_s$ to be any element of $(a_s,b_s)$, which has to verify $|a_s\meet b_s|<|x_s|$. Items 1, 2 and 3 are satisfied, as well as item 4 for $t_0$. By definition of $t_0$, if $\sigma_t\prec\sigma_s$, then either $t=t_0$ or $\sigma_t\prec\sigma_{t_0}$. In the former case, item 4 for $t$ is satisfied. In the latter case, as $a_{t_0}$, $\sigma_{t_0}$ and $b_{t_0}$ satisfy item 4, $x_{t_0}, a_{t_0}, b_{t_0}$ are in the interval specified by item 4. But then, so are $x_s, a_s, b_s$, so they satisfy item 4 for all $t$.

  We now define the embedding $\phi$: if $y\in F$, then $\phi(y)$ is defined to be $x_{|y\meet y|}$. It remains to show that $\phi$ is an embedding. An important fact is the following: If $x\ltlex y\in F$, and $t=|x\meet y|$, then $|a_t\meet b_t|\leq_\Nb|\phi(x)\meet\phi(y)|<_\Nb|x_t|$. Indeed, by \Cref{it:Dev-LaFlamme-4} $\phi(x)\in(a_t,x_t)$ and $\phi(y)\in(x_t, b_t)$. Combining the fact with \Cref{it:Dev-LaFlamme-3}, we get that $|x\meet y|<_\Nb|z\meet t|$ implies $|\phi(x)\meet\phi(y)|<_\Nb|\phi(z)\meet\phi(t)|$.

  Now suppose that for $x,y,z,t\in F$ with $x\ltlex y$ and $z\ltlex t$, $|x\meet y|=|z\meet t|=s$. This implies that $s_0=|x\meet z|>_\Nb s$ and $s_1=|y\meet t|>_\Nb s$. But then by \Cref{it:Dev-LaFlamme-3} and the fact of the previous paragraph, $|\phi(x)\meet\phi(z)|>_\Nb|a_{s_0}\meet b_{s_0}|>_\Nb|x_s|>_\Nb|\phi(x)\meet\phi(y)|$ and $|\phi(y)\meet\phi(t)|>_\Nb|a_{s_0}\meet b_{s_0}|>_\Nb|x_s|>_\Nb|\phi(z)\meet\phi(t)|$. Therefore, by \Jo2, $|\phi(x)\meet\phi(y)|=|\phi(z)\meet\phi(t)|$.

  Finally, suppose $x\ltlex y\in F$. Let $s=|x\meet y|$, by \Cref{it:Dev-LaFlamme-4} we have $\phi(x)\in (a_s,x_s)$ and $\phi(y)\in(x_s, b_s)$ so $\phi(x)\ltlex\phi(y)$.
\end{proof}

The \emph{age} \index{age of a structure}of a structure $\Mc$ is the set of all its finitely generated substructures.

\begin{corollary}
  The age of any DLO Joyce structure is the set of finite Joyce structures.
\end{corollary}

\begin{theorem}[$\RCA_0$]\label{thm:joyce-diagonalization}
There exists a DLO Joyce order \[(2^{<\omega}, <_T, \meetlevel{\cdot}{\cdot}_T)\]
such that for every coded Joyce order $X$,
the Joyce structures of \[(X,  <_T, \meetlevel{\cdot}{\cdot}_T)\] and \[(X, \ltlex, |\cdot \meet \cdot|)\] are isomorphic.
\end{theorem}
\begin{proof}
Let $(U, \ltlex, \meetlevel{\cdot}{\cdot}_U)$ be the DLO Joyce order defined in \Cref{thm:dlo-joyce-order-exists}, that is, $U = (000 \cup 100)^{*}01$ and $\meetlevel{\sigma}{\tau}_U = v(\sigma \meet \tau)$ for some injective function $v: 2^{<\omega} \to \omega$ such that
for every $\sigma, \tau \in 2^{<\omega}$, if $|\sigma| < |\tau|$ then $v(\sigma) < v(\tau)$.

Define the DLO Joyce order $(2^{<\omega}, <_T,  \meetlevel{\cdot}{\cdot}_T)$ as follows:
Given $\sigma \in 2^{<\omega}$, let $\hat{\sigma}$ be the binary string of length $3|\sigma|+2$
defined for every $j < |\sigma|$ by $\hat{\sigma}(3j) = \sigma(j)$, $\hat{\sigma}(3j+1) = \hat{\sigma}(3j+2) = 0$, and $\hat{\sigma}(3|\sigma|) = 0$ and $\hat{\sigma}(3|\sigma|+1) = 1$.
For instance, if $\sigma = 0110$ then $\hat{\sigma} = 00010010000001$.
Let $\sigma <_T \tau$ if and only if $\hat{\sigma} \ltlex \hat{\tau}$
and $\meetlevel{\sigma}{\tau}_T = \meetlevel{\hat{\sigma}}{\hat{\tau}}_U$.

Let $X$ be a coded Joyce order. We shall now show that $(X,  <_T, \meetlevel{\cdot}{\cdot}_T)$ and \mbox{$(X, \ltlex, |\cdot \meet \cdot|)$} are isomorphic via the identify function.  

Fix $\sigma, \tau \in X$. If $\sigma \ltlex \tau$, then $\hat{\sigma} \ltlex \hat{\tau}$, hence $\sigma <_T \tau$. Conversely, if $\sigma <_T \tau$, then $\hat{\sigma} \ltlex \hat{\tau}$, but since $\sigma$ and $\tau$ are incomparable with respect to the prefix relation, this implies that $\sigma \ltlex \tau$. Thus $\sigma <_T \tau$ if and only if $\sigma \ltlex \tau$.

Fix $\sigma, \tau, \rho, \mu \in X$. If $|\sigma \meet \tau| <_\Nb |\rho \meet \mu|$, then
$|\hat{\sigma} \meet \hat{\tau}| <_\Nb |\hat{\rho} \meet \hat{\mu}|$, then $v(\hat{\sigma} \meet \hat{\tau}) <_\Nb v(\hat{\rho} \meet \hat{\mu})$, hence $\meetlevel{\sigma}{\tau}_T <_\Nb \meetlevel{\rho}{\mu}_T$. Conversely, assume $\meetlevel{\sigma}{\tau}_T <_\Nb \meetlevel{\rho}{\mu}_T$. Unfolding the definition, $v(\hat{\sigma} \meet \hat{\tau}) <_\Nb v(\hat{\rho} \meet \hat{\mu})$. If $|\hat{\sigma} \meet \hat{\tau}| \neq |\hat{\rho} \meet \hat{\mu}|$, then by definition of $v$, $|\hat{\sigma} \meet \hat{\tau}| <_\Nb |\hat{\rho} \meet \hat{\mu}|$, hence $|\sigma \meet \tau| <_\Nb |\rho \meet \mu|$. If $|\hat{\sigma} \meet \hat{\tau}| = |\hat{\rho} \meet \hat{\mu}|$, then, since $X$ is a coded Joyce order, $\sigma \meet \tau = \rho \meet \mu$, so $\hat{\sigma} \meet \hat{\tau} = \hat{\rho} \meet \hat{\mu}$ and $v(\hat{\sigma} \meet \hat{\tau}) = v(\hat{\rho} \meet \hat{\mu})$, contradiction.
\end{proof}

\begin{definition}\index{Joyce order!diagonalization}\index{diagonalization!Joyce order}
A \emph{Joyce order diagonalization} for some Joyce order
\[\mbox{$(U,  <_U, \meetlevel{\cdot}{\cdot}_U)$}
\]
is a function $h: 2^{<\omega} \to U$, such that for every coded Joyce order $X$,
\[\mbox{$(h[X], <_U, \meetlevel{\cdot}{\cdot}_U)$}\]
is isomorphic to $(X, \ltlex, |\cdot \meet \cdot|)$.
\end{definition}

\begin{corollary}[$\RCA_0$]\label{cor:joyce-diagonalization-exists}
Every DLO Joyce order $(U, <_U, \meetlevel{\cdot}{\cdot}_U)$ has a Joyce order diagonalization.
\end{corollary}
\begin{proof}
Let $(2^{<\omega}, <_T, \meetlevel{\cdot}{\cdot}_T)$ be the Joyce order of \Cref{thm:joyce-diagonalization}. By \Cref{th:4.1LaflammeDevlin}, there is an embedding $h: 2^{<\omega} \to U$. By definition of an embedding, for every coded Joyce order $X \subseteq 2^{<\omega}$, $(h[X], <_U, \meetlevel{\cdot}{\cdot}_U)$ is isomorphic to $(X, <_T, \meetlevel{\cdot}{\cdot}_T)$. By \Cref{thm:joyce-diagonalization}, $(X, <_T, \meetlevel{\cdot}{\cdot}_T)$ is isomorphic to $(X, \ltlex, |\cdot \meet \cdot|)$. Thus $h$ is a Joyce order diagonalization.
\end{proof}

The following lemma bridges finite coded Joyce orders of size $n$
and strong subtrees of $2^{<\omega}$ of height $2n-1$, by showing that any coded Joyce order of size $n$ is a subset of a strong subtree of size $2n-1$, and, conversely, any strong subtree of height $2n-1$ is a superset of at most one coded Joyce order of size $n$.

\begin{lemma}\label{lem:coded-joyce-order-to-strong-subtree}
Let $F$ be a finite coded Joyce order of size $n$ and $T \in \Subtree{\omega}{2^{<\omega}}$.
Then every $E \in \Subtree{2n-1}{T}$ contains at most one coded Joyce order isomorphic to $F$. Moreover, every coded Joyce order $H \subseteq T$ isomorphic to $F$
	is included in some $E \in \Subtree{2n-1}{T}$.
\end{lemma}
\begin{proof}
  Let $E\in\Subtree{2n-1}T$, $h$ its level function and $F_0, F_1\subseteq E$ be two coded Joyce orders isomorphic to $F$. By \Cref{it:JO-are-trees-4} of \Cref{lem:JO-are-trees}, the set $\{|\sigma\meet\tau|:\sigma,\tau\in F_0\}$ has cardinality $2n-1$, %moreover it is included in $h[2n-1]$, therefore it is equal to $h[2n-1]$; 
  and similarly for $F_1$.  %\peter{What is $h[2n-1]$?}

  Remark that $F_0$ and $F_1$ are uniquely identified as the set of leaves of respectively $\meetclosure F_0$ and $\meetclosure F_1$, and that the isomorphism between $F_i$ and $F$ can be extended to an isomorphism between $\meetclosure F_i$ and $\meetclosure F$, where the element of length $h(j)$ (or equivalently of level $j$) of $\meetclosure F_i$ is mapped to the element of level $j$ of $F$. Let $\ell$ be the first level, if it exists, where $\meetclosure F_0\uh h(\ell)\neq\meetclosure F_1\uh h(\ell)$. Let $\sigma_i\in \meetclosure F_i$ be the unique element of $\meetclosure F_i(\ell)$, and $\sigma$ the unique element of $F(\ell)$. For every $j<\ell$, the values of $\sigma_0(h(j))=\sigma_1(h(j))$ are determined: 0 iff there is a $\tau$ where $\sigma\ltlex\tau\in \meetclosure F(j)$ and 1 iff there is a $\tau$ where $\sigma\gtlex\tau\in \meetclosure F(j)$. %; and 0 otherwise.
   %Indeed, as $j\in\{|\sigma\meet\tau|:\sigma,\tau\in F_i\}$, the order $\ltlex$ being preserved ensures the two first cases, and the definition a coded Joyce order ensures that $\sigma(j)=0$ at levels of meets that are not predecessor of $\sigma$, that is, the last case.
    As $E$ is a strong subtree, determining the values of $\sigma$ at levels $h(j)$ for $h(j)<|\sigma_i|$ entirely defines $\sigma_i$.
%    \peter{There are issues with the levels here.  You cannot assume the h(j)=j.  I started to correct this but will await.} \pelliot{I do not assume $h(j)=j$, however $\sigma\uh h(j)$ is at the $j$-th level of the subtree.} \peter{OK}
% such that  Moreover, any embedding from $F_i$ to $E$ must map the $j$-th level of $F_i$ to the $j$-th level of $E$. We show by induction that $\meetclosure{F_0}=\meetclosure{F_1}$. Toward a contradiction, let $\ell$ be the first level where they differ. Let $a_i$ be the unique element of $F_i$ at level $\ell$, for $i<2$. Let $b_i$ be the unique element at level 
%  There are $\frac{n(n-1)}2+n = \frac{n(n+1)}2$ possible pair of distinct elements for a meet, and at most $2n-1-n-1=n-2$ levels for the meet. By the pigeon-hole principle, there must exists $\frac n2 +1$ distinct pairs with the same level of meet.
  Therefore, $\meetclosure F_0=\meetclosure F_1$ and $F_0=F_1$.
%  \ludovic{TODO, prove the first part of the lemma.}

For the second part, let $H \subseteq T$ be a coded Joyce order isomorphic to $F$. We claim that $H$ is included in some $E \in \Subtree{2n-1}{T}$.
Let $H^{\meet} = \{\sigma \meet \tau: \sigma, \tau \in H\}$ be the $\meet$-closure of $H$.
In particular, $H^{\meet}$ is a finite tree of height $2n-1$ with exactly one string at each level. Since $T \in \Subtree{\omega}{2^{<\omega}}$ and $H \subseteq T$ then $H^{\meet} \subseteq T$. Let $L = \{\ell_0 < \dots < \ell_{2n-2}\}$ be the set of levels of the nodes of $H^{\meet}$ in $T$. Let $E$ be the largest (in the sense of inclusion) subtree of $T$ of height $2n-1$ containing $H^{\meet}$ such that for every $i < 2n-1$, $E(i) \subseteq T(\ell_i)$. We claim that for every $i < 2n-2$, every node $\sigma \in E(i)$ is 2-branching in $E$. Since $E \subseteq 2^{<\omega}$ is a tree, it is $\meet$-closed, $\sigma$ is at most 2-branching. Since $T \in \Subtree{\omega}{2^{<\omega}}$, every node in $T$ is 2-branching. Then $\sigma$ has two extensions $\tau_0, \tau_1 \in T(\ell_{i+1})$ such that $\tau_0 \meet \tau_1 = \sigma$. By maximality of $E$, $\sigma$ is 2-branching in $E$.
 	Thus $E \in \Subtree{2n-1}{T}$.
\end{proof}

\begin{theorem}[$\ACA_0$]\label{thm:strong-devlin-one-type}
  Let $\Xb$ be a countable DLO Joyce structure, and $\mathbb{F}$ be a finite Joyce structure. Then, the big Ramsey number of $\mathbb{F}$ in $\Xb$ is 1.
\end{theorem}
\begin{proof}
Let $X$ be a countable coded DLO Joyce order and $F$ be a finite coded Joyce order of size $n$. Fix a coloring $f: {X \choose F} \to k$. Here, ${X \choose F}$ denotes all the subcopies of $F$ in $X$.

Let $h: 2^{<\omega} \to X$ be a Joyce order diagonalization, which exists by \Cref{cor:joyce-diagonalization-exists}. Let $g: \Subtree{2n-1}{2^{<\omega}} \to k$ be defined for every $E \in \Subtree{2n-1}{2^{<\omega}}$ by $g(E) = f(h(H))$ where $H \subseteq E$ is the unique element coded Joyce order isomorphic to $F$, if it exists. Otherwise let $g(E) = 0$. This coloring is well-defined by \Cref{lem:coded-joyce-order-to-strong-subtree}.

By Milliken's tree theorem for height $2n-1$, there is a strong subtree $S \in \Subtree{\omega}{2^{<\omega}}$ such that $g$ restricted to $\Subtree{2n-1}{S}$ is monochromatic for some color $i < k$. In particular, by \Cref{lem:coded-joyce-order-to-strong-subtree}, for every coded Joyce order $H \subseteq S$ isomorphic to $F$, there is some $E \in \Subtree{2n-1}{S}$ containing $H$, and $g(E) = f(h(H)) = i$. 

Since $S \in \Subtree{\omega}{2^{<\omega}}$, there is an injective function $\phi: 2^{<\omega} \to S$ such that $\phi[X]$ is a coded Joyce order isomorphic to $X$. 
%\peter{I think we should be able to cite one of the above lemmas to get phi. } 
In particular, since $h$ is a Joyce diagonalization, $Y = h[\phi[X]]$ is a DLO coded Joyce order isomorphic to $X$, hence a subcopy of~$X$. Note that $Y$ is a coded Joyce order since it is a subset of $X$ which is a coded Joyce order.  
%\peter{I do not think diagonalization preserved coded.  Maybe into coded U they do.  Needs a lemma or comment}\ludovic{Yes they don't and Yes they do.}

We claim that $f$ restricted to ${Y \choose F}$ is monochromatic for color~$i$.
Let $\hat{H}$ be a copy of $F$ in $Y = h[\phi[X]]$. Let $H \subseteq \phi[X]$ be such that $h[H] = \hat{H}$. In particular since $\phi[X]$ is a coded Joyce order, so is $H$, so since $h$ is a Joyce order diagonalization, $\hat{H} = h[H]$ is a coded Joyce order isomorphic to $H$. In other words, $H$ is a copy of $F$ in $\phi[X] \subseteq S$, so $H$ is a copy of $F$ in $S$. 
	By \Cref{lem:coded-joyce-order-to-strong-subtree}, there is some $E \in \Subtree{2n-1}{S}$ containing $H$, and by definition of $g$, $g(E) = f(h(H))$.  By choice of $S$, $g$ restricted to $\Subtree{2n-1}{S}$ is homogeneous for color~$i$, so $g(E) = f(h[H]) = i$, so $f(\hat{H}) = i$. 
%	\peter{This is not exactly how you defined your coloring g.  There is an Ehat where Hhat lives.} \ludovic{Right. I fixed it.}
\end{proof}

\begin{statement}[Joyce Devlin's theorem for $n$-tuples and $\ell$ colors]\index{statement!$\JDT{n}{k, \ell}$}
  $\JDT{n}{k, \ell}$ is the statement: ``For any Joyce structure $\Xb$ and coloring $f: [\Xb]^n \to k$, there exists a strong subcopy of $\Xb$ such that $f$ uses at most $\ell$ colors''.
\end{statement}

\begin{corollary}[Tight bounds on Joyce Devlin's theorem]\label{cor:strong-devlin-tight-joyce-orders}
  For any $n$, $(\forall k)\JDT{n}{k, \ell}$ holds, $\ell$ being the number of Joyce orders with $n$ elements, and this bound is tight.
\end{corollary}
\begin{proof}
Let $\ell$ be the number of Joyce order structures with $n$ elements. Let $F_0, \dots, F_{\ell-1}$ be a finite enumeration of all the finite coded Joyce orders of size $n$.

We first prove that $(\forall k)\JDT{n}{k, \ell}$ holds.
Fix a coloring $f: [X]^n \to k$ for some countable DLO Joyce structure $(X, <, \JRel)$. By \Cref{thm:strong-devlin-one-type}, build a finite decreasing sequence of subsets $X = X_0 \supseteq X_1 \supseteq \dots \supseteq X_\ell$ of $X$ such that for every $s < \ell$:
\begin{enumerate}
	\item $(X_{s+1},  <, \JRel)$ is a subcopy of $(X_s, <, \JRel)$;
	\item every copy of $F_s$ in $(X_{s+1},  <, \JRel)$ is monochromatic for $f$ for some color $i_s < k$.
\end{enumerate}
The Joyce structure $(X_\ell, <, \JRel)$ is a subcopy of $(X, <, \JRel)$.
Moreover, for every $E \in [X_\ell]^n$, $(E, <, \JRel)$ is isomorphic to $F_s$ for some $s < k$,
so $f(E) = i_s$. It follows that $f[X_\ell]^n \subseteq \{ i_s: s < \ell \}$, hence $|f[X_\ell]^n| \leq \ell$.

We now show that the bound is tight. Let $f: [X]^n \to k$ be defined by $f(E) = s$ for the unique $s < \ell$ such that $(E, <, \JRel)$ is isomorphic to $F_s$.
Let $(Y, <, \JRel)$ be a subcopy of $(X, <, \JRel)$. In particular, $(Y, <, \JRel)$ is a DLO Joyce structure, so by \Cref{th:4.1LaflammeDevlin}, for every $s < \ell$, there is an embedding of $F_s$ into $(Y, <, \JRel)$. Therefore, $|f[Y]^n| \geq \ell$.
\end{proof}

It is clear that Joyce Devlin's theorem for $n$-tuples and $\ell$ colors implies Devlin's theorem for $n$-tuples and $\ell$ colors: indeed, by the existence of a DLO Joyce structure and computable categoricity of dense linear orders without endpoints, any such order can be turned into a DLO Joyce structure (see \Cref{cor:order-can-be-enriched}). The following theorem shows the converse:

\begin{theorem}[$\RCA_0$]\label{thm:suborder-to-subjoyce}
  Let $\Xb=(X,<,\JRel)$ be a DLO Joyce structure. Let $\Xb'=(X',<)$ be an isomorphic subcopy of $(X,<)$, that is, a dense linear order with no endpoints. Then, there exists a subcopy $(X'',<)$ of $(X',<)$ such that $(X'',<,\JRel)$ is a subcopy of $\Xb$.
\end{theorem}
\begin{proof}
  The structure $\hat\Xb'=(X',<,\JRel)$ is a DLO Joyce structure, even if it might not be isomorphic to $\Xb$. By \Cref{th:4.1LaflammeDevlin}, there exists an embedding of $\Xb$ into $\hat\Xb'$. The image of the embedding is $\Xb''$.
\end{proof}

\begin{corollary}[$\RCA_0$]
  Devlin's theorem for $n$-tuples and $\ell$ colors implies Joyce Devlin's theorem for $n$-tuples and $\ell$ colors.
\end{corollary}

\begin{corollary}[$\RCA_0$]
  The tight bound for Devlin's theorem and Joyce Devlin's theorem  for $n$ elements are the same, that is, the number of Joyce structures with $n$ elements, or the number of Joyce trees with $n$ leaves, or the odd tangent number of $n$.
\end{corollary}
\begin{proof}
Let $b_0$ and $b_1$ be the tight bound for Devlin's theorem and Joyce Devlin's theorem for $n$ elements, respectively.

We first claim that $b_0 \leq b_1$.
Let $(X, <)$ be a dense linear order with no endpoints. By \Cref{cor:order-can-be-enriched}, one can enrich this order with a relation $\JRel$ so that $(X, <, \JRel)$ is a DLO Joyce structure. Let $f: [X]^n \to k$ be a coloring. By choice of $b_1$, there is a Joyce subcopy $(Y, <, \JRel)$ of $(X, <, \JRel)$ such that $|f[Y]^n| \leq b_1$. In particular, $(Y, <)$ is a subcopy of $(X, <)$ so $b_0 \leq b_1$.

We then claim that $b_1 \leq b_0$.
Let $(X, <, \JRel)$ be a DLO Joyce structure.  Let $f: [X]^n \to k$ be a coloring. By choice of $b_0$, there is a subcopy $(Y, <)$ of $(X, <)$ such that $|f[Y]^n| \leq b_0$.
By \Cref{thm:suborder-to-subjoyce}, there is a subcopy $(Z, <)$ of $(Y, <)$ such that $(Z, <, \JRel)$ is a Joyce subcopy of $(X, <, \JRel)$. In particular, $|f[Z]^n| \leq b_0$. Thus $b_1 \leq b_0$.

It follows that $b_0 = b_1$. Moreover, by \Cref{cor:strong-devlin-tight-joyce-orders}, this tight bound is the number of Joyce structures with $n$ elements, that is, the odd tangent number of $n$ (see \cite[p.~147]{Todorcevic2010Ramsey}).
\end{proof}

\section{Lower bounds on Devlin's theorem}\label{sect:lower-bound-devlin}

A coloring that witness the need for 2 colors for Devlin's theorem for pairs is the coloring $f_0$ defined as follows. Let $(q_n)_{n\in\Nb}$ be an enumeration of the rationals, and define $f_0: [\mathbb Q]^2\to 2$ by letting $f_0(q_n, q_m) = 0$ if $q_n<q_m\iff n<m$, and $f_0(q_n, q_m) = 1$ otherwise. Now every subset $S\subseteq\mathbb Q$ of order-type $\mathbb Q$ (or even $\mathbb Z$) must contain pairs of both colors under $f_0$, as every element of a dense linear order has infinitely many element both below it and above it.

Recall the ordering $<_\Qb$ on $\cantor$ from \Cref{def:leq-Q}. An explicit embedding of $<_\Qb$ into $\Qb$ is given by the following function: $\sigma\mapsto\sum_{i<|\sigma|}(\sigma(i)-\frac12)2^{-i}$. Thus, the $i$th bit of $\sigma$ contributes to the sum either $-2^{-i-1}$ or $2^{-i-1}$, depending as it is $0$ or $1$.

\begin{theorem}\label{th:lower-bound-devlin}
  There is a computable instance of $\DT{2}{4,3}$ all of whose solutions compute the halting set.
\end{theorem}
\begin{proof}
  Recall the order $<_\Qb$ from Definition~\ref{def:leq-Q}, and that $(\Qb, <)\cong(\cantor, <_\Qb)$ via a computable bijection. Therefore, the rationals will now be considered as finite strings. %Fix an enumeration $(\sigma_n)_{n\in\Nb}$ of the elements of $\cantor$.

  Let $f_{<_\Qb}:[\cantor]^2\to2$ be the function such that $f_{<_\Qb}(\sigma, \tau)=1$ if and only if $|\sigma|<|\tau|\iff \sigma<_{\Qb}\tau$. Any dense (in the sense of $<_\Qb$) subset of $\cantor$ must contain a pair with both 0, and a pair with color 1. Let also $f_J:[\Nb]^3\to2$ be such that for any $x<y<z$, $f_J(x,y,z)=1$ if and only if $K_y\upharpoonright x = K_z\upharpoonright x$, where $K$ is a complete $\Si01$ set with fixed computable enumeration $(K_s)_{s \in \omega}$. (The function $f_J$ was devised by Jockusch~\cite[Theorem 5.7]{Jockusch1972Ramseys} to show the analogue of the present theorem for Ramsey's theorem for triples.)

  The function of interest for us is the product function $f=f_{<_\Qb}\times f_J:(\sigma,\tau)\mapsto \langle f_{<_\Qb}(\sigma, \tau),f_J(|\sigma\meet\tau|, |\sigma|, |\tau|)\rangle$. This is a $4$-coloring, so by $\DT{2}{4,3}$ let $S\subseteq\cantor$ be a dense linear ordering for $<_\Qb$ such that $f$ uses at most three colors on $[S]^2$. Suppose for instance that for some $c \in 2$ and for every $\sigma,\tau\in S$,  we have $f(\sigma,\tau)\neq(1, 1-c)$. (The case where the color $(0, 1-c)$ is avoided is symmetric). This means that any $\sigma_0<_\Qb\sigma_1$ in $S$ with $|\sigma_0|<|\sigma_1|$ must have color $c$ under $f_J$.

  The remainder of the proof consists of two parts. The first is the proof that $c$ must be 1, and the second is an argument to show how to compute $K$ from $S$. The main ingredient will be the fact that for every $n\in\Nb$, we can find arbitrarily long strings $\sigma$ and $\tau$ in $S$ with $|\sigma \meet \tau| > n$. This is depicted in Figure~\ref{fig:Devlin-to-ACA}.

  Given two strings $\sigma<_\Qb\tau$, define $\interval[open]\sigma\tau = \{\rho \in \cantor:\sigma<_\Qb\rho<_\Qb\tau\}$. Note that if $I\subseteq\cantor$ is a dense linear ordering without endpoints under $<_\Qb$, then so is $I\cap\interval[open]\sigma\tau$. Also, note that if $\xi,\rho \in \interval[open]\sigma\tau$ then $\xi \meet \rho = \sigma \meet \tau$.

  \begin{fact}\label{fact:dlo-blossom}
    If $I\subseteq\cantor$ is a dense linear ordering without endpoints under $<_\Qb$, then $I$ contains a pair of incompatible strings, $\sigma$ and $\tau$. Moreover, for every $n\in\Nb$, we can find such $\sigma$ and $\tau$ so that $|\sigma\meet\tau|>n$.
  \end{fact}
  \begin{proof}
  	Fix $n \in \Nb$. As $I$ is infinite but $2^n$ is finite, there exist $\rho_0,\rho_1\in I$ such that $\rho_0\upharpoonright n=\rho_1\upharpoonright n$. If $\rho_0$ and $\rho_1$ are incompatible, then these can serve as $\sigma$ and $\tau$. So suppose otherwise, say $\rho_0 <_\Qb \rho_1$. Fix any $\xi \in I\cap\interval[open]{\rho_0}{\rho_1}$. Since $I\cap\interval[open]{\rho_0}{\rho_1}$ is a dense linear order without endpoints, there are infinitely many $\sigma, \tau\in I\cap\interval[open]{\rho_0}{\rho_1}$ with $\sigma<_\Qb\xi<_\Qb\tau$, so these can be chosen so that $|\sigma|>|\xi|$ and $|\tau|>|\xi|$. But now, if $\sigma$ and $\tau$ were compatible, then by definition of $<_\Qb$ they would both be above or both below $\xi$, a contradiction. Thus, $\sigma$ and $\tau$ are incomparable elements of $I$. Furthermore, since $\rho_0<_\Qb\sigma <_\Qb \tau<_\Qb\rho_1$, we have $\sigma\upharpoonright n = \tau\upharpoonright n$, so $|\sigma\meet\tau|>n$.
  \end{proof}
  \begin{fact}\label{fact:dlo-blossom2}
    If $I\subseteq\cantor$ is a dense linear ordering without endpoints under $<_\Qb$, then for every $n\in\Nb$ there exists four pairwise incompatible strings $\alpha^i_j\in I$ for $i,j\in 2$ such that $\alpha^0_0<_\Qb\alpha^0_1<_\Qb\alpha^1_0<_\Qb\alpha^1_1$, the strings $\alpha^0_0\meet\alpha^0_1$ and $\alpha^1_0\meet\alpha^1_1$ are incompatible, and $|\alpha^0_{0}\meet\alpha^0_{1}\meet\alpha^1_{0}\meet\alpha^1_{1}|>n$. (See Figure~\ref{fig:Devlin-to-ACA}.)
  \end{fact}
  \begin{proof}
  		Fix $n \in \Nb$. First, suppose that whenever $\rho_0,\rho_1 \in I$ satisfy $|\rho_0\meet\rho_1|>n$ then they are incompatible. Since $I$ is infinite and $2^n$ is finite, we can then pick $\alpha^0_0 <_\Qb \alpha^0_1 <_\Qb \alpha^1_0 <_\Qb \alpha^1_1$ in $I$ with $|\alpha^0_{0}\meet\alpha^0_{1}\meet\alpha^1_{0}\meet\alpha^1_{1}|>n$. Then by assumption, all the $\alpha^i_j$ must be pairwise incompatible, as must $\alpha^0_0 \meet \alpha^0_1$ and $\alpha^1_0 \meet \alpha^1_1$.

  		So suppose otherwise, and fix $\rho_0 <_\Qb \rho_1$ with $|\rho_0\meet\rho_1|>n$. Fix $\gamma\in S\cap\interval[open]{\rho_0}{\rho_1}$ (represented in grey in Figure~\ref{fig:Devlin-to-ACA}). As $S\cap\interval[open]{\rho_0}\gamma$ and $S\cap\interval[open]\gamma{\rho_1}$ are two dense linear orderings without endpoints, we can apply the preceding fact to find incompatible $\alpha^0_{0},\alpha^0_{1}\in S\cap\interval[open]{\rho_0}\gamma$ and incompatible $\alpha^1_{0},\alpha^1_{1}\in S\cap\interval[open]\gamma{\rho_1}$ with $|\alpha^0_0\meet\alpha^0_1|>|\gamma|$ and $|\alpha^1_0\meet\alpha^1_1| > |\gamma|$.

  		Since $|\alpha^i_0\meet\alpha^i_1| > |\gamma|$ for each $i \in 2$, we have  $|\alpha^i_j| > |\gamma|$ for all $i,j \in 2$. Hence, $\alpha^0_j$ and $\alpha^1_j$ are incompatible for each $j \in 2$, being on opposite sides of $\gamma$ under~$<_\Qb$.

  		Since $\alpha^0_0,\alpha^0_1 <_\Qb \gamma <_\Qb \alpha^1_0,\alpha^1_1$, we have necessarily $\alpha^0_0\meet\alpha^0_1 \leq_\Qb \gamma \leq_\Qb \alpha^1_0,\alpha^1_1$, but since $|\alpha^i_0\meet\alpha^i_1| > |\gamma|$ for each $i \in 2$ these inequalities must be strict. It follows that $\alpha^0_0\meet\alpha^0_1$ and $\alpha^1_0\meet\alpha^1_1$ are incompatible, as desired.

  		Finally, as $\rho_0 <_\Qb \alpha^i_j <_\Qb \rho_1$ for all $i,j \in 2$, we have $\rho_0 \leq_\Qb \alpha^0_{0}\meet\alpha^0_{1}\meet\alpha^1_{0}\meet\alpha^1_{1} \leq_\Qb \rho_1$, meaning that $\alpha^0_{0}\meet\alpha^0_{1}\meet\alpha^1_{0}\meet\alpha^1_{1} = \rho_0 \meet \rho_1$ and hence $|\alpha^0_{0}\meet\alpha^0_{1}\meet\alpha^1_{0}\meet\alpha^1_{1}|>n$.
  \end{proof}

  We now use Fact \ref{fact:dlo-blossom2} to prove that the color 1 for $f_J$ cannot be avoided. Fix any $n$, and find $\alpha^i_j \in S$ for $i,j\in 2$ as in Fact \ref{fact:dlo-blossom2}. Fix $l$ such that $K_l\upharpoonright N= K\upharpoonright N$, where $N=|\alpha^0_0\meet\alpha^0_1\meet\alpha^1_0\meet\alpha^1_1|$. Pick $\sigma_n\in S\cap\interval[open]{\alpha^0_{0}}{\alpha^0_{1}}$ and $\sigma_m\in S\cap\interval[open]{\alpha^1_{0}}{\alpha^1_{1}}$ such that $n<m$ and $|\sigma_n|,|\sigma_m|>l$.
  %Then $f_{<_\Qb}(\sigma_n, \sigma_m)=1$.
  Now, as $\sigma_n\meet\sigma_m=\alpha^0_0\meet\alpha^0_1\meet\alpha^1_0\meet\alpha^1_1$, we have $f_J(|\sigma_n\meet\sigma_m|, |\sigma_n|,|\sigma_m|) = 1$ as the approximation for $K\upharpoonright N$ does not change after stage $l$. Therefore, the product coloring $f$ assigns $(\sigma_n,\sigma_m)$ the color $(1,1)$. In particular, $c = 1$, as desired.

  It remains to show that $K$ is $S$-computable. Given $n$, we uniformly compute $K \upharpoonright n$ from $S$. First, search for four strings $(\alpha^i_j)_{i,j\in2}$ in $S$ satisfying Fact~\ref{fact:dlo-blossom2}, which will be found as they exist. Then, output $K_{|\alpha^0_{0}|}\upharpoonright n$. Indeed, if it were the case that $K_{|\alpha^0_{0}|}\upharpoonright n\neq K\upharpoonright n$, then it would also be true that $K_{|\alpha^0_{0}|}\upharpoonright n\neq K_l \upharpoonright n$ for all sufficiently large $l$. But then we would have $f_J(\alpha^0_{0}, \sigma)= 0$ for any $\sigma\in S\cap\interval[open]{\alpha^1_{0}}{\alpha^1_{1}}$ with $|\sigma|$ sufficiently big, contradicting that fact that $c = 1$.
\end{proof}

\begin{corollary}\label{cor:dt2to3-implies-aca}
	Over $\RCA_0$, $\DT{2}{4,3}$ implies $\ACA$.
\end{corollary}

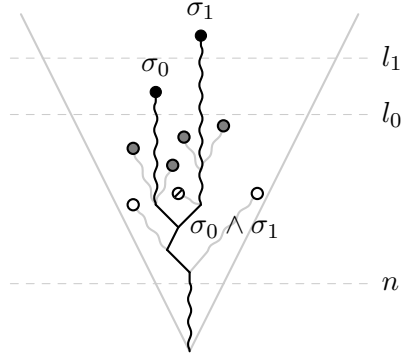
\begin{figure}[h!]
\begin{center}
	\begin{tikzpicture}[scale=1.5]
		\tikzset{
			empty node/.style={circle,inner sep=0,outer sep=0,fill=none},
			solid node/.style={circle,draw,inner sep=1.5,fill=black},
			hollow node/.style={circle,draw,inner sep=1.5,fill=white,thick},
			gray node/.style={circle,draw={rgb:black,1;white,4},inner sep=1,fill={rgb:black,1;white,4},thick},
			split node/.style={forbidden sign,draw,inner sep=1.5,fill=white,thick},
		}
		\tikzset{snake it/.style={decorate, decoration=snake, line cap=round}}
		\tikzset{gray line/.style={line cap=round,thick,color={rgb:black,1;white,4}}}
		\tikzset{gray thin line/.style={line cap=round,color={rgb:black,1;white,4}}}
		\tikzset{thick line/.style={line cap=round,rounded corners=0.1mm,thick}}
		\tikzset{thin line/.style={line cap=round,rounded corners=0.1mm}}
		\node (a)[empty node] at (0,0) {};
		\node (a0)[empty node] at (-1.5,3) {};
		\node (a1)[empty node] at (1.5,3) {};
		\node (b0)[empty node] at (-1.6,0.6) {};
		\node (b1)[empty node] at (1.6,0.6) {};
		\node (c0)[empty node] at (-1.6,2.1) {};
		\node (c1)[empty node] at (1.6,2.1) {};
		\node (d0)[empty node] at (-1.6,2.6) {};
		\node (d1)[empty node] at (1.6,2.6) {};
		\node (1)[empty node] at (0,0.7){};
		\node (2)[empty node] at (-0.2,0.9){};
		\node (3)[empty node,label=right:{$\sigma_0 \meet \sigma_1$}] at (-0.1,1.1){};
		\node (4)[empty node] at (-0.3,1.3){};
		\node (5)[empty node] at (0.1,1.3){};
		\node (6)[solid node,label=above:{$\sigma_0$}] at (-0.3,2.3){};
		\node (7)[solid node,label=above:{$\sigma_1$}] at (0.1,2.8){};
		\node (8)[hollow node] at (-0.5,1.3){};
		\node (9)[hollow node] at (0.6,1.4){};
		\node (10)[split node] at (-0.1,1.4){};
		\node (11)[solid node,fill=gray,thick] at (-0.5,1.8){};
		\node (12)[solid node,fill=gray,thick] at (-0.15,1.65){};
		\node (13)[solid node,fill=gray,thick] at (-0.05,1.9){};
		\node (14)[solid node,fill=gray,thick] at (0.3,2){};
		\node (15)[empty node] at (0.1,1.6){};
		%\draw[thick line] (a.center) to (-1.3,2.7);
		%\draw[thick line,decorate,decoration={snake,amplitude=-.3mm,segment length=2.5mm,pre length=3mm}] (b0.center) to (0.13-0.45,2.45);
		\begin{pgfonlayer}{background}
		\draw[gray line] (a0) to (a) to (a1);
		\draw[gray thin line,dash pattern={on 3pt off 3pt}] (b0) to (b1);
		\draw[gray thin line,dash pattern={on 3pt off 3pt}] (c0) to (c1);
		\draw[gray thin line,dash pattern={on 3pt off 3pt}] (d0) to (d1);
		\draw[gray line,decorate,decoration={snake,amplitude=.2mm,segment length=3mm}] (2.center) to (8.center);
		\draw[gray line,decorate,decoration={snake,amplitude=.2mm,segment length=3mm,pre length=3.5mm}] (1.center) to (9.center);
		\draw[gray line,decorate,decoration={snake,amplitude=.2mm,segment length=3mm}] (5.center) to (10.center);
		\draw[gray line,,decorate,decoration={snake,amplitude=.2mm,segment length=3mm}] (4.center) to (11.center);
		\draw[gray line,decorate,decoration={snake,amplitude=.2mm,segment length=3mm}] (4.center) to (12.center);
		\draw[gray line,decorate,decoration={snake,amplitude=.2mm,segment length=3mm}] (15.center) to (13.center);
		\draw[gray line,decorate,decoration={snake,amplitude=.2mm,segment length=3mm}] (15.center) to (14.center);
		\draw[thick line,decorate,decoration={snake,amplitude=.2mm,segment length=3mm}] (a.center) to (1.center);
		\draw[thick line] (1.center) to (2.center);
		\draw[thick line] (2.center) to (3.center);
		\draw[thick line] (3.center) to (4.center);
		\draw[thick line] (3.center) to (5.center);
		\draw[thick line,decorate,decoration={snake,amplitude=.2mm,segment length=3mm}] (4.center) to (6.center);
		\draw[thick line,decorate,decoration={snake,amplitude=.2mm,segment length=3mm}] (5.center) to (7.center);
		\end{pgfonlayer}
		\node (n)[empty node,label=right:{$n$}] at (1.6,0.6) {};
		\node (l0)[empty node,label=right:{$l_0$}] at (1.6,2.1) {};
		\node (l1)[empty node,label=right:{$l_1$}] at (1.6,2.6) {};
	\end{tikzpicture}
\caption{Finding $\sigma_0$ and $\sigma_1$ above $l_0$ and $l_1$, with a meet above $n$. The nodes $\rho_0$ and $\rho_1$ from Fact~\ref{fact:dlo-blossom} are represented as hollow nodes, the node $\gamma$ from the proof of Fact~\ref{fact:dlo-blossom2} is represented by a slashed node, and the nodes $\alpha^i_j$ from Fact~\ref{fact:dlo-blossom2} are in grey.}
\label{fig:Devlin-to-ACA}
\end{center}
\end{figure}
\begin{theorem}\label{thm:dt2-implies-rt2}
For every $k, \ell \geq 1$, $\RT{2}{k, \ell} \leq_c \DT{2}{2k,2\ell+1}$.
\end{theorem}
\begin{proof}
Let $f: [\omega]^2 \to k$ be an instance of $\RT{2}{k, \ell}$.
Let $\Qb = \{x_0, x_1, \dots \}$ be a computable enumeration of all the rationals.
Define $g: [\Qb]^2 \to 2k$ for every pair $\{x_p, x_q\} \in [\Qb]^2$
by $g(x_p, x_q) = \langle 0, f(p, q) \rangle$ if $x_p <_\Qb x_q$ and $g(x_p, x_q) = \langle 1, f(p, q) \rangle$ otherwise.
Let $U \subseteq \Qb$ be a solution to the instance $g$ of $\DT{2}{2k,2\ell+1}$,
that is, $(U, <_\Qb)$ is a DLO order and $|g[U]^2| \leq 2\ell+1$. Let $d < 2$ and $I \subseteq \{0, \dots, 2k-1\}$ with $|I| \leq \ell$  be such that $\{ i: \langle d, i \rangle \in g[U]^ 2\} \subseteq I$. Say $d  = 0$, the other case is symmetrical.
Build $U$-computably an infinite sequence $x_{p_0} <_\Qb x_{p_1} <_\Qb \dots$ such that $p_n < p_{n+1}$. Such a sequence exists since $(U, <_\Qb)$ has no endpoints.
For every $s < t \in \omega$, $f(x_{p_s}, x_{p_t}) = \langle 0, f(p_s, p_t)\rangle$.
Since $\{ i: \langle d, i \rangle \in g[U]^ 2\} \subseteq I$, $f(p_s, p_t) \in I$. Thus, letting $H = \{ p_s: s \in \omega \}$, $f[H]^2 \subseteq I$ so $|f[H]^2| \leq \ell$.
\end{proof}

We now give.a better lower bound to Devlin's theorem for pairs by constructing a computable instance of it with no $\Sigma^0_3$ solution.
%show that the exact arithmetical complexity of the simplest
%solution to Devlin's Theorem is $\Delta^0_4$ relatively to the instance.

\begin{definition}\index{thin set}\index{set!thin}\index{homogeneous set!for a tree}
  A set $H\subseteq\Nb$ is \emph{thin} for a $\sigma\in k^{<\Nb}$ is there exists
  some $i<2$ such that for all $n\in H,$ $n<|\sigma|\implies \sigma(n)\neq i$. It is
 \emph{thin} for a tree $T\subseteq2^{<\Nb}$ if the tree $\{\sigma\in T:H\text{
  is thin for }\sigma\}$ is infinite.
\end{definition}

Whenever $k = 2$, a thin set is also called \emph{homogeneous}.

\begin{definition}\index{approximation!$\Delta^0_2$}\index{approximation!$\Delta^0_3$}\index{$\Delta^0_2$ approximation}
  A \emph{$\Delta^0_2$ approximation} of a sequence $\sigma\in k^{\leq\om}$ is a sequence $(\sigma_s)_{s\in\Nb}$ of finite sequence such
  that for every $n$, $\lim \sigma_s(n)$ exists and has value $\sigma(n)$.

  A \emph{$\Delta^0_3$ approximation} of a sequence $\sigma$ is a sequence $(\sigma_{s,t})_{s,t\in\Nb}$
  such that for every $s\in \Nb$, $(\sigma_{s,t})_{t\in\Nb}$ is a $\Delta^0_2$
  approximation of a sequence $\sigma_s$, and $(\sigma_s)_{s\in\Nb}$ is a $\Delta^0_2$
  approximation of $\sigma$.
\end{definition}
% \begin{statement}
%   $\mathrm{RWKL}$ is the assertion that for every tree $T\subseteq2^{<\Nb}$,
%   there is a homogeneous set for $T$.   $\mathrm{RWKL}$ is the assertion that for every tree $T\subseteq2^{<\Nb}$,
%   there is a homogeneous set for $T$.

% \end{statement}
\begin{theorem}\label{th:dt-implies-rwkl''}
  Let $F$ be a finite Joyce order with two elements, $\Jb$ be a DLO Joyce structure and $k$ be an integer.
  For every $\Delta^0_3$ approximation of an infinite tree $T\subseteq k^{<\infty}$, there exists a coloring
  $f:{\Jb \choose F}\to k$ such that for every DLO Joyce suborder $S\subseteq \Jb$, if $f$ avoids $1$ color in ${S\choose F}$ then $S$ computes a thin set for $T$.
\end{theorem}
\begin{proof}
  \newcommand{\femb}{f_{<_\Qb}}
  We can always suppose $\Jb$ is a coded Joyce order. Let $m,M$ be such that $\{m;M\}=F$, and $|m|<|M|$.
  Let $(T_{s,t})_{s,t\in\Nb}$ be a $\Delta^0_3$ approximation of an infinite
  tree, that is for every $s\in\Nb$, $T_s=\lim_t T_{s,t}$ exists and $T=\lim_s
  T_s$ exists. Let $P_{s,t}$ be the
  leftmost path of $T_{s,t}$ of length $s$. Note that $P_s=\lim_tP_{s,t}$ is the
  leftmost path of $T_s$ of length $s$, and $P=\lim_s P_s$ is the leftmost path
  of $T$. If $\{\sigma,\tau\}\in {\Jb\choose F}$ with $|\sigma|>|\tau|$, define \[f(\sigma, \tau)=P_{|\sigma|,|\tau|}(|\sigma\meet\tau|),\] 
  a computable coloring of $\Jb\choose F$ in $k$ colors.

  % Define $\femb(\sigma,\tau)$ by $\femb(\sigma,\tau)=1$ iff
  % $\sigma\ltlex\tau\iff|\sigma|<|\tau|$, and define $f$ to be the product of $f_P$
  % and $\femb$: \[f(\sigma,\tau)=(\femb(\sigma,\tau),f_P(\sigma,\tau))\] %The set $(\cantor,\ltlex)$ has order-type $\Qb$, by
  % % $\DT{}2$, let
  Now, suppose that $S\subseteq\cantor$ is of order-type $\Qb$ and such that $S\choose F$ avoids some color $i<k$ for $f$.
   The claim is that the set \[H=\{|a\meet c|:(\exists a,b,c,d\in S)[a\ltlex
   b\ltlex c\ltlex d\land |a\meet c|<|a\meet b|,|c\meet d|]\}\] is thin for
  $P$, and thus for $T$.

  % Let $i$ such that the color $(1, 1-i)$ does not appear in $[S]^2$, suppose
  % for instance that $j=1$.
  % Let $m,M\in F$ be such that $|m|<|M|$. 
  Here, we suppose $m\ltlex M$, so that if $x,y\in\Jb$ satisfies $|x|<|y|$, then $\{x,y\}\in{\Jb\choose F}$ iff $x\ltlex y$.
  Let $\ell\in H$, fix $a,b,c,d$ witnessing it. % $\rho\in\cantor$ such
  %that $\ell=|f(\rho)|$.
  Let $s_0>\ell$ be such that $P_{s}(\ell)$ has settled for every
  $s\geq s_0$. Let $\sigma\in S$ in the interval with bounds $a$ and $b$ such
  that $s_1=|\sigma|\geq s_0$, which exists are there are infinitely many
  elements of $S$ in this interval.
  %$\sigma\succ\rho0$ be such that $s_1 = |f(\sigma)|\geq s_0$.
  Let $t_0$ be such
  that $P_{s_1,t_0}(\ell)$ has settled for every $t\geq t_0$, and let $\tau\in
  S$ be in the interval with bounds $c$ and $d$ with $t_1=|\tau|\geq
  \max(t_0,\ell, s_1)$.
  % $\tau\succ\rho1$ such that $t_1 = |f(\tau)| \geq t_0$. 
  Then, $\{\sigma,\tau\}\in{S\choose F}$ and
  $f_P(\sigma,\tau)=P_{s_1,t_1}(\ell)\neq i$ as $S\choose F$ avoids color $i$. By our choice of $t_1$,
  $P_{s_1,t_1}(\ell)=P_{s_1}(\ell)$, and by our choice of $s_1$,
  $P_{s_1}(\ell)=P(\ell)\neq i$, that is, $H$ is thin for $P$, and thus for $T$.

  If $m\gtlex M$, we do the same argument, but we take $\sigma$ in the interval with bounds $c$ and $d$, and $\tau$ in the interval with bound $a$ and $b$, to get the same conclusion.
  
  We proved that $H$ is thin
  for $T$. As $H$ is c.e. in $S$, it contains an infinite computable subset,
  which is thin for $T$ as well.
\end{proof}
\begin{corollary}
  Let $k$ be an integer.
  For every $\Delta^0_3$ approximation of an infinite tree $T\subseteq k^{<\infty}$, there exists a coloring
  $f:[\Qb]^2\to 2k$ such that for every DLO Joyce suborder $S\subseteq \Jb$, if $f$ takes only $2k-1$ color on $[S]^2$ then $S$ computes a thin set for $T$.
\end{corollary}
\begin{proof}
  Let $F_0$ and $F_1$ be the two Joyce structure with two elements. Let $f_0$ and $f_1$  be given by \Cref{th:dt-implies-rwkl''} for $F_0$ and $F_1$. Define $f:\Qb\to2k$ by enriching $\Qb$ to a Joyce order, and if $\sigma,\tau\in\Qb$, then $f(\sigma,\tau)=(i,f_i(\sigma,\tau))$ if and only if $\{\sigma,\tau\}$ is isomorphic to $F_i$.

  If $S\subseteq\Qb$ is an isomorphic substructure such that $f$ takes at most $2k-1$ color on $[S]$, then let $(i,j)$ with $i<2$ and $j<k$ be an avoided color. Then, $f_i$ avoids color $j$ on ${S\choose F_i}$, and by our choice of $f_i$, $S$ computes a thin set for $T$.
\end{proof}

\begin{definition}\index{DNC!function}\index{DNC!degree}\index{function!DNC}
A function $f: \omega \to \omega$ is \emph{DNC relative to $X$}
if for every $e$, $f(e) \neq \Phi^X_e(e)$. Here, $f(e)$ can be any value if $\Phi^X_e(e)\uparrow$. A Turing degree is DNC relative to $X$ if it computes such a function.
\end{definition}

\begin{lemma}\label{lem:rwkl-dnc}
For every $k \geq 2$ and set $X$, there exists an $X$-computable tree $T \subseteq k^{<\Nb}$
such that every infinite set thin for $T$ is of DNC degree relative to $X$.
\end{lemma}
\begin{proof}
Let $T \subseteq k^{<\Nb}$ be an infinite $X$-computable tree such that every infinite path is a Martin-L\"of random relative to $X$ in base $k$. Let $H$ be an infinite set thin for $T$. In particular, there is some path $P \in [T]$ and some color $i < k$ such that $H \subseteq \{ x: P(x) \neq i \}$. Let $Z$ be the Martin-L\"of random $P$ in base 2. The set $H$ computes an infinite subset of $Z$ or of $\overline{Z}$, hence is of DNC degree relative to $X$.
\end{proof}

\begin{corollary}\label{cor:dt2-implies-dnczpp}
For every $\ell \geq 2$, there exists a computable instance of $(\forall k)\DT{2}{k,\ell}$ such that every solution is of DNC degree relative to $\emptyset''$.
\end{corollary}
\begin{proof}
Fix $\ell \geq 2$.
  By \Cref{lem:rwkl-dnc} relativized to $\emptyset''$, there exists a computable
  $\Delta^0_3$-approximation of a tree $T \subseteq \ell^{<\Nb}$ % \ludovic{Todo, check the number of colors.}
  such that every infinite set thin for $T$ is of DNC degree relative to $\emptyset''$. By
  \Cref{th:dt-implies-rwkl''}, let $f$ be a computable instance of $\DT{2}{2\ell,\ell}$ such
  that every solution compute a set $H$ thin for $T$. Then every solution is of DNC degree relative to $\emptyset''$.
  \end{proof}

\begin{lemma}\label{lem:ce-dlo-has-computable-subcopy}
For every $X$-c.e.\ dense linear order with no endpoints $(D, <_D)$,
there is an $X$-computable subset $S \subseteq D$ such that $(S, <_D)$ is a sub-copy of $(D, <_D)$.
\end{lemma}
\begin{proof}
We build an $X$-computable $<_\N$-increasing sequence \[x_0 <_\N x_1 <_\N \dots\] of elements of $D$ such that letting $S = \{x_n: n \in \N\}$, $(S, <_D)$ is a dense linear order without endpoints. Start with $x_0 \in D$ being any element. Having defined $F = \{x_0 <_\N x_1, \dots <_\N x_n\}$,
consider a minimal interval in $F \cup \{-\infty, +\infty\}$ with respect to $<_D$, that is, an interval $(a, b)$ with $a <_D b \in F \cup \{-\infty, +\infty\}$ such that $(a, b) \cap F = \emptyset$. Then wait until some element $x_{n+1}$ appears in $W \cap (a, b)$ with $x_{n+1} >_\N x_n$. Such element must be found since, as $(D, <_D)$ is a DLO with no endpoints, there are infinitely many elements in $W \cap (a, b)$, so elements of arbitrary large value with respect to $<_\N$. By choosing the minimal interval in an appropriate way, we can ensure that $(S, <_D)$ is a DLO with no endpoints. Since $D$ is $X$-c.e., searching for $x_{n+1}$ is done $X$-computably, so $S \leq_T X$.
\end{proof}

\begin{corollary}\label{cor:dt2-no-sigma3}
For every $\ell \geq 2$, there exists a computable instance of $(\forall k)\DT{2}{k,\ell}$ 
with no $\Sigma^0_3$ solution.
\end{corollary}
\begin{proof}
By \Cref{cor:dt2-implies-dnczpp}, there is a computable instance $f: [\Qb]^2 \to k$ of $(\forall k)\DT{2}{k,\ell}$ such that every solution is of DNC degree relative to $\emptyset''$.
Suppose for the contradiction that there is a $\Sigma^0_3$ sub-copy $(U, <_{\Qb})$ of $(\Qb, <_{\Qb})$ such that $|f[U]^2| \leq \ell$.
By \Cref{lem:ce-dlo-has-computable-subcopy}, there is a $\Delta^0_3$ subset $H \subseteq U$ such that $(H, <_{\Qb})$ is a sub-copy of $(\Qb, <_{\Qb})$.
Since $H$ is of DNC degree relative to $\emptyset''$, is computes a function $f: \omega \to \omega$ such that for all $e$, $f(e) \neq \Phi^{\emptyset''}_e(e)$. Since $H$ is $\Delta^0_3$, so is $f$, hence there is some $e$ such that $\Phi_e^{\emptyset''} = f$.
But then $f(e) = \Phi_e^{\emptyset''}(e)$, contradiction.
\end{proof}

%\begin{corollary}
%  There exists a computable instance of $\DT{2}{4,2}$ with no $\Delta^0_3$ solution.
%\end{corollary}
%\begin{proof}
%  By \todo{ref} relativized to $\emptyset''$, there exists a computable
%  $\Delta^0_3$-approximation of a tree $T$ with no $\Delta^0_3$ solution. By
%  \Cref{th:dt-implies-rwkl''}, let $f$ be a computable instance of $\DT{2}{4,2}$ such
%  that every solution compute a set $H$ homogeneous for $T$. Then, $f$ cannot
%  have a $\Delta^0_3$ solution.
%\end{proof}

%\section{$\Delta^0_4$ solutions to Devlin's theorem for pairs}
%
%
%\begin{theorem}
%  For every $f:\Qb\to k$, there exists a $\Delta^0_4$ set $S$ such that $f$
%  takes only 2 colors on $[S]^2$.
%\end{theorem}
%\begin{proof}
%  \pelliot{todo}
%\end{proof}

\section{Above the big Ramsey number of Devlin's theorem}

In the case of Devlin's theorem for pairs, the existence of the big Ramsey number implies $\ACA_0$. By \Cref{cor:dt2to3-implies-aca}, this is also the case when weakening the statement by allowing 3 instead of 2 colors in the solution.  We shall now conclude the chapter about Devlin's theorem by proving that this bound is tight, in that the statement $(\forall k)\DT{2}{k, 4}$ does not imply $\ACA_0$ over $\RCA_0$. The proof consists essentially of replacing the use of Milliken's tree theorem for height 3 by the statement $(\forall k)\PMT{3}{k,2}$ which admits cone avoidance by \Cref{thm:pmtt3k2-cone-avoidance}. The cost of this substitution is an increase in the number of colors allowed in the solution.

\begin{theorem}[$\RCA_0 \wedge (\forall k)\PMT{3}{k,2}$]\label{thm:dt2to4-one-embedding-type-cone-avoidance}
Let $X$ be a DLO Joyce structure and $F$ be a Joyce structure of size 2. Then for every $k \in \omega$ and every coloring $f: {X \choose F} \to k$,
there is a subcopy $Y$ of $X$ such that $f$ uses at most 2 colors over ${Y \choose F}$.
\end{theorem}
\begin{proof}
The proof is exactly the same as the one of \Cref{thm:strong-devlin-one-type},
but replacing an application of Milliken's tree theorem for height 3 by
$(\forall k)\PMT{3}{k,2}$.
\end{proof}

\begin{theorem}\label{thm:pmt-to-dt-more-colors}
$(\forall k)\PMT{3}{k,2}$ implies $(\forall k)\JDT{2}{k, 4}$ over $\RCA_0$.
\end{theorem}
\begin{proof}
Let $F_0$ and $F_1$ be the two coded Joyce orders of size $2$.
Let $X$ be a coded DLO Joyce order and let $f: [X]^2 \to k$ be a coloring.
By \Cref{thm:dt2to4-one-embedding-type-cone-avoidance},
there is a subcopy $X_0$ of $X$ such that $f$ uses at most 2 colors $i_0, i_1$ over ${X_0 \choose F_0}$.
Again by \Cref{thm:dt2to4-one-embedding-type-cone-avoidance},
there is a subcopy $X_1$ of $X_0$ such that $f$ uses at most 2 colors $j_0, j_1$ over ${X_1 \choose F_1}$.
We claim that $f[X_1]^2 \subseteq \{i_0, i_1, j_0, j_1\}$. Let $E \in [X_1]^2$. In particular, $E$ is isomorphic to $F_0$ or $F_1$. In the first case, $f(E) \in \{i_0, i_1\}$ and in the second case, $f(E) \in \{j_0, j_1\}$.
Thus $X_1$ is a subcopy of $X$ such that $|f[X]^2| \leq 4$.
\end{proof}

\begin{corollary}\label{cor:jdt4-cone-avoidance}
$(\forall k)\JDT{2}{k, 4}$ admits cone avoidance.
\end{corollary}
\begin{proof}
By \Cref{thm:pmtt3k2-cone-avoidance}, $(\forall k)\PMT{3}{k,2}$ admits cone avoidance
hence so does $(\forall k)\JDT{2}{k, 4}$ by \Cref{thm:pmt-to-dt-more-colors}.
\end{proof}

\begin{corollary}\label{cor:cone-avoidance-joyce-devlin-4-colors}
$(\forall k)\JDT{2}{k, 4}$  does not imply $\ACA_0$ over $\RCA_0$.
\end{corollary}
\begin{proof}
By \Cref{thm:pmtt3k2-cone-avoidance}, $(\forall k)\PMT{3}{k,2}$ admits cone avoidance,
hence there is a model $\Mc$ of $\RCA_0 \wedge (\forall k)\PMT{3}{k,2}$ which is not a model of $\ACA_0$.
In particular, $\Mc \models (\forall k)\JDT{2}{k, 4}$ by \Cref{thm:pmt-to-dt-more-colors}.
\end{proof}

\section{The Erd\H{o}s Rado theorem}\label{subsect:er-theorem}

% Devlin's theorem also implies Erd\H{o}s Rado Theorem, a disjunctive statement with flavours both from Ramsey's theorem for pairs and Devlin's Theorem for pairs, in a very straightforward way.
Erd\H{o}s and Rado proved that it is always possible to obtain either a copy of $\Qb$ of one color, or else an infinite homogeneous set (in the sense of Ramsey's theorem) of the other color.
\begin{theorem}[Erd\H{o}s Rado theorem]\index{theorem!Erd\H{o}s Rado theorem}\index{Erd\H{o}s Rado theorem}
  For every $f:[\mathbb Q]^2\to 2$, there exists a subset $S\subseteq\mathbb Q$ such that either $S$ is infinite and $f$-homogeneous of color 0, or $S$ is of order-type $\Qb$ and $f$-homogeneous of color 1.
\end{theorem}
\begin{statement}\index{statement!$\mathrm{ER}^2$}
  $\mathrm{ER}^2$ is the statement denoting the Erd\H{o}s Rado theorem.
\end{statement}
This statement was studied by \cite{Chong2019Strengthc, Frittaion2017Coloring, Dzhafarov2017Coloring} in the setting of reverse mathematics. One would expect it to be a consequence of Devlin's theorem by the optimality of the bounds noted above. We give a direct combinatorial proof of $\mathrm{ER}^2$ from Devlin's theorem for pairs of rationals.
\begin{theorem}\label{th:dt242-implies-er2}
  $\DT{2}{4,2}$ implies $\mathrm{ER}^2$.
\end{theorem}
\begin{proof}
  Let $f:[\mathbb Q]^2\to 2$ be a coloring of pairs of rationals, regarded as a given instance of $\mathrm{ER}^2$. Let $f_0$ be the $2$-coloring of $[\Qb]^2$ witnessing the fact that big Ramsey degree of the pairs of rationals is 2, that is, $f_0$ is such that for every sub-copy $S$ of the rationals, $|f_0[S]^2| = 2$. An explicit construction of $f_0$ is given at the start of \Cref{sect:lower-bound-devlin}.

Apply $\DT{2}{4,2}$ to the $4$-coloring $f\times f_0:(q, r)\mapsto \langle f(q, r), f_0(q,r) \rangle$ to get a subcopy of the rationals $S$ such that $f \times f_0$ uses at most $t_{\DT{}{}}(2) = 2$ colors on $[S]^2$. As $[S]^2$ must have two colors for $f_0$, the two colors of $[S]^2$ for $f\times f_0$ must be of the form $(c_0, 0)$ and $(c_1,1)$. The rest of the proof is split into 3 cases.

 \case{1}{$c_0 = c_1 = 1$.} In this case, $[S]^2$ is monochromatic with color 1 for $f$, and since $S$ has order-type $\mathbb Q$ it is a solution to $f$ as an instance $\mathrm{ER}^2$.

 \case{2}{$c_0 = 0$ and $c_1 = 1$.} Then $f_0(q,r) = 0$ implies $f(q,r) = 0$, for all $q,r \in S$. We build an infinite set $T=\{q_{n_i}:i\in\Nb\}$ such that $[T]^2$ is monochromatic for $f_0$ with color $0$, and therefore also for $f$ with color $0$. To this end, we build an increasing sequence of rationals $(q_{n_i})_{i\in\Nb}$ in $S$, such that $(n_i)_{i\in\Nb}$ is also increasing. Fix any $q_{n_0}\in S$, and suppose $q_{n_i}$ has been defined. As there exists infinitely many rationals in $S$ above $q_{n_i}$, there exists $n_{i+1}>n_i$ such that $q_{n_{i+1}}> q_{n_i}$ and $q_{n_{i+1}}\in S$. This completes the construction. Now, $T$ is an infinite $f$-homogeneous set with color $0$, and hence a solution to $f$ as an instance of $\mathrm{ER}^2$.

 \case{3}{$c_1 = 0$ and $c_0 = 0$.} Symmetric to Case 2.
\end{proof}

%In fact, already $\DT2{<\infty,4}$ proves $\mathrm{ER}^2$.

%\peter{I am also skipping the rest of this subsection}

However, $\mathrm{ER}^2$ admits cone avoidance.
As a warm-up before the proof of this result, we prove the following:
\begin{lemma}\label{le:ER-construction-example}
  Let $f_J:[\cantor]^2\to2$ be such that $f_J(\sigma,\tau)=1$ iff $\emptyset'[|\sigma|]\uh|\sigma\meet\tau|=\emptyset'[|\tau|]\uh|\sigma\meet\tau|$. Then, there exists two computable infinite sets $X_0$ and $X_1$ such that $[X_i]^2$ is monochromatic of color $i$ for $f_J$.
\end{lemma}
\begin{proof}
  First, we do the construction for $i=0$. The set $X_0$ is defined as $\{0^{n_i}1:i\in\om\}$ for an increasing sequence $(n_i)_{i\in\om}$ verifying that $\emptyset'[n_i+1]\uh n_i\neq \emptyset'[n_{i+1}+1]\uh n_i$. Suppose that $n_i$ is defined. Then $n_{i+1}$ is the first integer $n>n_i$ found such that $\emptyset'[n+1]\uh n\neq \emptyset'\uh n$ and $\emptyset'[n_i+1]\uh n_i\neq \emptyset'[n+1]\uh n_i$, which must exists as otherwise $\emptyset'$ would be computable. Then, $[X_0]^2$ is monochromatic of color 0 by construction.

  For the other case, define $X_1=\{\emptyset'[n]\uh n:n\in\om\}$, we claim that $[X_1]^2$ is monochromatic for $f_J$ of color 1. Let $\sigma_0,\sigma_1\in X_1$ and for $i<2$, the length $n_i=|\sigma_i|$ is such that $\sigma_i=\emptyset'[n_i]\uh n_i$. We have $f_J(\sigma_0,\sigma_1)$ iff $\emptyset'[|\sigma_0|]\uh|\sigma_0\meet\sigma_1|=\emptyset'[|\sigma_1|]\uh|\sigma_0\meet\sigma_1|$, which we claim is true. Indeed, as $|\sigma_0|>|\sigma_0\meet\sigma_1|$, we have
  \begin{equation*}
    \begin{aligned}
      \emptyset'[|\sigma_0|]\uh|\sigma_0\meet\sigma_1|  & = (\emptyset'[n_0]\uh n_0)\uh|\sigma_0\meet\sigma_1| \\
       & = \sigma_0\uh|\sigma_0\meet\sigma_1|
    \end{aligned}
  \end{equation*}
  and as $|\sigma_1|>|\sigma_0\meet\sigma_1|$:
  \begin{equation*}
    \begin{aligned}
      \emptyset'[|\sigma_1|]\uh|\sigma_0\meet\sigma_1|  & = (\emptyset'[n_1]\uh n_1)\uh|\sigma_0\meet\sigma_1| \\
       & = \sigma_1\uh|\sigma_0\meet\sigma_1|.
    \end{aligned}
  \end{equation*}
  By definition of the meet operator, $\sigma_0\uh|\sigma_0\meet\sigma_1|=\sigma_1\uh|\sigma_0\meet\sigma_1|$, therefore $f_J(\sigma_0,\sigma_1)=1$ and $[X_1]^2$ is monochromatic of color 1.
  %\peter{This proof is fine but I think this would be easier to understand if you proved this by contradiction. if they are not equal then there meet is too long.}\pelliot{If you think so, I'll let you modify it}
\end{proof}

\begin{theorem}
  The statement $\mathrm{ER}^2$ admits cone avoidance.
\end{theorem}
\begin{proof}
  \newcommand\ER{\mathrm{ER}}
  \newcommand\femb{f_{<_\Qb}}
  \newcommand\iemb{{i_{<_\Qb}}}
  \renewcommand\ltlex{<_\Qb}
  Let $Z$, and $C$ with $C\not\leq_T Z$. Let $f: [\cantor]^2\to 2$ be a $Z$-computable coloring, seen as an instance of $\ER^2$ as $(\Qb,<)$ and $(\cantor,\ltlex)$ are computably isomorphic. Define $i_\infty=0$ and $i_\Qb=1$, so that the goal is to find either an infinite set homogeneous for color $i_\infty$, or a set of order-type $\Qb$ homogeneous for color $i_\Qb$.
%  First, suppose that $\ACA_0$ is true. In this case, $\ER^2$ is true, as $\ACA_0$ implies $\DT2{<\infty, 2}$ by \Cref{thm:strong-devlin-one-type}, which itself implies $\mathrm{ER}^2$ by \Cref{th:dt242-implies-er2}. If $\ACA_0$ does not hold, there exists $X$ such that $\forall Y$, $Y\neq X'$, in other words the jump of $x$ does not exists. Fix such an $X$.
%  More precisely, we reason depending on the following. Even in $\RCA_0$, we know that either $\ACA_0$ is true, or $\exists X, \forall Y$, $Y\neq X'$. In the first case, $\ER^2$ is true, as $\ACA_0$ implies $\DT2{<\infty, 3}$ by \pelliot{ref}, which itself implies $\mathrm{ER}^2$ by \pelliot{ref}.

  The proof goes as follows: first, we build a set $G$ such that $C\not\leq_TZ\oplus G$ but $C\leq_T(Z\oplus G)'$. Then, we apply cone avoidance of $\DT2{<\infty, 4}$ to the product of three colorings: the initial instance of $\ER^2$, the Jockush coloring relativized to $Z\oplus G$, and the coloring witnessing the fact that at least two colors must remain. Finally, we reason depending on which are the four remaining colors, with the two main constructions being linked with the two constructions of \cref{le:ER-construction-example}. Let us start with the existence of $G$. %, a result that is also a corollary of 

%  The idea of the proof is to reduce the coloring $f$ to a relativization of the Jockusch coloring, in a way that preserves the cone avoidance of $C$. First, we find $G$ such that $C\not\leq_TZ\oplus G$ but $C\leq_T(Z\oplus G)'$.
  \begin{claim}
    There exists $G$ such that  $C\not\leq_TZ\oplus G$ but $C\leq_T(Z\oplus G)'$.
  \end{claim}
  \begin{proof}
  Define a forcing, whose conditions are the tuples $(p,n)$ where $p:\om\times\om\to 2$ has finite domain, and $n$ is an integer. A condition $(q,m)$ extends a condition $(p,n)$ if $q\supset p$, and for every $(x,y)\in\dom(q)\setminus\dom(p)$, if $x<n$ then $q(x,y)=C(x)$. It is clear that if $G$ is generic enough for this forcing, then $G'\geq_T C$: Indeed, for every $i$, the set of conditions $\{(p,n):n\geq i\}$ is dense. Therefore, $\lim_{s\to\infty}G(i,s)$ is always defined with value $C(i)$.

  It remains to show that $C\not\leq_TZ\oplus G$. We prove that for every $e$, the set of conditions $(p,n)$ for which there is an $i$ such that $\Phi_e^{Z\oplus p}(i)\downarrow\neq C(i)$ or there is an $i$ such that for all $(q,m)$ extending $(p,n)$, $\Phi_e^{Z\oplus q}(i)\uparrow$, is dense. Indeed, fix $(p_0,n_0)$. If there exists $(p,n)\leq (p_0,n_0)$ and $i$ such that $\Phi^{Z\oplus p}(i)\downarrow\neq C(i)$, then $(p,n_0)$ extends $(p_0,n_0)$ and forces $\Phi_e^{Z\oplus G}$ not to compute $C$. If there is an $i$ such that no $(q,n)\leq (p_0,n_0)$ are such that $\Phi_e^{Z\oplus q}(i)\downarrow$, then already $(p_0, n_0)$ forces partiality of $\Phi_e^{Z\oplus G}$. If none of the two previous cases happen, then $Z$ computes $C$: to know the value of $C(i)$, guess the first $n_0$ values of $C$, using these find a $(q,n)\leq (p_0, n_0)$ such that $\Phi_e^{Z\oplus g}(i)\downarrow$, we have $\Phi_e^{Z\oplus g}(i)=C(i)$. This contradicts that $C\not\leq_T Z$.
\end{proof}
% Define $i_\infty=0$ and $i_\Qb=1$, so that the goal is to find either an infinite set homogeneous for color $i_\infty$, or a set of order-type $\Qb$ homogeneous for color $i_\Qb$.
Let $f_J^G$ be the coloring defined in the proof of \Cref{th:lower-bound-devlin} relativized to $Z\oplus G$, that is, $f_J^G(\sigma,\tau)=1$ iff $(Z\oplus G)'[|\sigma|]\uh |\sigma\meet\tau| = (Z\oplus G)'[|\tau|]\uh |\sigma\meet\tau|$. If $|\tau|>|\sigma|$, we can see color 1 for $f_J^G$ as saying: $|\tau|$ witness that the interval from $|\sigma\meet\tau|$ to $|\sigma|$ is ``large'' (relatively to $Z\oplus G$). To reflect this, we define $i_s=0$ the ``small'' color, and $i_\ell=1$ the ``large'' color. As in \Cref{th:lower-bound-devlin}, let also $\femb(\sigma,\tau)=1$ iff $(\sigma\ltlex\tau\iff|\sigma|<|\tau|)$, note that $\femb$ can be seen as the coloring which outputs the finite Joyce structure of $\{\sigma,\tau\}$. For the colors of $\femb$, we will use the variable $\iemb$.

Finally, define $g:[\cantor]^2\to(2\times2\times2)$ by
\[
g(\sigma,\tau)=(f(\sigma,\tau),f_J^G(\sigma,\tau), \femb(\sigma,\tau)).\]
We apply cone avoidance of $\DT2{<\infty,4}$, \Cref{{cor:cone-avoidance-joyce-devlin-4-colors}}, to the coloring $g$ to get a set $S\subseteq\cantor$ such that $S\oplus G\oplus Z\not\geq_T C$ and $(S,\ltlex)$ is a dense linear order with no endpoints, and such that $g$ takes at most 4 colors on $[S]^2$. % . The first one is $f$, the instance of $\ER^2$. The second and third are the witness of $\DT2{<\infty,3}$ implies $\ACA_0$ as defined in \Cref{th:lower-bound-devlin}: $f_J$, and $\femb$.

  Recall that none of the colors of $\femb$ can be avoided in a subset of $\cantor$ of order-type $\Qb$, therefore the two sets $S_\iemb=\{(\sigma,\tau):\femb(\sigma,\tau)=\iemb\}$ for $\iemb<2$ must be non empty; and the sum of the number of colors taken by $g$ on them is at most 4. Start by supposing that for each $\iemb<2$, $g$ takes at most 2 colors on $S_\iemb$.

  We reason depending on the following cases:
 
  \case{1}{There exists $\iemb<2$, such that $S_\iemb$ is monochromatic for $f_J^G$.}
 
  \case{2}{There exists $\iemb<2$, such that on $S_\iemb$, $f=f_J^G$.}
 
  \case{3}{There exists $\iemb<2$, such that on $S_\iemb$, either $f$ is homogeneous of color $i_\infty$, or $f=1-f_J^G$.}
 
  \case{4}{For all $\iemb<2$, $S_\iemb$ is monochromatic of color $i_\Qb$ for $f$.}
%  \item[\emph{Case 2.}] For all $i<2$, either $f$ is monochromatic of color $i_\Qb$ on $S_i$, or $f=f_J^G$ on $S_i$.
 % \end{itemize}
 
 \medskip
  We now prove the four cases in three different construction, Case~4 being trivial. The first construction is the one from \Cref{th:lower-bound-devlin}, and shows that Case~1 cannot happen. The second and third construction correspond to the two constructions of \Cref{le:ER-construction-example}. To separate them more clearly, the proof is divided in claims.
  \begin{claim}
    In Case 1, $S\oplus Z\oplus G$ computes $(Z\oplus G)'$.
  \end{claim}
  \begin{proof}
    The first paragraph after the proof of Fact \ref{fact:dlo-blossom2} asserts that the function $f^G_J$ must be monochromatic for color $i_\ell$. The second paragraph asserts that in this case, $S\oplus G\oplus Z$ computes $(Z\oplus G)'$.% \pelliot{make a precise reference}
  \end{proof}
  By our choice of $G$, $(Z\oplus G)'$ computes $C$, and thus $S\oplus Z\oplus G\geq_TC$, a contradiction with our choice of $S$, so Case 1 cannot happen.
  % \begin{claim}
  %   In case 2, either $S_0$ computes a subset $S'\subseteq S_0$ of order-type $\Qb$ homogeneous for color $i_\Qb$, or $S_0$ computes an infinite subset homogeneous for color $i_\infty$.
  % \end{claim}
  \begin{claim}
    In Case 2, there exists a set $\hat S\subseteq S$ computable in $S\oplus Z\oplus G$, such that $[\hat S]^2\subseteq S_\iemb$ is an infinite subset monochromatic for color $i_\infty$.
  \end{claim}
  \begin{proof}
%    We prove the following fact: if $f=f_J^G$ on $S_i$, then either there exists an infinite subset $\hat S$ with $\hat S\oplus Z\not\geq_T C$ and $[\hat S]^2\subseteq S_i$, such that either is monochromatic of color $i_\infty$, or $\hat S$ is of order-type $\Qb$ and $[\hat S]^2$ is monochromatic of color $i_\Qb$. Given this, the result is clear. Now, we prove this claim.

    %The strategy is the following: first, if for each $i$, $f$ is monochromatic of color $i_\Qb$ on $S_i$, then $S$ is a witness of
    Note that by the fact that we are in Case 2, a set $\hat S$ with $[\hat S]^2\subseteq S_0$ is such that $[\hat S]^2$ is monochromatic of color $i_\infty$ for $f$ if and only if it is monochromatic of color $i_s$ for $f_J^G$. The following construction corresponds to the first case of \Cref{le:ER-construction-example}. We computably in $S\oplus G\oplus Z$ define a sequence $(F_n, A_n)$, where $F_n$ is a finite approximation to $\hat S$ and $A_n$ a reservoir for future addition to $F_n$, such that for all $n\in\om$ the following holds:
    \begin{enumerate}
    \item $F_n$ is a finite set such that $[F_n]^2\subseteq S_\iemb$;
    \item $A_n\subseteq S$ is of order-type $\Qb$;
    \item $F_n\subsetneq F_{n+1}$ and $A_{n+1}\subseteq A_n$;
    \item for all $\sigma\in F_n$ and all $\tau\in F_n\cup A_n$, $(\sigma,\tau)\in S_\iemb$ and $f_J^G(\sigma,\tau)=i_s$;
    \item for all $\sigma\in F_{n+1}\setminus F_n$, $\sigma\in A_n$.
%    \item       A condition is a couple $(F,D)$ such that , and $D\subseteq S$ is is computable in $S$ and of order-type $\Qb$, and such that:

%      A condition $(F_1,D_1)$ extends a condition $(F_0,D_0)$ if $D_1\subseteq D_0$ and

    \end{enumerate}
Suppose $F_n,A_n$ are defined. If there is no $\sigma, \tau_0, \tau_1\in A_n$ with $(\sigma, \tau_i)\in S_\iemb$ for $i<2$ such that $(G\oplus Z)'[|\sigma|]\uh \ell\neq (G\oplus Z)'\uh \ell$ where $\ell=\min|\sigma\meet \tau_0|,|\sigma\meet \tau_1|$, then $A_n\oplus G\oplus Z$ would compute $(G\oplus Z)'\geq_T C$. % There must exists $\sigma\in A_n$ and $A$ such that $\{\sigma\}\times A\subseteq S_\iemb$ and $\{\sigma\meet\tau:\tau\in A\}$ is a singleton $\sigma\meet A$ such that $(G\oplus Z)'[|\sigma|]\uh |\sigma\meet A|\neq (G\oplus Z)'\uh |\sigma\meet A|$, as otherwise $A_n\oplus G\oplus Z$ would compute $(G\oplus Z)'\geq_T C$, a contradiction.
Define $F_{n+1}=F_n\cup\{\sigma\}$ and \[A_{n+1}=\{\tau\in A_n:(G\oplus Z)'\uh \ell=(G\oplus Z)'[|\tau|]\uh \ell\land \tau_0\ltlex \tau\ltlex\tau_1\}.\]
By construction, all items are satisfied. Define $\hat S=\bigcup_n F_n$. By Item 3, $\hat S$ is infinite, and by Item 4 and 5, $[\hat S]^2$ is monochromatic of color $i_s$ for $f_J^G$, and thus monochromatic of color $i_\infty$ for $f$.
%But then, by defining $A_{n+1}$ to be the element of $A$ of length sufficiently big to witness the inequality, we satisfy all items.
  \end{proof}
  \begin{claim}
    In Case 3, there exists $\hat S\subseteq S$  such that $\hat S\oplus Z\not\geq_T C$ and $[\hat S]^2\subseteq S_i$ is an infinite set monochromatic of color $i_\infty$ for $f$.
  \end{claim}
  \begin{proof}
    If $f$ is homogeneous of color $i_\infty$ on $S_\iemb$ then $S_\iemb$ is already a witness of the claim. Otherwise, $f(\sigma,\tau)=i_\infty$ if and only if $f_J^X(\sigma, \tau)=i_\ell$, so all we need is to find a subset $\hat S\subseteq S$ such that $[\hat S]^2\subseteq S_\iemb$ and $[\hat S]^2$ is monochromatic of color $i_\ell$ for $f_J^G$. %and similarly $f(\sigma,\tau)=i_\Qb$ if and only if $f_J^X(\sigma, \tau)=i_s$, for all $(\sigma,\tau)\in S_0$.
    % Define $\hat S=\{X'_{|\sigma|}\uh |\sigma|:\sigma\in s\}$.
    % that is $i_\infty=i_l$ and $i_\Qb=i_s$.

         % We suppose in the following that $\iemb=1$, the other case being similar.
         The following construction is roughly analogous to the second case in the proof of \Cref{le:ER-construction-example}, however the number of time we can take a lower meet to avoid having color $i_s$ is not anymore equal to the number of times we might have to do it. This prevents us from doing the construction computably, however we can still make it cone avoiding. We build $\hat S$ using the following forcing:
    \begin{definition}
      A \emph{condition} is a couple $(F,D)$ such that $F$ is a finite set with $[F]^2\subseteq S_{\iemb}$, and $D\subseteq S$ is computable in $S$ and of order-type $\Qb$, and such that: for all $\sigma\in F$ and all $\tau\in F\cup D$, $(\sigma,\tau)\in S_\iemb$ and $f_J^G(\sigma,\tau)=i_\ell$.

      A condition $(F_1,D_1)$ \emph{extends} a condition $(F_0,D_0)$ if $D_1\subseteq D_0$ and for all $\sigma\in F_1\setminus F_0$, $\sigma\in D_0$.  We write $(F_1,D_1)\leq(F_0, D_0)$.
    \end{definition}
    % Suppose $(F, D)$ is a condition for this forcing. Let $D_0\ltlex D_1$ be two subsets of $D$ computable in $S$ of order type $\Qb$: For instance, pick $x_0\ltlex x_1\ltlex x_2$ in $D$, and define $D_0=\{x\in D:x_0\ltlex x\ltlex x_1\}$ and $D_1=\{x\in D:x_1\ltlex x\ltlex x_2\}$. Fix $e\in\om$. If there exists $x\in\om$ and $E\subseteq D_{1-\iemb}$ such that $\forall\sigma,\tau\in E$, $f_J^G(\sigma,\tau)=i_\ell$ % \pelliot{add the conditions on $E$} $E\subseteq D_0$
    If $\Fc$ is a filter for this forcing, then we let $S_\Fc=\bigcup\{ F: (\exists D)[(F,D)\in\Fc]\}$. We have that $[S_\Fc]^2\subseteq S_\iemb$ is monochromatic of color $i_\ell$ for $f_j^G$. So we need to find a filter ensuring that $S_\Fc$ is infinite and $S_\Fc\oplus Z$ does not compute $C$.
    \begin{definition}
      Let $(F,D)$ be a condition and $\varphi$ be a $\Delta^{0,Z}_0$ formula with a free set parameter $\hat S$. We say that:
      \begin{enumerate}
      \item $(F,D)\Vdash (\exists x)\varphi(\hat S, x)$ if $\varphi(F,x)$ holds for some $x\in\om$; 
      \item $(F,D)\Vdash (\forall x)\varphi(\hat S, x)$ if $\varphi(F\cup E,x)$ holds for every $x$, and for every $E\subseteq D$ with $[F\cup E]^2\subseteq S_\iemb$ monochromatic of color $i_\ell$ for $f_J^G$. 
      \end{enumerate}
    \end{definition}
    
    We claim that for every Turing functional, for every condition $(F, D)$, there is a condition $(F',D')\leq(F,D)$ such that $(F',D')\Vdash\Gamma^{\hat S\oplus Z}\neq C$. Let $D_0\ltlex D_1$ be two subsets of $D$ computable in $S$ of order type $\Qb$: For instance, pick $x_0\ltlex x_1\ltlex x_2$ in $D$, and define $D_0=\{x\in D:x_0\ltlex x\ltlex x_1\}$ and $D_1=\{x\in D:x_1\ltlex x\ltlex x_2\}$. Fix $e\in\om$.

    Define the following c.e. set, where by ``$E$ is compatible with $F$'' we mean that for all $\sigma,\tau\in F\cup E$, $f_J^G(\sigma,\tau)=i_\ell$ and $[F\cup E]^2\subseteq S_{\iemb}$: \[W=\{\langle x,i\rangle: (\exists E\subseteq_{\mathrm{fin}} D_{1-\iemb})\text{ compatible with $F$})[\Phi_e^{F\cup E\oplus G\oplus Z}(x)\downarrow=i]\}.\]
    We consider the three following cases.
    
    \case{1}{There exists $x\in\om$ such that
      $\langle x, 1-C(x)\rangle\in W$.} Let $E$ be a witness of
      this. The condition $(F\cup E, \hat D_\iemb)$ forces $\Phi_e$ to
      be different from $C$, where $\hat D_\iemb$ is $D_\iemb$ with a
      finite number of elements removed, so that for all $\sigma\in E$ and all $\tau\in \hat D_\iemb$,
      $(G\oplus Z)'\uh|\sigma| = (G\oplus Z)'[|\tau|]\uh|\sigma|$.

    \case{2}{There exists $x\in\om$ such that for each $i\in 2$, $\langle x,i\rangle\not\in W$.} The condition $(F,D_{1-\iemb})$ already forces divergence of $\Phi_e$.

    \case{3}{Otherwise.} Thus, for every $x$ there is an $i$ such that $\langle x,i\rangle\in W$ and $\langle x,i\rangle\in W\implies i=C(x)$. But as $W$ is c.e, this implies that $C$ is computable, a contradiction.
    
    % If there exists $x\in\om$ and $E\subseteq D_{1-\iemb}$ such that $\forall\sigma,\tau\in E$, $f_J^G(\sigma,\tau)=i_\ell$ % \pelliot{add the conditions on $E$} $E\subseteq D_0$
    % and $\Phi_e^{F\cup E\oplus G\oplus Z}(x)\downarrow=1-C(x)$, then the condition $(F\cup E, \hat D_\iemb)$ forces $\Phi_e$ to be different from $C$, where $\hat D_\iemb$ is $D_\iemb$ with a finite number of elements removed, so that for all $\sigma\in E$, $\tau\in \hat D_\iemb$, $(G\oplus Z)'\uh|\sigma| = (G\oplus Z)'[|\tau|]\uh|\sigma|$.

%      Otherwise, if there is no $E\subseteq D_{1-\iemb}$ with $\forall\sigma,\tau\in E$, $f_J^G(\sigma,\tau)=i_\ell$ such that $\Phi_e^{F\cup E}(x)\downarrow$, the condition $(F,D_0)$ forces divergence of $\Phi_e$. Finally, if none of these cases happen, then $(F,\sigma)\oplus G\oplus Z$ computes $C$, a contradiction.

	\medskip
      Finally, any sufficiently generic for this forcing is infinite: indeed, consider the functional $\Gamma$ which halts if and only if its oracle has at least $n$ elements. It is impossible to have a condition forcing $\Gamma$ to halt, so any sufficiently generic has at least $n$ elements, and so for every $n$. %  \pelliot{say why}
%    \pelliot{Ici soit on a deuja gagner, soit on construit le $\{X'_s\uh s:s\in\om\}$ et on a gagner (sic)}
    \end{proof}
    By the previous claims, Case 1 cannot happen, and Cases 2 and 3 validate the theorem with an infinite and cone avoiding set homogeneous for color $i_\infty$. In Case 4, we are also done, as $S$ is a homogeneous set for color $i_\Qb$, of order-type $i_\Qb$.

    In the making of the four cases, we supposed that each $S_{i_{w_\Qb}}$ takes at most two colors. It remains the case when for some $\iemb$, $S_\iemb$ takes only one color by $g$, and $S_{1-\iemb}$ takes three colors by $g$. But then, Case 1 holds for $\iemb$, a contradiction.
\end{proof}

%%% Local Variables:
%%% mode: latex
%%% TeX-master: "../embryon"
%%% End:

\chapter{The Rado graph theorem}\label{sec:radomain}
\section{A big Ramsey structure for the Rado graph}
\index{Rado Graph theorem}
\begin{definition}\index{Joyce!graph}\index{Joyce graph}\index{graph!Joyce}
  A \emph{Joyce graph} is a graph $\mathcal{G} = (G,E)$ together with an order $<$ on $G$ and a symmetric function $\meetlevel\cdot\cdot:G^2\to\Nb$ such that $(G,<,\meetlevel\cdot\cdot)$ is a Joyce order and
  \begin{itemize}
  \item[\Jo4] for all $x,y,z\in G$, $\meetlevel xx < \meetlevel yz \implies (xEy\iff xEz)$.
  \end{itemize}
\end{definition}

As in the previous chapter, the function $\meetlevel\cdot\cdot$ has to be taken as the height of a meet. The axiom \Jo4 states that if two elements have a meet above the height of a third element, they are both linked to it or none are linked to it. In this sense, the axiom states some compatibility between the $\meetlevel\cdot\cdot$ operator and the edge relation. However, compared to the axiom \Jo2 which states a compatibility between the order and the $\meetlevel\cdot\cdot$ operator, the crucial height is the one of the element and not of the meet. In other words, the relevant height to decide whether $x<y$ is at the level of the meet, while the relevant height to decide the edge relation between $x$ and $y$ is at the level of $x$ or $y$.

% \begin{theorem}
%   Let $(G,E,\meetlevel\cdot\cdot)$ be a Rado Joyce graph. Then, the induced Joyce order is a DLO.
% \end{theorem}
% \begin{proof}
%   \pelliot{False!}
% \end{proof}
% \begin{definition}
%   A \emph{dense Rado graph} is a Joyce graph $(G,E,<)$ such that for every $G_0, G_1\subseteq G$ finite disjoint set, and for every $x_0, x_1\in G$, there exists $a_0,a_1,a_2\in G$ such that:
%   \begin{enumerate}
%   \item $a_0<x_0<a_1<x_1<a_2$;
%   \item $\forall i<3$, $\forall b\in F_0$, $a_iEb$ and $\forall b\in F_1$, $\lnot a_iEb$.
%   \end{enumerate}
% \end{definition}
% \pelliot{corriger}
% It is clear from the definition of a dense Joyce Rado Graph that the underlying order is a dense linear order with no endpoints, by taking $F_0, F_1$ to be empty.

%\peter{Some figures of examples and non examples are needed.  I now see that happens in the next figure.  Please move up.  Also for the last 2 sentences above I think you want the intended tree reflect the edge relation.  Please explain. Please explain how to take a finite graph and get a tree. You use that below. Figrure 7 needs to expand to show how to take a finite graph and get a Joyce graph via a Joyce xtree.} \ludovic{I think illustration in Figure 7 is sufficient. I don't want to move it up since we have not yet defined the notion of coded Joyce graph. The notion of coded Joyce graph and in particular the proof that every countable Joyce graph is isomorphic to a coded Joyce graph already explains how you switch from a graph to a tree. } 

\begin{definition}\index{Joyce!Rado graph}\index{Rado graph!Joyce}\index{graph!Joyce Rado}
  A \emph{Joyce Rado graph} is a Joyce graph $(G,E,<, \meetlevel\cdot\cdot)$ such that $(G,E)$ is a Rado graph.
\end{definition}

\index{$\Epn$}In what follows, define the relation $\Epn$ on strings of different length by $\sigma\Epn\tau$ if and only if $\sigma(|\tau|)=1$ and $|\tau|<|\sigma|$, or $\tau(|\sigma|)=1$ and $|\sigma|<|\tau|$.

\begin{theorem}[$\RCA_0$]\label{thm:joyce-rado-graph-exists}
  There exists a Joyce Rado graph.
\end{theorem}
\begin{proof}
  %Let $X=(000+101)^{*}01$, that is, the set of strings $\sigma \in 2^{<\omega}$ of length $3n+2$ for some $n \in \omega$, such that $\sigma(3n) = 0$, $\sigma(3n+1) = 1$, and for every $j < n$, $\sigma(3j+1) = 0$ and $\sigma(3j) = \sigma(3j+2)$. For example, $00010110101 \in X$. In particular, $X$ is an infinite antichain with respect to the prefix order. Let $\ltlex$ be the lexicographic order restricted to $X$, that is, $\sigma \ltlex \tau$ if $\sigma(|\sigma \meet \tau|) <_\Nb \tau(|\sigma \meet \tau|)$. Let also $\Epn$ defined for $\sigma,\tau\in X$ by $\sigma\Epn\tau$ if and only if $|\sigma|<|\tau|$ and $\tau(|\sigma|)=1$, or $|\tau|<|\sigma|$ and $\sigma(|\tau|)=1$.
  Consider $g:\cantor\to\cantor$ to be the function such that $g(\sigma)=\tau$, where $|\tau|=3|\sigma|+2$, for all $n<|\sigma|$, $\tau(3n)=\tau(3n+1)=\tau(3n+2)=\sigma(n)$, and $\tau(3|\sigma|)=0$, $\tau(3|\sigma|+1)=1$. The image of $g$ is a antichain. %In other words, $g'$ is $\ltlex$-preserving and $\meet$-preserving and maps $\cantor$ to $G$
  Fix an injective function $v: 2^{<\omega} \to \omega$ such that 
for every $\sigma, \tau \in 2^{<\omega}$, if $|\sigma| < |\tau|$ then $v(\sigma) < v(\tau)$,
and for every $\sigma, \tau \in \cantor$, define $\meetlevel{\sigma}{\tau} = v(\sigma \meet \tau)$.
Last, fix a cofinal set $S\subseteq\cantor$ such that for all $\sigma,\tau\in S$, $|\sigma|\neq|\tau|$. The claim is that $(g[S], \Epn, \ltlex, \meetlevel{\cdot}{\cdot})$ is a Joyce Rado graph.

We prove that $(g[S], \Epn, \ltlex, \meetlevel{\cdot}{\cdot})$ satisfies axioms \Jo{1}, \Jo{2}, \Jo3 and \Jo{4}.
Let $x, y, z, t \in g[S]$, not all equal, with $x \lelex y$ and $z \lelex t$.

\Jo{1}: Suppose $\meetlevel{x}{y} < \meetlevel{x}{z}$. By definition, $v(x \meet y) < v(x \meet z)$. By choice of the map $v$,  $|x \meet y| \leq |x \meet z|$, so $x \ltlex y$ iff $z \ltlex y$.

\Jo{2}: Suppose $\meetlevel{x}{y} < \meetlevel{x}{z}$. By definition, $v(x \meet y) < v(x \meet z)$. By choice of the map $v$,  $|x \meet y| \leq |x \meet z|$, so $x \meet y = z \meet y$, hence $v(x \meet y) = v(z \meet y)$.

\Jo{3}: Suppose $\meetlevel{x}{y} = \meetlevel{z}{t}$. By definition, $v(x \meet y) = v(z \meet t)$. By injectivity of the map $v$, $x \meet y = z \meet t$, so $x \meet y \prec x \meet z$ and $x \meet y \prec y \meet t$, hence $v(x \meet y) <_\Nb \min(v(x \meet z), v(y \meet t))$.

\Jo{4}: Let $x,y,z\in g[S]$. Suppose $\meetlevel{x}{x} < \meetlevel{y}{z}$. By definition, $v(x)=v(x \meet x) < v(y \meet z)$. By choice of the map $v$,  $|x| \leq |y \meet z|$, but as the length of elements of $g[S]$ and the length of proper meets of $g[S]$ are different by construction, $|x| < |y \meet z|$. Therefore $y(|x|)=z(|x|) $ so $x \Epn y$ iff $z \Epn y$.

It remains to show that $(g[S], \Epn)$ is a Rado graph. Let $F_0,F_1\subseteq g[S]$ be finite disjoint sets. As $S$ contains at most one element of each length, and is cofinal, let $\sigma$ be a string in $S$ such that $\sigma(\ell)=i$ whenever there exists $\tau\in S$ of size $\ell$ with $g(\tau)\in F_i$. By definition of $g$, $\sigma\Epn\tau$ iff $g(\sigma)\Epn g(\tau)$, so $g(\sigma)$ is linked with all of $F_1$ and none of $F_0$. Therefore, $(g[S], \Epn)$ is a Rado graph.

\end{proof}

%\peter{Where is $E_{pn}$ defined? (now I see after 6.7)}\ludovic{Fixed. I added the definition right before Theorem 6.3.}\peter{Why are all the rules satisfied? Please add.}\pelliot{I added why the rules are satisfied}

\begin{corollary}[$\RCA_0$]\label{cor:rado-can-be-enriched}
Every Rado graph $\mathcal{G} = (G,E)$ can be ordered and equipped with a function $\meetlevel{\cdot}{\cdot}: G^2\to \Nb$ to form a Joyce Rado graph.
\end{corollary}
\begin{proof}
Let $(X, \Epn, \ltlex, \meetlevel{\cdot}{\cdot}_X)$ be the Joyce Rado graph of \Cref{thm:joyce-rado-graph-exists}. By computable categoricity of the Rado graph, there
exists a graph isomorphism $f$ between $\mathcal{G} = (G, E)$ and $(X, \Epn)$.
Define $x < y$ for $x, y \in G$ if and only if $f(x) \ltlex f(y)$.
Also define $\meetlevel{\cdot}{\cdot}: G^2 \to \Nb$ by $\meetlevel{x}{y} = \meetlevel{f(x)}{f(y)}_X$.
Then $(G, E, <, \meetlevel{\cdot}{\cdot})$ is a Joyce Rado graph.
\end{proof}

The first-order structure that is of interest for us is the following.
\begin{definition}\index{structure!Joyce graph}\index{Joyce graph!structure}\index{big Ramsey structure!Rado graph}
  The \emph{Joyce (Rado) graph structure} of a Joyce (Rado) graph $(G,E,<, \meetlevel\cdot\cdot)$ is the structure $(G,E,<,\JRel)$ such that $(G,<,\JRel)$ is the Joyce structure of the
  %induced
  Joyce order $(G,<, \meetlevel\cdot\cdot)$. % A \emph{dense Joyce Rado graph structure} is the Joyce graph structure of a Joyce Rado Graph.
  % A \emph{DLO Joyce Rado graph structure} is the Joyce graph structure of a Rado Graph whose underlying Joyce order is a DLO. A \emph{strong Joyce Rado graph structure} is the Joyce graph structure of a Rado Graph whose underlying Joyce order is a DLO.
\end{definition}
We shall prove later that Joyce Rado graphs structures have big Ramsey degree 1 for every finite Joyce graph structure.

\begin{statement}\index{statement!$\JRG^n_{k,\ell}$}
	For all $n,k,\ell \geq 1$, $\JRG^n_{k,\ell}$ is the assertion that for every Joyce Rado graph structure $\mathcal{G}$ and every coloring $f:[\mathcal{G}]^n\to k$, there exists an isomorphic substructure $\mathcal{G}'$ of $\mathcal{G}$ satisfying $|f[\mathcal{G}']^n|\leq \ell$.
\end{statement}

As every Joyce graph is in particular a Joyce order, every finite Joyce graph of size $n$ can be fully specified by a finite Joyce order and a finite graph, both of size $n$, or equivalently by a finite Joyce tree with $n$ leaves and a finite graph of size $n$. In particular, for a fixed graph $G$ of size $n$, there are at most as many Joyce graphs isomorphic to it as there are Joyce trees with $n$ leaves. On the other hand, as we shall see in \Cref{fig:graph-representation}, if a finite graph $G$ of size $n$ is neither the clique, nor the anti-clique with $n$ vertices, there are some Joyce orders of size $n$ which cannot be enriched to form a Joyce graph isomorphic to $G$.
%\peter{What the relation between Joyce trees and finite Joyce graphs?}\ludovic{I added some explanations.}

\begin{theorem}[Joyce Rado graph theorem]
	For all $n,k \geq 1$, $\JRG^n_{k,J_n}$ holds, where $J_n$ is the number of non isomorphic Joyce graphs with $n$ elements. Moreover, this bound is tight: $\JRG^n_{k,\ell}$ does not hold for any $\ell < J_n$.
\end{theorem}

%\textit{Representing a Joyce graph as a set of strings.}
Just as we did for the Joyce order, we can canonically represent any countable Joyce graph as a set of strings equipped with the lexicographic order and  $|{\cdot}\meet{\cdot}|$, but also the relation $\Epn$.

\begin{definition}\index{Joyce graph!coded}\index{coded!Joyce graph}
%  The relation $\Epn\subseteq\{(\sigma,\tau):\sigma,\tau\in\cantor\land|\sigma|\neq|\tau|\}$ is defined by $\sigma\Epn\tau$ if and only if $\sigma(|tau|)=1$ and $|\tau|<|\sigma|$, or $\tau(|\sigma|)=1$ and $|\sigma|<|\tau|$.
  A \emph{coded Joyce graph} is a Joyce graph of the form
  \[
  	(X,\Epn,\ltlex,|\cdot\meet\cdot|)
  \]
  such that for all $\sigma,\tau,\rho\in X$ with $\sigma\neq\tau$, $|\rho|>|\sigma\meet\tau|$ and $\sigma \meet \tau \not \preceq \rho$, we have $\rho(|\sigma\meet\tau|)=0$.
\end{definition}
Note that if $(X,\Epn,\ltlex,|\cdot\meet\cdot|)$ is a coded Joyce graph, that does not mean that $(X,\ltlex,|\cdot\meet\cdot|)$ is a coded Joyce order. Indeed, there is no restriction on $\rho(|\sigma\meet\sigma|)$ in the case of a coded Joyce graph, while this value must be 0 in the case of a coded Joyce order. The two notions thus coincides  if and only if $\lnot \sigma\Epn\tau$ for every $\sigma,\tau\in X$, by definition of $\Epn$. %, as in a coded Joyce order, we have that $\rho(|\sigma\meet\tau|)=0$ even for $\sigma=\tau$, that is $\rho(|\sigma|)=0$ for any $\sigma,\rho\in X$ with $|\sigma|<|\rho|$, so $\lnot \sigma\Epn\rho$.

\begin{figure}[htbp]
  \begin{center}
    \begin{tikzpicture}[scale=1.5,font=\normalsize]
		\tikzset{
			empty node/.style={circle,inner sep=0,fill=none},
			solid node/.style={circle,draw,inner sep=1.5,fill=black},
			hollow node/.style={circle,draw,inner sep=1.5,fill=white},
			gray node/.style={circle,draw={rgb:black,1;white,4},inner sep=1.5,fill={rgb:black,1;white,4}}
		}
		\tikzset{snake it/.style={decorate, decoration=snake, line cap=round}}
		\tikzset{gray line/.style={line cap=round,thick,color={rgb:black,1;white,4}}}
		\tikzset{thick line/.style={line cap=round,rounded corners=0.1mm,thick}}
		\tikzset{thin line/.style={line cap=round,rounded corners=0.1mm}}
		\node (a)[solid node,label=below:{$x_0$}] at (-0.5,0.5) {};
		\node (b)[solid node,label=below:{$x_1$}] at (0,0.5) {};
		\node (c)[solid node,label=below:{$x_2$}] at (0.5,0.5) {};
		\node (label)[empty node,label=below:{(a)}] at (0,-0.15) {};
		\draw[thick line] (a) to[out=50,in=130] (b);
	\end{tikzpicture}
	\hspace{5mm}
	\begin{tikzpicture}[scale=1.5,font=\normalsize]
		\tikzset{
			empty node/.style={circle,inner sep=0,fill=none},
			solid node/.style={circle,draw,inner sep=1.5,fill=black},
			hollow node/.style={circle,draw,inner sep=1.5,fill=white},
			gray node/.style={circle,draw={rgb:black,1;white,4},inner sep=1.5,fill={rgb:black,1;white,4}}
		}
		\tikzset{snake it/.style={decorate, decoration=snake, line cap=round}}
		\tikzset{gray line/.style={line cap=round,thick,color={rgb:black,1;white,4}}}
		\tikzset{thick line/.style={line cap=round,rounded corners=0.1mm,thick}}
		\tikzset{thin line/.style={line cap=round,rounded corners=0.1mm}}
		\node (root)[solid node] at (0,0) {};
		\node (a)[solid node] at (-0.5,0.5) {};
		\node (b)[solid node,label=right:{$0$}] at (0.5,0.5) {};
		\node (c)[solid node,label=left:{$0$}] at (-0.75,1) {};
		\node (d)[solid node,label=right:{$0$}] at (-0.25,1) {};
		\node (e)[solid node,label=above:{$x_2$}] at (0.25,1) {};
		\node (f)[solid node,label=above:{$x_0$}] at (-0.9,1.5) {};
		\node (g)[solid node,label=right:{$1$}] at (-0.4,1.5) {};
		\node (h)[solid node,label=above:{$x_1$}] at (-0.25,2) {};
		\draw[-,thin line,dash pattern={on 3pt off 3pt}] (-1,0.5) to (1,0.5);
		\draw[-,thin line,dash pattern={on 3pt off 3pt}] (-1,1) to (1,1);
		\draw[-,thin line,dash pattern={on 3pt off 3pt}] (-1,1.5) to (1,1.5);
		\draw[thick line] (root) to (a);
		\draw[thick line] (root) to (b);
		\draw[thick line] (a) to (c);
		\draw[thick line] (a) to (d);
		\draw[thick line] (b) to (e);
		\draw[thick line] (c) to (f);
		\draw[thick line] (d) to (g);
		\draw[thick line] (g) to (h);
		\node (label)[empty node,label=below:{(b)}] at (0,-0.15) {};
	\end{tikzpicture}
	\hspace{5mm}
	\begin{tikzpicture}[scale=1.5,font=\normalsize]
		\tikzset{
			empty node/.style={circle,inner sep=0,fill=none},
			solid node/.style={circle,draw,inner sep=1.5,fill=black},
			hollow node/.style={circle,draw,inner sep=1.5,fill=white},
			gray node/.style={circle,draw={rgb:black,1;white,4},inner sep=1.5,fill={rgb:black,1;white,4}}
		}
		\tikzset{snake it/.style={decorate, decoration=snake, line cap=round}}
		\tikzset{gray line/.style={line cap=round,thick,color={rgb:black,1;white,4}}}
		\tikzset{thick line/.style={line cap=round,rounded corners=0.1mm,thick}}
		\tikzset{thin line/.style={line cap=round,rounded corners=0.1mm}}
		\node (root)[solid node] at (0,0) {};
		\node (a)[solid node] at (-0.5,0.5) {};
		\node (b)[solid node,label=right:{$0$}] at (0.5,0.5) {};
		\node (c)[solid node,label=left:{$1$}] at (-0.75,1) {};
		\node (d)[solid node,label=above:{$x_1$}] at (-0.25,1) {};
		\node (e)[solid node,label=right:{$0$}] at (0.25,1) {};
		\node (f)[solid node,label=above:{$x_0$}] at (-0.6,1.5) {};
		\node (g)[solid node,label=right:{$0$}] at (0.1,1.5) {};
		\node (h)[solid node,label=above:{$x_2$}] at (0,2) {};
		\draw[-,thin line,dash pattern={on 3pt off 3pt}] (-1,0.5) to (1,0.5);
		\draw[-,thin line,dash pattern={on 3pt off 3pt}] (-1,1) to (1,1);
		\draw[-,thin line,dash pattern={on 3pt off 3pt}] (-1,1.5) to (1,1.5);
		\draw[thick line] (root) to (a);
		\draw[thick line] (root) to (b);
		\draw[thick line] (a) to (c);
		\draw[thick line] (a) to (d);
		\draw[thick line] (b) to (e);
		\draw[thick line] (c) to (f);
		\draw[thick line] (e) to (g) to (h);
		\node (label)[empty node,label=below:{(c)}] at (0,-0.15) {};
	\end{tikzpicture}
	\\
	\smallskip
	\begin{tikzpicture}[scale=1.5,font=\normalsize]
		\tikzset{
			empty node/.style={circle,inner sep=0,fill=none},
			solid node/.style={circle,draw,inner sep=1.5,fill=black},
			hollow node/.style={circle,draw,inner sep=1.5,fill=white},
			gray node/.style={circle,draw={rgb:black,1;white,4},inner sep=1.5,fill={rgb:black,1;white,4}}
		}
		\tikzset{snake it/.style={decorate, decoration=snake, line cap=round}}
		\tikzset{gray line/.style={line cap=round,thick,color={rgb:black,1;white,4}}}
		\tikzset{thick line/.style={line cap=round,rounded corners=0.1mm,thick}}
		\node (a)[empty node,label=below:{$1$}] at (0,0) {};
		\node (b)[empty node,label=above:{$2$}] at (-0.4,0.5) {};
		\node (c)[empty node,label=below:{$3$}] at (0.4,0.5) {};
		\node (d)[empty node,label=above:{$5$}] at (0,1) {};
		\node (e)[empty node,label=above:{$4$}] at (0.8,1) {};
		\draw[thick line] (a.center) to (b.center);
		\draw[thick line] (c.center) to (e.center);
		\draw[thick line] (a.center) to (c.center) to (d.center);
		\node (label)[empty node,label=below:{(d)}] at (0,-0.4) {};
	\end{tikzpicture}
	\hspace{5mm}
	\hspace{5mm}
	\begin{tikzpicture}[scale=1.5,font=\normalsize]
		\tikzset{
			empty node/.style={circle,inner sep=0,fill=none},
			solid node/.style={circle,draw,inner sep=1.5,fill=black},
			hollow node/.style={circle,draw,inner sep=1.5,fill=white},
			gray node/.style={circle,draw={rgb:black,1;white,4},inner sep=1.5,fill={rgb:black,1;white,4}}
		}
		\tikzset{snake it/.style={decorate, decoration=snake, line cap=round}}
		\tikzset{gray line/.style={line cap=round,thick,color={rgb:black,1;white,4}}}
		\tikzset{thick line/.style={line cap=round,rounded corners=0.1mm,thick}}
		\node (a)[empty node,label=below:{$1$}] at (0,0) {};
		\node (b)[empty node,label=below:{$2$}] at (-0.4,0.5) {};
		\node (c)[empty node,label=above:{$3$}] at (0.4,0.5) {};
		\node (d)[empty node,label=above:{$4$}] at (-0.8,1) {};
		\node (e)[empty node,label=above:{$5$}] at (0,1) {};
		\draw[thick line] (a.center) to (b.center) to (d.center);
		\draw[thick line] (b.center) to (e.center);
		\draw[thick line] (a.center) to (c.center);
		\node (label)[empty node,label=below:{(e)}] at (0,-0.4) {};
	\end{tikzpicture}
	\hspace{5mm}
	\hspace{5mm}
	\begin{tikzpicture}[scale=1.5,font=\normalsize]
		\tikzset{
			empty node/.style={circle,inner sep=0,fill=none},
			solid node/.style={circle,draw,inner sep=1.5,fill=black},
			hollow node/.style={circle,draw,inner sep=1.5,fill=white},
			gray node/.style={circle,draw={rgb:black,1;white,4},inner sep=1.5,fill={rgb:black,1;white,4}}
		}
		\tikzset{snake it/.style={decorate, decoration=snake, line cap=round}}
		\tikzset{gray line/.style={line cap=round,thick,color={rgb:black,1;white,4}}}
		\tikzset{thick line/.style={line cap=round,rounded corners=0.1mm,thick}}
		\node (a)[empty node,label=below:{$1$}] at (0,0) {};
		\node (b)[empty node,label=below:{$2$}] at (-0.4,0.5) {};
		\node (c)[empty node,label=above:{$5$}] at (0.4,0.5) {};
		\node (d)[empty node,label=above:{$4$}] at (-0.8,1) {};
		\node (e)[empty node,label=above:{$3$}] at (0,1) {};
		\draw[thick line] (a.center) to (b.center) to (d.center);
		\draw[thick line] (b.center) to (e.center);
		\draw[thick line] (a.center) to (c.center);
		\node (label)[empty node,label=below:{(f)}] at (0,-0.4) {};
	\end{tikzpicture}
\end{center}
\caption{In (a), a finite graph $G = (\{x_0, x_1, x_1\}, \{ \{x_0, x_1\}\})$. In (b) and (c), two coded Joyce graphs isomorphic to $G$. The trees (e) and (f) are Joyce trees corresponding the coded Joyce graphs (b) and (c), respectively. In (d), a Joyce tree which cannot represent the graph $G$. Indeed, since there is an edge between $x_0$ and $x_1$ but not between $x_0$ and $x_1$, then for any coded Joyce graph $\{\sigma_0, \sigma_1, \sigma_2\}$ representing $G$, $|\sigma_1 \meet \sigma_2| < |\sigma_0|$.}
\label{fig:graph-representation}
\end{figure}
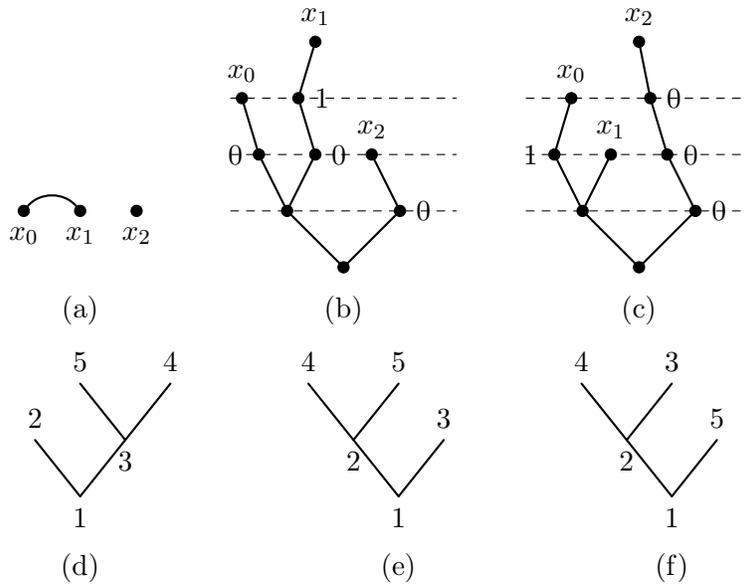

\begin{theorem}[$\RCA_0$]\label{thm:joyce-graph-to-coded}
  Every countable Joyce graph is isomorphic to a coded Joyce graph.
\end{theorem}
\begin{proof}
  Let $(G, E, <, \meetlevel{\cdot}{\cdot})$ be a countable Joyce graph.
%  Let $L$ be the set of labels over $G^2$. Let $L_0=\{\meetlevel xx:x\in G\}$ and $L_1=\{\meetlevel xy:x,y\in G\land x\neq y\}$. %For every $x \in G$, let $L_x$ be the set of labels $\ell \in L$ such that $\ell <_\Nb \meetlevel{x}{x}$ and such that there is some $y \in G$ such that $y < x$ and $\meetlevel{y}{x} = \ell$.
  Let $\sigma_x \in 2^{<\omega}$ be the unique string of length $\meetlevel{x}{x}$, such that for any $j < \meetlevel{x}{x}$:
  \begin{enumerate}
  \item\label{it:jctc1} if $j=\meetlevel xy$ for some $y\in G$, then $\sigma_x(j) = 1$ if and only if $y<x$;
  \item\label{it:jctc2} if $j=\meetlevel yy$ for some $y\in G$, then $\sigma_x(j) = 1$ if and only if $xEy$;
  \item\label{it:jctc3} $\sigma_x(j)=0$ otherwise.
  \end{enumerate}
  One first need to show that $\sigma_x$ is well-defined. First, there is no $y,z\in G$ such that $\meetlevel yz = \meetlevel zz$ by \Jo3, so \Cref{it:jctc1} and \Cref{it:jctc2} are compatible. \Cref{it:jctc1} do not contradict itself as there is no $y,z\in G$ such that $\meetlevel yy=\meetlevel zz$, also by \Jo3. It remains to show that \Cref{it:jctc2} does not contradict itself: Let $x,y,z\in G$ be such that $\meetlevel xy = \meetlevel xz<\meetlevel xx$. Then, by \Jo3 $\meetlevel xy<_\Nb\meetlevel yz$ and by \Jo1 we have $x<z$ iff $x<z$. So $\sigma_x$ is well-defined.
  
Let $X = \{ \sigma_x: x \in G \}$.
We claim that $(X, \Epn, \ltlex, |\cdot \meet \cdot|)$ is a Joyce graph whose structure is isomorphic to the Joyce structure of $(G, E, <, \meetlevel{\cdot}{\cdot})$.

Let $\sigma_x, \sigma_y\in X$. We have $|\sigma_x|=\meetlevel xx\neq \meetlevel yy=|\sigma_y|$, we suppose $|\sigma_x|<|\sigma_y|$. But then, $\sigma_x\Epn\sigma_y$ iff $\sigma_y(|\sigma_x|)=1$ iff $xEy$ by \Cref{it:jctc2}.

The rest of the proof follows the same argument that the construction in the proof of \Cref{thm:joyce-orders-can-be-coded} works. 
%\pelliot{Prove the claim}
We prove that for all $x,y,z,t\in G$, $\meetlevel xy<\meetlevel zt\implies |\sigma_x\meet \sigma_y|<_\Nb|\sigma_z\meet \sigma_t|$. We actually prove the stronger fact that for every $x,y\in G$, $\meetlevel xy = |\sigma_x\meet\sigma_y|$. If $x=y$, it is clear as by construction, $\sigma_x$ is of length $\meetlevel xx$. If $x\neq y$, we first prove that $\meetlevel xy \leq |\sigma_x\meet\sigma_y|$: indeed, for all $n<\meetlevel xy$, by \Jo1, \Jo4 and the construction, we have $\sigma_x(n)=1$ iff $\sigma_y(n)=1$. It remains to show $\meetlevel xy \geq |\sigma_x\meet\sigma_y|$: if $x<y$ we have that $\sigma_x(\meetlevel xy)=0\neq 1=\sigma_y(\meetlevel xy)$ by \Cref{it:jctc1}, so $\meetlevel xy \geq |\sigma_x\meet\sigma_y|$, and similarly for $x>y$.

Let $x<y\in G$. Then, $\meetlevel xy=|\sigma_x\meet\sigma_y|$, thus $\sigma_x(|\sigma_x\meet\sigma_y|)=0\neq 1=\sigma_y(|\sigma_x\meet\sigma_y|)$ by \Cref{it:jctc2}. It follows that $\sigma_x\ltlex\sigma_y$.%as if $z$ is such that $\meetlevel xz=\meetlevel xy$, then by \Jo3 $\meetlevel xy<_\Nb\meetlevel yz$ and by \Jo1 and the fact that $x<y$, we have $x<z$. So $\sigma_x(|x\meet y|)=\sigma_x(\meetlevel xy)=0$. However, $\meetlevel xy\not\in L_y$ as witnessed by $x$, so $\sigma_y(|x\meet y|)=\sigma_y(\meetlevel xy)=1$. Therefore, $\sigma_x\ltlex\sigma_y$.

\end{proof}
\begin{corollary}
  There exists a computably coded Joyce Rado graph.
\end{corollary}
\begin{proof}
  Immediate by \Cref{thm:joyce-graph-to-coded} and \Cref{thm:joyce-rado-graph-exists}.
\end{proof}

\section{Joyce blossom graphs and an embedding theorem}

\begin{definition}\label{def:blossom-tree}\index{blossom!tree}\index{tree!blossom}\index{blossom!Joyce graph}\index{Joyce graph!blossom}
  A \emph{blossom tree} is a pair $(f, g)$
  where $f: \cantor \to \cantor$ is a $\prec$-preserving, $\ltlex$-preserving and $\meet$-preserving function, such that for every $\sigma, \tau \in \cantor$:
  \begin{enumerate}
  \item $g(\sigma) \succ f(\sigma)$; % $f(\sigma 0) \ltlex g(\sigma) \ltlex f(\sigma 1)$;
  \item if $|\tau| > |\sigma|$ then $|f(\tau)| > |g(\sigma)|$;
    % \item the function $\{\sigma, \tau\} \mapsto |g(\sigma) \meet g(\tau)|$ is injective.
  \item\label{it:blossom-rado} if $|\tau| = |\sigma|$ then
   $f(\tau 0)(|g(\sigma)|) \neq f(\tau 1)(|g(\sigma)|)$.
  \end{enumerate}
  % A \emph{blossom Joyce order} is a Joyce order of the form $(g[\cantor], \ltlex, |\cdot \meet \cdot|)$ for some blossom tree $(f, g)$.
  A \emph{Joyce blossom graph} is a structure $(g[S],\Epn,\ltlex,|\cdot\meet\cdot|)$ for some blossom tree $(f,g)$ and some set $S$ cofinal in $\cantor$ such that
   for all $\sigma,\tau\in\meetclosure S$, $|\sigma|\neq|\tau|$.
%  A \emph{Joyce blossom graph} is a Joyce blossom graph such that $(G,\Epn,\ltlex,|\cdot\meet\cdot|)$ such that there exists $(f,g)$ and $S$ such that:
\end{definition}
Note that a Joyce blossom graph $\mathcal{G}$ is a Joyce Rado graph: indeed, let $F_0$ and $F_1$ be two disjoint finite sets of vertices of $\mathcal{G}$, and let $f,g$ and $S$ be the witnesses of the fact that $\mathcal{G}$ is a Joyce blossom graph. By \Cref{it:blossom-rado} of the definition, there exists $\sigma\in f[\cantor]$ with $|\sigma|>\max\{|\tau|:\tau\in F_0\cup F_1\}$ and such that for every $\tau\in F_0\cup F_1$ $\sigma(|\tau|)=0$ iff $\tau\in F_0$. By the fact that $S$ is cofinal, let $\rho\succ f^{-1}(\sigma)$ in $S$. Then, $g(\rho)\succ \sigma$, and therefore $g(\rho)\Epn\tau$ if $\tau\in F_1$ and $\lnot g(\rho)\Epn\tau$ if $\tau\in F_0$. The Joyce requirements are satisfied by the fact that the relations are $\Epn$, $\ltlex$ and $|\cdot\meet\cdot|$, and that for all $\sigma,\tau\in\meetclosure S$, $|\sigma|\neq|\tau|$.

%Moreover, since for a blossom embedding $(f,g)$, 

%More than that, a Joyce blossom graph is a Joyce graph: By the properties of $f$ and $g$, and the fact that $\meetclosure S$ has at most one element at each level, so is for $g[S]$. The rest is ensured

Note that if $(f, g)$ is a blossom tree, then letting $B = g[\cantor]$ it follows by Item 1 that 
$(B, \preceq)$ is an antichain and $(B, \ltlex)$ contains a dense linear order with no endpoints.
From a computability-theoretic viewpoint, $\prec$-preservation of $f$ and Item 1 of  \Cref{def:blossom-tree}  ensures that if $(f, g)$ is computable, then so is $g[\cantor]$. Conversely, Item 2 of \Cref{def:blossom-tree} implies that $(f, g)$ is computable from $g[\cantor]$. One can therefore switch from one notion to the other in the computability realm.

\begin{comment}
\begin{definition}\label{def:blossom-rado-tree}
A \emph{Joyce blossom graph} is a graph $(g[S],\Epn)$ where $S$ is a cofinal set with $\forall\sigma,\tau\in\meetclosure S$, $\sigma|\neq|\tau|$ and $g$ is part of a blossom tree $(f, g)$
such that 
$$
\forall \sigma, \rho \in 2^\ell \mbox{, } f(\rho 0)(|g(\sigma)|) \neq f(\rho 1)(|g(\sigma)|)
$$
%A \emph{Joyce blossom graph} is a Joyce blossom graph of the form $(g[\cantor], \Epn, \ltlex, |\cdot \meet \cdot|)$ for some blossom Rado tree $(f, g)$.
\end{definition}
\pelliot{Make figures for these definitions}

% Note that if $(f, g)$ is a blossom Rado tree, then letting $B = g[\cantor]$, 
% $(B, \Epn)$ is a Rado graph, hence every Joyce blossom graph is a Joyce Rado graph.
%Note that if $(f, g)$ is a blossom tree, and $S\subseteq\cantor$ is a cofinal set with $\forall \sigma,\tau\in S$, $|\sigma|\neq|\tau|$, then $(g[S], \Epn)$ is a Rado graph. If $(f,g)$ is a Joyce blossom graph, then $(g[S],\Epn)$ is a Joyce Rado graph.%, hence every Joyce blossom graph is a Joyce Rado graph.

\end{comment}
\begin{lemma}\label{le:computable-bjrg}
  There exists a computable Joyce blossom graph, with a DLO induced order.
\end{lemma}
\begin{proof}
%  \pelliot{check with new definitions}
  Let $g$ and $S$ be the objects defined in the proof of \Cref{thm:joyce-rado-graph-exists}, and $\mathcal{G} =(g[S],\Epn,\ltlex,\meetlevel\cdot\cdot)$ be the Joyce graph defined in the same proof. By \Cref{thm:joyce-graph-to-coded}, let $\mathcal{G}'=(G',\Epn,\ltlex, |\cdot\meet\cdot|)$ be the coded Joyce Rado Graph computably isomorphic to $\mathcal{G}$ via $e$.%$(g[\cantor],\Epn,\ltlex,\meetlevel\cdot\cdot)$ via $e$.
%  Let $g:\cantor\to\cantor$ be the computable function defined in the proof of \Cref{thm:joyce-rado-graph-exists}. By \Cref{thm:joyce-graph-to-coded}, let $\mathcal{G}'=(G',\Epn,\ltlex, |\cdot\meet\cdot|)$ be the coded Joyce Rado Graph computably isomorphic to $(g[\cantor],\Epn,\ltlex,\meetlevel\cdot\cdot)$ via $e$.

  It remains to show that $\mathcal{G}'$ is a Joyce blossom graph. % Consider $g':\cantor\to\cantor$ to be the function such that $g(\sigma)=\tau$, where $|\tau|=3|\sigma|+2$, $\forall n<|\sigma|$, $\tau(3n)=\tau(3n+2)=\sigma(n)$ and $\tau(3n+1)=0$; and $\tau(3|\sigma|)=0$, $\tau(3|\sigma|+1)=1$. In other words, $g'$ is $\ltlex$-preserving and $\meet$-preserving and maps $\cantor$ to $G$.
  Define $g'=e\circ g$ and $f:\sigma\mapsto g'(\tau0)\meet g'(\tau1)$. It is easy to check that  $f$, $g'$ and the set $S$ are witnesses of the fact that $\mathcal{G}'$ is a Joyce blossom graph.%, via the functions. %$f,g:\cantor\to\cantor$ such that if $|\sigma|=n$, then $|f(\sigma)|=3n$, $f(\sigma)(3n)=f(\sigma)(3n+2)=\sigma(n)$ and $f(\sigma)(3n+1)=0$; and $g(\sigma)=f(\sigma)01$.
%\pelliot{Make a figure and refer to the figure.}  
\end{proof}

\begin{lemma}[$\RCA_0$]\label{thm:graph-embedding-cantor-exists}
For every Rado graph $(G, E)$,
there exists a graph embedding $e: (2^{<\omega}, \Epn) \to (G, E)$.
\end{lemma}
\begin{proof}
  Let $(\sigma_n)_{n\in\om}$ be an enumeration of $\cantor$ such that $|\sigma_n|>|\sigma_m|$ implies $n>m$. Suppose $e(\sigma_i)$ is defined for every $i<n$. Let $F_0=\{e(\tau):\tau\in\cantor, |\tau|<|\sigma_n|, \sigma_n(|\tau|)=0\}$ and $F_1=\{e(\tau):\tau\in\cantor, |\tau|<|\sigma_n|, \sigma_n(|\tau|)=1\}$. By the fact that $(G,E)$ is a Rado graph, there exists an $g\in G$ such that for all $a\in F_0$, $\lnot aEg$ and for all $a\in F_1$, $aEg$, and $g$ is not already in the image of $e$. Define $e(\sigma_n)=g$.
%\ludovic{TODO, Todorcevic}
\end{proof}

\begin{theorem}[$\ACA_0$]\label{thm:blossom-to-joyce-rado}
For every coded Joyce Rado graph $\mathcal{G}=(G,E,<,\meetlevel\cdot\cdot)$,
there is an embedding from a Joyce blossom graph to $G$.% \pelliot{why not: there is a blossom Rado subgraph?}
\end{theorem}
\begin{proof}
%  \pelliot{Adapt with new notation. One possibility is: theorem is now ``for every coded... there is an embedding from $(g[\cantor],...)$ for a blossom tree $(f,g)$''. Corollary: ``there is an embedding from a Joyce blossom graph'': same embedding works.}
We show the stronger result that there exists a blossom tree $(f,g)$ such that $(g[\cantor],\Epn,\ltlex,|\cdot\meet\cdot|)$ embeds into $\mathcal{G}$. Thus, for any $S\subseteq\cantor$ such that for all $\sigma,\tau\in \meetclosure S$ we have $|\sigma|\neq|\tau|$ whenever $\sigma\neq\tau$, we have $(g[S],\Epn,\ltlex,|\cdot\meet\cdot|)$ embeds into $\mathcal{G}$. 
  
Let $e: (\cantor, \Epn) \to (G, \Epn)$ be the graph embedding constructed in \Cref{thm:graph-embedding-cantor-exists}.
We say that $L \subseteq \cantor$ is \emph{large above $\tau \in \cantor$}
if $e^{-1}[L]$ is cofinal in $2^{<\omega}$ above $\tau$. The set $L$ is large if it is large above some $\tau \in \cantor$. Note that the collection of large sets is partition regular, that is, if $L_0 \cup \dots \cup L_{d-1}$ is large, then there is some $j < d$ such that $L_j$ is large. Moreover, if $L$ is large above $\tau$, then it is large above any $\rho \succeq \tau$. Last, $G$ is large above $\epsilon$.
Given a $\sigma \in \cantor$, we write $G \uh \sigma = \{ \tau \in G: \tau \succeq \sigma \}$.
Note that if $G \uh \sigma$ is large above $\tau$, then so is $G \uh \rho$ for every $\rho \preceq \sigma$.

The following claim is the combinatorial core of the theorem.

\begin{claim}\label{claim:blossom-to-joyce-rado-combi}
If $G \uh \sigma$ is large, then for cofinitely many $\rho \in \cantor$, there are some $\sigma_0, \sigma_1 \in G$ such that
$G \uh \sigma_0$ and $G \uh \sigma_1$ are large, $\sigma \preceq \sigma_0 \meet \sigma_1$ and $\sigma_0(\ell) \neq \sigma_1(\ell)$, where $\ell = |e(\rho)|$.
\end{claim}
\begin{proof}
Say $G \uh \sigma$ is large above some $\tau$. Pick any $\rho \in \cantor$ such that $|\rho| \geq |\tau|$. Since if $G \uh \sigma$ is large above $\tau$, $G \uh \sigma$ is large above any $\mu \succeq \tau$, then we can assume that $|\tau| = |\rho|$. Unfolding the definition of largeness,
the set $C = e^{-1}[G \uh \sigma]$ is cofinal above $\tau$. Since $e: (\cantor, \Epn) \to (G, \Epn)$ is a graph embedding and $|\tau| = |\rho|$,
then for every $\mu \succeq \tau 0$, $\neg(\mu \Epn \rho)$, hence $\neg(e(\mu) \Epn e(\rho))$ and for every $\mu \succeq \tau 1$, $\mu \Epn \rho$, hence $e(\mu) \Epn e(\rho)$. It follows that, letting $L_0 = \{\nu \in G \uh \sigma: \neg(\nu \Epn e(\rho))\}$ and $L_1 = \{\nu \in G \uh \sigma: \nu \Epn e(\rho)\}$, then  $C_0 = e^{-1}[L_0] = \{ \mu \succeq \tau 0: \mu \in C \}$ and $C_1 = e^{-1}[L_1] = \{ \mu \succeq \tau 1: \mu \in C \}$. Since $C$ is cofinal above $\tau$, then $C_0$ and $C_1$ are cofinal above $\tau 0$ and $\tau 1$, respectively, so $L_0$ and $L_1$ are large above $\tau 0$ and $\tau 1$, respectively.

Let $\ell = |e(\rho)|$. Note that since $L_0$ and $L_1$ are both non-empty, then $\ell \geq |\sigma|$.  $L_0 = \bigcup_{\nu \succeq \sigma: |\nu| = \ell} G \uh \nu 0$
and $L_1 = \bigcup_{\nu \succeq \sigma: |\nu| = \ell} G \uh \nu 1$.
By partition regularity of largeness, there are some $\nu_0, \nu_1 \succeq \sigma$
such that $|\nu_0| = |\nu_1| = \ell$, and $G \uh \nu_0 0$  and $G \uh \nu_1 1$ are large. Let $\sigma_0 = \nu_0 0$ and $\sigma_1 = \nu_1 1$. This proves our claim.
\end{proof}

We are now ready to prove \Cref{thm:blossom-to-joyce-rado}.
Using Claim \ref{claim:blossom-to-joyce-rado-combi}, we build a $\prec$-preserving and $\ltlex$-preserving function $\phi: \cantor \to \cantor$, together with a function $g: \cantor \to G$ such that for every $\sigma \in \cantor$:
\begin{enumerate}
	\item $G \uh \phi(\sigma)$ is large; $g(\sigma) \succeq \phi(\sigma 0) \meet \phi(\sigma 1) $;
	\item for every $\rho \in 2^{|\sigma|}$, $\phi(\rho 0)(|g(\sigma)|) \neq \phi(\rho 1)(|g(\sigma)|)$.
\end{enumerate}

Initially, $\phi(\epsilon) = \epsilon$ and $g$ is nowhere defined.
Assume $\phi$ is defined over $2^{\leq k}$ and $g$ over $2^{<k}$ for some $k \in \omega$.

\emph{Defining $\phi$}. Consider successively each $\sigma \in 2^k$. By partition regularity of largeness, $G \uh \nu$ is large for some $\nu \succeq \phi(\sigma)$ such that $|\nu|$ is bigger than any value considered so far. By Claim \ref{claim:blossom-to-joyce-rado-combi}, there is a $\rho$ two nodes $\mu_0, \mu_1$ extending $\nu$ such that, letting $\ell = |e(\rho)|$,  $G \uh \mu_0$ and $G \uh \mu_1$ are large and $\mu_0(\ell) < \mu_1(\ell)$.
Temporarily define $\phi(\sigma 0) = \mu_0$ and $\phi(\sigma 1) = \mu_1$. The actual value of $\phi(\sigma 0)$ and $\phi(\sigma 1)$ might change while defining $g$, but will be extensions of these strings. Since $\phi(\sigma 0) \meet \phi(\sigma 1) = \mu_0 \meet \mu_1 \succeq \nu$, $|\phi(\sigma 0) \meet \phi(\sigma 1)|$ is bigger than any value considered so far.

\emph{Defining $g$}. Consider successively each $\tau \in 2^k$. We need to define $g(\tau)$ so that it satisfies Item 2.
Since $G \uh \phi(\tau)$ is large, it is infinite, so by Claim \ref{claim:blossom-to-joyce-rado-combi}, there is a single $\rho$ such that $e(\rho) \succeq \phi(\tau)$ and $\ell = |e(\rho)|$ is bigger than any value considered so far, and for every $\sigma \in 2^k$ and $i < 2$ there are two extensions $\mu_0, \mu_1$ of $\phi(\sigma i)$ such that $\mu_0(\ell) < \mu_1(\ell)$. Then let $\phi(\sigma i) = \mu_i$ and $g(\tau) = e(\rho)$ and consider the next $\tau \in 2^k$.

\emph{Defining $f$}. We now define $f$ so that $(f, g)$ is a blossom tree.
For every $\sigma \in 2^{<\omega}$, let $f(\sigma) = \phi(\sigma 0) \meet \phi(\tau 0)$.
Then $f: \cantor \to \cantor$ is a $\prec$-preserving, $\ltlex$-preserving and $\meet$-preserving function such that for every $\sigma \in \cantor$:
\begin{enumerate}
	\item $g(\sigma) \succeq f(\sigma) \succeq \phi(\sigma)$;
	\item\label{it:f-g-are-rado} for every $\rho \in 2^{|\sigma|}$, $f(\rho 0)(|g(\sigma)|) \neq f(\rho 1)(|g(\sigma)|)$.
\end{enumerate}
Thus $(f, g)$ is a blossom tree, with $g:\cantor\to G$. Let $S\subseteq\cantor$ be a cofinal set such that $\meetclosure S$ has at most one string of each length. The structure $(g[S], \Epn, \ltlex, |\cdot \meet \cdot |)$ is a Joyce blossom graph. %since we ensured that the length of the elements of $g[\cantor]$ and of $f[\cantor]$ are fresh values, so are unique. Thus  $(g[\cantor], \Epn, \ltlex, |\cdot \meet \cdot |)$ is a Joyce blossom graph. Moreover, it is a substructure of $G$. This completes the proof of \Cref{thm:blossom-to-joyce-rado}.
\end{proof}

\begin{theorem}[$\RCA_0$]\label{thm:joyce-to-rado-blossom}
For every Joyce blossom graph $\mathcal{G}$
and every (finite or infinite) Joyce graph $\mathcal{F}$,
there is a Joyce structure embedding from $\mathcal{F}$ to $\mathcal{G}$.
\end{theorem}
\begin{proof}
Let $(f, g)$ be the blossom tree and $S\subseteq\cantor$ such that $\mathcal{G} = (g[S], \Epn, \ltlex, |\cdot \meet \cdot|)$.
  Let $F\subseteq\cantor$ be the domain of a coded Joyce graph isomorphic to $\mathcal{F}$. Let $D=\{d_i:i\in|\meetclosure F|\}$ be an enumeration of $\meetclosure F$ such that $i<j$ implies $|d_i|<|d_j|$. We first build a function $\phi:D\to S$. Define $\phi(d_0)$ to be any element of $S$. Suppose $\phi(d_i)$ is defined for $i<n$. Then $d_n$ is mapped to any element of $S$ extending $\sigma$, where $\sigma$ is the string of length $|\phi(n-1)|+1$, such that:
  \begin{enumerate}
  \item\label{it:jtrb0} if $k<|\sigma|$ and $k\neq|\phi(d_i)|$ for any $i<n$, then $\sigma(k)=0$;
  \item\label{it:jtrb1} if $d_i\prec d_n$, then $\sigma(|\phi(d_i)|)=1$ iff $d_i1\prec d_n$;
  % \item\label{it:jtrb2} Otherwise $\sigma(|\phi(d_i)|)=j$ so that $f((\sigma\uh |\phi(d_i)|)\cat j)(|g(\phi(d_i))|)=d_n(|d_i|)$.
  \item\label{it:jtrb2} Otherwise $\sigma(|\phi(d_i)|)=j$ so that $f(\sigma)(|g(\phi(d_i))|)=d_n(|d_i|)$.
  \end{enumerate}
  The last item can be satisfy for a single $j$ by \Cref{it:blossom-rado} of \Cref{def:blossom-tree}. The string $\sigma$ is uniquely defined, and $\phi(d_n)\in S$ extending $\sigma$ exists as $S$ is cofinal. Note that $\phi$ is $\prec$-preserving: if $d_n\prec d_m$, then by \Cref{it:jtrb0}, for any $k<|\phi(d_n)|$ if $k\not\in\{|\phi(d_i):i<n\}$, then $\phi(d_n)(k)=\phi(d_m)(k)$. By \Cref{it:jtrb1} for any $i$ such that $d_i\prec d_n$, $\phi(d_n)(|\phi(d_i)|)=\phi(d_m)(|\phi(d_i)|)$. By \Cref{it:jtrb2} and the fact that $f$ is $\prec$-preserving, for $i$ such that $d_i\not\prec d_n$, we again have $\phi(d_n)(|\phi(d_i)|)=\phi(d_m)(|\phi(d_i)|)$. So $\phi$ is $\prec$-preserving.

  Define $\psi = g\circ\phi$. We claim that $\psi:F\to g[S]$ preserves the Joyce graph structure. In order to show the claim, we have to prove that it preserves $\ltlex$, and that for any $d_{n_0},d_{n_1},d_{m_0},d_{m_1}\in F$, $|d_{n_0}\meet d_{n_1}|<|d_{m_0}\meet d_{m_1}|$ implies $|\psi(d_{n_0})\meet \psi(d_{n_1})|<|\psi(d_{m_0})\meet \psi(d_{m_1})|$, and finally that for any $d_n,d_m\in F$, $d_n\Epn d_m$ implies $\psi(d_n)\Epn \psi(d_m)$. The proof of these three facts are respectively in the three following paragraphs.
  
  The fact that $\phi$ is $\prec$-preserving implies that $\phi(d_n\meet d_m)\prec\phi(d_n)\meet\phi(d_m)$. By \Cref{it:jtrb1}, $\phi(d_n\meet d_m)\succ\phi(d_n)\meet\phi(d_m)$: indeed, $\phi(d_n)(|\phi(d_n\meet d_m)|)\neq \phi(d_m)(\phi(|d_n\meet d_m|))$. So $\phi$ is $\meet$-preserving, and by \Cref{it:jtrb1}, it is also $\ltlex$-preserving. The function $g$ is also $\ltlex$-preserving: $f$ is $\ltlex$-preserving, and $f(\sigma)\prec g(\sigma)$ for every $\sigma\in\cantor$. So $\psi=g\circ\phi$ is $\ltlex$-preserving.

  % By construction, $|d_n|<|d_m|\iff n<m\iff |\phi(d_n)|<|\phi(d_m)|$. 
  By the second item of \Cref{def:blossom-tree}, for all $\sigma,\tau\in\cantor$ with $|\sigma|\neq|\tau|$, $|\sigma|<|\tau|\iff |g(\sigma)|<|g(\tau)|\iff |f(\sigma)|<|f(\tau)|\iff |g(\sigma)|<|f(\tau)|\iff |f(\sigma)|<|g(\tau)|$. So $|d_n|<|d_m|\iff |\psi(d_n)|<|\psi(d_m)|$. Now, suppose $|d_{n_0}\meet d_{n_1}|<|d_{m_0}\meet d_{m_1}|$ for some $d_{n_0},d_{n_1},d_{m_0}d_{m_1}\in F$. Let $d_n=d_{n_0}\meet d_{n_1}$ and $d_m=d_{m_0}\meet d_{m_1}$. If $d_{n_0}\neq d_{n_1}$, then $\psi(d_{n_0})\meet\psi(d_{n_1})=f(\phi(d_n))$, otherwise $\psi(d_{n_0})\meet\psi(d_{n_1})=g(\phi(d_n))$; and similarly for $m_0, m_1, m$ with $d_m=d_{m_0}\meet d_{m_1}$. So, depending whether $d_n, d_m\in F$, we use one of the previous equivalence to get that  $|\psi(d_{n_0})\meet \psi(d_{n_1})|<|\psi(d_{m_0})\meet \psi(d_{m_1})|$.
  
%   Indeed, as $g$ preserves $\ltlex$ and the ordering of the meets: that is, by the second item of \Cref{def:blossom-tree}, $\forall\sigma,\tau\in\cantor$ with $|\sigma|\neq|\tau|$, $|\sigma|<|\tau|\iff |g(\sigma)|<|g(\tau)|$. Since these properties are also satisfied by $\phi$, they are satisfied by $\psi$.

  Finally, for any $n<m$,
  \[
  	\psi(d_m)(|\psi(d_n)|)= g(\phi(d_m))(|g(\phi(d_n))|)=f(\phi(d_m))(|g(\phi(d_n))|)
  \]
  as $g(\phi(d_m))\succ f(\phi(d_m))$. By \Cref{it:jtrb2} $\phi(d_m)$ is chosen so that \[f(\phi(d_m))(|\psi(d_n)|)= d_m(|d_n|).\qedhere\]
  % First, note that by construction, $\phi$ is $\prec$-preserving. Combined with the first item, we have that $\phi$ is $\meet$-preserving. By Item 1 again, it is also $\ltlex$-preserving.
\end{proof}

\begin{corollary}[$\ACA_0$]\label{th:4.1LaflammeRado}
  Let $\mathcal{G}$ be a Joyce Rado graph, and $\mathcal{F}$ be a (finite or infinite) Joyce graph. Then, there exists an embedding from $\mathcal{F}$ to $\mathcal{G}$.
\end{corollary}
\begin{proof}
By \Cref{thm:joyce-graph-to-coded}, we can assume that $\mathcal{G}$ is a coded Joyce Rado graph.
By \Cref{thm:blossom-to-joyce-rado}, there is an embedding of a Joyce blossom graph $\mathcal{B}$ to $\mathcal{G}$. By \Cref{thm:joyce-to-rado-blossom},
there is an embedding of $\mathcal{F}$ to $\mathcal{B}$. Thus there is an embedding of $\mathcal{F}$ to $\mathcal{G}$.
\end{proof}

Recall that the \emph{age} of a graph $\mathcal{G}$ is the collection of all finite graphs that are isomorphic to a subgraph of $\mathcal{G}$.

\begin{corollary}
  The age of a Joyce Rado graph is the set of finite Joyce graphs.
\end{corollary}

\begin{theorem}
There is a computable Joyce Rado graph $\mathcal{G}$ such that for every Joyce blossom graph $\mathcal{B}$,
every embedding of $\mathcal{B}$ to $\mathcal{G}$ computes $\emptyset'$.
\end{theorem}
\begin{proof}
  % We will build a computable Joyce Rado graph $(G,\Epn, \ltlex, \meetlevel\cdot\cdot)$ where $G\subseteq\cantor$, such that the following hold:
  % \begin{itemize}
  % \item[$\Rc$:] For every $\sigma,\tau\in G$, there exists infinitely many $\rho, \gamma$ such that $\sigma\meet\tau=\rho\meet\gamma$ if and only if $\emptyset'\uh|\sigma\meet\gamma|=\emptyset'_{|\sigma|}\uh|\sigma\meet\gamma|$
  % \end{itemize}
  % Suppose that such a $G$ is defined. Then, for any $f,g$ Joyce blossom graph with $g[\cantor]\subseteq G$, $g$ compute $\emptyset'$: For any $n$, there exists $\sigma,\tau\in\cantor$ such that $|g(\sigma)\meet g(\tau)|>n$, and as $f(\sigma\meet\tau) = g(\sigma)\meet g(\tau)$, for any $\rho,\gamma$ such that $\rho\meet\gamma=\sigma\meet\tau$, $g(\sigma)\meet g(\tau) = g(\rho)\meet g(\gamma)$ and as there are infinitely many such $\rho,\gamma$, by the requirement, $\emptyset'\uh|\sigma\meet\gamma|=\emptyset'_{|\sigma|}\uh|\sigma\meet\gamma|$.
%  Note that if  $|\sigma\meet\tau|<|\rho\meet\gamma|$ implies $\meetlevel\sigma\tau<\meetlevel\rho\gamma$, if $\mathcal{B}$ is a Joyce blossom graph and $e:\mathcal{B}\to(G,\Epn,\ltlex,\meetlevel\cdot\cdot)$ is an embedding, then $e\circ f:\cantor\to\meetclosure G$, $e\circ g:\cantor\to G$ form a blossom graph.
  
  The idea of the proof is to build a Joyce Rado graph with domain $G$ such that if $S\subseteq\cantor$ is a cofinal set with at most one meet of each length, $f:\cantor\to \meetclosure G$ and $g:\cantor\to G$ form a blossom tree, and $\sigma,\tau\in\cantor$ are such that $|g(\sigma)\meet g(\tau)|>3j$, then $\emptyset'(j)=1$ iff $g(\sigma)(3j+2)=g(\sigma)(3j+2)$ iff $g(\tau)(3j+2)=g(\tau)(3j+2)$. 
  
  % Here is the construction. 
  Let $(G_0,E)$ be a Rado graph, and let $(g_n)_{n\in\om}$ be an enumeration of $G_0$, and $(\emptyset'_s)_{s\in\om}$ a computable approximation of $\emptyset'$. Define $\sigma_n$ to be the unique string of length $3n+2$ such that:
  \begin{enumerate}
  \item $\sigma_n(3n)=0$ and $\sigma_n(3n+1)=1$;
  \item for any $j<n$, $\sigma_n(3j+1)=0$ and
    \begin{enumerate}
    \item if $\emptyset'_n(j)=0$ and $g_nEg_j$ then $\sigma_n(3j)=0$ and $\sigma_n(3j+2)=1$,
    \item if $\emptyset'_n(j)=0$ and $\lnot g_nEg_j$ then $\sigma_n(3j)=1$ and $\sigma_n(3j+2)=0$ ,
    \item if $\emptyset'_n(j)=1$ and $g_nEg_j$ then $\sigma_n(3j)=1$ and $\sigma_n(3j+2)=1$,
    \item if $\emptyset'_n(j)=1$ and $\lnot g_nEg_j$ then $\sigma_n(3j)=0$ and $\sigma_n(3j+2)=0$.
    \end{enumerate}
  \end{enumerate}
  Let $G = \{\sigma_n:n\in\Nb\}$. It is clear that $(G,\Epn)$ is a Rado graph, as it is in bijection with $G_0$ via $g_n\mapsto \sigma_n$ since $g_nEg_m$ iff $\sigma_n\Epn\sigma_m$. Define $\meetlevel{\sigma_n}{\sigma_m}=v(\sigma_n\meet\sigma_m)$ where $v$ is a fixed injective function $\cantor\to\om$ such that if $|\sigma|<|\tau|$ then $v(\sigma)<v(\tau)$. Then by construction, $\mathcal{G}=(G,\Epn,\ltlex,\meetlevel\cdot\cdot)$ is a Joyce Rado graph.

  Now, suppose that $f:\cantor\to \meetclosure G$ and $g:\cantor\to G$ form a Joyce blossom graph.
  The claim is the following: if $\sigma,\tau\in\cantor$ are such that $|g(\sigma)\meet g(\tau)|>3j$, then $\emptyset'(j)=1$ iff $g(\sigma)(3j+2)=g(\sigma)(3j+2)$ iff $g(\tau)(3j+2)=g(\tau)(3j+2)$.

  Indeed, recall that $g(\sigma)\meet g(\tau)=f(\sigma)\meet f(\tau)$. Let $\rho\succ\sigma$ be such that $|g(\rho)|\geq 3n+2$ where $\emptyset'_{n}(j)=\emptyset'(j)$. By construction, $g(\rho)(3j)=g(\rho)(3j+2)$ iff $\emptyset'_{n}(j)=1$ iff $\emptyset'(j)=1$. As $\rho\succ\sigma$, we have $g(\rho)\succ f(\sigma)\prec g(\sigma)$ so finally $g(\sigma)(3j)=g(\sigma)(3j+2)$ iff $\emptyset'(j)=1$. % As $g(\sigma)\succ f(\sigma)$, $g(\sigma)(3j)=g(\sigma)(3j+2)$ iff $\emptyset'(j)=1$.

  Therefore, given $g$, to know the value of $\emptyset'(j)$, it suffices to find $\sigma,\tau\in S$ such that $|g(\sigma)\meet g(\tau)|>3j$, and answer according to whether $g(\sigma)(3j)=g(\sigma)(3j+2)$.
\end{proof}

\begin{corollary}
\Cref{th:4.1LaflammeRado} implies $\ACA_0$.
\end{corollary}

\section{A proof of the Rado Graph theorem}\label{subsect:rado-from-mtt}

\begin{definition}\index{diagonalization!Joyce graph}\index{Joyce graph!diagonalization}
A \emph{Joyce graph diagonalization} for some Joyce graph $(U, E_U,  <_U, \meetlevel{\cdot}{\cdot}_U)$ is a function $h: 2^{<\omega} \to U$, such that for every coded Joyce graph $X$,
$(h[X], E_U, <_U, \meetlevel{\cdot}{\cdot}_U)$ is isomorphic to $X$.
\end{definition}

\begin{theorem}[$\RCA_0$]\label{thm:joyce-diagonalization-Rado}
There exists a Joyce Rado graph $(2^{<\omega}, E_T, <_T, \meetlevel{\cdot}{\cdot}_T)$
such that for every coded Joyce order $X$,
the Joyce structures of $(X, E_T, <_T, \meetlevel{\cdot}{\cdot}_T)$ and $(X, \Epn, \ltlex, |\cdot \meet \cdot|)$ are isomorphic.
\end{theorem}
\begin{proof}
Let $(U, \Epn, \ltlex, \meetlevel{\cdot}{\cdot}_U)$ be the Joyce Rado graph defined in \Cref{thm:joyce-rado-graph-exists}, that is, $U = (000 \cup 101)^{*}01$ and $\meetlevel{\sigma}{\tau}_U = v(\sigma \meet \tau)$ for some injective function $v: 2^{<\omega} \to \omega$ such that 
for every $\sigma, \tau \in 2^{<\omega}$, if $|\sigma| < |\tau|$ then $v(\sigma) < v(\tau)$.

Define the Joyce Rado graph $(2^{<\omega}, \Epn, <_T,  \meetlevel{\cdot}{\cdot}_T)$ as follows:
Given $\sigma \in 2^{<\omega}$, let $\hat{\sigma}$ be the binary string of length $3|\sigma|+2$
defined for every $j < |\sigma|$ by $\hat{\sigma}(3j) = \sigma(j)$, $\hat{\sigma}(3j+1) = \hat{\sigma}(3j+2) = 0$, and $\hat{\sigma}(3|\sigma|) = 0$ and $\hat{\sigma}(3|\sigma|+1) = 1$.
For instance, if $\sigma = 0110$ then $\hat{\sigma} = 00010010000001$.
Then let $\sigma <_T \tau$ if and only if $\hat{\sigma} \ltlex \hat{\tau}$
and $\meetlevel{\sigma}{\tau}_T = \meetlevel{\hat{\sigma}}{\hat{\tau}}_U$. 

Let $X$ be a coded Joyce order. We claim that $(X, \Epn, <_T, \meetlevel{\cdot}{\cdot}_T)$ and $(X, \Epn, \ltlex, |\cdot \meet \cdot|)$ are isomorphic.

Fix $\sigma, \tau \in X$. If $\sigma \ltlex \tau$, then $\hat{\sigma} \ltlex \hat{\tau}$, hence $\sigma <_T \tau$. Conversely, if $\sigma <_T \tau$, then $\hat{\sigma} \ltlex \hat{\tau}$, but since $\sigma$ and $\tau$ are incomparable with respect to the prefix relation, this implies that $\sigma \ltlex \tau$. Thus $\sigma <_T \tau$ if and only if $\sigma \ltlex \tau$.

Fix $\sigma, \tau, \rho, \mu \in X$. If $|\sigma \meet \tau| <_\Nb |\rho \meet \mu|$, then 
$|\hat{\sigma} \meet \hat{\tau}| <_\Nb |\hat{\rho} \meet \hat{\mu}|$, then $v(\hat{\sigma} \meet \hat{\tau}) <_\Nb v(\hat{\rho} \meet \hat{\mu})$, hence $\meetlevel{\sigma}{\tau}_T <_\Nb \meetlevel{\rho}{\mu}_T$. Conversely, assume $\meetlevel{\sigma}{\tau}_T <_\Nb \meetlevel{\rho}{\mu}_T$. Unfolding the definition, $v(\hat{\sigma} \meet \hat{\tau}) <_\Nb v(\hat{\rho} \meet \hat{\mu})$. If $|\hat{\sigma} \meet \hat{\tau}| \neq |\hat{\rho} \meet \hat{\mu}|$, then by definition of $v$, $|\hat{\sigma} \meet \hat{\tau}| <_\Nb |\hat{\rho} \meet \hat{\mu}|$, hence $|\sigma \meet \tau| <_\Nb |\rho \meet \mu|$. If $|\hat{\sigma} \meet \hat{\tau}| = |\hat{\rho} \meet \hat{\mu}|$, then since $X$ is a coded Joyce graph, $\sigma \meet \tau = \rho \meet \mu$, so $\hat{\sigma} \meet \hat{\tau} = \hat{\rho} \meet \hat{\mu}$ and $v(\hat{\sigma} \meet \hat{\tau}) = v(\hat{\rho} \meet \hat{\mu})$, contradiction. 
\end{proof}

\begin{corollary}[$\ACA_0$]\label{cor:joyce-diagonalization-exists-Rado}
  Every Joyce Rado graph $(U, E_U, <_U, \meetlevel\cdot\cdot_U)$ has a Joyce graph diagonalization.
\end{corollary}
\begin{proof}
Let $(2^{<\omega}, E_T, <_T, \meetlevel{\cdot}{\cdot}_T)$ be the Joyce Rado graph of \Cref{thm:joyce-diagonalization-Rado}. By \Cref{th:4.1LaflammeRado}, there is an embedding $h: 2^{<\omega} \to U$. By definition of an embedding, for every coded Joyce graph $X \subseteq 2^{<\omega}$, $(h[X], E_U, <_U, \meetlevel{\cdot}{\cdot}_U)$ is isomorphic to $(X, E_T, <_T, \meetlevel{\cdot}{\cdot}_T)$. By \Cref{thm:joyce-diagonalization-Rado}, $(X, E_T, <_T, \meetlevel{\cdot}{\cdot}_T)$ is isomorphic to $(X, \Epn, \ltlex, |\cdot \meet \cdot|)$. Thus $h$ is a Joyce graph diagonalization.
\end{proof}

In the case of Joyce blossom graphs, the existence of a Joyce graph diagonalization holds in $\RCA_0$;

\begin{corollary}[$\RCA_0$]\label{cor:joyce-diagonalization-exists-blossom}
  Every Joyce blossom graph has a Joyce graph diagonalization.
\end{corollary}
\begin{proof}
Similar to the proof of \Cref{cor:joyce-diagonalization-exists-Rado},
but apply \Cref{thm:joyce-to-rado-blossom} instead of \Cref{th:4.1LaflammeRado}.
\end{proof}

\begin{lemma}\label{lem:coded-joyce-order-to-strong-subtree-Rado}
Let $F$ be a finite coded Joyce graph of size $n$ and $T \in \Subtree{\omega}{2^{<\omega}}$.
Then every $E \in \Subtree{2n-1}{T}$ contains at most one coded Joyce graph isomorphic to $F$. Moreover, every coded Joyce graph $H \subseteq T$ isomorphic to $F$
	is included in some $E \in \Subtree{2n-1}{T}$.
\end{lemma}
\begin{proof}
The proof is a straightforward adaptation of \Cref{lem:coded-joyce-order-to-strong-subtree}, \emph{mutatis mutandis}.
\end{proof}

\begin{theorem}[$\ACA_0$]\label{thm:strong-rado-one-type}
  Let $\mathcal{G}$ be a Joyce Rado structure, and $\mathcal{F}$ be a finite Joyce graph. Then, the big Ramsey number of $\mathcal{F}$ in $\mathcal{G}$ is 1.
\end{theorem}
\begin{proof}
Let $X$ be a countable coded Joyce Rado graph and $F$ be a finite coded Joyce Rado graph of size $n$. Fix a coloring $f: {X \choose F} \to k$. Here, ${X \choose F}$ denotes all the subcopies of $F$ in $X$.

Let $h: 2^{<\omega} \to X$ be a Joyce graph diagonalization, which exists by \Cref{cor:joyce-diagonalization-exists-Rado}. Let $g: \Subtree{n}{2^{<\omega}} \to k$ be defined for every $E \in \Subtree{2n-1}{2^{<\omega}}$ by $g(E) = f(h(H))$ where $H \subseteq E$ is the unique element coded Joyce graph isomorphic to $F$, if it exists. Otherwise let $g(E) = 0$. This coloring is well-defined by \Cref{lem:coded-joyce-order-to-strong-subtree}.

By Milliken's tree theorem for height $2n-1$, there is a strong subtree $S \in \Subtree{\omega}{2^{<\omega}}$ such that $g$ restricted to $\Subtree{2n-1}{S}$ is monochromatic for some color $i < k$. In particular, by \Cref{lem:coded-joyce-order-to-strong-subtree-Rado}, for every coded Joyce graph $H \subseteq S$ isomorphic to $F$, there is some $E \in \Subtree{2n-1}{S}$ containing $H$, and $g(E) = f(h(H)) = i$.

Since $S \in \Subtree{\omega}{2^{<\omega}}$, there is an injective function $\phi: 2^{<\omega} \to S$ such that $\phi[X]$ is a coded Joyce graph isomorphic to $X$. In particular, since $h$ is a Joyce graph diagonalization, $Y = h[\phi[X]]$ is a coded Joyce Rado graph isomorphic to $X$, hence a subcopy of~$X$.

We claim that $f$ restricted to ${Y \choose F}$ is monochromatic for color~$i$.
Let $\hat{H}$ be a copy of $F$ in $Y = h[\phi[X]]$. Let $H \subseteq \phi[X]$ be such that $h[H] = \hat{H}$. In particular since $\phi[X]$ is a coded Joyce graph, so is $H$, so since $h$ is a Joyce graph diagonalization, $\hat{H} = h[H]$ is a coded Joyce graph isomorphic to $H$. In other words, $H$ is a copy of $F$ in $\phi[X] \subseteq S$, so $H$ is a copy of $F$ in $S$. By choice of $S$, $f(h[H]) = i$, so $f(\hat{H}) = i$. This completes the proof of \Cref{thm:strong-rado-one-type}.
\end{proof}

\begin{corollary}[$\ACA_0$]
  The statement $(\forall k) \JRG^n_{k,J_n}$ holds, where $J_n$ is the number of non isomorphic Joyce graph structure with $n$ elements, while $(\forall k)\JRG^n_{k,J_n-1}$ does not hold.
\end{corollary}
\begin{proof}
Let $\ell$ be the number of Joyce graph with $n$ elements. Let $F_0, \dots, F_{\ell-1}$ be a finite enumeration of all the finite coded Joyce graph structures of size $n$.

We first prove that $\JRG^n_{k, \ell}$ holds.
Fix a coloring $f: [X]^n \to k$ for some countable Joyce Rado graph structure $(X, E, <, \JRel)$. By \Cref{thm:strong-devlin-one-type}, build a finite decreasing sequence of subsets $X = X_0 \supseteq X_1 \supseteq \dots \subseteq X_\ell$ of $X$ such that for every $s < \ell$:
\begin{enumerate}
	\item $(X_{s+1},  E, <, \JRel)$ is a subcopy of $(X_s, E, <, \JRel)$;
	\item every copy of $F$ in $(X_{s+1},  E, <, \JRel)$ is monochromatic for $f$ for some color $i_s < k$
\end{enumerate}
The Joyce graph structure $(X_\ell, E, <, \JRel)$ is a subcopy of $(X, <, \JRel)$.
Moreover, for every $E \in [X_\ell]^n$, $(E, <, \JRel)$ is isomorphic to $F_s$ for some $s < k$,
so $f(E) = i_s$. It follows that $f[X_\ell]^n \subseteq \{ i_s: s < \ell \}$, hence $|f[X_\ell]^n| \leq \ell$.

We now show that the bound is tight. Let $f: [X]^n \to k$ be defined by $f(H) = s$ for the unique $s < \ell$ such that $(H, E, <, \JRel)$ is isomorphic to $F_s$. 
Let $(Y, E, <, \JRel)$ be a subcopy of $(X, E, <, \JRel)$. In particular, $(Y, E, <, \JRel)$ is a Joyce Rado graph structure, so by \Cref{th:4.1LaflammeRado}, for every $s < \ell$, there is an embedding of $F_s$ into $(Y, E, <, \JRel)$. Therefore, $|f[Y]^n| \geq \ell$.
\end{proof}

\begin{theorem}[$\ACA_0$]\label{thm:subgraph-to-subjoyce-Rado}
  Let $\mathcal{G}=(G,E,<,\JRel)$ be a Joyce graph structure. Let $\mathcal{G}'=(g',E)$ be an isomorphic subcopy of $(G,E)$, that is, a Rado graph. Then, there exists a subcopy $(G'',E)$ of $(G',E)$ such that $(G'',E,<,\JRel)$ is a subcopy of $\mathcal{G}$.
\end{theorem}
\begin{proof}
  The structure $\hat\Xb'=(X',E, <,\JRel)$ is a Joyce Rado graph structure, even if it might not be isomorphic to $\Xb$. By \Cref{th:4.1LaflammeRado}, there exists an embedding of $\Xb$ into $\hat\Xb'$. The image of the embedding is $\Xb''$.
\end{proof}

Note that contrary to the case of Joyce orders, for which the proof that Devlin's theorem implies the Joyce Devlin theorem holds in $\RCA_0$, the following corollary holds in $\ACA_0$. The difference comes from proof of \Cref{th:4.1LaflammeRado} which is more complex than \Cref{th:4.1LaflammeDevlin}.

\begin{corollary}[$\ACA_0$]
  The Rado Graph theorem for $n$-tuples and $\ell$ colors implies the Joyce Rado graph theorem for $n$-tuples and $\ell$ colors.
\end{corollary}

\begin{corollary}[$\ACA_0$]
  The tight bound for the Rado graph theorem and the Joyce Rado graph theorem for $n$ elements are the same, that is, the number of Joyce graph structures with $n$ elements.
\end{corollary}
\begin{proof}
Let $b_0$ and $b_1$ be the tight bound for the Rado graph theorem and the Joyce Rado graph theorem for $n$ elements, respectively. 

We first claim that $b_0 \leq b_1$. 
Let $(X, E)$ be a Rado graph. By \Cref{cor:rado-can-be-enriched}, one can enrich this graph with an order $<$ and a relation $\JRel$ so that $(X, E, <, \JRel)$ is a Joyce Rado graph structure. Let $f: [X]^n \to k$ be a coloring. By choice of $b_1$, there is a Joyce subcopy $(Y, E, <, \JRel)$ of $(X, E, <, \JRel)$ such that $|f[Y]^n| \leq b_1$. In particular, $(Y, E)$ is a subcopy of $(X, E)$ so $b_0 \leq b_1$.

We then claim that $b_1 \leq b_0$.
Let $(X, E, <, \JRel)$ be a Joyce Rado graph structure.  Let $f: [X]^n \to k$ be a coloring. By choice of $b_0$, there is a subcopy $(Y, E)$ of $(X, E)$ such that $|f[Y]^n| \leq b_0$. 
By \Cref{thm:subgraph-to-subjoyce-Rado}, there is a subcopy $(Z, E)$ of $(Y, E)$ such that $(Z, E, <, \JRel)$ is a Joyce subcopy of $(X, E, <, \JRel)$. In particular, $|f[Z]^n| \leq b_0$. Thus $b_1 \leq b_0$.

It follows that $b_0 = b_1$. Moreover, by \Cref{cor:strong-devlin-tight-joyce-orders}, this tight bound is the number of Joyce graph structures with $n$ elements.
\end{proof}

\section{Cone avoidance of the Rado Graph theorem for pairs}

\begin{theorem}\label{thm:blossom-limit-cone-avoidance}
Fix two sets $C$ and $Z$ such that $C \not \leq_T Z$.
Let
\[
	\mathcal{B} = (B, \Epn, \ltlex, |\cdot \meet \cdot|)
\]
be a $Z$-computable Joyce blossom graph. For every $Z$-computable function $f: [B]^2 \to k$, there exists a subcopy $(U, \Epn, \ltlex, |\cdot \meet \cdot|)$ of $\mathcal{B}$ and a finite set of colors $I \subseteq k$ such that $C \not \leq_T U \oplus Z$, $|I| \leq 4$, and 
$$
(\forall \ell_0)( \forall^{\infty} \ell_1)( \forall^{\infty} \ell_2)[ f[U(\ell_0, \ell_1, \ell_2)]^2 \subseteq I],
$$
where $U(\ell_0, \ell_1, \ell_2)$ is the set of Joyce subgraphs of size 2 whose labels are exactly $\ell_0, \ell_1, \ell_2$, that is,  $U(\ell_0, \ell_1, \ell_2) = \{ \{\sigma, \tau\} \in [U]^2: |\sigma \meet \tau| = \ell_0, |\sigma| = \ell_1, |\tau| = \ell_2 \}$.
\end{theorem}
\begin{proof}
By \Cref{cor:joyce-diagonalization-exists-blossom}, there is a $Z$-computable Joyce graph diagonalization $h: \cantor \to B$.
Let $\mathcal{F}_0, \mathcal{F}_1, \mathcal{F}_2, \mathcal{F}_3$ be the 4 Joyce graph structures of size 2.

For every $j < k$, let $g_j: \Subtree{3}{2^{<\omega}} \to k$ be defined for every $E \in \Subtree{3}{2^{<\omega}}$ by $g(E) = f(h(H))$ where $H \subseteq E$ is the unique element coded Joyce graph isomorphic to $\mathcal{F}_j$, if it exists. Otherwise let $g(E) = 0$. This coloring is well-defined by \Cref{lem:coded-joyce-order-to-strong-subtree}.

By 4 successive applications of \Cref{thm:cmtt-admits-strong-cone-avoidance}, there exists a strong subtree $R \in \Subtree{\omega}{2^{<\omega}}$ such that for each $j < 4$, $g_j$ restricted to $\Subtree{3}{R}$ is stable. For each $j < 4$, let $\hat{g}_j: \Subtree{2}{R} \to k$
 be the (non-computable) limit coloring of $g_j$. 
 
 Again, by 4 successive applications of \Cref{thm:cmtt-admits-strong-cone-avoidance}, there exists a strong subtree $S \in \Subtree{\omega}{R}$ such that for each $j < 4$, $\hat{g}_j$ restricted to $\Subtree{2}{S}$ is stable. For each $j < 4$, let $\mu_j: S \to k$
 be the (non-computable) limit coloring of $\hat{g}_j$.
 
 Last, by 4 successive applications of \Cref{thm:hl-strong-cone-avoidance},
 there exists a strong subtree $T \in \Subtree{\omega}{T}$ such that for each $j < 4$,
 $\mu_j$ restricted to $T$ is monochromatic for some color $i_j < 4$. Let $I = \{i_j: j < 4\}$.

In particular, by \Cref{lem:coded-joyce-order-to-strong-subtree-Rado}, for every coded Joyce graph $H \subseteq T$ isomorphic to $\mathcal{F}_j$, there is some $E \in \Subtree{3}{S}$ containing $H$, and $g_j(E) = f(h(H))$.

Let $(X, \Epn, \ltlex, |\cdot \meet \cdot|)$ be a $Z$-computable coded Joyce graph isomorphic to $\mathcal{B}$, which exists by \Cref{thm:joyce-graph-to-coded}.
Since $T \in \Subtree{\omega}{2^{<\omega}}$, there is an injective function $\phi: 2^{<\omega} \to T$ such that $\phi[X]$ is a coded Joyce graph isomorphic to $X$, hence to $\mathcal{B}$. In particular, since $h$ is a Joyce graph diagonalization, $U = h[\phi[X]]$ is a coded Joyce Rado graph isomorphic to $X$, hence a subcopy of~$\mathcal{B}$.

We claim that the statement of the theorem holds for $U$, $f$ and $I$.
Given any $\ell \in \omega$, there is at most one level $n \in \omega$ such that for every $\sigma, \tau \in T$ with $|\sigma \meet \tau| = n$, $|h(\phi(\sigma)) \meet h(\phi(\tau))| = \ell$. We call $n$ the \emph{preimage} of $\ell$. Moreover, if $\ell$ is a label of $U$, that is, there is some $\rho, \nu \in U$ such that $|\rho \meet \nu| = \ell$, then it has a preimage. 

Fix $\ell_0 \in \omega$. If $\ell_0$ has no preimage, then $U(\ell_0, \ell_1, \ell_2) = \emptyset$ for every $\ell_1, \ell_2$ and the property is vacuously satisfied. Let $n_0$ be the preimage of $\ell_0$. Since for every $j < 4$, $\hat{g}_j$ is stable over $\Subtree{\omega}{T}$ with limit color $i_j$, there is some threshold $t_0 \in \omega$ such that for every strong subtree $E$ of $U$ of height 2 whose first level is $n_0$ and second is higher than $t_0$, $\hat{g}_j(E) = i_j$. For all but finitely many $\ell_1$, the preimage of $\ell_1$, if it exists, is larger than $t_0$. Fix any such $\ell_1$ with preimage $n_1$. For every $j < 4$, since $\hat{g}_j$ is the limit coloring of $g_j$, there is a threshold $t_1 \in \omega$ such that for every strong subtree $E$ of height 3 whose first two levels are $n_0$ and $n_1$, respectively, and whose last level is higher than $t_1$, $g_j(E) = i_j$. 
	For all but finitely many $\ell_2$, its preimage is larger than $t_1$. 
	
	Fix any such $\ell_2$, and let $\hat{H} \in U(\ell_0, \ell_1, \ell_2)$. 
	Let $j < 4$ be such that $\hat{H}$ is isomorphic to $\mathcal{F}_j$. Let $H \subseteq \phi[X]$ be such that $h[H] = \hat{H}$. In particular since $\phi[X]$ is a coded Joyce graph, so is $H$, so since $h$ is a Joyce graph diagonalization, $\hat{H} = h[H]$ is a coded Joyce graph isomorphic to $H$. In other words, $H$ is a copy of $\hat{H}$ in $\phi[X] \subseteq T$, so $H$ is a copy of $\mathcal{F}_j$ in $T$. By choice of $T$, $f(h[H]) = i_j$, so $f(\hat{H}) = i_j \in I$.
This completes the proof of \Cref{thm:blossom-limit-cone-avoidance}.
\end{proof}

\begin{theorem}\label{thm:weak-rado-cone-avoidance}
$(\forall k)\RG^2_{k, 4}$ admits cone avoidance.
\end{theorem}
\begin{proof}
Fix two sets $Z, C$ such that $C \not \leq_T Z$.
Let $(V, E)$ be a $Z$-computable Rado graph and $h: [V]^2 \to k$ be a $Z$-computable coloring.

By computable categoricity of the Rado graph, $(V, E)$ is $Z$-computably isomorphic to the graph of a computable Joyce blossom graph \[\mathcal{B} = (B, \Epn, \ltlex, |\cdot \meet \cdot|).\]
This induces a $Z$-computable coloring $\hat{h}: [B]^2 \to k$ by composing the coloring $h$ with the isomorphism.
By \Cref{thm:blossom-limit-cone-avoidance}, there is a subcopy $(U, \Epn, \ltlex, |\cdot \meet \cdot|)$ of $\mathcal{B}$ and a finite set of colors $I \subseteq k$ such that $C \not \leq_T U \oplus Z$, $|I| \leq 4$, and 
\begin{equation}\label{eqn:weak-rado-pairs-ca-1}
(\forall \ell_0)( \forall^{\infty} \ell_1)( \forall^{\infty} \ell_2)[ \hat{h}[U(\ell_0, \ell_1, \ell_2)]^2 \subseteq I],
\end{equation}
where $U(\ell_0, \ell_1, \ell_2)$ is the set of Joyce subgraphs of size 2 whose labels are exactly $\ell_0, \ell_1, \ell_2$, that is, $U(\ell_0, \ell_1, \ell_2) = \{ \{\sigma, \tau\} \in [U]^2: |\sigma \meet \tau| = \ell_0, |\sigma| = \ell_1, |\tau| = \ell_2 \}$.
Let $(f, g)$ be a $U \oplus Z$-computable blossom tree and $D \subseteq \cantor$ be a $Z$-computable set cofinal in $\cantor$ such that $g[D] = B$. 
In particular, by \Cref{eqn:weak-rado-pairs-ca-1}, the following holds:
\begin{equation}\label{eqn:weak-rado-pairs-ca-2}
(\forall \rho \in \cantor)(\forall^{\infty} \sigma_0 \in D)(\forall^{\infty} \sigma_1 \in D)[\sigma_0 \meet \sigma_1 = \rho \Rightarrow \hat{h}(g(\sigma_0), g(\sigma_1)) \in I].
\end{equation}

We are going to build a by forcing infinite set $G \subseteq D$ such that $(g[G], \Epn)$ is a Rado graph and $\hat{h}[g[G]]^2 \subseteq I$. 

\begin{definition}
A string $\sigma \in \cantor$ \emph{witnesses} a finite 2-partition $F_0 \sqcup F_1 \subseteq \cantor$
if for every $i < 2$ and every $\rho \in F_i$, $\sigma(|\rho|) = i$.
\end{definition}

In other words, $\sigma$ witnesses $F_0 \sqcup F_1 \subseteq \cantor$ if in the graph $(\cantor, \Epn)$, $\sigma$ is connected to all the elements of $F_1$ and disconnected from all the elements of $F_0$.
Note that if $\sigma$ witnesses  $F_0 \sqcup F_1 \subseteq D$, then so does any $\tau \succeq \sigma$.

\begin{lemma}\label{lem:weak-rado-cone-avoidance-Rado-preservation}
If $(G, \Epn)$ is a Rado graph for some $G \subseteq D$, then so is $(g[G], \Epn)$.	
\end{lemma}
\begin{proof}
Let $\hat{F}_0 \sqcup \hat{F}_1 \subseteq g[G]$ be a finite 2-partition. 
Let $F_0 = g^{-1}[\hat{F}_0]$ and $F_1 = g^{-1}[\hat{F}_1]$. In particular, $F_0$ and $F_1$ are disjoint. 
Since $(G, \Epn)$ is a Rado graph, then there is some $\sigma \in G$ witnessing the 2-partition $F_0 \sqcup F_1 \subseteq G$. We claim that $g(\sigma)$ witnesses the 2-partition $\hat{F}_0 \sqcup \hat{F}_1 \subseteq g[G]$. By definition of a blossom tree $(f,g)$, the function $f$ is $\meet$-preserving, so $f(\sigma)$ witnesses the 2-partition $\hat{F}_0 \sqcup \hat{F}_1 \subseteq g[G]$. Since $g(\sigma) \succeq f(\sigma)$, then so does $g(\sigma)$.
\end{proof}

Given $\sigma \in \cantor$, we write $D \uh \cantor = \{ \tau \in D: \tau \succeq \sigma \}$.
Given a finite set $R \subseteq \cantor$, we write $D \uh R = \{ \tau \in D: (\exists \sigma \in R)[\tau \succeq \sigma] \}$.

\begin{definition}
A \emph{condition} is a pair $(F, R)$ where $F \subseteq D$ and $R \subseteq \cantor$
are both finite sets such that $R$ is prefix-free and:
\begin{enumerate}[(1)]
	\item every finite 2-partition $F_0 \sqcup F_1 = F$ is witnessed by some $\sigma \in R$;
	\item for every $\sigma \in F$ and $\tau \in F \cup (D \uh R)$, $\hat{h}(g(\sigma), g(\tau)) \in I$;
%	\item For every $\sigma_0 \neq \sigma_1 \in R$ and $\tau_0 \in D \uh \sigma_0$ and $\tau_1 \in D \uh \sigma_1$, $\hat{h}(g(\tau_0), g(\tau_1)) \in I$.
\end{enumerate}
A condition $(E, S)$ \emph{extends} $(F, R)$ (written $(E, S) \leq (F, R)$) if $F \subseteq E$, $E \setminus F \subseteq D \uh R$ and for every $\tau \in S$, there is some $\sigma \in R$ such that $\tau \succeq \sigma$.
\end{definition}

One can see a condition $(F, R)$ as the Mathias condition $(F, D \uh R)$.
In particular, for every filter $\Fc$ for this notion of forcing, letting $G_\Fc = \bigcup \{ F: (F, R) \in \Fc \}$, if $(F, R) \in \Fc$ then $F \subseteq G_\Fc \subseteq F \cup (D \uh R)$.
Structurally, we ensured that $\hat{h}[g[G_\Fc]]^2 \subseteq I$. Note that $(\emptyset, \{\epsilon\})$ is a valid condition.

\begin{lemma}\label{lem:weak-rado-cone-avoidance-condition-is-rado}
For every condition $(F, R)$, $(F \cup (D \uh R), \Epn)$ is a Rado graph.	
\end{lemma}
\begin{proof}
Fix any finite 2-partition  $E_0 \sqcup E_1 \subseteq F \cup D \uh R$.
By definition of a condition, there is some $\sigma \in R$ witnessing the 2-partition $E_0 \cap F, E_1 \cap F$. Since $D$ is cofinal in $\cantor$, there is some $\tau \in D$ such that $\tau \succeq \sigma$, and for every $i < 2$ and $\rho \in E_i \setminus F$, $\tau(|\rho|) = i$. 
Thus $\tau$ witnesses the 2-partition $E_0 \sqcup E_1 \subseteq F \cup D \uh R$. Moreover $\tau \in D \uh \sigma$, hence $\tau \in D \uh R$.
\end{proof}

The following lemma shows that if $\Fc$ is a sufficiently generic filter, then $(G_\Fc, \Epn)$ is a Rado graph.

\begin{lemma}\label{lem:weak-rado-cone-avoidance-progress}
For every condition $(F, R)$ and every 2-partition $F_0 \sqcup F_1 \subseteq F$,
an extension $(E, S)$ such that $E$ contains an element witnessing the 2-partition.
\end{lemma}
\begin{proof}
Let $\ell_0 \in \omega$ be larger than the length of any string in $R$.
Let $\ell_1 \in \omega$ be sufficiently large with respect to $\ell$ so that
\begin{equation}\label{eqn:weak-rado-pairs-ca-3}
(\forall \rho \in 2^{\leq \ell_0})(\forall \sigma_0 \in D^{\geq \ell_1})(\forall^{\infty} \sigma_1 \in D)[\sigma_0 \meet \sigma_1 = \rho \Rightarrow \hat{h}(g(\sigma_0), g(\sigma_1)) \in I].
\end{equation}
Such an $\ell_1$ exists by \Cref{eqn:weak-rado-pairs-ca-2}.
Let $\hat{R}$ be obtained from $R$ by extending each string $\sigma \in R$ into a string $\tau$ of length $\ell_1$ such that $\tau(\ell_0) = 0$.
Then $(F, \hat{R})$ is again a condition.

By \Cref{lem:weak-rado-cone-avoidance-condition-is-rado}, since $(F \cup (D \uh \hat{R}), \Epn)$ is a Rado graph, there is some $\tau \in F \cup (D \uh \hat{R})$ witnessing the 2-partition.
If $\tau \in F$ then we are done since $(F, \hat{R})$ then satisfies the lemma, so assume $\tau \in D \uh \hat{R}$. % By cofinality of $D$ in $\cantor$, we can assume that $|\tau|$ is bigger than the maximal length of the strings in $R$. 
Let $E = F \cup \{\tau\}$. Since $|\tau| \geq \ell_1$, by \Cref{eqn:weak-rado-pairs-ca-3}, there is some $\ell_2 \in \omega$ so that
\begin{equation}\label{eqn:weak-rado-pairs-ca-4}
(\forall \rho \in 2^{\leq \ell_0})(\forall \mu \in D^{\geq \ell_2})[\tau \meet \mu = \rho \Rightarrow \hat{h}(g(\tau), g(\mu)) \in I].
\end{equation}
For every $\sigma \in R$, let $\sigma_0, \sigma_1 \in R$ be strings of length at least $\ell_2$ such that $\sigma_0(|\tau|) = 0$, $\sigma_1(|\tau|) = 1$, $\sigma_0(\ell_0) = \sigma_1(\ell_0) = 1$, $\sigma_0 \meet \sigma_1 \succeq \sigma$, and define $S = \{\sigma_0, \sigma_1: \sigma \in R \}$.

We claim that $(E, S)$ is a condition.
We first prove Item (1). Let $F_0 \sqcup F_1 \subseteq E$ be a 2-partition. Since $(F, R)$ is a condition, there is some $\sigma \in R$ witnessing the 2-partition $(F_0 \setminus \{\tau\}) \sqcup  (F_1 \setminus \{\tau\})$. If $\tau \in F_i$ for some $i < 2$, then $\sigma_i \succeq \sigma$ witnesses the 2-partition  $F_0 \sqcup F_1 \subseteq E$. If $\tau \not \in F_0 \sqcup F_1$, then any $\sigma_i \succeq \sigma$ witnesses it.

We now prove Item (2). Fix $\nu \in E$ and $\mu \in E \cup (D \uh S)$ with $|\nu| \leq |\mu|$. If $\nu \neq \tau$, then $\nu \in F$ and $\hat{h}(g(\nu), g(\mu)) \in I$ by Item (2) for $(F, R)$.
If $\nu = \tau$, then $\mu \in (D \uh S)$. By choice of $S$, $\mu(\ell_0) = 1$, and $\nu(\ell_1) = 0$. Thus, $|\mu \meet \nu| \leq \ell_0$. Moreover, since every string in $S$ has length at least $\ell_2$, $|\mu| \geq \ell_2$, so by \Cref{eqn:weak-rado-pairs-ca-4}, $\hat{h}(g(\nu), g(\mu)) \in I$.
\end{proof}

\begin{definition}
Let $c = (F, R)$ be a condition and let $\varphi(G, x)$ be a $\Delta^{0,Z}_0$ formula with a free set parameter $G$ and a free integer parameter $x$.
\begin{enumerate}[1.]
	\item $c \Vdash (\exists x)\varphi(G, x)$ if $\varphi(F, x)$ holds for some $x \in \omega$;
	\item $c \Vdash (\forall x)\varphi(G, x)$ if $\varphi(F \cup E, x)$ holds for every $x \in \omega$ and every $E \subseteq D \uh R$ such that $\hat{h}[g[E]]^2 \subseteq I$.
\end{enumerate}
\end{definition}

In particular, the forcing relation is closed under extension, and if $\Fc$ is a filter and $c \Vdash \varphi(G)$ for some $\Sigma^{0,Z}_1$ or $\Pi^{0,Z}_1$ formula and some $c \in \Fc$, then $\varphi(G_\Fc)$ actually holds. As usual, we write $c \Vdash \Gamma^{G \oplus Z} \neq C$ if either $c \Vdash \Gamma^{G \oplus Z}(x)\uparrow$ or $c \Vdash \Gamma^{G \oplus Z}(x)\downarrow \neq C(x)$ for some $x \in \omega$.

\begin{lemma}\label{lem:weak-rado-cone-avoidance-req}
For every condition $c$ and every Turing functional $\Gamma$,
there is an extension $d$ of $c$ such that $d \Vdash \Gamma^{G \oplus Z} \neq C$.
\end{lemma}
\begin{proof}
As in the proof of \Cref{lem:weak-rado-cone-avoidance-progress}, let $\ell_0$ and $\ell_1$ be sufficiently large to satisfy \Cref{eqn:weak-rado-pairs-ca-3}. Again, let $\hat{R}$ be obtained from $R$ by extending each string $\sigma \in R$ into a string $\tau$ of length $\ell_1$ such that $\tau(\ell_0) = 0$.

Let $W$ be the set of all pairs $(x, v) \in \omega \times 2$ such that there is a finite set $E \subseteq D \uh \hat{R}$ satisfying $\hat{h}[g[E]]^2 \subseteq I$ and such that $\Phi^{(F \cup E) \oplus Z}(x) \downarrow = v$. Note that the set $W$ is $Z$-c.e. We have three cases.

\case{1}{$(x, 1-C(x)) \in W$ for some $x \in \omega$.} Let $E \subseteq D \uh \hat{R}$ witness that $(x, 1-C(x)) \in W$, that is, $\hat{h}[g[E]]^2 \subseteq I$ and $\Phi^{(F \cup E) \oplus Z}(x) \downarrow = 1-C(x)$. Since for every $\tau \in E$, $|\tau| \geq \ell_1$, then by \Cref{eqn:weak-rado-pairs-ca-3}, there is some $\ell_2 \in \omega$ so that
\begin{equation}\label{eqn:weak-rado-pairs-ca-5}
(\forall \rho \in 2^{\leq \ell_0})(\forall \tau \in E)(\forall \mu \in D^{\geq \ell_2})[\tau \meet \mu = \rho \Rightarrow \hat{h}(g(\tau), g(\mu)) \in I].
\end{equation}

For every $\sigma \in R$, and every 2-partition $E_0 \sqcup E_1 = E$, let $\sigma_{E_0, E_1}$ be a string of length at least $\ell_2$ extending $\sigma$, such that $\sigma_{E_0, E_1}(|\tau|) = 0$ for every $\tau \in E_0$ and $\sigma_{E_0, E_1}(|\tau|) = 1$ for every $\tau \in E_1$. Let $S = \{\sigma_{E_0, E_1}: \sigma \in R, E_0 \sqcup E_1 = E \}$.

We claim that $(F \cup E, S)$ is a condition.
We first prove Item (1). Let $(F_0 \cup E_0) \sqcup (F_1 \cup E_1) \subseteq F \cup E$ be a 2-partition with $F_0 \sqcup F_1 \subseteq F$ and $E_0 \sqcup E_1 \subseteq E$. Since $(F, R)$ is a condition, there is some $\sigma \in R$ witnessing the 2-partition $F_0 \sqcup F_1 \subseteq F$. By construction, $\sigma_{E_0, E_1} \in S$ witnesses the 2-partition $E_0 \sqcup E_1 \subseteq E$ and extends $\sigma$, so $\sigma_{E_0, E_1}$ witnesses the 2-partition $(F_0 \cup E_0) \sqcup (F_1 \cup E_1) \subseteq F \cup E$.

We now prove Item (2). Fix $\nu \in F \cup E$ and $\mu \in F \cup E \cup (D \uh S)$ with $|\nu| \leq |\mu|$. If $\nu \in F$, then $\hat{h}(g(\nu), g(\mu)) \in I$ by Item (2) for $(F, R)$.
If $\nu \in E$ and $\mu \in F \cup E$, then $\mu \in E$ since $|\nu| \leq |\mu|$, and $\hat{h}(g(\nu), g(\mu)) \in I$ since $\hat{h}[g[E]]^2 \subseteq I$.
If $\nu \in \tau$ and $\mu \in (D \uh S)$, then by choice of $S$, $\mu(\ell_0) = 1$, and $\nu(\ell_1) = 0$. Thus, $|\mu \meet \nu| \leq \ell_0$. Moreover, since every string in $S$ has length at least $\ell_2$, $|\mu| \geq \ell_2$, so by \Cref{eqn:weak-rado-pairs-ca-5}, $\hat{h}(g(\nu), g(\mu)) \in I$.

Moreover, the condition $(F \cup E, S)$ forces $\Phi^{G \oplus Z}(x)\downarrow = 1-C(x)$, thus satisfies the lemma.

\case{2}{$(x, C(x)) \not \in W$ for some $x \in \omega$.} Then the condition $(F, \hat{R})$ is an extension of $(F, R)$ forcing $\Phi^{G \oplus Z}(x)\uparrow \vee \Phi^{G \oplus Z}(x)\downarrow \neq C(x)$, and we are done.

\case{3}{otherwise.} Then $W$ is a $Z$-c.e.\ graph of the characteristic function of $C$, hence $C$ is $Z$-computable, contradicting the hypothesis. This case therefore cannot happen. This completes the proof of \Cref{lem:weak-rado-cone-avoidance-req}.
\end{proof}

We are now ready to prove \Cref{thm:weak-rado-cone-avoidance}.
Let $\Fc$ be a sufficiently generic filter for this notion of forcing.
By definition of a forcing condition, $\hat{h}[g[F]]^2 \subseteq I$ for every $(F, R) \in \Fc$, so $\hat{h}[g[G_\Fc]]^2 \subseteq I$. By \Cref{lem:weak-rado-cone-avoidance-req}, $C \not \leq_T G_\Fc \oplus Z$. Since $g \leq_T Z$ then $C \not \leq_T g[G_\Fc] \oplus Z$. By \Cref{lem:weak-rado-cone-avoidance-progress}, $(G_\Fc, \Epn)$ is a Rado graph, and by \Cref{lem:weak-rado-cone-avoidance-Rado-preservation}, so is $(g[G_\Fc], \Epn)$. 
The image of $(g[G_\Fc], \Epn)$ be the $Z$-computable isomorphism between the Rado graph $(V, E)$ and the Joyce blossom graph $\mathcal{B}$ yields a Rado subgraph $(\hat{V}, E)$ of $(V, E)$ such that $h[\hat{V}]^2 \subseteq I$ and $C \not \leq_T \hat{V} \oplus Z$.
This completes the proof of \Cref{thm:weak-rado-cone-avoidance}.
\end{proof}

\begin{corollary}
$(\forall k)\RG^2_{k, 4}$ does not imply $\ACA_0$ over $\RCA_0$.
\end{corollary}
\begin{proof}
Immediate by \Cref{thm:weak-rado-cone-avoidance} and \Cref{lem:cone-avoidance-not-aca}.
\end{proof}

\section{Lower bound on the Rado Graph theorem}
In order to show lower bounds, we  use the notion of Joyce blossom graph: Indeed, one can computably embed any Joyce graph in a Joyce blossom graph by \Cref{thm:joyce-to-rado-blossom}, and there is a computable one by \Cref{le:computable-bjrg}.  %we embed are a dense linear order, or a linear order, into a computable Joyce blossom graph.
The embeddings of a Joyce complete graph with order type $\Nb$ or $\Qb$ allow to show that $\JRG^2_{8,7}$ implies respectively $\RT22$ and the Devlin Theorem. Even though it is weaker, We include the proof of $\RT22$ from $\JRG^2_{8,7}$ as it yields a computable reduction.
% \begin{corollary}
%   For any Joyce graph $\mathcal{F}$ with two elements, the statement $\JRG^F_{<\infty}$ implies $\RT22$.
% \end{corollary}
\begin{theorem}
  For every $k,n$, the statement $\JRG^2_{4k, 4n+3}$ implies $\RT2{k,n}$. In particular, $\JRG^2_{8, 7}$ implies $\RT22$.
\end{theorem}
\begin{proof}
  We first give the insight about what $4n+3$ corresponds to: if the union of $4$ disjoints sets is of cardinality $4n+3$, then one of them is of cardinality at most $N$.
  \begin{claim}\label{cl:Joyce-complete-linear-graph}
    There exists a Joyce graph $(G,E,<,\meetlevel\cdot\cdot)$ such that $(G,E)$ is the complete graph and for all $a,b\in G$, $a<b\iff \meetlevel aa<\meetlevel bb$.
  \end{claim}
  \begin{proof}
    Let $G=\{1^{2n}0:n\geq 1\}$, then $(G,\Epn, \ltlex, |\cdot\meet\cdot|)$ is a witness of the claim.    
  \end{proof}
Let $f:\Nb^2\to k$ be a symmetric coloring. By \Cref{le:computable-bjrg}, let $\mathcal{G}=(G,E,<,\meetlevel\cdot\cdot)$ be a computable Joyce blossom graph. %In particular, there exists an embedding $e$ from $(G,\ltlex)$ to $\Xb$.

  Let $(\mathcal{F}_i)_{i<4}$ be an enumeration of the finite Joyce graph structures with two elements. If $a,b\in G$, define $g(a,b)=\langle f(\meetlevel aa, \meetlevel bb), i\rangle$ where $i$ is such that $\{a,b\}$ is isomorphic to $\mathcal{F}_i$. By $\JRG^2_{<\infty, 2n+1}$, let $\mathcal{G}'$ be a subcopy of $\mathcal{G}$ in $\mathcal{G}$, using at most $4n+3$ colors. Either $\{(a,b):aEb\}$ or $\{(a,b):\lnot aEb\}$ uses at most $2n+1$ colors, suppose for instance that it is the former. Similarly, either $\{(a,b):a<b\land \meetlevel aa<\meetlevel bb\land aEb\}$ or $\{(a,b):a<b\land \meetlevel aa>\meetlevel bb\land aEb\}$ uses at most $n$ colors, suppose that it is the former. By Claim \ref{cl:Joyce-complete-linear-graph} there exists a Joyce complete graph $\Kb$ such that  for all $a,b\in \Kb$, $a<b\iff \meetlevel aa<\meetlevel bb$. By \Cref{thm:joyce-to-rado-blossom}, there exists a computable subcopy $\mathcal{G}''$ of $\Kb$ inside $\mathcal{G}'$. In particular, $\mathcal{G}''$ uses at most $n$ colors for $g$. As a consequence, $\{\meetlevel aa:a\in\mathcal{G}''\}$ uses at most $n$ colors for $f$.
  % Thus, $\JRG^2_{8, 7}$ implies $\DT2{4,3}$ which implies $\ACA_0$.
  % Let $f:\Nb\to2$ be a coloring. By \Cref{th:exists-dense-rado-graph}, let $E,<,\meetlevel\cdot\cdot$ be such that $(\Nb,E,<,\meetlevel\cdot\cdot)$ is a densely Rado graph. Let $(s_i)_{i<4}$ be an enumeration of the Joyce graph structure with $2$ elements, and define $g:[\Nb]^2\to 8$ such that $g(n,m)=(f(n,m), i)$ where $(\{n,m\},E,<,\meetlevel\cdot\cdot)$ is isomorphic to $s_i$.
  % By $\JRG^2_{<\infty, 7}$, let $S\subseteq\Nb$ such that $|g[\Nb^2]|\leq7$. As $7<2\times 4$, let $i$ be such that $\{(n,m):n,m\in S$ and $\{n;m\}$ is isomorphic to $s_i\}$ is monochromatic for $g$, and thus also for $f$. The rest of the proof is the same as the previous corollary for $F=s_i$.
\end{proof}

\begin{theorem}
  For every $n,k$, the statement $\JRG^2_{2k, 2n+1}$ implies $\DT2{k,n}$ and thus $\JRG^2_{8, 7}$ implies $\ACA_0$.
\end{theorem}
\begin{proof}
  We start with a claim.
  \begin{claim}\label{cl:Joyce-complete-dlo-graph}
    There exists a Joyce graph $(G,E,<,\meetlevel\cdot\cdot)$ such that $(G,E)$ is the complete graph and $(G,<)$ is a DLO.
  \end{claim}
  \begin{proof}
    Let $(G,<,\meetlevel\cdot\cdot)$ be a DLO Joyce order. Define $E=\{(a,b):a,b\in G\}$. Then $(G, E,<,\meetlevel\cdot\cdot)$ is a witness of the claim.
  \end{proof}
  Let $\Xb=(X,<_X)$ be a computable dense linear order with no endpoints, and $f:X^2\to k$ be a symmetric coloring. By \Cref{le:computable-bjrg}, let $\mathcal{G}$ be a computable Joyce blossom graph. In particular, there exists an embedding $e$ from $(G,\ltlex)$ to $\Xb$.

  If $a,b\in G$, define $g(a,b)=\langle f(e[\{a;b\}]), 1\rangle$ if $aEb$ and
  \[
  g(a,b)=\langle f(e[\{a;b\}]), 0\rangle
  \]
  if $\lnot aEb$. By $\JRG^2_{<\infty, 2n+1}$, let $\mathcal{G}'$ be a subcopy of $\mathcal{G}$ in $\mathcal{G}$, using at most $2n+1$ colors. Either $\{(a,b):aEb\}$ or $\{(a,b):\lnot aEb\}$ uses at most $n$ colors, suppose for instance that it is the former. By Claim \ref{cl:Joyce-complete-dlo-graph} there exists a Joyce complete graph $\Kb$ whose order is a DLO. By \Cref{thm:joyce-to-rado-blossom}, there exists a computable subcopy $\mathcal{G}''$ of $\Kb$ inside $\mathcal{G}'$. In particular, $\mathcal{G}''$ is a DLO and uses at most $n$ colors for $g$. As a consequence, $e[\mathcal{G}'']\subseteq\Xb$ is a DLO and uses at most $n$ colors for $f$.

  Thus, $\JRG^2_{8, 7}$ implies $\DT2{4,3}$ which implies $\ACA_0$.
\end{proof}

% \begin{theorem}
%   Let $\mathcal{G}=(G,E,<,\meetlevel\cdot\cdot)$ be a DLO Joyce Rado graph. Then, there exists $G'\subseteq G$ such that $(G',E)$ is the completely disconnected countable graph, and $(G', <)$ is a dense linear order.
% \end{theorem}
% \begin{corollary}
%   The Joyce Rado Theorem implies Devlin's theorem, and therefore $\ACA_0$.
% \end{corollary}

Larson~\cite{MR2488745} computed the big Ramsey number for small subgraphs of the Rado graph.
The sum of the big Ramsey numbers for subgraphs of size 3 is equal to the number of Joyce graphs of size 3, that is, 112. In other words, $(\forall k)\RG^3_{k, 112}$ holds, while $(\forall k)\RG^3_{k, 111}$ does not. In the particular case of the complete graph $K_3$ of size 3,
$(\forall k)\RG^{K_3}_{k, 16}$ holds, while $(\forall k)\RG^{K_3}_{k, 15}$ does not.

%\begin{theorem}[{Hirschfeldt and Jockusch~\cite[Theorem 2.1]{Hirschfeldt2016notions}}]\label{thm:unbalanced-rt32-pazp}
%There exists a computable coloring $h: [\omega]^3 \to 2$, such that every infinite $h$-homogeneous set is of PA degree over $\emptyset'$. Moreover, there is a function $\mu: \omega \to \omega$ such that for every $x, y, z$ such that $\mu(x) < y$ and $\mu(y) < z$, then $h(\{x, y, z\}) = 1$.
%\end{theorem}

The proof technique for the following theorem stems from Jockusch~\cite{Jockusch1972Ramseys} who constructed a computable instance of Ramsey's theorem for triples whose solutions compute the halting set. It was later refined by Hirschfeldt and Jockusch~\cite[Theorem 2.1]{Hirschfeldt2016notions} who showed the existence of a computable such instance such that every solution is of PA degrees over $\emptyset'$. Although it is unknown whether the Rado graph statement for graphs of size $n$ implies Ramsey's theorem for $n$-tuples, we can adapt the argument of Hirschfeldt and Jockusch to graphs.

\begin{theorem}\label{thm:weak-rado-triples-aca}
Let $(G, E)$ be a computable Rado graph and $F$ be a finite graph of size $3$. Let $b \in \omega$ be the Ramsey degree of $F$ in the Rado graph theorem, that is, the number of Joyce graphs isomorphic to $F$. There exists a computable coloring $f: {G \choose F} \to 2b$, such that for every $\hat{G} \subseteq G$ for which $(\hat{G}, E)$ is a Rado graph and $f$ restricted to ${\hat{G} \choose F}$ has at most $b$ colors, $\hat{G}$ is of PA degree over $\emptyset'$.
\end{theorem}
\begin{proof}
By \Cref{cor:rado-can-be-enriched}, there exists an order $<$ over $G$ and a function $\meetlevel{\cdot}{\cdot}: G^2 \to \omega$ such that $(G, E, <, \meetlevel{\cdot}{\cdot})$ is a Joyce Rado graph. Let $J_0, \dots, J_{b-1}$ be an enumeration of all the Joyce graph structures of size 3 isomorphic to $F$. Let $g: {G \choose F} \to 2b$ be the coloring which to $\{x, y, z\}$ associates the index $i < b$ so that the Joyce graph structure of $(\{x, y, z\}, E, <, \meetlevel{\cdot}{\cdot})$ is isomorphic of $J_i$.

Let $h: [\omega]^3 \to 2$ be defined for every $x < y < z$ by $h(x, y, z) = 1$ iff for every $e < x$, $\Phi_e^{\emptyset'[y]}(e)[y] = \Phi_e^{\emptyset'[z]}(e)[z]$. Last, let $f: {G \choose F} \to 2 \times b$ be defined by 
$$
f(x, y, z) = \langle h(\meetlevel{x}{x}, \meetlevel{y}{y}, \meetlevel{z}{z}), g(x, y, z) \rangle.
$$
Let $\hat{G} \subseteq G$ be such that $(\hat{G}, E)$ is a Rado graph and $f$ restricted to ${\hat{G} \choose F}$ has at most $b$ colors.

\begin{claim}\label{claim:weak-rado-triples-aca-1}
$f{\hat{G} \choose F} = \{\langle 1, i\rangle: i < b\}$.
\end{claim}
\begin{proof}
Let $H \subseteq \hat{G}$ be a (non-computable) set such that $(H, E)$ is a Rado graph, and $H$ is sparse enough so that for every $x, y \in H$ such that $\meetlevel{x}{x} <_\Nb \meetlevel{y}{y}$, then for every $e < x$ such that $\Phi_e^{\emptyset'}(e)\downarrow$,  $\Phi_e^{\emptyset'[y]}(e)[y]\downarrow$. In particular, $\{ \meetlevel{x}{x}: x \in H \}$ is $h$-homogeneous for color~1.
By \Cref{th:4.1LaflammeRado}, for every $i < b$, $J_i$ embeds into $(H, E, <, \meetlevel{\cdot}{\cdot})$, so for every $i < b$, there is a unique $j < 2$ such that $\langle j, i\rangle \in f{\hat{G} \choose F}$. Moreover, by sparsity of $H$, $j = 1$. Thus $\{\langle 1, i\rangle: i < b\} \subseteq f{H \choose F} \subseteq f{\hat{G} \choose F}$, and by cardinality, $f{\hat{G} \choose F} = \{\langle 1, i\rangle: i < b\}$.
\end{proof}

Let $F_0$ be the finite graph induced by the two first elements of $F$.

\begin{claim}\label{claim:weak-rado-triples-aca-2}
For every $\{x, y\} \in {\hat{G} \choose F_0}$ with $\meetlevel{x}{x} <_\Nb \meetlevel{y}{y} <_\Nb \meetlevel{z}{z}$, and every $e < \meetlevel{x}{x}$, if $\Phi_e^{\emptyset'}(e)\downarrow$ then $\Phi_e^{\emptyset'[\meetlevel{y}{y}]}(e)[\meetlevel{y}{y}]\downarrow$
\end{claim}
\begin{proof}
Since $\hat{G}$ is a Rado graph, we can find a $z$ with $\meetlevel{z}{z}$ sufficiently large such that $\{x, y, z\} \in {\hat{G} \choose F}$ and for every $e < \meetlevel{x}{x}$, if $\Phi_e^{\emptyset'}(e)\downarrow$ then $\Phi_e^{\emptyset'[\meetlevel{z}{z}]}(e)[\meetlevel{z}{z}]\downarrow$.
By Claim \ref{claim:weak-rado-triples-aca-1}, $h(\meetlevel{x}{x}, \meetlevel{y}{y}, \meetlevel{z}{z}) = 1$, so by definition of $h$, for every $e < \meetlevel{x}{x}$, $\Phi_e^{\emptyset'[\meetlevel{y}{y}]}(e)[\meetlevel{y}{y}] = \Phi_e^{\emptyset'[\meetlevel{z}{z}]}(e)[\meetlevel{z}{z}]$. In particular, if $\Phi_e^{\emptyset'}(e)\downarrow$ then $\Phi_e^{\emptyset'[\meetlevel{y}{y}]}(e)[\meetlevel{y}{y}]\downarrow$.
\end{proof}

We are now ready to prove that $\hat{G}$ is of PA degree relative to $\emptyset'$.
For this, we prove that $\hat{G}$ computes a completion of the universal partial $\emptyset'$-computable function $e \mapsto \Phi^{\emptyset'}_e(e)$. 
Given $e$, search $\hat{G}$-computably for a pair $\{x, y\} \in {\hat{G} \choose F_0}$ such that $e < \meetlevel{x}{x}$, and return $\Phi_e^{\emptyset'[\meetlevel{y}{y}]}(e)[\meetlevel{y}{y}]$ if it halts, otherwise return 0.
Such a pair $\{x, y\}$ is always found since $\hat{G}$ is a Rado graph. By Claim \ref{claim:weak-rado-triples-aca-2}, if $\Phi_e^{\emptyset'[\meetlevel{y}{y}]}(e)[\meetlevel{y}{y}]$ does not halt, then $\Phi_e^{\emptyset'}(e)\uparrow$, so this is a valid completion.
This completes the proof of \Cref{thm:weak-rado-triples-aca}.
\end{proof}

\begin{corollary}\label{cor:rado-graph-3-aca}
For every finite graph $F$ of size $3$, letting $b$ be the number of Joyce graphs isomorphic to $F$, $(\forall k)\RG^{F}_{k, b}$ implies $\ACA_0$ over $\RCA_0$.
\end{corollary}

\begin{corollary}\label{cor:rado-graph-complete-3-aca}
$(\forall k)\RG^{K_3}_{k, 16}$ implies $\ACA_0$ over $\RCA_0$.
\end{corollary}

\begin{theorem}\label{thm:rado-graph-tight-tuple-to-graph}
For every $n \geq 1$, every finite graph $F$ of size $n$,
let $b_n$ be the tight bound of $(\forall k)\RG^n_{k, b_n}$ and $c_n$ be the tight bound of $(\forall k)\RG^F_{k, c_n}$, that is, $b_n$ is the number of Joyce graphs of size $n$
and $c_n$ is the number of Joyce graphs isomorphic to $F$.
Then $(\forall k)\RG^n_{k, b_n}$ implies  $(\forall k)\RG^F_{k, c_n}$ over $\RCA_0$.
\end{theorem}
\begin{proof}
Let $f: {G \choose F} \to k$ be an instance of $(\forall k)\RG^F_{k, c_n}$.
Let $g: [G]^n \to b_n$ be the coloring witnessing the tightness of the bound $b_n$.
In particular, $g$ restricted to ${\hat{G} \choose F}$ uses exactly $c_n$ many colors.
Given $H \in [G]^n$, let $h(H) = \langle g(H), f(H) \rangle$ if the graph $H$ is isomorphic to $F$, and $h(H) = \langle g(H), \bot \rangle$ otherwise. 
By $(\forall k)\RG^n_{k, b_n}$, there is a Rado subgraph $\hat{G} \subseteq G$
such that $h[\hat{G}]^n$ has at most $b_n$ colors. By choice of $g$,
for every $i < b_n$, there is some $v$ such that $\langle i, v\rangle \in h[\hat{G}]^n$.
Thus for every $i < b_n$, there is exactly one $v$ such that $\langle i, v\rangle \in h[\hat{G}]^n$. In particular, $f$ restricted to ${\hat{G} \choose F}$ has at most $b_n$ many colors.
\end{proof}

\begin{corollary}
$(\forall k)\RG^3_{k, 112}$ implies $\ACA_0$ over $\RCA_0$.
\end{corollary}
\begin{proof}
Immediate by \Cref{cor:rado-graph-3-aca} and \Cref{thm:rado-graph-tight-tuple-to-graph}.
112 is the tight bound for the Rado graph theorem for triples, and 16 the tight bound for the Rado graph theorem restricted to the complete graph of size 3.
\end{proof}

We do not know whether $(\forall k)\RG^2_{k,4}$ implies $\RT{2}{2}$ over $\RCA_0$. \Cref{thm:rg284-implies-srt22} however is a partial result towards that direction. Ramsey's theorem for pairs $\mathsf{RT}^2_k$ was decomposed by Cholak, Jockusch and Slaman~\cite{Cholak2001strength} into a stability version ($\mathsf{SRT}^2_k$) and a cohesiveness principle ($\COH$) in order to simplify the computability-theoretic analysis of the theorem.

\begin{definition}
An infinite set $C$ is \emph{cohesive}\index{set!cohesive} for a sequence of sets $R_0, R_1, \dots \subseteq \N$ if for every $n$, $C \subseteq^{*} R_n$ or $C \subseteq^{*} \overline{R}_n$. A coloring $f: [\omega]^2 \to k$ is \emph{stable}\index{coloring!stable} if for every $x$, $\lim_y f(x, y)$ exists.
\end{definition}

\begin{statement}[Cohesiveness]
$\COH$\index{statement!$\mathsf{COH}$} is the statement \qt{Every countable sequence of set has an infinite cohesive set}.s
\end{statement}

\begin{statement}[Stable Ramsey's theorem for pairs]
$\mathsf{SRT}^2_k$\index{statement!$\mathsf{SRT}^2_2$} is the restriction of $\mathsf{RT}^2_k$ to stable colorings.
\end{statement}

Cholak, Jockusch and Slaman~\cite[Lemma 7.11]{Cholak2001strength} and Mileti~\cite[Corollary A.1.4]{Mileti2004Partition} proved the equivalence between $\RT 2k$ and $\mathsf{SRT}^2_k \wedge \COH$ over $\RCA_0$.
The following theorem shows that  $(\forall k)\RG^2_{k,4}$ implies $\mathsf{SRT}^2_2$, hence any proof of separation would be a proof of separation of  $(\forall k)\RG^2_{k,4}$ from $\COH$. 

\begin{theorem}\label{thm:rg284-implies-srt22}
$\RG^2_{8,4}$ implies $\mathsf{SRT}^2_2$ over $\RCA_0$.
\end{theorem}
\begin{proof}
Let $f: [\omega]^2 \to 2$ be a stable coloring. 
Fix a computable coded Joyce Rado graph $(G, \ltlex, \Epn, |\cdot \meet \cdot|)$,
let $F_0, F_1, F_2, F_3$ be the 4 Joyce graphs of size 2, 
and define an instance $g: [G]^2 \to 4 \times 2$ of $\RG^2_{8,4}$ by $g(\{\sigma, \tau\}) = (i, f(\{|\sigma|, |\tau|\}))$ where $F_i$ is the unique Joyce graph isomorphic to $(\{\sigma, \tau\}, \ltlex, \Epn, |\cdot \meet \cdot|)$. 

Let $H \subseteq G$ be a $\RG^2_{8,4}$-solution to $g$, that is, 
$(H, \Epn)$ is a Rado graph and $|g[H]^2| \leq 4$.
By \Cref{th:4.1LaflammeRado}, for every $i < 4$, there is an embedding from $F_i$ to $(H, \ltlex, \Epn, |\cdot \meet \cdot|)$, so for every $i < 4$, there is a unique value $v(i) < 2$ such that $(i, v(i)) \in g[H]^2$. Thus $g[H]^2 = \{(i, v(i)): i < 4\}$.

\begin{claim}\label{claim:rg284-implies-srt22-monochromatic}
$g[H]^2 = \{(i, j): i < 4\}$ for some $j < 2$.
\end{claim}
\begin{proof}
Fix $i_0 < i_1 < 4$. We need to show that $v(i_0) = v(i_1)$.
By \Cref{th:4.1LaflammeRado}, there is an embedding $h$ of a Joyce blossom graph $B$ in $(H, \ltlex, \Epn, |\cdot \meet \cdot|)$. One can find some $\sigma \in B$, $\tau_0, \tau_1 \in B$
such that $\{\sigma, \tau_0\}, \ltlex, \Epn, |\cdot \meet \cdot|)$ and $\{\sigma, \tau_1\}, \ltlex, \Epn, |\cdot \meet \cdot|)$ are isomorphic to $F_{i_0}$ and $F_{i_1}$, respectively.
Then $g(\{h(\sigma), h(\tau_0)\}) = (i_0, v(i_0))$ and $g(\{h(\sigma), h(\tau_1)\}) = (i_1, v(i_1))$, so $f(\{h(\sigma), h(\tau_0)\}) = v(i_0)$ and $f(\{h(\sigma), h(\tau_1)\}) = v(i_1)$.
Moreover, $\tau_0$ and $\tau_1$ can be chosen so that $|h(\tau_0)|$ and $|h(\tau_1)|$ is large enough to witness stability of $\lim_s f(|h(\sigma)|, s)$.
Then $f(\{h(\sigma), h(\tau_0)\}) = v(i_0) = f(\{h(\sigma), h(\tau_1)\}) = v(i_1)$.
\end{proof}

Let $j < 2$ be such that Claim \ref{claim:rg284-implies-srt22-monochromatic} holds.
It follows that $Y = \{|\sigma|: \sigma \in H \}$ is $f$-homogeneous for color $j$.
This completes the proof of \Cref{thm:rg284-implies-srt22}.
\end{proof}

We now prove that for every $\ell \geq 1$, $\RT{2}{2}$ does not imply $(\forall k)\RG^2_{k,\ell}$ over $\RCA_0$. For this, we need two essential notions in computability theory, namely, lowness and hyperimmunity. Lowness is a weakness property over Turing degrees, saying informally that a low Turing degree behaves like the computable Turing degree from the viewpoint of the halting set. Hyperimmunity is aa strength property about the ability to compute fast-growing functions, not dominated by any computable function.

\begin{definition}\index{low set}\index{set!low}
  A set $A$ is \emph{low} if $A'\equiv_T\emptyset'$. A set $A$ is \emph{low relative to $X$} for a set $X$ if $(A\oplus X)'\equiv_T X'$.
\end{definition}

\begin{definition}
A function $f: \N \to \N$ is \emph{hyperimmune}\index{hyperimmune!function}\index{function!hyperimmune} relative to $X$ if it is not dominated by any $X$-computable function. 
An infinite set $H \subseteq \Nb$ is \emph{hyperimmune}\index{hyperimmune!set}\index{set!hyperimmune} relative to $X$ if its principal function, that is, the function which to $n$ associates the $n$th element of $H$, is hyperimmune relative to $X$.
\end{definition}

\begin{theorem}\label{th:rg2-no-low-over-zp}
  Let $P$ be low relative to $\emptyset'$. For every $\ell$, there exists a computable instance of $\RG^2_{\ell+1,\ell}$ with no solution computable in $P$.
\end{theorem}
\begin{proof}
  We first need the following definition and lemma to build the instance.
  \begin{lemma}
    There exists a $\Delta^0_3$ coloring $g:\Nb\to\ell+1$ such that for every $k\leq\ell$, the set $A_k=\{n\in\Nb:g(n)\neq k\}$ is hyperimmune relative to $P$.
  \end{lemma}
  \begin{proof}
%    \pelliot{todo}
    We prove that there exists a $\Delta^0_2$ coloring $g:\Nb\to\ell+1$ such that for all $k$, $\{n:g(n)\neq k\}$ is hyperimmune. The relativization to $P$ of this proof yields the result of the lemma, as a $\Delta^0_2$ relative to $P$ is a $\Delta^0_3$ set by lowness relative to $\emptyset'$ of $P$.

    We build uniformly in $\emptyset'$ a sequence $\sigma_s$ of compatible strings of increasing length. Start with $\sigma_0$ the empty string. Suppose $\sigma_s$ is defined, with $s=\langle k,e\rangle$. Let $n_s=|\{i<|\sigma_s|:\sigma_s(i)\neq k\}|$ Using $\emptyset'$, we can decide if $\varphi_e(n_s)\downarrow$. If so we define $\sigma_{s+1}$ to be $\sigma_s$ followed by the string consisting of $\varphi_e(n_s)+1$ times $k$. Otherwise, $\sigma_{s+1}$ is $\sigma_s$ followed by $0$.

    Define $g=\bigcup_s \sigma_s$. By construction, for $s=\langle k, e\rangle$, $\sigma_{s+1}$ diagonalizes against $\varphi_e$ dominating $\{n\in\Nb:g(n)\neq k\}$. Thus, for any $k<\ell$, $\{n\in\Nb:g(n)\neq k\}$ is hyperimmune.
  \end{proof}
  Fix $g$ as in the lemma. Let $(g_{s,t})_{s,t\in\Nb}$ be a $\Delta^0_3$ approximation of $g$. One can always arrange the approximation so that for every $s,t$, $g_{s,t}$ is a coloring of $\Nb$ in $\ell+1$ colors.
  %This yields a family of colorings of $\sigma_{s,t}:\Nb\to\ell+1$ such that $\sigma_{s,t}(n)$ is the only $k\leq\ell$ with $n\in X_{s,t}^k$.
  Let $\Gb=(G, \Epn)$ with $G\subseteq\cantor$ a coded Rado graph. Define $f$ to be the following: for every $\sigma,\tau\in\cantor$ with $|\sigma|<|\tau|$,
  \[ f(\sigma,\tau)=g_{|\sigma|,|\tau|}(|\sigma\meet\tau|)\]
  Let $S\subseteq\cantor$ be such that $(S,\Epn)$ is a Rado graph, and $f$ takes at most $\ell$ colors on $[S]^2$. Let $k$ be the avoided color. We prove that $S$ computes a function bounding the principal function of $A_k$, so by hyperimmunity relative to $P$ of $A_k$, $P$ cannot compute $S$.

  Let $h\leq_T S$ be the function so that $h(n)$ is the smallest such that $S\cap2^{<h(n)}$ has $n$ elements. We prove that for every $n$, $A_k$ contains at least $n$ elements smaller than $h(n)$. Indeed, let $e_0,\dots,e_{n-1}$ be the $n$ elements of $S\cap 2^{<h(n)}$. Let $(\tau_j)_{j<2^n}$ be an enumeration of the set $2^n$, we say that a string $\rho$ realize $\tau_j$ if for all $i<n$ $(\rho\Epn e_i\iff\tau_j(i)=1)$, note that for any $j$ there are infinitely many elements of $S$ realizing $\tau_j$, so there are some of arbitrarily big length. Let $\rho_0$ be the first element of $S$ such that $g_{s}(m)=g(m)$ for all $s\geq |\rho_0|$ and $m<h(n)$, and $\rho_0$ realizes $\tau_0$. % Let $s_0$ be such that $c_{s}(m)=c(m)$ for all $s\geq s_0$ and $m\leq g(n)$.
  If $\rho_j$ is defined, then define $\rho_{j+1}$ to be the first element of $S$ realizing $\tau_{j+1}$ such that for all $t\geq|\rho_{j+1}|$, $g_{|\rho_j|, t}(m)=g(m)$ for all $m<h(n)$.

  This defines $2^n$ many different strings $\rho_j$ and thus $2^n-1$ many meets, which must be of height below $h(n)$ by construction: given $\rho_{j_0},\rho_{j_1}$, they must differ at $|e_i|$ for some $i<n$ such that $\tau_{j_0}(i)\neq \tau_{j_1}(i)$.
  Moreover, as $\rho_j\in S$ for all $j$, if $|\rho_{j_0}|<|\rho_{j_1}|$, we have $g_{|\rho_{j_0}|,|\rho_{j_1}|}(|\rho_{j_0}\meet\rho_{j_1}|)\neq k$, but also $g_{|\rho_{j_0}|,|\rho_{j_1}|}(|\rho_{j_0}\meet\rho_{j_1}|)=g(|\rho_{j_0}\meet\rho_{j_1}|)$ as $|\rho_{j_0}\meet\rho_{j_1}|<h(n)$.

  Thus, for every $j_0\neq j_1<2^n$, $|\rho_{j_0}\meet\rho_{j_1}|\in A_k$ and is below $h(n)$. There are $2^n-1>n$ many such meet. So $h$ is a function computable in $S$ that dominates the principal function of $A_k$. As $A_k$ is hyperimmune relative to $P$, $P$ cannot compute $h$, and so cannot compute $S$.
\end{proof}

\begin{lemma}[{Simpson~\cite[Theorem 6.5]{Simpson1977Degrees}}]\label{lem:pa-degree-dense}
For every pair of sets $A$ and $C$ such that $A$ is of PA degree over $C$,
there is a set $B$ such that $A$ is PA over $B$ and $B$ is PA over $C$.
\end{lemma}
\begin{proof}
Fix a universal $\Pi^0_1$ class functional $\Cc \subseteq 2^\omega$ such that for every $X$, $\Cc^X \neq \emptyset$ and all the members of $\Cc^X$ are of PA degree relative to $X$. For example, take $\Cc^X$ to be the class of all $\{0,1\}$-valued DNC functions relative to $X$. Then the class $\Dc = \{B\oplus Z: B \in \Cc^{C} \mbox{ and } Z \in \Cc^B \}$ is a non-empty $\Pi^{0,C}_1$ class. In particular, $A$ computes some $B \oplus Z \in \Dc$, and since $Z$ is of PA degree relative to $B$ and $B$ is of PA degree relative to $C$, the result follows.
\end{proof}

\begin{lemma}\label{lem:model-rt22-jump-pa-zp}
For every set $P$ of PA degree over $\emptyset'$, there exists an $\omega$-model $\Mc$ of $\RT 22$ such that for every $X \in \Mc$, $X' \leq_T P$.
\end{lemma}
\begin{proof}
Fix $P$. We inductively define $A_0, A_1, \dots \subseteq \Nb$ as follows.
Let $A_0 = \emptyset$, and suppose we have defined $A_s$ for some $s \in \Nb$ and that $P$ is of Pa degree above $A_s'$. By \Cref{lem:pa-degree-dense}, there is some $Q$ be such that $P$ is PA relative to $Q$ and $Q$ is PA relative to $A_s'$. 
If $s \neq (e,t)$ for some $e \in \Nb$ and some $t < s$, or if $\Phi_e^{A_t}$ is not a coloring $f: [\omega]^2 \to 2$, then let $A_{s+1} = A_s$. Otherwise, by Cholak, Jockusch and Slaman~\cite{Cholak2001strength} (see Hirschfeldt~\cite[Corollary 6.58]{Hirschfeldt2015Slicing} for an explicit formulation), there is an infinite $f$-homogeneous set $H$ such that $(H \oplus A_s)' \leq_T Q$. Let $A_{s+1} = A_s \oplus H$. In particular, $A_{s+1}' \leq_T Q$, so $P$ is of PA degree over $A_{s+1}'$. s

Let $\Sc = \{ Z: (\exists s)[Z \leq_T A_s]\}$, which is a Turing ideal since $A_t \leq_T A_s$ for all $t \leq s$. By construction, if $f: [\omega]^2 \to 2$ is any instance of $\RT 22$ in $\Sc$ then $\Sc$ contains a solution to $f$. (Indeed, if $f = \Phi_e^{A_t}$, say, leet $s = (e, t)$; then a solution to $X$ is computable from $A_{s+1}$.) It follows that $\Mc = (\Nb, \Sc)$ is a model of $\RCA_0 \wedge \RT 22$, and by construction, for every $X \in \Sc$, $X' \leq_T P$. 
\end{proof}

We are now ready to prove our separation theorem.

\begin{theorem}
For every $\ell \geq 1$, $\RT{2}{2}$ does not imply $(\forall k)\RG^2_{k,\ell}$ over $\RCA_0$.
\end{theorem}
\begin{proof}
By the low basis theorem relativized to $\emptyset'$ (see Jockusch and Soare \cite[Theorem 2.1]{Jockusch197201}), there is a set $P$ of PA degree over $\emptyset'$, such that $P' \leq_T \emptyset''$. 
By \Cref{lem:model-rt22-jump-pa-zp}, there is an $\omega$-model $\Mc$ of $\RT 22$ such that for every $X \in \Mc$, $X' \leq_T P$.
By \Cref{th:rg2-no-low-over-zp}, there is a computable instance $f$ of $(\forall k)\RG^2_{k, \ell}$ with no $P$-computable solution. In particular, $f \in \Mc$, but $\Mc$ does not contain a solution to $f$, so $\Mc$ is not a model of $(\forall k)\RG^2_{k, \ell}$.
\end{proof}

%%% Local Variables:
%%% mode: latex
%%% TeX-master: "../embryon"
%%% End:

\chapter{A generalized tree theorem}\label{sec:GenCHMTT}

%\todo[inline]{I (Ludovic) can make the blabla about the tree theorem with references in literature to motivate this section, unless Benoit, you feel like to make it.}

In this chapter, we study the principle $\mathrm{CHMTT^n_{k,l}}$ stating that given a $k$-coloring of $[2^{<\omega}]^n$, there is a subtree $S\subseteq 2^{<\omega}$ such that $(S, \preceq)$ is isomorphic to $(2^{<\omega}, \preceq)$ and such that $[S]^n$ uses at most $l$ colors. Note that we do not require $S$ to be a strong subtree of $T$ or even $S$ to be meet-closed --- thus enlarging a bit for this chapter the original definition of a tree given with \cref{def:trees}. 

The existence for every $n$ of a finite big Ramsey degree associated with these structures --- a smallest number $l_n$ such that $\mathrm{CHMTT^n_{k,l_n}}$ holds for every $k$ --- easily follows from the Milliken's tree theorem. We will try to identify more precisely the specific values of these big Ramsey degrees $l_n$, a new sequence of numbers, which does not seem to have appeared before in combinatorics.

In order to pursue this study, we introduce first a simpler principle, for which we now require the subtree to be a strong subtree.

\begin{theorem}[Strong generalized CHM tree theorem]\index{generalized CHM tree theorem!strong}\index{tree theorem!strong generalized CHM tree theorem}
For every $n \geq 1$ there exists $\ell \geq 1$ such that for every $k \geq 1$ and every $f: [2^{<\omega}]^n \to k$ there is a strong subtree $S \subseteq 2^{<\omega}$ such that $|f ([S]^n)| \leq \ell$. 
\end{theorem}
\begin{statement}\index{statement!$\mathrm{SCHMTT^n_{k,l}}$}
  We call $\mathrm{SCHMTT^n_{k,l}}$ the statement of the Strong generalized CHM tree theorem.
\end{statement}

Note that both $\mathrm{SCHMTT^n_{k,l}}$ and $\mathrm{CHMTT^n_{k,l}}$ would work exactly the same way if we start from a coloring of any perfect tree $T$ rather than $2^{<\omega}$: via an isomorphism between $T$ and $2^{<\omega}$, the theorem applied to $2^{<\omega}$ also gives via the isomorphism a solution for $T$. We start by introducing the notion of embedding types, useful in the conduct of our study.

\section{Embedding types}

We shall try to identify what can be used by a coloring of $[2^{<\omega}]^n$ to identify some structure one will never be able to remove in any strong subtree. A first step for that is the identification of the concept of embedding type, for which we introduce the following preliminary notions.

\begin{definition}
Let $S$ be a set of strings.
\begin{enumerate}\index{meet-closed}\index{level-closed}
\item $S$ is \emph{meet-closed} if for every $\sigma,\tau\in S$, $\sigma \wedge \tau \in S$. 
\item $S$ is \emph{level-closed} if for every $\sigma,\tau\in S$, $\tau\uh|\sigma|\in S$.
\end{enumerate}
\end{definition}

We will be interested in finite trees which are both meet-closed and level closed. 

\begin{definition}[closure]
Let $S$ be a set of strings.
\begin{enumerate}\index{meet closure}\index{level closure}\index{full closure}\index{closure!meet}\index{closure!full}\index{closure!level}\index{$A^\wedge$}\index{$\lvlclosure A$}\index{$A^{\mathrm{cl}}$}
\item The \emph{meet closure} of $S$ is the set $S^\wedge=\{\sigma \wedge \tau:\sigma,\tau\in S\}$.
\item The \emph{level closure} of $S$ is the set $\lvlclosure S=\{\sigma\uh|\tau|:\sigma,\tau\in S\}$.
\item The \emph{full closure} of $S$ is the set $S^{\mathrm{cl}} = \lvlclosure {(S^\wedge)}$.
\end{enumerate}
\end{definition}

Note that $S \subseteq S^\wedge$ and $S \subseteq \lvlclosure S$ by taking $\sigma=\tau$ in the above definitions. Any strong subtree of $2^{<\omega}$ is meet-closed and level-closed but not conversely, as witnessed by the following example $S=\{\emptystr;0;00;01;1;11\}$. In Figure~\ref{fig:some-trees}, we give an example of a subtree, and its full closure.
\begin{figure}[h]
  \centering
  \centering
	\begin{tikzpicture}[scale=1.5,font=\footnotesize]
		\tikzset{
			empty node/.style={circle,inner sep=0,fill=none},
			solid node/.style={circle,draw,inner sep=1.5,fill=black},
			hollow node/.style={circle,draw,inner sep=1.5,fill=white},
			gray node/.style={circle,draw={rgb:black,1;white,4},inner sep=1.5,fill={rgb:black,1;white,4}}
		}
		\tikzset{snake it/.style={decorate, decoration=snake, line cap=round}}
		\tikzset{gray line/.style={line cap=round,thick,color={rgb:black,1;white,4}}}
		\tikzset{thick line/.style={line cap=round,rounded corners=0.1mm,thick}}
		\node (a)[solid node] at (0,0) {};
		\node (a1)[gray node] at (0.3,0.5) {};
		\node (a0)[gray node] at (-0.3,0.5) {};
		\node (a00)[gray node] at (-0.45,1) {};
		\node (a01)[solid node] at (-0.15,1) {};
		\node (a10)[solid node] at (0.15,1) {};
		\node (a11)[solid node] at (0.45,1) {};
		\draw[-,thick line] (a) to (-0.3,0.5) to (a01);
		\draw[-,thick line] (a) to (0.3,0.5) to (a10);
		\draw[-,thick line] (a) to (0.3,0.5) to (a11);
		\draw[-,gray line] (a0) to (a00);
		\draw[-,thick line] (0,-0.2) to (a);
		\node at (0,-0.5) {(a)};
	\end{tikzpicture}
	\hspace{5mm}
	\begin{tikzpicture}[scale=1.5,font=\footnotesize]
		\tikzset{
			empty node/.style={circle,inner sep=0,fill=none},
			solid node/.style={circle,draw,inner sep=1.5,fill=black},
			hollow node/.style={circle,draw,inner sep=1.5,fill=white},
			gray node/.style={circle,draw={rgb:black,1;white,4},inner sep=1.5,fill={rgb:black,1;white,4}}
		}
		\tikzset{snake it/.style={decorate, decoration=snake, line cap=round}}
		\tikzset{gray line/.style={line cap=round,thick,color={rgb:black,1;white,4}}}
		\tikzset{thick line/.style={line cap=round,rounded corners=0.1mm,thick}}
		\node (a)[solid node] at (0,0) {};
		\node (a1)[solid node] at (0.3,0.5) {};
		\node (a0)[gray node] at (-0.3,0.5) {};
		\node (a00)[gray node] at (-0.45,1) {};
		\node (a01)[solid node] at (-0.15,1) {};
		\node (a10)[solid node] at (0.15,1) {};
		\node (a11)[solid node] at (0.45,1) {};
		\draw[-,thick line] (a) to (-0.3,0.5) to (a01);
		\draw[-,thick line] (a) to (0.3,0.5) to (a10);
		\draw[-,thick line] (a) to (0.3,0.5) to (a11);
		\draw[-,gray line] (a0) to (a00);
		\draw[-,thick line] (0,-0.2) to (a);
		\node at (0,-0.5) {(b)};
	\end{tikzpicture}
	\hspace{5mm}
	\begin{tikzpicture}[scale=1.5,font=\footnotesize]
		\tikzset{
			empty node/.style={circle,inner sep=0,fill=none},
			solid node/.style={circle,draw,inner sep=1.5,fill=black},
			hollow node/.style={circle,draw,inner sep=1.5,fill=white},
			gray node/.style={circle,draw={rgb:black,1;white,4},inner sep=1.5,fill={rgb:black,1;white,4}}
		}
		\tikzset{snake it/.style={decorate, decoration=snake, line cap=round}}
		\tikzset{gray line/.style={line cap=round,thick,color={rgb:black,1;white,4}}}
		\tikzset{thick line/.style={line cap=round,rounded corners=0.1mm,thick}}
		\node (a)[solid node] at (0,0) {};
		\node (a1)[solid node] at (0.3,0.5) {};
		\node (a0)[solid node] at (-0.3,0.5) {};
		\node (a00)[gray node] at (-0.45,1) {};
		\node (a01)[solid node] at (-0.15,1) {};
		\node (a10)[solid node] at (0.15,1) {};
		\node (a11)[solid node] at (0.45,1) {};
		\draw[-,thick line] (a) to (-0.3,0.5) to (a01);
		\draw[-,thick line] (a) to (0.3,0.5) to (a10);
		\draw[-,thick line] (a) to (0.3,0.5) to (a11);
		\draw[-,gray line] (a0) to (a00);
		\draw[-,thick line] (0,-0.2) to (a);
		\node at (0,-0.5) {(b)};
	\end{tikzpicture}
  \caption{The set of nodes (a) is level-closed but not meet-closed. The set of nodes (b) is the meet-closure of (a). Note that it is now not level-closed. The set of nodes (c) is the level-closure of (b) which corresponds to the full closure of (a).}
  \label{fig:some-trees}
\end{figure}
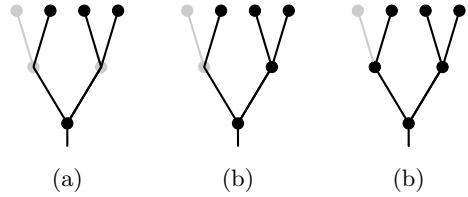

%\begin{lemma}The following holds:
%  \begin{enumerate}
%%  \item If $T$ is a tree and $S$ is a subtree of $T$, then $S^\meet$ is a $\meet$-closed subtree of $T$, and $\lvlclosure{S}$ is a level-closed subtree of $T$.
%  \item If $T$ is a tree and $S$ is a subtree of $T$, then $\lvlclosure{S}$ is a level-closed subtree of $T$.
%%  \item For every tree $T$ and subtree $S$ of $T$, $S\subseteq S^\meet\subseteq T$ and $S\subseteq \lvlclosure S\subseteq T$.
%  \item For every tree $T$ and subtree $S$ of $T$, $S\subseteq \lvlclosure S\subseteq T$.
%  \end{enumerate}
%\end{lemma}
  
The idea is the following: given a set of strings $S = \{\sigma_1, \dots, \sigma_n\} \in [2^{<\omega}]^n$, one can easily compute the tree $S^{\mathrm{cl}}$. A coloring can then identify which \emph{type} of tree arise from $S$ and give a different color to each of them. The number of these type is defined below as the \emph{embedding types} that we now formally define.
  
\begin{definition}\index{tree!strong isomorphism}\index{isomorphism!fully closed tree}\index{embedding type}\index{type!embedding}
Two finite fully closed trees $F_0,F_1 \subseteq 2^{<\omega}$ are \emph{strongly isomorphic} if there is a bijection $f:F_0 \rightarrow F_1$ such that $\sigma i \preceq \tau \leftrightarrow f(\sigma) i \preceq f(\tau)$ for any $\sigma, \tau \in F_0$. The \emph{embedding types} are the equivalence classes of the strong isomorphism relation on finite fully closed trees.
\end{definition}  
 
Any embedding type has a minimal element with respect its height. We usually use this minimal element as a canonical representative of the class.

\Cref{fig:same-emb-types,fig:some-emb-height,fig:some-emb-types} illustrate the notion of embedding types. \Cref{fig:some-emb-types} consists of example of different embedding types. \Cref{fig:same-emb-types} shows several level-closed subtrees with the same embedding type. \Cref{fig:some-emb-height} illustrates the height of an embedding type.

\begin{figure}[h]
%  \centering \includegraphics[scale=0.5]{embTypeExamples.eps}
  \centering
  \begin{tikzpicture}[scale=1.5,font=\footnotesize]
		\tikzset{
			empty node/.style={circle,inner sep=0,fill=none},
			solid node/.style={circle,draw,inner sep=1.5,fill=black},
			hollow node/.style={circle,draw,inner sep=1.5,fill=white},
			gray node/.style={circle,draw={rgb:black,1;white,4},inner sep=1.5,fill={rgb:black,1;white,4}}
		}
		\tikzset{snake it/.style={decorate, decoration=snake, line cap=round}}
		\tikzset{gray line/.style={line cap=round,thick,color={rgb:black,1;white,4}}}
		\tikzset{thick line/.style={line cap=round,rounded corners=0.1mm,thick}}
		\node (a)[solid node] at (0,0) {};
		\node (a1)[solid node] at (0.3,0.5) {};
		\node (a0)[gray node] at (-0.3,0.5) {};
		\node (a00)[gray node] at (-0.45,1) {};
		\node (a01)[gray node] at (-0.15,1) {};
		\node (a10)[solid node] at (0.15,1) {};
		\node (a11)[gray node] at (0.45,1) {};
		\begin{pgfonlayer}{background}
		\draw[-,gray line] (a) to (a0);
		\draw[-,thick line] (a) to (a1);
		\draw[-,gray line] (a0) to (a00);
		\draw[-,gray line] (a0) to (a01);
		\draw[-,thick line] (a1) to (a10);
		\draw[-,gray line] (a1) to (a11);
		\end{pgfonlayer}
		\node at (0,-0.5) {(d)};
	\end{tikzpicture}
	\hspace{5mm}
	\begin{tikzpicture}[scale=1.5,font=\footnotesize]
		\tikzset{
			empty node/.style={circle,inner sep=0,fill=none},
			solid node/.style={circle,draw,inner sep=1.5,fill=black},
			hollow node/.style={circle,draw,inner sep=1.5,fill=white},
			gray node/.style={circle,draw={rgb:black,1;white,4},inner sep=1.5,fill={rgb:black,1;white,4}}
		}
		\tikzset{snake it/.style={decorate, decoration=snake, line cap=round}}
		\tikzset{gray line/.style={line cap=round,thick,color={rgb:black,1;white,4}}}
		\tikzset{thick line/.style={line cap=round,rounded corners=0.1mm,thick}}
		\node (a)[solid node] at (0,0) {};
		\node (a1)[solid node] at (0.3,0.5) {};
		\node (a0)[gray node] at (-0.3,0.5) {};
		\node (a00)[gray node] at (-0.45,1) {};
		\node (a01)[gray node] at (-0.15,1) {};
		\node (a10)[gray node] at (0.15,1) {};
		\node (a11)[solid node] at (0.45,1) {};
		\begin{pgfonlayer}{background}
		\draw[-,gray line] (a) to (a0);
		\draw[-,thick line] (a) to (a1);
		\draw[-,gray line] (a0) to (a00);
		\draw[-,gray line] (a0) to (a01);
		\draw[-,gray line] (a1) to (a10);
		\draw[-,thick line] (a1) to (a11);
		\end{pgfonlayer}
		\node at (0,-0.5) {(e)};
	\end{tikzpicture}
	\hspace{5mm}
	\begin{tikzpicture}[scale=1.5,font=\footnotesize]
		\tikzset{
			empty node/.style={circle,inner sep=0,fill=none},
			solid node/.style={circle,draw,inner sep=1.5,fill=black},
			hollow node/.style={circle,draw,inner sep=1.5,fill=white},
			gray node/.style={circle,draw={rgb:black,1;white,4},inner sep=1.5,fill={rgb:black,1;white,4}}
		}
		\tikzset{snake it/.style={decorate, decoration=snake, line cap=round}}
		\tikzset{gray line/.style={line cap=round,thick,color={rgb:black,1;white,4}}}
		\tikzset{thick line/.style={line cap=round,rounded corners=0.1mm,thick}}
		\node (a)[solid node] at (0,0) {};
		\node (a1)[solid node] at (0.3,0.5) {};
		\node (a0)[solid node] at (-0.3,0.5) {};
		\node (a00)[gray node] at (-0.45,1) {};
		\node (a01)[gray node] at (-0.15,1) {};
		\node (a10)[solid node] at (0.15,1) {};
		\node (a11)[solid node] at (0.45,1) {};
		\begin{pgfonlayer}{background}
		\draw[-,thick line] (a) to (a0);
		\draw[-,thick line] (a) to (a1);
		\draw[-,gray line] (a0) to (a00);
		\draw[-,gray line] (a0) to (a01);
		\draw[-,thick line] (a1) to (a10);
		\draw[-,thick line] (a1) to (a11);
		\end{pgfonlayer}
		\node at (0,-0.5) {(f)};
	\end{tikzpicture}
	\hspace{5mm}
	\begin{tikzpicture}[scale=1.5,font=\footnotesize]
		\tikzset{
			empty node/.style={circle,inner sep=0,fill=none},
			solid node/.style={circle,draw,inner sep=1.5,fill=black},
			hollow node/.style={circle,draw,inner sep=1.5,fill=white},
			gray node/.style={circle,draw={rgb:black,1;white,4},inner sep=1.5,fill={rgb:black,1;white,4}}
		}
		\tikzset{snake it/.style={decorate, decoration=snake, line cap=round}}
		\tikzset{gray line/.style={line cap=round,thick,color={rgb:black,1;white,4}}}
		\tikzset{thick line/.style={line cap=round,rounded corners=0.1mm,thick}}
		\node (a)[solid node] at (0,0) {};
		\node (a1)[gray node] at (0.3,0.5) {};
		\node (a0)[solid node] at (-0.3,0.5) {};
		\node (a00)[solid node] at (-0.45,1) {};
		\node (a01)[solid node] at (-0.15,1) {};
		\node (a10)[gray node] at (0.15,1) {};
		\node (a11)[gray node] at (0.45,1) {};
		\begin{pgfonlayer}{background}
		\draw[-,thick line] (a) to (a0);
		\draw[-,gray line] (a) to (a1);
		\draw[-,thick line] (a0) to (a00);
		\draw[-,thick line] (a0) to (a01);
		\draw[-,gray line] (a1) to (a10);
		\draw[-,gray line] (a1) to (a11);
		\end{pgfonlayer}
		\node at (0,-0.5) {(g)};
	\end{tikzpicture}
	\hspace{5mm}
	\begin{tikzpicture}[scale=1.5,font=\footnotesize]
		\tikzset{
			empty node/.style={circle,inner sep=0,fill=none},
			solid node/.style={circle,draw,inner sep=1.5,fill=black},
			hollow node/.style={circle,draw,inner sep=1.5,fill=white},
			gray node/.style={circle,draw={rgb:black,1;white,4},inner sep=1.5,fill={rgb:black,1;white,4}}
		}
		\tikzset{snake it/.style={decorate, decoration=snake, line cap=round}}
		\tikzset{gray line/.style={line cap=round,thick,color={rgb:black,1;white,4}}}
		\tikzset{thick line/.style={line cap=round,rounded corners=0.1mm,thick}}
		\node (a)[solid node] at (0,0) {};
		\node (a1)[solid node] at (0.3,0.5) {};
		\node (a0)[solid node] at (-0.3,0.5) {};
		\node (a00)[solid node] at (-0.45,1) {};
		\node (a01)[solid node] at (-0.15,1) {};
		\node (a10)[solid node] at (0.15,1) {};
		\node (a11)[solid node] at (0.45,1) {};
		\begin{pgfonlayer}{background}
		\draw[-,thick line] (a) to (a0);
		\draw[-,thick line] (a) to (a1);
		\draw[-,thick line] (a0) to (a00);
		\draw[-,thick line] (a0) to (a01);
		\draw[-,thick line] (a1) to (a10);
		\draw[-,thick line] (a1) to (a11);
		\end{pgfonlayer}
		\node at (0,-0.5) {(h)};
	\end{tikzpicture}
  \caption{
    A few embedding types of height 3, the underlying grey tree being $\cantor$. % , by their smallest representative, a strong subtree
    All of them are different.
  }
  \label{fig:some-emb-types}
\end{figure}
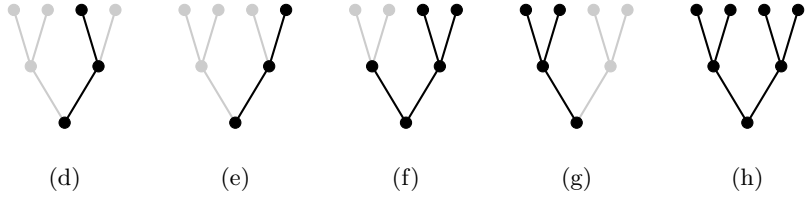

%\begin{figure}[h]
%%  \centering \includegraphics[scale=0.5]{embTypeExamples.eps}
%  \centering \includegraphics[scale=0.4]{figures/embType3Branching.pdf}
%  \caption{
%    Another set of different embedding types, in a more branching underlying tree. % , by their smallest representative, a strong subtree
%    All of them are again different.
%  }
%  \label{fig:some-emb-types-3-branching}
%\end{figure}

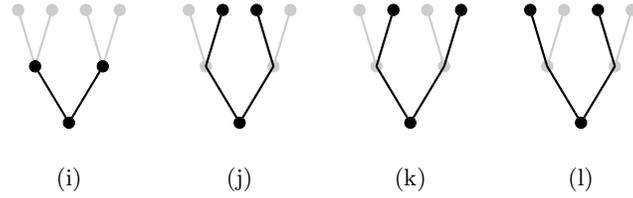
\begin{figure}[h]
%  \centering \includegraphics[scale=0.5]{embTypeExamples.eps}
  \centering
  \begin{tikzpicture}[scale=1.5,font=\footnotesize]
		\tikzset{
			empty node/.style={circle,inner sep=0,fill=none},
			solid node/.style={circle,draw,inner sep=1.5,fill=black},
			hollow node/.style={circle,draw,inner sep=1.5,fill=white},
			gray node/.style={circle,draw={rgb:black,1;white,4},inner sep=1.5,fill={rgb:black,1;white,4}}
		}
		\tikzset{snake it/.style={decorate, decoration=snake, line cap=round}}
		\tikzset{gray line/.style={line cap=round,thick,color={rgb:black,1;white,4}}}
		\tikzset{thick line/.style={line cap=round,rounded corners=0.1mm,thick}}
		\node (a)[solid node] at (0,0) {};
		\node (a1)[solid node] at (0.3,0.5) {};
		\node (a0)[solid node] at (-0.3,0.5) {};
		\node (a00)[gray node] at (-0.45,1) {};
		\node (a01)[gray node] at (-0.15,1) {};
		\node (a10)[gray node] at (0.15,1) {};
		\node (a11)[gray node] at (0.45,1) {};
		\begin{pgfonlayer}{background}
		\draw[-,thick line] (a) to (a0);
		\draw[-,thick line] (a) to (a1);
		\draw[-,gray line] (a0) to (a00);
		\draw[-,gray line] (a0) to (a01);
		\draw[-,gray line] (a1) to (a10);
		\draw[-,gray line] (a1) to (a11);
		\end{pgfonlayer}
		\node at (0,-0.5) {(i)};
	\end{tikzpicture}
	\hspace{5mm}
	\begin{tikzpicture}[scale=1.5,font=\footnotesize]
		\tikzset{
			empty node/.style={circle,inner sep=0,fill=none},
			solid node/.style={circle,draw,inner sep=1.5,fill=black},
			hollow node/.style={circle,draw,inner sep=1.5,fill=white},
			gray node/.style={circle,draw={rgb:black,1;white,4},inner sep=1.5,fill={rgb:black,1;white,4}}
		}
		\tikzset{snake it/.style={decorate, decoration=snake, line cap=round}}
		\tikzset{gray line/.style={line cap=round,thick,color={rgb:black,1;white,4}}}
		\tikzset{thick line/.style={line cap=round,rounded corners=0.1mm,thick}}
		\node (a)[solid node] at (0,0) {};
		\node (a1)[gray node] at (0.3,0.5) {};
		\node (a0)[gray node] at (-0.3,0.5) {};
		\node (a00)[gray node] at (-0.45,1) {};
		\node (a01)[solid node] at (-0.15,1) {};
		\node (a10)[solid node] at (0.15,1) {};
		\node (a11)[gray node] at (0.45,1) {};
		\draw[-,gray line] (a) to (a0);
		\draw[-,gray line] (a) to (a1);
		\draw[-,gray line] (a0) to (a00);
		\draw[-,gray line] (a0) to (a01);
		\draw[-,gray line] (a1) to (a10);
		\draw[-,gray line] (a1) to (a11);
		\draw[-,thick line] (a) to (0.3,0.5) to (a10);
		\draw[-,thick line] (a) to (-0.3,0.5) to (a01);
		\node at (0,-0.5) {(j)};
	\end{tikzpicture}
	\hspace{5mm}
	\begin{tikzpicture}[scale=1.5,font=\footnotesize]
		\tikzset{
			empty node/.style={circle,inner sep=0,fill=none},
			solid node/.style={circle,draw,inner sep=1.5,fill=black},
			hollow node/.style={circle,draw,inner sep=1.5,fill=white},
			gray node/.style={circle,draw={rgb:black,1;white,4},inner sep=1.5,fill={rgb:black,1;white,4}}
		}
		\tikzset{snake it/.style={decorate, decoration=snake, line cap=round}}
		\tikzset{gray line/.style={line cap=round,thick,color={rgb:black,1;white,4}}}
		\tikzset{thick line/.style={line cap=round,rounded corners=0.1mm,thick}}
		\node (a)[solid node] at (0,0) {};
		\node (a1)[gray node] at (0.3,0.5) {};
		\node (a0)[gray node] at (-0.3,0.5) {};
		\node (a00)[gray node] at (-0.45,1) {};
		\node (a01)[solid node] at (-0.15,1) {};
		\node (a10)[gray node] at (0.15,1) {};
		\node (a11)[solid node] at (0.45,1) {};
		\draw[-,gray line] (a) to (a0);
		\draw[-,gray line] (a) to (a1);
		\draw[-,gray line] (a0) to (a00);
		\draw[-,gray line] (a0) to (a01);
		\draw[-,gray line] (a1) to (a10);
		\draw[-,gray line] (a1) to (a11);
		\draw[-,thick line] (a) to (0.3,0.5) to (a11);
		\draw[-,thick line] (a) to (-0.3,0.5) to (a01);
		\node at (0,-0.5) {(k)};
	\end{tikzpicture}
	\hspace{5mm}
	\begin{tikzpicture}[scale=1.5,font=\footnotesize]
		\tikzset{
			empty node/.style={circle,inner sep=0,fill=none},
			solid node/.style={circle,draw,inner sep=1.5,fill=black},
			hollow node/.style={circle,draw,inner sep=1.5,fill=white},
			gray node/.style={circle,draw={rgb:black,1;white,4},inner sep=1.5,fill={rgb:black,1;white,4}}
		}
		\tikzset{snake it/.style={decorate, decoration=snake, line cap=round}}
		\tikzset{gray line/.style={line cap=round,thick,color={rgb:black,1;white,4}}}
		\tikzset{thick line/.style={line cap=round,rounded corners=0.1mm,thick}}
		\node (a)[solid node] at (0,0) {};
		\node (a1)[gray node] at (0.3,0.5) {};
		\node (a0)[gray node] at (-0.3,0.5) {};
		\node (a00)[solid node] at (-0.45,1) {};
		\node (a01)[gray node] at (-0.15,1) {};
		\node (a10)[solid node] at (0.15,1) {};
		\node (a11)[gray node] at (0.45,1) {};
		\draw[-,gray line] (a) to (a0);
		\draw[-,gray line] (a) to (a1);
		\draw[-,gray line] (a0) to (a00);
		\draw[-,gray line] (a0) to (a01);
		\draw[-,gray line] (a1) to (a10);
		\draw[-,gray line] (a1) to (a11);
		\draw[-,thick line] (a) to (-0.3,0.5) to (a00);
		\draw[-,thick line] (a) to (0.3,0.5) to (a10);
		\node at (0,-0.5) {(l)};
	\end{tikzpicture}
  \caption{
    A few finite level-closed subtrees with the same embedding type. The fact that they are level-closed depends on the underlying grey tree.
  }
  \label{fig:same-emb-types}
\end{figure}

\begin{figure}[h]
%  \centering \includegraphics[scale=0.5]{embTypeExamples.eps}
  \centering
	\begin{tikzpicture}[scale=1.5,font=\footnotesize]
		\tikzset{
			empty node/.style={circle,inner sep=0,fill=none},
			solid node/.style={circle,draw,inner sep=1.5,fill=black},
			hollow node/.style={circle,draw,inner sep=1.5,fill=white},
			gray node/.style={circle,draw={rgb:black,1;white,4},inner sep=1.5,fill={rgb:black,1;white,4}}
		}
		\tikzset{snake it/.style={decorate, decoration=snake, line cap=round}}
		\tikzset{gray line/.style={line cap=round,thick,color={rgb:black,1;white,4}}}
		\tikzset{thick line/.style={line cap=round,rounded corners=0.1mm,thick}}
		\node (a)[gray node] at (0-0.6,0) {};
		\node (a1)[gray node] at (0.3-0.6,0.5) {};
		\node (a0)[solid node] at (-0.3-0.6,0.5) {};
		\node (a00)[gray node] at (-0.45-0.6,1) {};
		\node (a01)[gray node] at (-0.15-0.6,1) {};
		\node (a10)[gray node] at (0.15-0.6,1) {};
		\node (a11)[gray node] at (0.45-0.6,1) {};
		\draw[-,thick line] (0-0.6,0) to (a0);
		\draw[-,gray line] (a) to (a1);
		\begin{pgfonlayer}{background}
		\draw[-,gray line] (a0) to (a00);
		\draw[-,gray line] (a0) to (a01);
		\end{pgfonlayer}
		\draw[-,gray line] (a1) to (a10);
		\draw[-,gray line] (a1) to (a11);
		\node (b)[solid node] at (0+0.6,0) {};
		\node (b1)[gray node] at (0.3+0.6,0.5) {};
		\node (b0)[gray node] at (-0.3+0.6,0.5) {};
		\node (b00)[gray node] at (-0.45+0.6,1) {};
		\node (b01)[gray node] at (-0.15+0.6,1) {};
		\node (b10)[gray node] at (0.15+0.6,1) {};
		\node (b11)[gray node] at (0.45+0.6,1) {};
		\begin{pgfonlayer}{background}
		\draw[-,gray line] (b) to (b0);
		\draw[-,gray line] (b) to (b1);
		\draw[-,gray line] (b0) to (b00);
		\draw[-,gray line] (b0) to (b01);
		\draw[-,gray line] (b1) to (b10);
		\draw[-,gray line] (b1) to (b11);
		\end{pgfonlayer}
		\node at (0,-0.5) {height 1};
	\end{tikzpicture}
	\hspace{5mm}
	\begin{tikzpicture}[scale=1.5,font=\footnotesize]
		\tikzset{
			empty node/.style={circle,inner sep=0,fill=none},
			solid node/.style={circle,draw,inner sep=1.5,fill=black},
			hollow node/.style={circle,draw,inner sep=1.5,fill=white},
			gray node/.style={circle,draw={rgb:black,1;white,4},inner sep=1.5,fill={rgb:black,1;white,4}}
		}
		\tikzset{snake it/.style={decorate, decoration=snake, line cap=round}}
		\tikzset{gray line/.style={line cap=round,thick,color={rgb:black,1;white,4}}}
		\tikzset{thick line/.style={line cap=round,rounded corners=0.1mm,thick}}
		\node (a)[solid node] at (0-0.6,0) {};
		\node (a1)[gray node] at (0.3-0.6,0.5) {};
		\node (a0)[gray node] at (-0.3-0.6,0.5) {};
		\node (a00)[gray node] at (-0.45-0.6,1) {};
		\node (a01)[solid node] at (-0.15-0.6,1) {};
		\node (a10)[gray node] at (0.15-0.6,1) {};
		\node (a11)[solid node] at (0.45-0.6,1) {};
		\draw[-,thick line] (a) to (-0.3-0.6,0.5) to (a01);
		\draw[-,thick line] (a) to (0.3-0.6,0.5) to (a11);
		\draw[-,gray line] (a0) to (a00);
		\draw[-,gray line] (a1) to (a10);
		\node (b)[solid node] at (0+0.6,0) {};
		\node (b1)[solid node] at (0.3+0.6,0.5) {};
		\node (b0)[solid node] at (-0.3+0.6,0.5) {};
		\node (b00)[gray node] at (-0.45+0.6,1) {};
		\node (b01)[gray node] at (-0.15+0.6,1) {};
		\node (b10)[gray node] at (0.15+0.6,1) {};
		\node (b11)[gray node] at (0.45+0.6,1) {};
		\begin{pgfonlayer}{background}
		\draw[-,thick line] (b) to (b0);
		\draw[-,thick line] (b) to (b1);
		\draw[-,gray line] (b0) to (b00);
		\draw[-,gray line] (b0) to (b01);
		\draw[-,gray line] (b1) to (b10);
		\draw[-,gray line] (b1) to (b11);
		\end{pgfonlayer}
		\node at (0,-0.5) {height 2};
	\end{tikzpicture}
	\hspace{5mm}
	\begin{tikzpicture}[scale=1.5,font=\footnotesize]
		\tikzset{
			empty node/.style={circle,inner sep=0,fill=none},
			solid node/.style={circle,draw,inner sep=1.5,fill=black},
			hollow node/.style={circle,draw,inner sep=1.5,fill=white},
			gray node/.style={circle,draw={rgb:black,1;white,4},inner sep=1.5,fill={rgb:black,1;white,4}}
		}
		\tikzset{snake it/.style={decorate, decoration=snake, line cap=round}}
		\tikzset{gray line/.style={line cap=round,thick,color={rgb:black,1;white,4}}}
		\tikzset{thick line/.style={line cap=round,rounded corners=0.1mm,thick}}
		\node (a)[solid node] at (0-0.6,0) {};
		\node (a1)[solid node] at (0.3-0.6,0.5) {};
		\node (a0)[gray node] at (-0.3-0.6,0.5) {};
		\node (a00)[gray node] at (-0.45-0.6,1) {};
		\node (a01)[gray node] at (-0.15-0.6,1) {};
		\node (a10)[solid node] at (0.15-0.6,1) {};
		\node (a11)[gray node] at (0.45-0.6,1) {};
		\draw[-,gray line] (a) to (a0);
		\draw[-,thick line] (a) to (a1);
		\draw[-,gray line] (a0) to (a00);
		\draw[-,gray line] (a0) to (a01);
		\draw[-,thick line] (a1) to (a10);
		\draw[-,gray line] (a1) to (a11);
		\node (b)[solid node] at (0+0.6,0) {};
		\node (b1)[solid node] at (0.3+0.6,0.5) {};
		\node (b0)[solid node] at (-0.3+0.6,0.5) {};
		\node (b00)[solid node] at (-0.45+0.6,1) {};
		\node (b01)[solid node] at (-0.15+0.6,1) {};
		\node (b10)[solid node] at (0.15+0.6,1) {};
		\node (b11)[solid node] at (0.45+0.6,1) {};
		\draw[-,thick line] (b) to (b0);
		\draw[-,thick line] (b) to (b1);
		\draw[-,thick line] (b0) to (b00);
		\draw[-,thick line] (b0) to (b01);
		\draw[-,thick line] (b1) to (b10);
		\draw[-,thick line] (b1) to (b11);
		\node at (0,-0.5) {height 3};
	\end{tikzpicture}
  \caption{ Some subtrees with embedding type of different height. The two first have the same embedding type, the unique embedding type of height 1. The two in the middle have the same embedding type, of height 2. The last pair consists of level-closed subtrees with two different embedding types of height 3. }
  \label{fig:some-emb-height}
\end{figure}
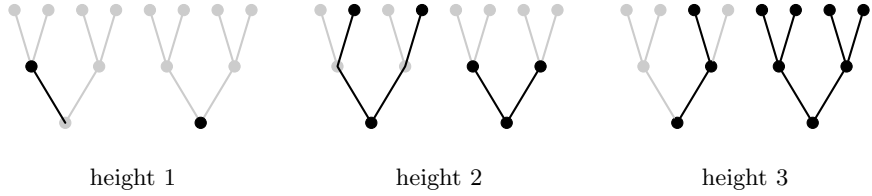

\section{Strong generalized CHM tree theorem}

No embedding type can be avoided in a strong subtree of $2^{<\omega}$. For this reason the number of colors that cannot be avoided by $n$-tuples of elements of $2^{<\omega}$ is at least the number of embedding types that can be generated by these tuples.

\begin{definition}\index{$e_{\sTT}$}
Let $e_{\sTT}:\omega \rightarrow \omega$ be the function which to $n$ associates the number of embedding types that can be generated by $n$ distinct strings. 
\end{definition}

Let us provide an example with \Cref{fig:tuple-to-height}: all the possible embedding types that are generated by two strings. We have $e_{\sTT}(2) = 7$.

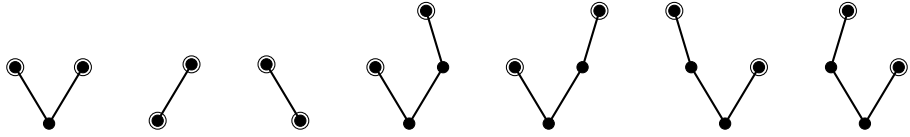
\begin{figure}[h]
  \begin{center}
    \begin{tikzpicture}[scale=1.5,font=\footnotesize]
		\tikzset{
			empty node/.style={circle,inner sep=0,fill=none},
			solid node/.style={circle,draw,inner sep=1.5,fill=black},
			hollow node/.style={circle,draw,inner sep=1.5,fill=white},
			gray node/.style={circle,draw={rgb:black,1;white,4},inner sep=1.5,fill={rgb:black,1;white,4}}
		}
		\tikzset{snake it/.style={decorate, decoration=snake, line cap=round}}
		\tikzset{gray line/.style={line cap=round,thick,color={rgb:black,1;white,4}}}
		\tikzset{thick line/.style={line cap=round,rounded corners=0.1mm,thick}}
		\begin{pgfonlayer}{background}
		\node (a)[solid node] at (0,0) {};
		\node (a1)[solid node] at (0.3,0.5) {};
		\node (a0)[solid node] at (-0.3,0.5) {};
		\draw[fill=none] (a0) circle [radius=0.75mm];
		\draw[fill=none] (a1) circle [radius=0.75mm];
		\end{pgfonlayer}
		\draw[-,thick line] (a) to (a0);
		\draw[-,thick line] (a) to (a1);
	\end{tikzpicture}
	\hspace{5mm}
	\begin{tikzpicture}[scale=1.5,font=\footnotesize]
		\tikzset{
			empty node/.style={circle,inner sep=0,fill=none},
			solid node/.style={circle,draw,inner sep=1.5,fill=black},
			hollow node/.style={circle,draw,inner sep=1.5,fill=white},
			gray node/.style={circle,draw={rgb:black,1;white,4},inner sep=1.5,fill={rgb:black,1;white,4}}
		}
		\tikzset{snake it/.style={decorate, decoration=snake, line cap=round}}
		\tikzset{gray line/.style={line cap=round,thick,color={rgb:black,1;white,4}}}
		\tikzset{thick line/.style={line cap=round,rounded corners=0.1mm,thick}}
		\begin{pgfonlayer}{background}
		\node (a)[solid node] at (0,0) {};
		\node (a1)[solid node] at (0.3,0.5) {};
		\draw[fill=none] (a) circle [radius=0.75mm];
		\draw[fill=none] (a1) circle [radius=0.75mm];
		\end{pgfonlayer}
		\draw[-,thick line] (a) to (a1);
	\end{tikzpicture}
	\hspace{5mm}
	\begin{tikzpicture}[scale=1.5,font=\footnotesize]
		\tikzset{
			empty node/.style={circle,inner sep=0,fill=none},
			solid node/.style={circle,draw,inner sep=1.5,fill=black},
			hollow node/.style={circle,draw,inner sep=1.5,fill=white},
			gray node/.style={circle,draw={rgb:black,1;white,4},inner sep=1.5,fill={rgb:black,1;white,4}}
		}
		\tikzset{snake it/.style={decorate, decoration=snake, line cap=round}}
		\tikzset{gray line/.style={line cap=round,thick,color={rgb:black,1;white,4}}}
		\tikzset{thick line/.style={line cap=round,rounded corners=0.1mm,thick}}
		\begin{pgfonlayer}{background}
		\node (a)[solid node] at (0,0) {};
		\node (a0)[solid node] at (-0.3,0.5) {};
		\draw[fill=none] (a) circle [radius=0.75mm];
		\draw[fill=none] (a0) circle [radius=0.75mm];
		\end{pgfonlayer}
		\draw[-,thick line] (a) to (a0);
	\end{tikzpicture}
	\hspace{5mm}
	\begin{tikzpicture}[scale=1.5,font=\footnotesize]
		\tikzset{
			empty node/.style={circle,inner sep=0,fill=none},
			solid node/.style={circle,draw,inner sep=1.5,fill=black},
			hollow node/.style={circle,draw,inner sep=1.5,fill=white},
			gray node/.style={circle,draw={rgb:black,1;white,4},inner sep=1.5,fill={rgb:black,1;white,4}}
		}
		\tikzset{snake it/.style={decorate, decoration=snake, line cap=round}}
		\tikzset{gray line/.style={line cap=round,thick,color={rgb:black,1;white,4}}}
		\tikzset{thick line/.style={line cap=round,rounded corners=0.1mm,thick}}
		\begin{pgfonlayer}{background}
		\node (a)[solid node] at (0,0) {};
		\node (a1)[solid node] at (0.3,0.5) {};
		\node (a0)[solid node] at (-0.3,0.5) {};
		\node (a10)[solid node] at (0.15,1) {};
		\draw[fill=none] (a0) circle [radius=0.75mm];
		\draw[fill=none] (a10) circle [radius=0.75mm];
		\end{pgfonlayer}
		\draw[-,thick line] (a) to (a0);
		\draw[-,thick line] (a) to (a1);
		\draw[-,thick line] (a1) to (a10);
	\end{tikzpicture}
	\hspace{5mm}
	\begin{tikzpicture}[scale=1.5,font=\footnotesize]
		\tikzset{
			empty node/.style={circle,inner sep=0,fill=none},
			solid node/.style={circle,draw,inner sep=1.5,fill=black},
			hollow node/.style={circle,draw,inner sep=1.5,fill=white},
			gray node/.style={circle,draw={rgb:black,1;white,4},inner sep=1.5,fill={rgb:black,1;white,4}}
		}
		\tikzset{snake it/.style={decorate, decoration=snake, line cap=round}}
		\tikzset{gray line/.style={line cap=round,thick,color={rgb:black,1;white,4}}}
		\tikzset{thick line/.style={line cap=round,rounded corners=0.1mm,thick}}
		\begin{pgfonlayer}{background}
		\node (a)[solid node] at (0,0) {};
		\node (a1)[solid node] at (0.3,0.5) {};
		\node (a0)[solid node] at (-0.3,0.5) {};
		\node (a11)[solid node] at (0.45,1) {};
		\draw[fill=none] (a0) circle [radius=0.75mm];
		\draw[fill=none] (a11) circle [radius=0.75mm];
		\end{pgfonlayer}
		\draw[-,thick line] (a) to (a0);
		\draw[-,thick line] (a) to (a1);
		\draw[-,thick line] (a1) to (a11);
	\end{tikzpicture}
	\hspace{5mm}
	\begin{tikzpicture}[scale=1.5,font=\footnotesize]
		\tikzset{
			empty node/.style={circle,inner sep=0,fill=none},
			solid node/.style={circle,draw,inner sep=1.5,fill=black},
			hollow node/.style={circle,draw,inner sep=1.5,fill=white},
			gray node/.style={circle,draw={rgb:black,1;white,4},inner sep=1.5,fill={rgb:black,1;white,4}}
		}
		\tikzset{snake it/.style={decorate, decoration=snake, line cap=round}}
		\tikzset{gray line/.style={line cap=round,thick,color={rgb:black,1;white,4}}}
		\tikzset{thick line/.style={line cap=round,rounded corners=0.1mm,thick}}
		\begin{pgfonlayer}{background}
		\node (a)[solid node] at (0,0) {};
		\node (a1)[solid node] at (0.3,0.5) {};
		\node (a0)[solid node] at (-0.3,0.5) {};
		\node (a00)[solid node] at (-0.45,1) {};
		\draw[fill=none] (a1) circle [radius=0.75mm];
		\draw[fill=none] (a00) circle [radius=0.75mm];
		\end{pgfonlayer}
		\draw[-,thick line] (a) to (a0);
		\draw[-,thick line] (a) to (a1);
		\draw[-,thick line] (a0) to (a00);
	\end{tikzpicture}
	\hspace{5mm}
	\begin{tikzpicture}[scale=1.5,font=\footnotesize]
		\tikzset{
			empty node/.style={circle,inner sep=0,fill=none},
			solid node/.style={circle,draw,inner sep=1.5,fill=black},
			hollow node/.style={circle,draw,inner sep=1.5,fill=white},
			gray node/.style={circle,draw={rgb:black,1;white,4},inner sep=1.5,fill={rgb:black,1;white,4}}
		}
		\tikzset{snake it/.style={decorate, decoration=snake, line cap=round}}
		\tikzset{gray line/.style={line cap=round,thick,color={rgb:black,1;white,4}}}
		\tikzset{thick line/.style={line cap=round,rounded corners=0.1mm,thick}}
		\begin{pgfonlayer}{background}
		\node (a)[solid node] at (0,0) {};
		\node (a1)[solid node] at (0.3,0.5) {};
		\node (a0)[solid node] at (-0.3,0.5) {};
		\node (a01)[solid node] at (-0.15,1) {};
		\draw[fill=none] (a1) circle [radius=0.75mm];
		\draw[fill=none] (a01) circle [radius=0.75mm];
		\end{pgfonlayer}
		\draw[-,thick line] (a) to (a0);
		\draw[-,thick line] (a) to (a1);
		\draw[-,thick line] (a0) to (a01);
	\end{tikzpicture}
    \caption{The seven possible embedding types generated by two nodes (shown as circled). That is, these are the embedding types of their full closure. The maximal height of the embedding types here is 3.}
     \label{fig:tuple-to-height}
\end{center}
\end{figure}

Given an element of $[2^{<\omega}]^n$ one can computably recognize which embedding type it generates. It follows that given an enumeration $\embfont e_1, \embfont e_2, \dots, \embfont e_{e_{\sTT}(n)}$ of the embedding types that can be generated by $n$ distinct strings, one can define the color $c$ on $[2^{<\omega}]^n$ which to each element generating the embedding type $\embfont e_i$ associates $i$. No strong subtree of $2^{<\omega}$ can avoid any embedding type and thus at least $e_{\sTT}(n)$ colors are used by $c$ within any strong subtree of $2^{<\omega}$.

We can in fact force even more colors: Given a finite strong subtree $F \subseteq 2^{<\omega}$, there might be distinct tuples $\overline{\sigma_1},\overline{\sigma_2} \in [F]^n$ such that $\overline{\sigma_1}^{cl} = \overline{\sigma_2}^{cl} = F$. Note that such a phenomenon does not happen for $n=2$ and below, but start to happen from $n=3$. For instance the tuples $\{\sigma0,\sigma1,\sigma00\}$ and $\{\sigma,\sigma1,\sigma00\}$ generate the same embedding type. This leads to the following definition:

\begin{definition}\index{tuple type}\index{type!tuple}
A \emph{tuple type} is an equivalence class on the following relation defined on $\bigcup_n [2^{<\omega}]^n \times [2^{<\omega}]^n$: We say that $\overline{\sigma}, \overline{\tau} \in [2^{<\omega}]^n$ are equivalent if there is a strong isomorphism $f$ from $\overline{\sigma}^{\mathrm{cl}}$ to $\overline{\tau}^{\mathrm{cl}}$ which associates elements of $\overline{\sigma}$ to elements of $\overline{\tau}$.
\end{definition}

Note that the tuple types are a refinement of the embedding types. Also for $n>2$ this refinement is strict, by the example given above with $\{\sigma0,\sigma1,\sigma00\}$ and $\{\sigma,\sigma1,\sigma00\}$: no strong isomorphism from $\{\sigma0,\sigma1,\sigma00\}^{\mathrm{cl}}$ to $\{\sigma,\sigma1,\sigma00\}^{\mathrm{cl}}$ can map elements of $\{\sigma0,\sigma1,\sigma00\}$ to those of $\{\sigma,\sigma1,\sigma00\}$ as no string is a prefix of the other two in the former but one is in the latter.

\begin{definition}\index{$t_{\sTT}$}
Let $t_{\sTT}:\omega \rightarrow \omega$ be the function which to $n$ associates the number of tuple types that can be generated by $n$ distinct strings. 
\end{definition}

Just like a coloring on $[2^{<\omega}]^n$ can recognize the generated embedding types, it can recognize the corresponding tuple types: the embedding type together with the role played by each string generating it. And just like no embedding type can be avoided in a strong perfect tree, also no tuple type can be avoided in a strong perfect tree. It follows that if $c$ is a coloring of $[2^{<\omega}]^n$ which associates to an element its corresponding tuple type, then any strong perfect subtree of $T$ needs at least $t_{\sTT}(n)$ colors. In the next section we show that this number is optimal.

% \begin{definition}
%   The statement ``strong generalized tree theorem'' is the following statement: ``for every coloring $f:[\cantor]^n\to l$, there exists a perfect tree $T$ such that $[T]^n$ uses at most $k$ color''...

%   The statement ``strong generalized tree theorem'' is the following statement: ``for every coloring $f:[\cantor]^n\to l$, there exists a perfect tree $T$ such that $[T]^n$ uses at most $k$ color''...
% \end{definition}

We are now ready to formally state and prove the strong generalized CHM tree theorem.

\begin{theorem}[Strong generalized CHM tree theorem] \label{th:strong_gen_treeth}
For every $n$, the principle $(\forall k) \mathrm{SCMHTT^n_{k,t_{\sTT}(n)}}$ is provable in $\ACA_0$, and $\RCA_0$ proves that the principle $(\forall k) \mathrm{SCMHTT^n_{k,t_{\sTT}(n)-1}}$ is false.
\end{theorem}
\begin{proof}
We already saw that this principle is false when the maximal number of color is $t_{\sTT}(n)-1$, with as an example the coloring on $[2^{<\omega}]^n$ which on each element associates an integer representing its tuple type.

Let us now show that one can prove the statement within $\ACA_0$ for $t_{\sTT}(n)$ colors. Let $\embfont t_1, \embfont t_2, \dots, \embfont t_{t_{\sTT}(n)}$ be an enumeration of the tuple types of size $n$. Let $c$ be any coloring on the elements of $[2^{<\omega}]^n$.

Let $T_0 = 2^{<\omega}$ and inductively for $i < t_{\sTT}(n)$, let $\embfont e_{i}$ be the embedding type that $\embfont t_{i}$ belongs to. Let $m_i$ be the height of the canonical representative of $\embfont t_{i}$. Let $c_i$ be the color on $T_i$ which on any strong subtree $F$ of height $m_i$ associates the color $c$ gives on the unique element of $[F]^{n}$ with tuple type $\embfont t_i$. Using \cref{thm:milliken-aca} stating that Milliken theorem for height $m_i$ is provable in $\ACA_0$, let then $T_{i+1}$ be a strong perfect subtree of $T_i$ which belongs to $\Mc$ and which is monochromatic for $c_{i}$ and let $k_{i}$ be the corresponding color. For this step, note that even if Milliken theorem is stated for strong subtrees of $2^{<\omega}$, we can also apply it for strong subtrees of $T$ where $T$ is itself a strong subtree of $2^{<\omega}$.

Let $S = T_{t_{\sTT}(n)}$. By induction we have that $S$ is a strong subtree of $2^{<\omega}$. Any $\overline{\sigma} \in [S]^n$ belongs to some tuple type $\embfont t_{i}$ and thus has color $k_i$. Thus at most $t_{\sTT}(n)$ color are used in $S$.
\end{proof}

\section{Avoiding types}

We now turn to the study of $\mathrm{CMHTT^n_{k,}}$. In particular, we do not require our subtrees to be strings anymore. As expected we need fewer colors, basically because we can avoid some embedding types, and within the embedding types which cannot be avoided, we can avoid some tuple types. 

Note that one can easily create a perfect tree which avoids almost all embedding types. Suppose we force for instance every node to be of different length. Formally such a tree has only embedding types which consists of comparable nodes, because any embedding type with two incomparable nodes contains two distinct nodes of the same length. However, this does not help: what we want is to avoid the embedding types (resp. the tuple types) which can be generated by the $n$-tuples of the tree, even though elements in the strong closure of the $n$-tuple are not necessarily all in the tree.

We will show given any strong tree $S$ how to compute a perfect subtree $T$ of $S$ which avoids as many tuple types as possible. We in fact give right away the syntactic property a tree must have to avoid as many tuple types as possible.

\begin{definition} \label{def:syntactic_minimize}\index{syntactical minimization}
We say that a perfect tree $T$ \emph{syntactically minimizes the number of tuple types} if:
\begin{enumerate}
\item[(1)] any two nodes of $T^{\wedge}$ is of different length;
\item[(2)] for any nodes $\sigma,\tau \in T$ with $\sigma \prec \tau$ we have $\sigma 0 \preceq \tau$;
\item[(3)] for any nodes $\sigma,\tau \in T^{cl}$ with $\sigma \notin T^{\wedge}$ and $\sigma \prec \tau$ we have $\sigma 0 \preceq \tau$.
\end{enumerate}
\end{definition}

Given (1), note that (3) in the previous definition is equivalent to have for any incomparable nodes $\sigma,\tau \in T^{\wedge}$ with $|\sigma| < |\tau|$ that $\tau(|\sigma|) = 0$.

\begin{lemma} \label{lem:existance_smntt}
Given any strong perfect subtree $S \subseteq 2^{<\omega}$, there is an $S$-computable perfect subtree $T \subseteq S$ which syntactically minimizes the number of tuple types.
\end{lemma}
\begin{proof}
Without loss of generality we consider that we work with $S = 2^{<\omega}$. The subtree that we build can then be pulled back in $S$ using some isomorphism between $2^{<\omega}$ and $S$.

We start by computing a meet-closed subtree $T' \subseteq 2^{\omega}$ such that (1) and (3) are satisfied. We put in $T_{0}$ the root of $2^{<\omega}$. Then inductively suppose we have a finite perfect tree $T_n$ such that each of its leaf is of level $n$ and such that for $\tau_1,\tau_2 \in T_{n}$ we have $|\tau_1| +1 < |\tau_2|$ or $|\tau_2| +1 < |\tau_1|$. Let $\sigma_1, \dots, \sigma_k$ be the leaves of $T_n$ such that $|\sigma_i|+1 < |\sigma_{i+1}|$. We define $T_{n+1, 0}$ to be $T_{n}$. Inductively for $i \leq k$ suppose we have defined a perfect tree $T_{n+1, i} \supseteq T_{n+1,0}$ such that for $\tau_1,\tau_2 \in T_{n+1, i}$ we have $|\tau_1| +1 < |\tau_2|$ or $|\tau_2| +1 < |\tau_1|$ and such that $|\sigma_i|$ is the smallest among the leaves of $T_{n+1, i}$. We let $\tau_0$ be the lexicographically smallest such that:

\begin{itemize}
\item $|\sigma_{i+1}0\tau_0| - 1$ is bigger than every string in $T_{n+1, i}$;
\item for every $\sigma \in T_{n+1, i}$ different from $\sigma_{i+1}$ we have $\sigma_{i+1}0\tau_0(|\sigma|) = 0$. Note that by the induction hypothesis we can find such a string.
\end{itemize}
Then let $\tau_1$ be the lexicographically smallest such that:
\begin{itemize}
\item $|\sigma_{i+1}1\tau_1| - 1$ is bigger than every string in $T_{n+1, i} \cup \{\tau_0\}$;
\item for every string $\sigma \in T_{n+1, i} \cup \{\tau_0\}$ different from $\sigma_{i+1}$ we have $\sigma_{i+1}1\tau_1(|\sigma|) = 0$. Note that by the induction hypothesis we can find such a string.
\end{itemize}

Let us then define $T_{n+1, i+1} = T_{n+1, i} \cup \{\tau_0, \tau_1\}$. Note that $\sigma_{i+1}$ becomes the smallest leaf of $T_{n+1, i+1}$. Once we have defined $T_{n+1, i}$ for every $i \leq k$ we define $T_{n+1} = T_{n+1, k}$.

We finally define $T' = \bigcup_n T_n$. By construction $T'$ has the desired properties. Note that for any perfect subtree $T \subseteq T'$ then also every node is of different length, so (1) is preserved. Furthermore as $T'$ is meet closed then also for any perfect tree $T \subseteq T'$ and any two incomparable nodes $\sigma,\tau \in T^{\wedge}$ with $|\sigma| < |\tau|$ we have $\tau(|\sigma|) = 0$, so (3) is preserved.

We finally find a perfect subtree $T \subseteq T'$ such that for any $\sigma,\tau \in T$ with $\sigma \prec \tau$ we have $\sigma0 \preceq \tau$. Given an isomorphism $f:2^{<\omega}  \rightarrow T'$ we define $T$ to be the range of $f$ on strings of the form $\sigma00$ or $\sigma01$ for $\sigma \in 2^{<\omega}$. One can easily verify that $T$ is a perfect subtree of $T'$ on which (1) (2) and (3) are verified.
\end{proof}

See \Cref{fig:lastminute1} for an illustration.
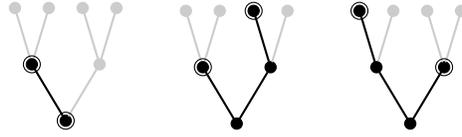
\begin{figure}
	\begin{center}
	  \begin{tikzpicture}[scale=1.5,font=\footnotesize]
			\tikzset{
				empty node/.style={circle,inner sep=0,fill=none},
				solid node/.style={circle,draw,inner sep=1.5,fill=black},
				hollow node/.style={circle,draw,inner sep=1.5,fill=white},
				gray node/.style={circle,draw={rgb:black,1;white,4},inner sep=1.5,fill={rgb:black,1;white,4}}
			}
			\tikzset{snake it/.style={decorate, decoration=snake, line cap=round}}
			\tikzset{gray line/.style={line cap=round,thick,color={rgb:black,1;white,4}}}
			\tikzset{thick line/.style={line cap=round,rounded corners=0.1mm,thick}}
			\node (a)[solid node] at (0,0) {};
			\node (a1)[gray node] at (0.3,0.5) {};
			\node (a0)[solid node] at (-0.3,0.5) {};
			\node (a00)[gray node] at (-0.45,1) {};
			\node (a01)[gray node] at (-0.15,1) {};
			\node (a10)[gray node] at (0.15,1) {};
			\node (a11)[gray node] at (0.45,1) {};
			\draw[-,thick line] (0,0) to (a0);
			\draw[-,gray line] (a) to (a1);
			\begin{pgfonlayer}{background}
			\draw[-,gray line] (a0) to (a00);
			\draw[-,gray line] (a0) to (a01);
			\end{pgfonlayer}
			\draw[-,gray line] (a1) to (a10);
			\draw[-,gray line] (a1) to (a11);
			\draw[fill=none] (a) circle [radius=0.75mm];
			\draw[fill=none] (a0) circle [radius=0.75mm];
		\end{tikzpicture}
		\hspace{5mm}
		\begin{tikzpicture}[scale=1.5,font=\footnotesize]
			\tikzset{
				empty node/.style={circle,inner sep=0,fill=none},
				solid node/.style={circle,draw,inner sep=1.5,fill=black},
				hollow node/.style={circle,draw,inner sep=1.5,fill=white},
				gray node/.style={circle,draw={rgb:black,1;white,4},inner sep=1.5,fill={rgb:black,1;white,4}}
			}
			\tikzset{snake it/.style={decorate, decoration=snake, line cap=round}}
			\tikzset{gray line/.style={line cap=round,thick,color={rgb:black,1;white,4}}}
			\tikzset{thick line/.style={line cap=round,rounded corners=0.1mm,thick}}
			\node (a)[solid node] at (0,0) {};
			\node (a1)[solid node] at (0.3,0.5) {};
			\node (a0)[solid node] at (-0.3,0.5) {};
			\node (a00)[gray node] at (-0.45,1) {};
			\node (a01)[gray node] at (-0.15,1) {};
			\node (a10)[solid node] at (0.15,1) {};
			\node (a11)[gray node] at (0.45,1) {};
			\draw[-,thick line] (0,0) to (a0);
			\draw[-,thick line] (a) to (a1);
			\begin{pgfonlayer}{background}
			\draw[-,gray line] (a0) to (a00);
			\draw[-,gray line] (a0) to (a01);
			\end{pgfonlayer}
			\draw[-,thick line] (a1) to (a10);
			\draw[-,gray line] (a1) to (a11);
			\draw[fill=none] (a0) circle [radius=0.75mm];
			\draw[fill=none] (a10) circle [radius=0.75mm];
		\end{tikzpicture}
		\hspace{5mm}
		\begin{tikzpicture}[scale=1.5,font=\footnotesize]
			\tikzset{
				empty node/.style={circle,inner sep=0,fill=none},
				solid node/.style={circle,draw,inner sep=1.5,fill=black},
				hollow node/.style={circle,draw,inner sep=1.5,fill=white},
				gray node/.style={circle,draw={rgb:black,1;white,4},inner sep=1.5,fill={rgb:black,1;white,4}}
			}
			\tikzset{snake it/.style={decorate, decoration=snake, line cap=round}}
			\tikzset{gray line/.style={line cap=round,thick,color={rgb:black,1;white,4}}}
			\tikzset{thick line/.style={line cap=round,rounded corners=0.1mm,thick}}
			\node (a)[solid node] at (0,0) {};
			\node (a1)[solid node] at (0.3,0.5) {};
			\node (a0)[solid node] at (-0.3,0.5) {};
			\node (a00)[solid node] at (-0.45,1) {};
			\node (a01)[gray node] at (-0.15,1) {};
			\node (a10)[gray node] at (0.15,1) {};
			\node (a11)[gray node] at (0.45,1) {};
			\draw[-,thick line] (0,0) to (a0);
			\draw[-,thick line] (a) to (a1);
			\begin{pgfonlayer}{background}
			\draw[-,thick line] (a0) to (a00);
			\draw[-,gray line] (a0) to (a01);
			\end{pgfonlayer}
			\draw[-,gray line] (a1) to (a10);
			\draw[-,gray line] (a1) to (a11);
			\draw[fill=none] (a00) circle [radius=0.75mm];
			\draw[fill=none] (a1) circle [radius=0.75mm];
		\end{tikzpicture}
	\end{center}\caption{An example of three tuple types on $[T]^{2}$, for a tree $T$ that syntactically minimizes the number of types.}\label{fig:lastminute1}
\end{figure}

We shall see that the number above is optimal. Of course, it is not the case that one of the above types can never be omitted in a perfect tree and it is in fact one difficulty in showing that a tree $T$ syntactically minimizing the number of tuple types really does so: every tuple type can be omitted in some perfect tree. Of course omitting a type may force some other type to become unavoidable. In order to overcome this difficulty, we need to introduce a third equivalence relation, within which we erase the part of a tuple type which can be omitted.

\begin{definition}\index{type!weak tuple}\index{weak tuple type}
The \emph{weak tuple types} are the equivalence classes of the following relation: $\overline{\sigma}, \overline{\tau}$ have the same weak tuple type if there is a bijection $f$ from $\overline{\sigma}^{cl}$ to $\overline{\tau}^{cl}$ such that:
\begin{itemize}
\item for $\sigma_1,\sigma_2 \in \overline{\sigma}^{cl}$ we have $\sigma_1 \preceq \sigma_2$ iff $f(\sigma_1) \preceq f(\sigma_2)$;
\item for $\sigma_1,\sigma_2,\sigma_3 \in \overline{\sigma}^{cl}$ we have $\sigma_1 0 \preceq \sigma_2$ and $\sigma_1 1 \preceq \sigma_3$ iff $f(\sigma_1) 0 \preceq f(\sigma_2)$ and $f(\sigma_1) 1 \preceq f(\sigma_3)$;
\item elements of $\overline{\sigma}$ are sent to $\overline{\tau}$.
\end{itemize}
\end{definition}

In other word, weak tuple types are tuple types, modulo the fact that whenever a node is not branching, it does not matter for its extension to go left or right. It is clear from the definition that the tuple types are a refinement of the weak tuple types. But the weak tuple types are not a refinement of the embedding types, nor are the embedding types a refinement of the weak tuple types.

\begin{lemma} \label{lem:coincide_smntt}
Let $T$ be a perfect tree which syntactically minimizes the number of tuple types. Then its tuple types and weak tuple types coincide.
\end{lemma}
\begin{proof}
We already have that the tuple types are a refinement of the weak tuple types. All we have to do is to show that restricted to $T$, the weak tuple types are a refinement of the tuple types.

For any $n$ consider any two $n$-tuples $\overline{\sigma},\overline{\tau}$ of $T$. Suppose they are in the same weak tuple type via some bijection $f:\overline{\sigma}^{cl} \rightarrow \overline{\tau}^{cl}$. Let us show that $f$ in fact witnesses that $\overline{\sigma}$ and $\overline{\tau}$ are in the same tuple types. For that it is enough to show that $\sigma_1 i \preceq \sigma_2$ implies $f(\sigma_1) i \preceq f(\sigma_2)$ for $\sigma_1,\sigma_2 \in \overline{\sigma}^{cl}$. Let $\sigma_1,\sigma_2 \in \overline{\sigma}^{cl}$ with $\sigma_1 i \preceq \sigma_2$. 

Suppose first that we have a string $\sigma_3 \in \overline{\sigma}^{cl}$ such that $\sigma_1 (1-i) \preceq \sigma_3$. Then by definition of a weak tuple type we have $f(\sigma_1) i \preceq f(\sigma_2)$. Otherwise there are two possibilities: either $\sigma_1 \in \overline{\sigma}$ or $\sigma_1 \in \overline{\sigma}^{cl}$ but $\sigma_1 \notin \overline{\sigma}^{\wedge}$. 

In the first case, note that we must have a string $\sigma_3 \in \overline{\sigma}$ with $\sigma_1 i \preceq \sigma_2 \preceq \sigma_3$. Note that as $\sigma_1,\sigma_3 \in \overline{\sigma}$ then also $\sigma_1,\sigma_3 \in T$. Then by property (2) in \Cref{def:syntactic_minimize} (the definition of syntactically minimizing the number of tuple type), we have $\sigma_1 0 \preceq \sigma_2 \preceq \sigma_3$. Note also that $f(\sigma_1) \preceq f(\sigma_1) \preceq f(\sigma_3)$ and that by hypothesis on $f$ we have $f(\sigma_1),f(\sigma_3) \in \overline{\tau} \subseteq T$. Therefore also we have $f(\sigma_1)0 \prec f(\sigma_3)$ by (2) in \Cref{def:syntactic_minimize} and thus we have $f(\sigma_1)0 \prec f(\sigma_2)$.

In the second case we have $\sigma_1,\sigma_2 \in T^{cl}$ and $\sigma_1 \notin T^{\wedge}$. Thus by property (3) in \Cref{def:syntactic_minimize} we have $\sigma_1 0 \preceq \sigma_2$. We shall argue that also $f(\sigma_1) \notin T^{\wedge}$. Suppose for contradiction that $f(\sigma_1) \in T^{\wedge}$. Then also $f(\sigma_1) \in \overline{\tau}^{\wedge}$. Recall that by hypothesis $\sigma_1 \notin \overline{\sigma}$ and thus $f(\sigma_1) \notin \overline{\tau}$ (as $f$ is a bijection between the two). It follows that $f(\sigma_1)$ must be the meet of two nodes in $\overline{\tau}$ and thus is branching in $\overline{\tau}^{cl}$. On the other hand as $\sigma_1 \notin T^{\wedge}$ it is not branching in $\overline{\sigma}^{cl}$ which contradicts the properties of $f$. Thus $f(\sigma_1) \notin T^{\wedge}$ and it follows by property (3) in the definition of syntactically minimizes the number of tuple type that $f(\sigma_1) 0 \preceq f(\sigma_2)$.

We then have that $\overline{\sigma}$ and $\overline{\tau}$ are in the same tuple types.
\end{proof}

We shall now identify the weak tuple types no perfect tree can omit. We shall then see that weak tuple types of a tree which syntactically minimizes the number of tuple types are all of this form. It will then follow that such a tree really minimizes the number of tuple types.

\begin{definition}\index{tuple type!length-injective}
A tuple type (resp. a weak tuple type) $\overline{\sigma}$ is \emph{length-injective} if $\sigma_1,\sigma_2 \in \overline{\sigma}^{\wedge}$ implies $|\sigma_1| \neq |\sigma_2|$.
\end{definition}

\begin{definition}\index{tuple type!meet-avoiding}
A tuple type (resp. a weak tuple type) $\overline{\sigma}$ is \emph{meet-avoiding} if for any incomparable $\sigma_1,\sigma_2 \in \overline{\sigma}$ we have $\sigma_1 \wedge \sigma_2 \notin \overline{\sigma}$.
\end{definition}

See \Cref{fig:lastminute2} for an illustration. 
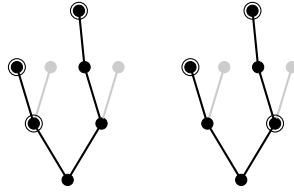
\begin{figure}
	\begin{center}
	  \begin{tikzpicture}[scale=1.5,font=\footnotesize]
			\tikzset{
				empty node/.style={circle,inner sep=0,fill=none},
				solid node/.style={circle,draw,inner sep=1.5,fill=black},
				hollow node/.style={circle,draw,inner sep=1.5,fill=white},
				gray node/.style={circle,draw={rgb:black,1;white,4},inner sep=1.5,fill={rgb:black,1;white,4}}
			}
			\tikzset{snake it/.style={decorate, decoration=snake, line cap=round}}
			\tikzset{gray line/.style={line cap=round,thick,color={rgb:black,1;white,4}}}
			\tikzset{thick line/.style={line cap=round,rounded corners=0.1mm,thick}}
			\node (a)[solid node] at (0,0) {};
			\node (a1)[solid node] at (0.3,0.5) {};
			\node (a0)[solid node] at (-0.3,0.5) {};
			\node (a00)[solid node] at (-0.45,1) {};
			\node (a01)[gray node] at (-0.15,1) {};
			\node (a10)[solid node] at (0.15,1) {};
			\node (a11)[gray node] at (0.45,1) {};
			\node (a100)[solid node] at (0.1,1.5) {};
			\draw[-,thick line] (a) to (a0);
			\draw[-,thick line] (a) to (a1);
			\begin{pgfonlayer}{background}
			\draw[-,thick line] (a0) to (a00);
			\draw[-,gray line] (a0) to (a01);
			\end{pgfonlayer}
			\draw[-,thick line] (a1) to (a10);
			\draw[-,thick line] (a10) to (a100);
			\draw[-,gray line] (a1) to (a11);
			\draw[fill=none] (a0) circle [radius=0.75mm];
			\draw[fill=none] (a00) circle [radius=0.75mm];
			\draw[fill=none] (a100) circle [radius=0.75mm];
		\end{tikzpicture}
		\hspace{5mm}
		\begin{tikzpicture}[scale=1.5,font=\footnotesize]
			\tikzset{
				empty node/.style={circle,inner sep=0,fill=none},
				solid node/.style={circle,draw,inner sep=1.5,fill=black},
				hollow node/.style={circle,draw,inner sep=1.5,fill=white},
				gray node/.style={circle,draw={rgb:black,1;white,4},inner sep=1.5,fill={rgb:black,1;white,4}}
			}
			\tikzset{snake it/.style={decorate, decoration=snake, line cap=round}}
			\tikzset{gray line/.style={line cap=round,thick,color={rgb:black,1;white,4}}}
			\tikzset{thick line/.style={line cap=round,rounded corners=0.1mm,thick}}
			\node (a)[solid node] at (0,0) {};
			\node (a1)[solid node] at (0.3,0.5) {};
			\node (a0)[solid node] at (-0.3,0.5) {};
			\node (a00)[solid node] at (-0.45,1) {};
			\node (a01)[gray node] at (-0.15,1) {};
			\node (a10)[solid node] at (0.15,1) {};
			\node (a11)[gray node] at (0.45,1) {};
			\node (a100)[solid node] at (0.1,1.5) {};
			\draw[-,thick line] (a) to (a0);
			\draw[-,thick line] (a) to (a1);
			\begin{pgfonlayer}{background}
			\draw[-,thick line] (a0) to (a00);
			\draw[-,gray line] (a0) to (a01);
			\end{pgfonlayer}
			\draw[-,thick line] (a1) to (a10);
			\draw[-,thick line] (a10) to (a100);
			\draw[-,gray line] (a1) to (a11);
			\draw[fill=none] (a1) circle [radius=0.75mm];
			\draw[fill=none] (a00) circle [radius=0.75mm];
			\draw[fill=none] (a100) circle [radius=0.75mm];
		\end{tikzpicture}
	\end{center}\caption{An example of two length-injective and meet-avoiding tuple types, generating the same embedding type.}\label{fig:lastminute2}
\end{figure}
%\todo[inline]{Paul-Elliot, can you please put the example: $0, 00, 100$ and $1, 00, 100$}

Up to symmetry and restricted to a tree which syntactically minimizes the number of types, these are the only tuple types generated by three strings and which are in the same embedding type. We will now see that the length-injective and meet-avoiding weak tuple types are exactly those which cannot be avoided by a perfect tree.

\begin{lemma} \label{lem:doesnotavoid_smntt}
Let $S$ be any perfect tree. Then $S$ has a member inside every length-injective and meet-avoiding weak tuple type.
\end{lemma}
\begin{proof}
Let $\overline{\sigma}$ be a length-injective and meet-avoiding tuple type. Let us define first an injection $f$ from $\overline{\sigma}^{\wedge}$ into $S^{\wedge}$ with the following properties:
\begin{enumerate}
\item[(1)] $f(\overline{\sigma}) \subseteq S$;
\item[(2)] for $\sigma_1,\sigma_2 \in \overline{\sigma}^{\wedge}$ we have $\sigma_1 \preceq \sigma_2$ iff $f(\sigma_1) \preceq f(\sigma_2)$;
\item[(3)] if $\sigma_1$ is branching in $\overline{\sigma}^{\wedge}$ then $f(\sigma_1)$ is branching in $S^\wedge$. Furthermore for $\sigma_1,\sigma_2, \sigma_3 \in \overline{\sigma}^{\wedge}$ we have $\sigma_1 0 \preceq \sigma_2$ and $\sigma_1 1 \preceq \sigma_3$ iff $f(\sigma_1) 0 \preceq f(\sigma_2)$ and $f(\sigma_1) 1 \preceq f(\sigma_3)$;
\item[(4)] for $\sigma_1,\sigma_2 \in \overline{\sigma}^{\wedge}$ we have $|\sigma_1| < |\sigma_2|$ iff $|f(\sigma_1)| < |f(\sigma_2)|$.
\end{enumerate}

Let $\sigma_1, \dots, \sigma_n$ be a list of the elements of $\overline{\sigma}^{\wedge}$ with $|\sigma_1| < |\sigma_2| < \dots < |\sigma_n|$. Note that we must have $\sigma_i \preceq \sigma_j$ implies $i \leq j$. Note also that as $\overline{\sigma}$ is meet-avoiding we must have that $\rho \in \overline{\sigma}^\wedge$ is branching in $\overline{\sigma}^\wedge$ iff $\rho \notin \overline{\sigma}$.

If $\sigma_1 \in \overline{\sigma}$ then find the lexicographically first $\tau \in S$ and let $f(\sigma_1) = \tau$. Otherwise find the lexicographically first $\tau \in S^\wedge$ which is branching in $S^\wedge$ and let $f(\sigma_1) = \tau$. Note that so far (1)(2)(3) and (4) are satisfied.

Suppose $f(\sigma_1), \dots, f(\sigma_k)$ have been defined with (1)(2)(3) and (4) satisfied so far. Consider $\sigma_{k+1}$. Let $j \leq k$ be the largest such that $\sigma_{j} \prec \sigma_{k+1}$. Suppose first $\sigma_j$ is branching in $\overline{\sigma}^{\wedge}$. Then by induction hypothesis (3) we must have that $f(\sigma_j)$ is branching in $S^{\wedge}$. In this case let $i \in \{0,1\}$ be such that $\sigma_{j} i \preceq \sigma_{k+1}$. If $\sigma_{k+1} \in \overline{\sigma}$ then find the lexicographically first $\tau \in S$ such that $|\tau| > |f(\sigma_k)|$ and such that $f(\sigma_j) i \preceq \tau$. Then let $f(\sigma_{k+1}) = \tau$. Otherwise find the lexicographically first branching $\tau \in S^{\wedge}$, such that $|\tau| > |f(\sigma_k)|$ and such that $f(\sigma_j) i \preceq \tau$. Then let $f(\sigma_{k+1}) = \tau$. Note that in any case (1) (2) (3) and (4) are satisfied so far.

Suppose now $\sigma_j$ is not branching in $\overline{\sigma}^{\wedge}$. If $\sigma_{k+1} \in \overline{\sigma}$ then find the lexicographically first $\tau \in S$ such that $|\tau| > |f(\sigma_k)|$ and such that $f(\sigma_j) \preceq \tau$. Then let $f(\sigma_{k+1}) = \tau$. Otherwise find the lexicographically first branching $\tau \in S^{\wedge}$, such that $|\tau| > |f(\sigma_k)|$ and such that $f(\sigma_j) \preceq \tau$. Then let $f(\sigma_{k+1}) = \tau$. Note that in any case (1) (2) (3) and (4) are satisfied so far. This ends the first part of the construction.

Note also that $f$ is a bijection between $\overline{\sigma}$ and $f(\overline{\sigma}) \subseteq S$. In order to show that $\overline{\sigma}$ and $f(\overline{\sigma})$ are in the same weak tuple type, we shall now extend $f$ to $\overline{\sigma}^{cl}$ such that $f$ becomes a bijection from $\overline{\sigma}^{cl}$ to $f(\overline{\sigma})^{cl}$.

Let us first argue that so far $f$ is a bijection from $\overline{\sigma}^{\wedge}$ to $f(\overline{\sigma})^{\wedge}$. We have $f(\overline{\sigma}) \subseteq f(\overline{\sigma}^{\wedge})$. By design we also have that $f(\overline{\sigma}^{\wedge})$ is meet closed and thus we have $f(\overline{\sigma})^{\wedge} = f(\overline{\sigma}^{\wedge})$. As $f$ is injective it is a bijection from $\overline{\sigma}^{\wedge}$ to $f(\overline{\sigma}^\wedge)$ and then it is a bijection from $\overline{\sigma}^{\wedge}$ to $f(\overline{\sigma})^\wedge$.

Let us now extend $f$ to $\overline{\sigma}^{cl}$: for incomparable $\sigma_1,\sigma_2 \in \overline{\sigma}^{\wedge}$ with $|\sigma_1| < |\sigma_2|$ we assign $f(\sigma_2 \upharpoonright {|\sigma_1|})$ to $f(\sigma_2)  \upharpoonright {|f(\sigma_1)|}$. Let us now show that $f(\overline{\sigma}^{cl}) = f(\overline{\sigma})^{cl}$. It is clear by definition of $f$ that $f(\overline{\sigma}^{cl}) \subseteq f(\overline{\sigma})^{cl}$. Let us now show $f(\overline{\sigma})^{cl} \subseteq f(\overline{\sigma}^{cl})$.

Suppose $\tau \in f(\overline{\sigma})^{cl}$. Then as $f(\overline{\sigma})^{cl} = (f(\overline{\sigma})^{\wedge})^{cl} = f(\overline{\sigma}^{\wedge})^{cl}$ we have $\tau \in f(\overline{\sigma}^{\wedge})^{cl}$. Then there exists $\sigma_1,\sigma_2 \in \overline{\sigma}^{\wedge}$ with $|f(\sigma_1)| < |f(\sigma_2)|$ and with $\tau = f(\sigma_2) \upharpoonright {|f(\sigma_1)|}$. By (4) we have $|\sigma_1| < |\sigma_2|$ and then $\sigma_2 \upharpoonright {|\sigma_1|} \in \overline{\sigma}^{cl}$. Thus $\tau \in f(\overline{\sigma}^{cl})$ and then $f(\overline{\sigma})^{cl} \subseteq f(\overline{\sigma}^{cl})$ and then $f(\overline{\sigma})^{cl} = f(\overline{\sigma}^{cl})$.

Let us now show that $f$ is injective on $\overline{\sigma}^{cl}$. Let $\sigma_1,\sigma_2, \rho_1, \rho_2 \in \overline{\sigma}^{\wedge}$ with $\sigma_1,\sigma_2$ and $\rho_1,\rho_2$ incomparable, with $|\sigma_1| < |\sigma_2|$ and with $|\rho_1| < |\rho_2|$. Suppose $\sigma_2 \upharpoonright {|\sigma_1|} \neq \rho_2 \upharpoonright {|\rho_1|}$. If $\sigma_1 \neq \rho_1$ then by (4) we must have $|f(\sigma_1)| \neq |f(\rho_1)|$ and thus $f(\sigma_2)  \upharpoonright {|f(\sigma_1)|} \neq f(\rho_2)  \upharpoonright {|f(\rho_1)|}$. Otherwise it must be that $\sigma_2 \upharpoonright {s} \neq \rho_2 \upharpoonright {s}$ for $s = |\sigma_1| = |\rho_1|$. By definition of $f$ it must be that $f((\sigma_2 \upharpoonright {s}) \wedge (\rho_2 \upharpoonright {s})) = f(\sigma_2 \upharpoonright {s}) \wedge f(\rho_2 \upharpoonright {s})$ and thus by (3) that $f(\sigma_2 \upharpoonright {s}) \neq f(\rho_2 \upharpoonright {s})$. 

It follows that $f$ is a bijection from $\overline{\sigma}^{cl}$ to $f(\overline{\sigma}^{cl})$ and thus that it is a bijection from $\overline{\sigma}^{cl}$ to $f(\overline{\sigma})^{cl}$.

It is clear that property (1) and (2) is still satisfied by $f$ on $\overline{\sigma}^{cl}$. Also as every branching node of $\overline{\sigma}^{cl}$ is already branching in $\overline{\sigma}^{\wedge}$ property (3) is till satisfied on $\overline{\sigma}^{cl}$. It follows that $\overline{\sigma}$ and $f(\overline{\sigma})$ are in the same weak tuple type.
\end{proof}

\begin{lemma} \label{lem:realizeall_smntt}
Let $T$ be a tree which syntactically minimizes the number of tuple types. Then every weak tuple type of $T$ is length-injective and meet-avoiding.
\end{lemma}
\begin{proof}
By definition we have that $\sigma_1,\sigma_2 \in T^{\wedge}$ implies $|\sigma_1| \neq |\sigma_2|$. Thus the weak tuple types of $T$ are length-injective. Suppose now $\sigma_1,\sigma_2 \in T$ with $\sigma_1,\sigma_2$ incomparable. Suppose for contradiction that $\sigma_1 \wedge \sigma_2 \in T$. Then we have $(\sigma_1 \wedge \sigma_2) 1 \preceq \sigma_1$ or $(\sigma_1 \wedge \sigma_2) 1 \preceq \sigma_2$. In any case we violate property (2) of syntactically minimizing the number of types. Thus for any $\sigma_1,\sigma_2 \in T$ with $\sigma_1,\sigma_2$ incomparable we have $\sigma_1 \wedge \sigma_2 \notin T$ which implies that the weak tuple types of $T$ are meet-avoiding.
\end{proof}

\section{Generalized CHM tree theorem}

We are now ready to study the generalized CHM tree theorem, where we do not necessarily required the subtree to be a strong subtree.

\begin{definition}\index{$t^T_{\TT}(n)$}
Given a perfect tree $T$, let $t^T_{\TT}(n)$ be the number of tuple types generated by $n$ distinct strings of $T$. Let 
$$t_{\TT}(n) = \min \{t^T_{\TT}(n)\:\ T \text{ is a perfect tree}\}.$$
\end{definition}

\begin{theorem} \label{th:reallyminimizes}
Suppose $T$ syntactically minimizes the number of tuple types, then $t_{\TT}(n) = t_{\TT}^T(n)$.
\end{theorem}
\begin{proof}
Suppose $T$ syntactically minimizes the number of tuple types. Then by \Cref{lem:coincide_smntt} the tuple types of $T$ coincide with its weak tuple types. By \Cref{lem:realizeall_smntt} every weak tuple type of $T$ is length-injective and meet-avoiding. By \Cref{lem:doesnotavoid_smntt} we then have that every weak-tuple type of $T$ is a weak-tuple type in any perfect tree $S$. Using the fact that the tuple types are a refinement of the weak tuple types, we then have that given any $n$, the number of tuple types of $[S]^{n}$ is bigger than the number of weak tuple types of $[S]^{n}$ and then bigger than the number of weak tuple types of $[T]^{n}$ and then bigger than the number of tuple types of $[T]^{n}$. Thus $t_{\TT}(n) = t_{\TT}^T(n)$.
\end{proof}

\begin{theorem}[Generalized CHM tree theorem]
For every $n$, the principle $\mathrm{CMHTT^n_{k,t_{\TT}(n)}}$ is provable in $\ACA_0$ but $\RCA_0$ proves that the principle $\mathrm{CMHTT^n_{k,t_{\TT}(n)-1}}$ is false.
\end{theorem}
\begin{proof}
Let $\Mc$ be a model of $\ACA_0$. Let $T \in \Mc$ be a perfect tree and $c \in \Mc$ be a color of $[T]^n$. Using \Cref{th:strong_gen_treeth}, there is a strong subtree $S \subseteq T$ such that every tuple type of $[S]^n$ is monochromatic for $c$. Using \Cref{lem:existance_smntt} let $R$ be a $S$-computable perfect subtree of $S$ which synctactically minimizes the number of tuple types. Note that $R \in \Mc$ and that by \Cref{th:reallyminimizes} $[R]^n$ has at most $t_{\TT}(n)$ many tuple types. It follows that $c$ uses at most $t_{\TT}(n)$ many colors on $R$.

To show optimality, and given an enumeration $\{\embfont e_i\}_{i \leq t_{\sTT}(n)}$ of the tuple types generated by $n$ strings, let us define a color on $2^{<\omega}$ which associates $i$ to $\overline{\sigma}$ of tuple type $\embfont e_i$. By minimality of $t_{\TT}(n)$ among $t^T_{\TT}(n)$ for a perfect tree $T$ we have that every perfect subtree of $2^{<\omega}$ uses at least $t_{\TT}(n)$ colors.
\end{proof}

\begin{theorem}[CHM tree theorem for $n$-tuple and $k$-colors]
For every coloring of $n$-tuples of pairwise comparable strings, there exists a perfect tree on which the coloring is monochromatic.
\end{theorem}
\begin{proof}
This follows from the fact that given any $n$, there is only one weak tuple type of size $n$ which contains only comparable strings.
\end{proof}

It is easy to determine the number $e_n$ of embedding types of height $n$, which is given by the following induction:

$$
\begin{array}{rcl}
e_0&=&1,\\
e_1&=&1,\\
e_{n+1}&=&2 \times e_n \times (\sum_{i<n} e_i) + e_n^2.
\end{array}
$$
The definition above is justified by the following observation: there is one tree of height $0$ (the emptyset), there is one tree of height $1$ (the empty string) and for any $n \geq 1$, the possibilities to build trees of height $n+1$ are as follow: having a left subtree of the root (the empty string) of height $n$ and a right subtree of the root of height $<n$, or the inverse of that, or have both a left and a right subtree of the root of height $n$. 

The number of embedding types generated by $n$ strings, namely $e_{\sTT}(n)$ appears much harder to compute. It is the same for $t_{\sTT}(n)$ and $t_{\TT}(n)$. We can also define the function $n \mapsto e_{\TT}(n)$ which to $n$ associates the minimal number of embedding type within any perfect tree (which is the number of embedding types generated by $n$ strings of a tree which syntactically minimizes the number of tuple types).

We computed the first values of each with the help of a computer program:

\[
\begin{array}{|c|c|c|c|c|}
\hline
&e_{\sTT}&t_{\sTT}&e_{\TT}&t_{\TT}\\
\hline
\hline
0&1&1&1&1\\
\hline
1&1&1&1&1\\
\hline
2&7&7&3&3\\
\hline
3&345&369&27&29\\
\hline
4&136949&145215&561&635\\
\hline
\end{array}
\]

None of these sequence appears in OEIS, The On-Line Encyclopedia of Integer Sequences \cite{OEIS}. It then seems that each of them is a new natural combinatorial sequence. Even if it seems that these sequences cannot be computed with an easy mathematical induction like for the number of embedding types of height $n$, we conjecture each of them to be polynomial time computable.

%%% Local Variables:
%%% mode: latex
%%% TeX-master: "../embryon"
%%% End:

\chapter{Open Questions}\label{sect:open-questions}

The computability-theoretic study of Milliken's tree theorem and its applications being completely new, this work leaves many questions open. We collect some here that seem most promising for follow-up research directions.

\section{Milliken's tree theorem in the arithmetical hierarchy}

When analyzing a mathematical problem from a computability-theoretic viewpoint, the first step usually consists in determining whether the computable instances of the problem admit arithmetical solutions, and if so, trying to identify the exact level in the arithmetical hierarchy where they stand. For example, Jockusch~\cite{Jockusch1972Ramseys} proved that every computable instance of Ramsey's theorem for $n$-tuples admits $\Pi^0_n$ solutions, and for each $n \geq 2$, constructed a computable instance of Ramsey's theorem for $n$-tuples and 2 colors with no $\Sigma^0_n$ solutions. Thus, the status of Ramsey's theorem with respect to the arithmetical hierarchy is fully determined. 

The case of Milliken's tree theorem is less clear. By \Cref{thm:milliken-arithmetic}, computable instances of Milliken's tree theorem admit arithmetical solutions. More precisely, every computable instance of Milliken's tree theorem for subtrees of height $n$ admits a $\Delta^0_{2n-1}$ solution. On the other hand, since Milliken's tree theorem generalizes Ramsey's theorem, for every $n \geq 2$, there exists a computable instance of Milliken's tree theorem for trees of height $n$ with no $\Sigma^0_n$ solutions. This leaves a gap between the lower and upper bound.

\begin{question}\label{quest:mttn-deltanp1}
Does every computable instance of Milliken's tree theorem for height $n$
admit a $\Delta^0_{n+1}$ solution?
\end{question}

The proof by Jockusch~\cite{Jockusch1972Ramseys} of the existence of a $\Pi^0_2$ solution for every computable instance of Ramsey's theorem for $n$-tuples is by an inductive argument based on the notion of prehomogeneous set. In particular, he proves that every PA degree relative to $\emptyset'$ is sufficient to compute a prehomogeneous set. Hirschfeldt and Jockusch~\cite[Theorem 2.1]{Hirschfeldt2016notions} actually proved a reversal, by constructing a computable instance of Ramsey's theorem for triples such that every prehomogeneous set is of PA degree relative to $\emptyset'$. This bound on prehomogeneous sets is sufficient to make increasing the level in the arithmetical hierarchy only by one when increasing the size of the colored tuples by one, by taking prehomogeneous sets of low degree over $\emptyset'$.

Similarly, the current upper bound of Milliken's tree theorem is proved using the corresponding notion of prehomogeneous tree, but \Cref{thm:milliken-prehomogeneous} yields only a $\Delta^0_3$ solution, which makes increase the level in the arithmetical hierarchy by 2 instead of 1 when coloring larger tuples. The following questions are still open:

\begin{question}
Given a computable instance of the Milliken's tree theorem for height $n$, does any PA degree relative to $\emptyset'$ compute a prehomogeneous infinite strong subtree? Is there always a prehomogeneous infinite strong subtree of low degree relative to $\emptyset'$?
\end{question}

A positive answer to either question would be sufficient to answer positively Question \ref{quest:mttn-deltanp1}.

\section{Larger degrees and cone avoidance}

Cone avoidance is a central notion in the computability-theoretic analysis of theorems. It is the main tool for separating a theorem from $\ACA_0$ over $\omega$-structures. It is in particular a desirable property to have, and given a statement which does not admit cone avoidance, one can ask whether there exists a natural weakening of it which admits it. The analysis of Ramsey's theorem gives a good example:  Jockusch~\cite{Jockusch1972Ramseys} constructed for every $n \geq 3$ a computable instance of Ramsey's theorem for $n$-tuples whose solutions compute the halting set. In particular, this shows that Ramsey's theorem for 3-tuples does not admit cone avoidance. On the other hand, Wang~\cite[Theorem 3.2]{Wang2014Some} proved that when weakening the notion of homogeneity in Ramsey's theorem by allowing a larger number of colors, then the resulting statement admits cone avoidance. In particular, he proved that $(\forall k)\RT{3}{k, 2}$ admits cone avoidance, where $(\forall k)\RT{n}{k, \ell}$ is the statement whose instances are colorings $f: [\omega]^n \to k$ for some $k$, and whose solutions are infinite sets $H$ such that $|f[H]| \leq \ell$. In general, Wang proved that for every $n$, and every $\ell$ sufficiently large with respect to $n$, the statement  $(\forall k)\RT{n}{k, \ell}$ admits cone avoidance. Cholak and Patey~\cite{Cholak2019Thin} computed the exact bound where this threshold phenomenon happens, which happens to be $(\forall k)\RT{n}{k, C_{n-1}}$, where $C_0, C_1, \dots$ is the Catalan sequence, starting with $1, 1, 2, 5, 14, 42, \dots$

Milliken's tree theorem behaves like Ramsey's theorem with many respects. Milliken's tree theorem for pairs admits cone avoidance, while there exists a computable instance of Milliken's tree theorem for trees of height 3 whose solutions compute the halting set. By a similar investigation, we proved in \Cref{subsect:thin-milliken} that $(\forall k)\PMT{3}{k,2}$ admits cone avoidance (\Cref{thm:pmtt3k2-cone-avoidance}), where $(\forall k)\PMT{n}{k,\ell}$ is the weakening of $(\forall k)\PMT{n}{}$ where $\ell$ colors are allowed in the solutions.

The proof of \Cref{thm:pmtt3k2-cone-avoidance} goes through the existence of a level-homo\-geneous strong subtree. Recall that a tree $T$ is level-homogeneous with respect to a coloring $f: \Subtree{n}{T}$ if strong subtrees with the same level function get assigned the same color. This notion reduces the problem of finding an infinite strong subtree monochromatic for $f$ to the problem of finding an infinite homogeneous set. Indeed, if $T$ is level-homogeneous with respect to $f$, the color depends only on the levels, hence becomes a coloring of $[\omega]^n$. The known counter-examples to cone avoidance of Milliken's tree theorem for trees of height at least 3 as all inherited from Ramsey's theorem by defining a coloring which depends only on the levels. We proved that the statement which to a finite coloring of $\Subtree{3}{T}$, associates an infinite level-homogeneous strong subtree, admits cone avoidance. This result goes towards the intuition that the strength of Milliken's tree theorem is mainly inherited from Ramsey's theorem. It is therefore natural to wonder whether the statement of the existence of a level-homogeneous infinite strong subtree admits cone avoidance, when considering colorings of finite subtrees of larger height. Since the proof from height $n$ to height $n+1$ is usually inductive, by first proving cone avoidance for height $n$, then strong cone avoidance for height $n$, and then only cone avoidance for height $n+1$, we wonder whether the statement of the existence of a level-homogeneous infinite strong subtree admits strong cone avoidance.

\begin{question}
Given two sets $C$ and $Z$ such that $C \not \leq_T Z$, and a finite sequence of $Z$-computable, $Z$-computably bounded, infinite trees with no leaves $T_0, \dots, T_{d-1}$, does every coloring $f: \Subtree{n}{T_0, \dots, T_{d-1}} \to k$ admit a level-homogeneous tuple $(S_0, \dots, S_{d-1}) \in \Subtree{\omega}{T_0, \dots, T_{d-1}}$ such that $C \not \leq_T Z \oplus S_0 \oplus \dots \oplus S_{d-1}$?
\end{question}

A positive answer to this question would enable to make it benefit from the computability-theoretic analysis for Ramsey's theorem, and in particular would imply that $(\forall k)\PMT{n}{k,C_{n-1}}$ admits cone avoidance.

\section{Comparing the statements for pairs in reverse mathematics}

Ramsey's theorem for pairs admits a special status with respect to full Ramsey's theorem in reverse mathematics, as it admits cone avoidance, while Ramsey's theorem for larger tuples is equivalent to $\ACA_0$ over~$\RCA_0$. This threshold phenomenon was also satisfied by the Chubb-Hirst-McNicholl tree theorem (see Dzhafarov and Patey~\cite{Dzhafarov2017Coloring}) whose statement for pairs admits cone avoidance, while is equivalent to $\ACA_0$ for larger tuples, or the Erd\"os-Rado theorem (see Chong, Liu, Liu and Yang~\cite{Chong2019Strengthc}).
We therefore naturally had a particular focus on the restriction of Milliken's tree theorem for trees of height 2, and on the applications of Milliken's tree theorem restricted to pairs.

As we can see in the proof of Devlin's theorem and the Rado graph theorem using Milliken's tree theorem, both Devlin's theorem for $n$-tuples and the Rado graph theorem for graphs of size $n$ involve applications of Milliken's tree theorem for strong subtrees of height $2n-1$. This is essentially due to \Cref{lem:coded-joyce-order-to-strong-subtree}. Informally, when representing rational numbers as strings, any coloring of a pairs of rationals induces a coloring of strong subtrees of height 3, by considering the tree whose first level is the length of their meet, the second level is the length of the shortest of the two strings representing the rationals, and the third level is the longest length. 

This yields two main questions, namely, (1) whether there exists another proof of these statements of size $n$ involving only applications of Milliken's tree theorem for trees of height $n$, and (2) whether these statements should be more considered as statement about pairs or about triples. The latter question is more informal, and depends on the aspects considered. 

One aspect separating Ramsey's theorem for pairs from larger tuples is the existence of cone avoiding solutions. With this respect, Devlin's theorem admits a computable instance whose solutions all compute the halting set, while the Rado graph theorem for graphs of size 2, the Erd\"os-Rado theorem and Milliken's tree theorem for pairs are all cone avoiding. This proves in particular that Milliken's tree theorem for pairs does not imply Devlin's theorem for pairs in $\RCA_0$, and answers the first question negatively for Devlin's theorem. Another aspect which could better capture the difference between statements about pairs and about larger tuples, is the position in the arithmetical hierarchy. As explained, Jockusch~\cite{Jockusch1972Ramseys} proved the existence of a computable instance of Ramsey's theorem for triples with no $\Sigma^0_3$ solution, while every computable instance of Ramsey's theorem for pairs admits a $\Pi^0_2$ solution. Here again, using this criterium, Devlin's theorem for pairs does not seem to be a statement about pairs. Indeed, by \Cref{cor:dt2-no-sigma3} in a computable instance of Devlin's theorem for pairs with no $\Sigma^0_3$ solution. The question for the Rado graph theorem for graphs of height 2 and for the Erdos-Rado theorem remains open:

\begin{question}
Is there a computable instance of the Rado graph theorem for graphs of size 2 with no $\Sigma^0_3$ solution? Same question for the Erd\"os-Rado theorem for pairs.
\end{question}

If the answer is yes, then this would answer negatively the corresponding part of the following question.

\begin{question}
  Does $\MT{2}{}$ imply $(\forall k)\RG^2_{k, 4}$ over $\RCA_0$? Same question for $\mathrm{ER}^2$.
\end{question}

Milliken's tree theorem for trees of height 2 is a natural generalization of Ramsey's theorem for pairs, and so is the Erd\"os-Rado theorem. By \Cref{thm:dt2-implies-rt2}, this is also the case of Devlin's theorem for pairs. It is however unknown whether the Rado graph theorem for graphs of height 2 also implies Ramsey's theorem for pairs. On the positive side, the Rado graph theorem for pairs implies a stable version of Ramsey's theorem for pairs (see \Cref{thm:rg284-implies-srt22}). Thus, by the decomposition of Ramsey's theorem for pairs in its stable version and the cohesiveness principle (see Cholak, Jockusch and Slaman~\cite{Cholak2001strength}, Section 7), the question can be rephrased as whether the Rado graph theorem for pairs implies the $\COH$ over~$\RCA_0$.

\begin{question}
Does $(\forall k)\RG^2_{k, 4}$ imply $\RT 22$ over $\RCA_0$?
Equivalently, does $(\forall k)\RG^2_{k, 4}$ imply $\COH$ over $\RCA_0$?
\end{question}

Devlin's theorem for pairs and the Erd\"os-Rado theorem are both statements about colorings of pairs of rationals. The former is symmetric, in that the nature of the solution does not depend on value of the color, and is thus arguably more natural than the Erd\"os-Rado theorem.
Since the statement $\DT{2}{<\infty,2}$ is somehow combinatorially optimal with respect to coloring of pairs of  dense linear orders with no endpoints, one could expect that it implies the Erd\"os-Rado theorem. This is actually the case by \Cref{th:dt242-implies-er2}: $\DT{2}{4,2}$ implies $\mathrm{ER}^2$ over $\RCA_0$. On the other hand, $\mathrm{ER}^2$ admits cone avoidance, while $\DT{2}{4,2}$ does not. This yields the following question: is there a natural statement which implies $\mathrm{ER}^2$ and does admit cone avoidance? The notion of naturality is kept informal. When increasing the number of colors allowed in the solutions of Devlin's theorem for pairs, the statement $\DT{2}{<\infty,4}$ is the first one admitting cone avoidance (\Cref{cor:jdt4-cone-avoidance}). This yields the following question:

\begin{question}
Does $\DT{2}{<\infty,4}$ imply $\mathrm{ER}^2$ over $\RCA_0$?
\end{question}

\noindent Part of the reason we might expect this to be true is because in \Cref{subsect:er-theorem}, we did actually prove cone avoidance of $\mathrm{ER}^2$ (first shown by Chong, Liu, Liu and Yang~\cite{Chong2019Strengthc}) using $\DT{2}{<\infty,4}$. However, our proof involved the existence of a generic set for a particular notion of forcing, and this may not belong to a given model of $\RCA_0$. Thus, it does not settle the question, but it makes an affirmative answer plausible.

%\appendix
%    Include appendix "chapters" here.
%\include{}

\backmatter
%    Bibliography styles amsplain or harvard are also acceptable.
\bibliographystyle{amsplain}
\bibliography{bibliography}
%    See note above about multiple indexes.
\printindex

\end{document}